%% file: Half-plan_boundary_BGk_model_final.tex
\documentclass[10pt]{amsart}
\usepackage[margin=1.0in]{geometry}
\usepackage{xcolor}
\usepackage{placeins}
\usepackage{hyperref}
\usepackage{graphicx, subfigure, caption2}
\usepackage{graphicx}
\usepackage{kotex}
\usepackage{amsmath,amsfonts,amssymb,amsthm, physics}
\usepackage{mathtools}
\usepackage{commath}
\usepackage{colortbl}
\usepackage{tikz}

\usetikzlibrary{arrows.meta,calc,intersections}

\newtheorem{theorem}{Theorem}[section]
\newtheorem{lemma}[theorem]{Lemma}
\newtheorem{proposition}[theorem]{Proposition}
\newtheorem{corollary}[theorem]{Corollary}
\theoremstyle{definition}
\newtheorem{definition}[theorem]{Definition}

\theoremstyle{remark}
\newtheorem{remark}[theorem]{Remark}

\numberwithin{equation}{section}

\DeclareMathOperator*{\esssup}{ess\,sup}

\begin{document}
\title[Instantaneous Exponential Ill-posedness of the BGK Model]{Instantaneous Exponential Ill-posedness of the BGK Model in a Half-Space with Inflow Boundary Conditions}
        
\author{Donghyun Lee}
\address{Department of Mathematics, Pohang University of Science and Technology, Pohang, South Korea}
\email{donglee@postech.ac.kr}

\author{Sungbin Park}
\address{Basic Science Research Institute(BRSI), Pohang University of Science and Technology, Pohang, South Korea}
\email{parksb2942@postech.ac.kr}

\author{Sung-Jun Son}
\address{Center for Mathematical Machine Learning and its Applications(CM2LA), Pohang University of Science and Technology, Pohang, South Korea}
\email{sungjun129@postech.ac.kr}

\author{Seok-Bae Yun}
\address{Department of Mathematics, Sungkyunkwan University, Suwon, South Korea}
\email{sbyun01@skku.edu}
    
\begin{abstract}
    In this paper, we study boundary-driven behavior of the BGK model and the Boltzmann equation in a half-space with vacuum initial data. For any $1\leq\beta\leq2$ and $\alpha>0$, we construct a unique local-in-time BGK solution whose inflow boundary data have an arbitrarily small norm with the tangential exponential weight $e^{\alpha|v_h|^\beta}$, but whose norm with weight $e^{\alpha'|v_h|^\beta}$ is infinite for every $0<\alpha'\leq\alpha$ and every positive time of existence, where $v=(v_h,v_z)\in\mathbb{R}^n$ and $v_h\in\mathbb{R}^{n-1}$ denotes the tangential velocity. Thus, this BGK solution instantaneously leaves the exponentially weighted spaces associated with the inflow data, even when the exponential weight is weakened. The construction relies on sharp estimates for the density, bulk velocity, and temperature on a singular scale together with a fixed-point argument in a time-dependent anisotropic space. For comparison, we establish local existence and uniqueness for the cutoff Boltzmann equation under isotropic exponential weights $e^{\alpha|v|^\beta}$. In the Appendix, we further clarify the role of these different weights by showing that the macroscopic stability required for the BGK construction can fail even for Maxwellian inflow data and that it can be recovered by suitably weakening the normal velocity decay of the inflow data. These observations highlight a structural difference between local Maxwellian relaxation and the Boltzmann collision operator in the presence of boundary inflow and vacuum initial data. This work extends the whole-space analysis of D. Lee, S. Park, and S.-B. Yun \cite{LPY2024} to the half-space setting.
\end{abstract}
\maketitle
\tableofcontents
        
\section{Introduction}
\subsection{Kinetic equations: Boltzmann and BGK models}
Kinetic equations describe gas dynamics from a statistical perspective and have been extensively studied. The Boltzmann equation is one of the most important equations in kinetic theory, describing the evolution of the particle distribution through binary collisions. It consists of a transport term and a collision operator representing particle interactions. More precisely, the Boltzmann equation takes the form
\begin{align}\label{eq:Boltzmann}
    \partial_t f + v \cdot \nabla_x f = \int_{\mathbb{R}^n} \int_{\mathbb{S}^{n-1}} 
    B(v-v_*,\omega) \left[f(v')f(v'_*) - f(v)f(v_*)\right] \, d\omega \, dv_*, \qquad n \ge 2,
\end{align}
where $f(t,x,v)$ is the velocity distribution function at time $t\geq 0$, position $x\in \Omega$, and velocity $v\in \mathbb{R}^n $, and the right-hand side of the equation represents the Boltzmann collision operator, where $(v', v_*')$ denote the post-collisional velocities determined by the collision law
\begin{align*}
    &v' = v - \langle v-v_*,\omega\rangle\omega,\quad v_*' = v_* + \langle v-v_*,\omega\rangle\omega,\quad \omega\in \mathbb{S}^{n-1}.
\end{align*}
We usually define the right-hand side of \eqref{eq:Boltzmann} by
\begin{align*}
    Q(f,f):= \int_{\mathbb{R}^n} \int_{\mathbb{S}^{n-1}} 
    B(v-v_*,\omega) \left[f(v')f(v'_*) - f(v)f(v_*)\right] \, d\omega \, dv_*.
\end{align*}
We can also write the $Q$ operator as the bilinear form
\begin{align*}
    Q(f_1,f_2):= \int_{\mathbb{R}^n} \int_{\mathbb{S}^{n-1}} 
    B(v-v_*,\omega) \left[f_1(v')f_2(v'_*) - f_1(v)f_2(v_*)\right] \, d\omega \, dv_*.
\end{align*}
$B$ is the collision kernel characterizing particle interactions. In this paper, we are interested in the collision kernel of the form
\begin{align}\label{eq:B_condition}
    B(v-v_*,\omega) = |v-v_*|^\kappa b(\cos\theta),\quad -n<\kappa\leq 1,\quad 0\leq b(\cos\theta)\leq C|\cos\theta|,\quad \cos\theta = \frac{v-v_*}{|v-v_*|}\cdot\omega.
\end{align}
When $n=3$ and $\kappa=1$, we can consider the hard-sphere collision kernel.

The important property of a solution of the Boltzmann equation is that the total mass, momentum, and energy are preserved in the following sense: for a suitable boundary-free space $\Omega$ such as $\mathbb{T}^n$ or $\mathbb{R}^n$, it formally satisfies
\begin{align}\label{eq:conservation_macro}
    \int_{\Omega\times \mathbb{R}^n}
    \left(\begin{array}{c}1\\ v\\|v|^2\end{array}\right)f(t,x,v) 
    \,dx dv=\int_{\Omega\times \mathbb{R}^n}
    \left(\begin{array}{c}1\\ v\\|v|^2\end{array}\right)f_0(x,v) 
    \,dx dv.
\end{align}
Another important property is the entropy dissipation property given by
\begin{align}\label{eq:entropy_ineq}
    \frac{d}{dt}\int_{\Omega\times\mathbb{R}^n}f\ln f\, dxdv \leq 0.
\end{align}
For the derivation of these properties, we refer to the well-known review paper \cite{Villani2002}.

Although the Boltzmann equation provides an accurate physical description, the collision operator is highly nonlinear and involves a complicated high-dimensional integral structure. This makes both theoretical analysis and numerical computation extremely challenging. To overcome these difficulties, several simplified models have been proposed that retain key physical properties such as conservation laws. 
One of the most widely used models is the BGK model, introduced by Bhatnagar, Gross, and Krook \cite{BGK1954}. In this model, the collision operator is replaced by a relaxation term driving the distribution toward a local Maxwellian equilibrium. The BGK model is given by
\begin{align*}
    \partial_t f + v \cdot \nabla_x f = \frac{1}{\tau} \big( \mathcal{M}(f) - f \big),
\end{align*}
where $\tau > 0$ is the relaxation time. In this paper, we choose $\tau = 1$.

The local Maxwellian $\mathcal{M}(f)$ is determined by the macroscopic moments of $f$ and is defined as
\begin{equation*}
    \mathcal{M}(f)(t,x,v) = \frac{\rho}{(2\pi T)^{n/2}} \exp\!\left(-\frac{|v-u|^2}{2T}\right),
\end{equation*}
where $\rho$, $u$, and $T$ denote the density, bulk velocity, and temperature, respectively.
These quantities are obtained from the moments of the distribution function $f$ as
\begin{align*}
    \begin{split}
        \rho(t,x) &= \int_{\mathbb{R}^n} f(t,x,v)\,dv,\\
        \rho(t,x) u (t,x) &= \int_{\mathbb{R}^n} f(t,x,v)v\,dv,\\
        n \rho(t,x) T (t,x) &= \int_{\mathbb{R}^n} |v-u(t,x)|^2 f(t,x,v)\,dv.\\
    \end{split}
\end{align*}
A direct computation shows that a solution of the BGK model also satisfies \eqref{eq:conservation_macro} and \eqref{eq:entropy_ineq} by the structure of the relaxation term. Under the same boundary-free assumptions, it implies that any solution of the BGK model also satisfies the conservation of total mass, momentum, and energy, as well as the entropy dissipation property.

Due to its relatively simple structure while keeping the key kinetic features, the BGK model has been widely studied in both mathematical analysis and numerical simulations. In particular, it provides a useful framework for understanding the behavior of the Boltzmann equation, including relaxation to equilibrium and hydrodynamic limits.

\subsection{Problem Motivation}

A fundamental issue in the mathematical analysis of the BGK model is the control of the local Maxwellian in suitable function spaces. A classical estimate obtained by Perthame and Pulvirenti \cite{PP1993} shows that the local Maxwellian satisfies a polynomial weighted bound of the form
\begin{align}\label{est:PP}
    \|(1+|v|^2)^q \mathcal{M}(f)\|_{L^\infty_v}\leq C \|(1+|v|^2)^q f\|_{L^\infty_v}
\end{align}
in some range of $q\geq 0$. This estimate has played a crucial role in many well-posedness results for the BGK model. However, since the Maxwellian is itself an exponential function in the velocity variable, this naturally raises the question of whether the above polynomial estimate can be improved to an exponential one. More precisely, one may ask whether an estimate of the form
\begin{align*}
    \|e^{\alpha |v|^\beta} \mathcal{M}(f)\|_{L^\infty_v} \leq C\|e^{\alpha |v|^\beta} f\|_{L^\infty_v}
\end{align*}
holds for some $\alpha>0$ and $\beta>0$. The question of whether such an exponential estimate holds for the initial value problem was first raised in \cite{LPY2024}, where they showed that the estimate fails in general. In particular, for any given $\alpha> 0$, $\beta>0$, and $1\leq p\leq \infty$, the authors constructed a sequence of functions $f_i(v)$ such that
\begin{align*}
    \sup_i \|f_i e^{\alpha |v|^\beta} \|_{L^p_v} < \infty \quad \text{while} \quad \lim_{i\rightarrow \infty}\| \mathcal{M}(f_i) e^{\alpha'|v|^\beta} \|_{L^p_v} = \infty
\end{align*}
for any $0<\alpha'\leq \alpha$. From this observation, the authors further established an instantaneous ill-posedness result for spatially inhomogeneous BGK solutions: for any $n\geq 1$, $\alpha>0$, $0<\beta\leq 2$, $1\leq p,q\leq \infty$, and $\epsilon>0$, there exists initial data $f_0(x,v)$ in the whole space such that
\begin{align*}
    &\|f_0(x,v) e^{\alpha |v|^\beta} \|_{L^p_v(\mathbb{R}^n, L^q_x(\mathbb{R}^n))} < \epsilon,\quad\text{but}\quad \| f(t,x,v) e^{\alpha' |v|^\beta} \|_{L^p_v(\mathbb{R}^n, L^q_x(\mathbb{R}^n))} =\infty\\
    &\left(\text{resp. }\|f_0(x,v) e^{\alpha |v|^\beta} \|_{L^q_x(\mathbb{R}^n, L^p_v(\mathbb{R}^n))} < \epsilon,\quad\text{but}\quad \| f(t,x,v) e^{\alpha' |v|^\beta} \|_{L^q_x(\mathbb{R}^n, L^p_v(\mathbb{R}^n))} =\infty\right)
\end{align*}
for any $0<\alpha'\leq\alpha$ and $t>0$.

Motivated by this observation, we further investigate mechanisms that may produce singular behavior in kinetic models. In particular, the role of boundary effects in the formation of singularities remains largely unexplored, although boundary interactions often generate strong non-equilibrium phenomena. In this paper, we study the BGK model in the half-space with vacuum initial data and an inflow boundary condition and show that even for small boundary data, the exponentially weighted norm of the solution becomes infinite instantaneously.

We consider the spatial domain $\mathbb{R}^n_+$ for $n\geq 2$, so the boundary of the domain is given by $\partial \mathbb{R}^n_+$. By the prescribed inflow boundary condition, particles enter through the boundary. Since the interior is initially vacuum, the distribution within the domain is determined entirely by the boundary inflow. Under free transport, particles propagate along characteristic trajectories.

For a phase point
\begin{align*}
    (x,v)\in \mathbb{R}^n_+\times\mathbb{R}^n,
\end{align*}
we define the backward exit time by
\begin{align*}
    t_b(x,v):=\sup\left\{\tau>0:x-\tau v\in\mathbb{R}^n_+\right\}.
\end{align*}
Whenever $t_b(x,v)\neq \infty$ and $t\geq t_b(x,v)$, the backward characteristic issued from $(t,x,v)$ reaches the boundary at
\begin{align*}
    x_b:=x-vt_b(x,v)
\end{align*}
at time $t-t_b(x,v)$. Thus, if we only consider the free transport effect, the value of the distribution function at \((t,x,v)\) is determined by the inflow boundary data at $(t-t_b, x-vt_b, v)$. This geometric configuration
is illustrated in Figure~\ref{fig:boundary_geometry}.

\input{boundary_geometry}

This boundary-driven setting exhibits several fundamental differences from the classical boundary-free setting. In standard kinetic problems, the evolution is primarily determined by the initial data. In the present problem, however, the evolution is driven entirely by the boundary inflow due to the vacuum initial data. This setting requires a substantial modification of the functional framework. In particular, an anisotropic framework adapted to the normal and tangential velocity structures is better suited to this problem; the usual isotropic velocity-weighted spaces can cause some problems. Accordingly, we redesign the local existence framework for the BGK model to accommodate the present boundary-driven setting.

In this paper, we consider the mild solution of the BGK and Boltzmann equations. For given inflow boundary data $g(t,x_h,v)\geq 0$ defined for $x_h\in\partial \mathbb{R}^n_+$ and $v_z>0$ and vacuum initial data, a function $f(t,x,v)\geq 0$ is called a mild solution of the BGK model for $0\leq t\leq t_0$ if it satisfies
\begin{align}\label{eq:BGK_sol_def}
    f(t,x,v) = e^{-t_b}g(t-t_b, x - vt_b, v)\mathbf{1}_{t\geq t_b} + \int_0^t e^{-(t - \tau)} \mathcal{M}(f)(\tau, x - v(t - \tau), v)\mathbf{1}_{z-v_z(t-\tau)> 0}\,d\tau
\end{align}
for a.e. $0\leq t\leq t_0$, $x\in\mathbb{R}^n_+$, and $v\in\mathbb{R}^n$. Note that $t_b = \frac{z}{v_z}$ for $v_z>0$ when $x\in \mathbb{R}^n_+$, so $t\geq t_b$ is equivalent to $v_z\geq \frac{z}{t}$. We also remark that the first term $e^{-t_b}g(t-t_b, x - vt_b, v)\mathbf{1}_{t\geq t_b}$ plays a role of the initial data $f_0(x,v)$ in the usual solution of the BGK model.

Likewise, $f(t,x,v)$ is a solution of the Boltzmann equation for $0\leq t\leq t_0$ if it satisfies
\begin{align}\label{eq:Boltzmann_sol_def}
    \begin{split}
        f(t,x,v) &= e^{-\int_{t-t_b}^t L(f)(\tau, x-v(t-\tau), v)\,d\tau}g(t-t_b, x - vt_b, v)\mathbf{1}_{t\geq t_b}\\
        & \quad + \int_0^t e^{-\int_\tau^t L(f)(s, x-v(t-s), v)\,ds} Q^+(f,f)(\tau, x - v(t - \tau), v)\mathbf{1}_{z-v_z(t-\tau)> 0}\,d\tau,
    \end{split}
\end{align} 
for a.e. $0\leq t\leq t_0$, $x\in\mathbb{R}^n_+$, and $v\in\mathbb{R}^n$. Here, we define the $Q^+(f_1,f_2)$ and $L(f)$ operators by
\begin{align*}
    Q^+(f_1,f_2) &:= \int_{\mathbb{R}^n} \int_{\mathbb{S}^{n-1}} B(v-v_*,\omega) f_1(v')f_2(v'_*) \, d\omega \, dv_*,\\
    L(f) &:= \int_{\mathbb{R}^n} \int_{\mathbb{S}^{n-1}} B(v-v_*,\omega) f(v_*) \, d\omega \, dv_*.
\end{align*}
The original $Q$ operator is given by $Q(f,f) = Q^+(f,f) - fL(f)$.

We define a function space for BGK solutions $\Omega_{\gamma, t_0}$ as follows.
\begin{definition}\label{def:BGK_sol_space}
For $\gamma>0$, $t_0>0$, and given inflow boundary data $g(t,x_h,v)$, $\Omega_{\gamma, t_0}$ is the collection of functions $f(t,x,v)\geq 0$ for $0\leq t\leq t_0$ such that
\begin{enumerate}
    \item it satisfies
    \begin{align*}
        \esssup_{0< t\leq t_0}\esssup_{x\in\mathbb{R}^n_+}\int_{v\in \mathbb{R}^n}\left(1+\left(|x_h|+\frac{z}{t}\right)^{2\gamma} + \left(v-\frac{z}{t}\hat{z}\right)^2\right)f(t,x,v)\,dv<\infty
    \end{align*}
    and $f(0,x,v) = 0$.
    \item We define the macroscopic fields by
    \begin{align*}
        \rho(t,x)(1, u(t,x), nT(t,x)) &= \int_{\mathbb{R}^n} (1, v, |v-u(t,x)|^2) f(t,x,v)\,dv,\\
        \rho_g(t,x)(1, u_g(t,x), nT_g(t,x)) &= \int_{\mathbb{R}^n} (1, v, |v-u_g(t,x)|^2) e^{-t_b}g(t-t_b, x - vt_b, v)\mathbf{1}_{t\geq t_b}\,dv,
    \end{align*}
    where $(\rho_g, u_g, T_g)$ are the macroscopic fields generated by the damped free transport of the inflow boundary data $e^{-t_b}g(t-t_b, x - vt_b, v)\mathbf{1}_{t\geq t_b}$. Then, there exists a constant $C_u>0$ such that
    \begin{align}\label{eq:X_T_A3}
        \begin{split}
            &|\rho-\rho_g|(t,x)\leq \frac{1}{2}\rho_g(t,x),\\
            &\left|u-\frac{z}{t}\hat{z}\right|(t,x)\leq C_u\sqrt{T_g(t,x)},\\
            &|T-T_g|(t,x)\leq \frac{1}{2}T_g(t,x)
        \end{split}
    \end{align}
    for almost every $0<t\leq t_0$ and $x\in\mathbb{R}^n_+$.
\end{enumerate}
\end{definition}
In the second condition, we do not include $t=0$ since $\rho(0,x) = \rho_g(0,x) = 0$, so $u_g(0,x)$, $T_g(0,x)$, $u(0,x)$ and $T(0,x)$ are not well-defined. To remedy this issue, one can consider $\rho u$ and $\rho T$ instead of $u$ and $T$. In fact, if $f$ satisfies \eqref{eq:BGK_sol_def}, then it should satisfy $f(0,x,v) = 0$.

We explain the conditions in Definition \ref{def:BGK_sol_space}. Because of the boundary effect $\mathbf{1}_{t\geq t_b} = \mathbf{1}_{v_z\geq \frac{z}{t}}$, we can guess that the particles will concentrate near $v_z = \frac{z}{t}$, so we need to account for this effect on the weight and the computation in the macroscopic fields; this is the reason why we use $|x_h|+\frac{z}{t}$ and $v - \frac{z}{t}\hat{z}$ as groups. The first condition is to choose a suitable function space where contraction mapping can be applied; see Section \ref{sec:sol_construction} and Proposition \ref{prop:maxwell} for this. The second condition is to ensure that the macroscopic fields follow the behavior of $(\rho_g, u_g, T_g)$, and to make sure that $\left|u-\frac{z}{t}\hat{z}\right|(t,x)$ is bounded by a constant multiple of $\sqrt{T}(t,x)$ for $0<t\leq t_0$. 

The second condition is a stability condition for the macroscopic fields. When we consider the pure transport equation, the macroscopic fields should be $(\rho_g, u_g, T_g)(t,x)$ (if we ignore the $e^{-t_b}$ term for small $t$), and the second condition ensures that the macroscopic fields are stable in the presence of the relaxation term $\mathcal{M}(f) - f$. When the spatial domain is $\mathbb{R}^n$ or $\mathbb{T}^n$, we can verify that the solutions of the BGK model constructed in the paper by Perthame and Pulvirenti \cite{PP1993} satisfy a version of this condition. For this case, we define the macroscopic fields generated by the damped free transport of the initial data $e^{-t} f_0(x - vt, v)$ by
\begin{align*}
    \rho_0(t,x)(1, u_0(t,x), nT_0(t,x)) &= \int_{\mathbb{R}^n} (1, v, |v-u_0(t,x)|^2) e^{-t} f_0(x - vt, v)\,dv.
\end{align*}
We claim that the macroscopic fields of the solution satisfy
\begin{align}\label{eq:macro_cond_0}
    \begin{split}
        &|\rho-\rho_0|(t,x)\leq \frac{1}{2}\rho_0(t,x),\\
        &|u|(t,x)\leq C_u\sqrt{T_0(t,x)},\\
        &|T-T_0|(t,x)\leq \frac{1}{2}T_0(t,x)
    \end{split}
\end{align}
for a constant $C_u>0$ and $0\leq t\leq t_0$ for some $t_0>0$. Now, we claim that the solutions of the BGK model in  \cite{PP1993} satisfy the condition for some $t_0$ in an \textit{a priori} sense. To prove it, as in \cite{PP1993}, we assume that the initial data $f_0(x,v)$ satisfies
\begin{align*}
    \inf_{0\leq t\leq 1}\inf_x \int_{\mathbb{R}^n} e^{-t}f_0(x-vt,v)\,dv>C_0\quad \text{and}\quad \|(1+|v|^q)f_0(x,v)\|_{L^\infty_{x,v}}<\infty
\end{align*}
for $C_0>0$ and some $q>n+2$. By the standard arguments using \eqref{est:PP} in \cite{PP1993} for example, we have
\begin{align*}
    \|(1+|v|^q)\mathcal{M}(f)(t,x,v)\|_{L^\infty_{x,v}}&\leq C\|(1+|v|^q)f(t,x,v)\|_{L^\infty_{x,v}},\\
    \|(1+|v|^q)f(t,x,v)\|_{L^\infty_{x,v}}&\leq \|(1+|v|^q)f_0(x,v)\|_{L^\infty_{x,v}} e^{Ct}
\end{align*}
for some constant $C$ for $0\leq t\leq 1$. Using these estimates, we can derive
\begin{align*}
    |\rho-\rho_0|(t,x),|u-u_0|(t,x), |T-T_0|(t,x)\leq O(t)
\end{align*}
for $0\leq t\leq 1$. Therefore, we can choose sufficiently small $t_0>0$ depending on $\rho_0, u_0, T_0$, $C_0$, and $\|(1+|v|^q)f_0(x,v)\|_{L^\infty_{x,v}}$ such that the macroscopic fields satisfy the first and third lines of \eqref{eq:macro_cond_0} for $0\leq t\leq t_0$.

To prove the second line, we again recall the inequality in \cite{PP1993}, which states that
\begin{align*}
    \sup_x \frac{\rho|u|^{2n}}{[(T+u^2)T]^{n/2}}\leq C_n \|(1+|v|^n)f(t,x,v)\|_{L^\infty_{x,v}}.
\end{align*}
If $|u|(t,x)\geq \sqrt{T}(t,x)$, then
\begin{align*}
    \frac{|u|^n}{T^{n/2}}&\leq \frac{\rho}{C_0/2} \frac{|u|^n}{T^{n/2}}\left(2^{n/2}\frac{|u|^n}{(T+u^2)^{n/2}}\right)= \frac{2^{1+n/2}}{C_0} \frac{\rho|u|^{2n}}{[(T+u^2)T]^{n/2}}\\
    &\leq \frac{2^{1+n/2}C_n}{C_0}\|(1+|v|^n)f(t,x,v)\|_{L^\infty_{x,v}}.
\end{align*}
Since $T(t,x)\leq \frac{3}{2}T_0(t,x)$, the solution $f$ satisfies the second line of \eqref{eq:macro_cond_0} for $0\leq t\leq t_0$ when
\begin{align*}
    C_u = \sqrt{\frac{3}{2}}\max\left\{1, \left(\frac{2^{1+n/2}C_n}{C_0}\|(1+|v|^n)f(t,x,v)\|_{L^\infty_{t,x,v}}\right)^{1/n}\right\}.
\end{align*}

Unfortunately, the integral of $e^{-t_b} g(t-t_b, x-vt_b, v)\mathbf{1}_{v_z\geq \frac{z}{t}}$ tends to zero as $\frac{z}{t}\rightarrow \infty$. Hence, $\rho_g(t,x)$ does not have a uniform lower bound with respect to $z$ for $t\simeq 0$. Therefore, even if $f$ satisfies the polynomial weighted original inequality of Perthame and Pulvirenti, it may fail to meet the second condition in Definition \ref{def:BGK_sol_space} due to the lack of a lower bound on the density. Proposition \ref{prop:BGK_contradiction} in the Appendix provides such an example even when the boundary data is given by a pure Maxwellian function. For this reason, we explicitly include the bounds on the macroscopic fields $(\rho,u,T)$ in the definition of the solution space. These bounds enable us to use some standard techniques in the BGK model.

The function space for the solution of the Boltzmann equation is simpler than the BGK case, so it will be specified in the main theorem.

\subsection{Notation}Throughout this paper, we use the following notations:

\begin{itemize}
    \item If not stated otherwise, all constants, usually denoted by $C$, are generic and may vary from line to line. 
    The notation $C_\eta$ indicates that the constant depends on $\eta$, not necessarily exclusively.
    
    \item For a vector $a \in \mathbb{R}^n$, we write
    \begin{align*}
        a := (a_h, a_z),
    \end{align*}
    where $a_h \in \mathbb{R}^{n-1}$ denotes the tangential component and $a_z \in \mathbb{R}$ denotes the normal component. If $a$ is a spatial variable $x$, we write $x = (x_h, z)$.
    Accordingly, we define
    \begin{align*}
        |a| := \text{the standard Euclidean norm in } \mathbb{R}^n.
    \end{align*}
    For this notation, $\mathbb{R}^n_+ := \{z>0\}$ and $\partial \mathbb{R}^n_+ := \{z=0\}$.
    
    \item We define the Japanese bracket
    \begin{align*}
        \langle v \rangle := (1 + |v|^2)^{1/2}, \qquad \langle v_h \rangle := (1 + |v_h|^2)^{1/2}.
    \end{align*}
    We define the following weight functions:
    \begin{align*}
        w^h_{\alpha,\beta,\delta}(v) &:= \langle v\rangle^\delta \exp\big( \alpha |v_h|^\beta \big),\\
        w_{\alpha,\beta,\delta}(v) &:= \langle v\rangle^\delta \exp\big( \alpha |v|^\beta \big),
    \end{align*}
    where $\alpha > 0$, $\beta > 0$, and $\delta \in \mathbb{R}$. The $w^h_{\alpha,\beta,\delta}(v)$ weight will be used for the analysis of the BGK model.
    
    \item For a function $f(x,y)$ defined on a product set $\Omega_1\times \Omega_2$, we define the iterated $L^p$ norm by
    \begin{align*}
        \|f\|_{L^p_x(\Omega_1, L^q_y(\Omega_2))}:=\left(\int_{\Omega_1} \left(\int_{\Omega_2} |f(x,y)|^q \,dy\right)^{p/q}\,dx\right)^{1/p}
    \end{align*}
    with the usual modifications when $p=\infty$ or $q=\infty$. If $p=q$, then we usually abbreviate the notation by $L^p_{x,y}(\Omega_1\times \Omega_2)$. If there is no confusion about the integral set, we will omit the integration set. Furthermore, we use the notation $L^\infty_{x,y}(\phi(x,y))$ or $L^p_xL^q_y(\phi(x,y))$ for a weight function $\phi(x,y)$ to mean that $f\in L^p_xL^q_y(\phi(x,y))$ if and only if
    \begin{align*}
        \phi(x,y)f(x,y)\in L^p_x(\Omega_1, L^q_y(\Omega_2)).
    \end{align*}
\end{itemize}

\subsection{Main results} We are now ready to state our main theorem. The following theorem shows that even arbitrarily small inflow boundary data can lead to an instantaneous blow-up of the solution in the exponentially weighted norms for the BGK solution. This demonstrates a strong exponential-type ill-posedness phenomenon driven purely by the boundary.

\begin{theorem}\label{thm:main_BGK}
    Let $n\geq 2$, $1\leq p,q\leq \infty$, $\alpha>0$, $1\leq \beta\leq 2$, $\delta\geq 0$, and assume vacuum initial data $f_0(x,v) = 0$. For any $\epsilon>0$, there exists inflow boundary data $g(t,x_h,v)$ such that
    \begin{align}\label{3_thm:main_BGK_1_cond}
        \begin{split}
            &\sup_{0\leq t<\infty}\left\|w^h_{\alpha,\beta,\delta}e^{-t_b}g(t-t_b,x_h-v_ht_b,v)\mathbf{1}_{t\geq t_b}\right\|_{L_v^p(\mathbb{R}^n, L_x^q(\mathbb{R}^n_+))}<\epsilon,\\
            &\left(\text{resp. }\sup_{0\leq t<\infty} \left\|w^h_{\alpha,\beta,\delta}e^{-t_b}g(t-t_b,x_h-v_ht_b,v)\mathbf{1}_{t\geq t_b}\right\|_{L_x^q(\mathbb{R}^n_+, L_v^p(\mathbb{R}^n))}<\epsilon,\right)
        \end{split}
    \end{align}
    but there exists a unique solution of the BGK model $f(t,x,v)$ in $\Omega_{\frac{1}{1+\beta}, t_0}$ for some $t_0$ under the boundary data $g$ satisfying
    \begin{align}\label{3_thm:main_BGK_1}
        \begin{split}
            &\|w^h_{\alpha',\beta,\delta}f(t,x,v)\|_{L_v^p(\mathbb{R}^n, L_x^q(\mathbb{R}^n_+))}=\infty\\
            &\left(\text{resp. }\|w^h_{\alpha',\beta,\delta}f(t,x,v)\|_{L_x^q(\mathbb{R}^n_+, L_v^p(\mathbb{R}^n))}=\infty\right)
        \end{split}
    \end{align}
    for any $0<\alpha'\leq \alpha$ and $0<t\leq t_0$. Furthermore, for any solution $\tilde{f}(t,x,v)$ in $t\in [0, t^*_0]$ with $t^*_0>t_0$, if $\tilde{f}(t,x,v) = f(t,x,v)$ for $0\leq t\leq t_0$, it still satisfies \eqref{3_thm:main_BGK_1} for $0<t\leq t^*_0$. Here, $\tilde{f}(t,x,v)$ does not need to be in $\Omega_{\frac{1}{1+\beta}, t_0^*}$.
\end{theorem}

In contrast to the BGK model case, the following theorem shows that the solution of the Boltzmann equation remains well-behaved near $t=0$ in the exponentially weighted norms in the same half-space inflow setting.
\begin{theorem}\label{thm:main_Boltzmann}
    For the collision kernel given by \eqref{eq:B_condition}, let $n\geq 2$, $\alpha>0$, $\max\left\{0, \frac{2}{n-1}\kappa, \frac{2}{3}(1+\kappa)\right\}<\beta\leq 2$, and $\delta>n+1$. If $n=2$, we additionally assume that $1+\kappa\leq \beta$. Given any $1\leq p,q\leq \infty$, the following holds. For the vacuum initial data $f_0(x,v) = 0$ and an inflow boundary datum $g(t,x_h,v)$ for $v_z>0$ satisfying
    \begin{align}\label{3_thm:main_Boltzmann_1_cond}
        \begin{split}
            w_{\alpha,\beta,\delta}g(t, x_h, v)&\in L^\infty_{t,x_h,v}\left([0, t^*_0]\times \partial \mathbb{R}^n_+\times \mathbb{R}^n\right),\\
            v_z^{1/q}w_{\alpha,\beta,\delta}g(t, x_h, v)&\in L^p_v\left(\mathbb{R}^n, L^q_{t, x_h}([0, t^*_0]\times \partial \mathbb{R}^n_+)\right)
        \end{split}
    \end{align}
    for some $t_0^*>0$, there exists a unique solution of the Boltzmann equation satisfying
    \begin{align*}
        w_{\alpha,\beta,\delta}f(t,x,v)\in L^\infty_{loc, t}\left([0, t_0), L^\infty_{x,v}(\mathbb{R}^n_+\times \mathbb{R}^n)\cap L^p_v\left(\mathbb{R}^n, L^q_x(\mathbb{R}^n_+)\right)\right)
    \end{align*}
    for some $0<t_0\leq t^*_0$. The $t_0$ depends on $\alpha$, $\beta$, $\delta$, $n$, $\kappa$, and $\|w_{\alpha,\beta,\delta}g(t, x_h, v)\|_{L^\infty_{t,x_h,v}\left([0, t^*_0]\times \partial \mathbb{R}^n_+\times \mathbb{R}^n\right)}$.
\end{theorem}
\begin{remark}
    For a function $\phi(t,x_h,v)$ for $x_h\in\partial \mathbb{R}^n_+$ and $\phi = 0$ for $v_z\leq 0$, we have
    \begin{align*}
        \|\phi(t-t_b, x_h-v_ht_b, v)\mathbf{1}_{t\geq t_b}\|_{L^p_v(\mathbb{R}^n, L^q_x(\mathbb{R}^n_+))}=\|v_z^{1/q} \phi(\tau,x_h,v)\|_{L^p_v(\mathbb{R}^n, L^q_{\tau,x_h}([0,t]\times \partial \mathbb{R}^n_+))}
    \end{align*}
    for any $1\leq p,q\leq \infty$. Therefore, the norm in \eqref{3_thm:main_Boltzmann_1_cond} corresponds to the $L_v^p(\mathbb{R}^n, L_x^q(\mathbb{R}^n_+))$ norm in \eqref{3_thm:main_BGK_1_cond} except the weight function $w^h_{\alpha,\beta,\delta}$ and $w_{\alpha,\beta,\delta}$.
\end{remark}
\begin{remark}
    For technical reasons, we consider the $L^p_v\left(\mathbb{R}^n, L^q_x(\mathbb{R}^n_+)\right)$ norm for the Boltzmann equation, which corresponds to the $L^p_vL^q_x$ case of Theorem \ref{thm:main_BGK}. If we use Minkowski's integral inequality for $q\geq p$, we can check $w_{\alpha,\beta,\delta}f(t,x,v)\in L^\infty_{loc, t}\left([0, t_0), L^q_x\left(\mathbb{R}^n_+, L^p_v(\mathbb{R}^n)\right)\right)$ under the assumption \eqref{3_thm:main_Boltzmann_1_cond}.
\end{remark}
Theorem~\ref{thm:main_BGK} demonstrates that an instantaneous exponential norm blow-up of the solution can occur even though the inflow boundary data is well-defined in the exponential norm at all times. In contrast to Theorem~\ref{thm:main_BGK}, the Boltzmann solution remains locally bounded in the corresponding isotropically weighted spaces by Theorem~\ref{thm:main_Boltzmann}. We should note that the exponential weight differs between models: we use the $w^h_{\alpha,\beta,\delta}$ weight for the BGK model, whereas the $w_{\alpha,\beta,\delta}$ weight is used for the Boltzmann equation. We will discuss this difference at the end of this section and in the Appendix.

\subsection{Literature} The Boltzmann equation is a fundamental equation in the kinetic theory for describing the dynamics of dilute gases, and its mathematical analysis has been widely studied in the literature. Classical works on the existence and qualitative theory of solutions include 
\cite{Carleman1933,Grad1958,Ukai1974,KS1978,IS1984}. Later developments established global existence results for large initial data, most notably the renormalized solution theory of DiPerna and Lions \cite{DL1989}. In the regime near a global Maxwellian, a well-developed theory has been established, providing global existence and stability results \cite{GuoVPB, GuoVMB, StrainJAMS}. For more results including large amplitude problem and boundary value problems, we refer to \cite{GDCB2010, KimLee, KimLeeNonconvex, KoKimLeeTorus, KLP2022, DHWY2017,DKL2023,DW2019,KS2025} and the references therein. For a comprehensive overview of the mathematical aspects of the Boltzmann equation, we refer to the review article \cite{Villani2002}. These results provide a broad analytical framework for the kinetic theory of dilute gases and serve as a fundamental point of comparison for simplified kinetic models.

Among such simplified models, the BGK model occupies a central position due to its relative simplicity and its wide applicability in physics and engineering. 
As mentioned above, a key issue in the mathematical analysis of the BGK model is the control of the local Maxwellian in suitable function spaces. 
In particular, the polynomial weighted estimate of Perthame and Pulvirenti \cite{PP1993} has played a fundamental role in the well-posedness theory of the BGK model; see \eqref{est:PP}. 
Building upon this estimate, numerous works have established existence and uniqueness results for the BGK model in various settings \cite{CZ2016,Mischler1996,PY2016,PY2019,SY2022,SY2023,Son2025,Yun2015-2,WZ2012,Zhang2007,Zhang2013,Zhang2010}. Furthermore, discrete versions of similar estimates also play an essential role in the convergence analysis of numerical schemes for BGK-type equations \cite{BCRY2022,Bellouquid2003,Issautier1996,RSY2012,RY2018}. The polynomial weighted estimate of Perthame and Pulvirenti \cite{PP1993} can be avoided in certain special settings. Typical examples include the near-equilibrium framework \cite{BKLY2023,Yun2010,Yun2015-1,Yun2019} and stationary boundary problems such as the slab problem \cite{BY2020,BY2016,BY2024,HY2019}. More recently, local well-posedness near a global Maxwellian for the BGK equation in a bounded domain with inflow boundary conditions was established in a polynomially weighted $L^\infty$ space in \cite{LZ2026}.
Consequently, much of the mathematical theory of the BGK model relies on the observation that the relaxation operator preserves polynomial velocity decay in an appropriate sense.

Although the well-posedness theory of kinetic equations has been extensively studied, much less is known about possible ill-posedness phenomena. Understanding such instability mechanisms is important for clarifying the limitations of kinetic models and the functional settings in which they are mathematically valid. Several works have investigated ill-posedness behavior in various kinetic equations. In particular, instability and blow-up phenomena for the gain-term-only Boltzmann equation have been studied in \cite{IS1987,ACI2004}. More recently, the well-posedness and ill-posedness of the Boltzmann equation in Sobolev spaces were investigated in \cite{CH2024,CSZ2024}, where the authors identified functional settings in which the equation becomes ill-posed.

A closely related line of research concerns the creation and propagation of velocity moments. For the spatially homogeneous Boltzmann equation, polynomial and exponential moment creation and propagation have been extensively studied under various collision kernel assumptions based on Povzner-type estimates \cite{Povzner1962,Elmroth1983,Desvillettes1993,Bobylev1997,GPV2009,Mouhot2006,AG2008,ACGM2013}. Recently, exponential moment creation and propagation in angular non-cutoff settings have attracted significant attention \cite{LM2012,TAGP2018,GPT2019,Fournier2021,CHJ2024}. Similar propagation results are also available for the Landau equation \cite{Villani1998,DV2000,Desvillettes2015,GS2025}. These results show that if the initial data possess exponential decay in the velocity variable, then the solution preserves some exponential moments for all positive times, which can be smaller than the original exponential moment. In \cite{LPY2024}, the authors showed that the solution maintains the same exponential moment of the initial data locally in time.

In contrast to the Boltzmann and Landau equations, where the symmetry structure of the collision operators generates the moment propagation properties, the mechanism of the moment propagation in the BGK model is fundamentally different. For the polynomial moment propagation, it is developed using the interpolation argument for the local Maxwellian, and we refer to the papers dealing with well-posedness theory above, such as \cite{Perthame1989}. The question around the propagation of the exponential moments was recently investigated in \cite{LPY2024}, where it was shown that the BGK model can exhibit ill-posedness in exponentially weighted spaces. Motivated by this result, in the present work, we further investigate this ill-posedness mechanism from the viewpoint of boundary effects. More precisely, we consider the BGK model in a half-space domain with inflow boundary conditions and show that boundary interactions alone can trigger an instantaneous blow-up of the solution's exponential moment. Also, we check that the Boltzmann solution preserves the same exponential moment locally in time.\\

\subsection{Proof strategy and main difficulties} This work extends the previous paper \cite{LPY2024} to the half-space inflow boundary condition with vacuum initial data. For the BGK model for the whole space setting, for example, the authors considered initial data schematically of the form
\begin{align}\label{eq:init_whole_space}
    f_0(x,v) = e^{-\alpha |v|^2}\mathbf{1}_{|v|\geq |x|^{\frac{1}{3}}}
\end{align}
and constructed a unique local-in-time solution exhibiting instantaneous ill-posedness with respect to the exponential weights $e^{\alpha'|v|^2}$ for every $0<\alpha'\leq\alpha$. As the density $\rho(t,x)$ vanishes as $|x|\rightarrow \infty$ while the local temperature $T(t,x)$ blows up, the existence and uniqueness of the solution is very non-trivial, as we can not use the standard estimates as in \cite{Perthame1989}. The main idea for constructing the solution was to tightly control the integral of the local Maxwellian term to ensure the stability of the macroscopic fields and to employ a contraction mapping argument. The ill-posedness result then follows from the spatial growth rate of the local temperature.

For our boundary problem, due to the vacuum initial data and the nature of the boundary effect, the original difficulty intertwines with the singularity generated by the boundary effect, making the problem much harder. Compared to the whole space problem, the role of the $e^{-t} f_0(x-vt, v)$ in the definition of the mild solution is replaced by $e^{-t_b} g(t-t_b, x-vt_b, v)\mathbf{1}_{t\geq t_b}$ as in \eqref{eq:BGK_sol_def}, where the cutoff function is equivalent to $v_z\geq\frac{z}{t}$. If we directly try to copy \eqref{eq:init_whole_space} in choosing $g(t,x_h,v)$, then one natural candidate would be
\begin{align*}
    g(t,x_h,v) = e^{-|v|^2}\mathbf{1}_{|v|\geq |x_h|^{\frac{1}{3}}}\mathbf{1}_{v_z\geq 0}
\end{align*}
as there is no way to include the $z$ variable in the boundary data. Then, we have
\begin{align*}
    \rho_g(t,x) = \int_{\mathbb{R}^n} e^{-t_b} g(t-t_b, x-vt_b, v)\mathbf{1}_{t\geq t_b}\,dv = \int_{\mathbb{R}^n} e^{-t_b} e^{-|v|^2}\mathbf{1}_{|v|\geq |x_h-v_ht_b|^{\frac{1}{3}}}\mathbf{1}_{v_z\geq \frac{z}{t}}\,dv.
\end{align*}
If $|x_h|\rightarrow \infty$ or $\frac{z}{t}\rightarrow \infty$, the density $\rho_g(t,x)$ goes to $0$, and there is a possibility that the macroscopic fields become unstable. In fact, the second one is trickier since, for instance, it is singular as $t\rightarrow 0$; see Proposition~\ref{prop:BGK_contradiction} in the Appendix for an example of the instantaneous macroscopic instability even when we choose the inflow boundary data $g(t,x_h,v) = e^{-|v|^2}$, which is regarded as the natural function in the kinetic theory. 

To overcome the instability of the macroscopic fields, we choose special boundary data $g$, whose prototype is given by
\begin{align*}
    g(t,x_h,v) = e^{-|v_h|^2}\mathbf{1}_{|v_h|\geq (|x_h| + v_z)^{1/3}}\mathbf{1}_{v_z\geq 0}.
\end{align*}
We also adapt the analysis to account for the boundary effects. The main idea is to suppress the $v_z$ directional component in the exponential function and include it in the cutoff function. We also need to separate the horizontal component $(x_h,v_h)$ and vertical component $(z, v_z)$ in the analysis and treat the $\frac{z}{t}$ term generated from the boundary effect very carefully. By the cutoff function, replacing $v_z = \frac{z}{t}$, we regard $|x_h| + \frac{z}{t}$ as the group of spatial variables behaving like $|x|$. For the boundary data for the general weight $w^h_{\alpha,\beta,\delta}$, we refer to \eqref{eq:boundary_data}. In Section \ref{sec:Macro_quan}, we will demonstrate that this modification gives the stability of the macroscopic fields of the solution. After the estimate, in Sections \ref{sec:sol_construction} and \ref{sec:Ill_posed}, we construct a unique solution and establish the ill-posedness result using the blow-up of the temperature as $|x_h| + \frac{z}{t}\rightarrow \infty$. In fact, these ideas are reflected in the conditions of the solution space for the BGK model in Definition~\ref{def:BGK_sol_space}. Proposition \ref{prop:suppressed_v_z} also shows that suppressing the $v_z$ component for the exponential decay in the general boundary data helps to construct a unique solution in the general case. An additional novelty of this work is that we simplify the proof of the estimate of the local Maxwellian operator (Proposition \ref{prop:M[f]_integral}) using an elementary inequality in Lemma \ref{lem:concaten}, compared to the one in \cite{LPY2024}.

For the Boltzmann equation, in contrast to the BGK model, we can import the $Q^+(f,f)$ estimate for the weight function $w_{\alpha,\beta,\delta}(v)$ in \cite{LPY2024} without considering the singular boundary effect since we do not need to compute the macroscopic fields. Using the $Q^+$ estimate, we establish the local well-posedness result for the exponentially weighted space. Interestingly, we can not use the weight function $w^h_{\alpha,\beta,\delta}$ since the proof requires the symmetry property of collision velocities such as $|v|^2 + |v_*|^2 = |v'|^2 + |v'_*|^2$, and $|v_h|$ lacks $v_z$ information. It suggests that the two model equations experience the boundary data in a different way for the vacuum initial data. We will give further discussion about this problem in the Appendix.\\

{\bf Organization.} In Section~\ref{sec:technical_lemmas}, we check some well-definedness problems such as the definition of the mild solution and the boundary trace problem, and then we establish several technical lemmas which will be used throughout the paper. In Section~\ref{sec:Macro_quan}, we verify that the macroscopic fields of the solution are stable near $(\rho_g, u_g, T_g)$ in an \textit{a priori} sense. With this \textit{a priori} estimate, we prove the local existence and uniqueness of solutions in Section~\ref{sec:sol_construction}. Subsequently, in Section \ref{sec:Ill_posed}, we establish the main ill-posedness result in the exponential class for the constructed solution. Finally, in Section~\ref{sec:Well_Bol}, we discuss the local well-posedness of the Boltzmann equation. In the Appendix, we show that the macroscopic fields of the solution can not be stable if we choose the inflow boundary data by $e^{-|v|^2}$ and prove that the macroscopic fields become stable if we suppress the $v_z$ component in the exponent of a decaying function like $e^{-|v_h|^2 - |v_z|}$.

\section{Technical lemmas}\label{sec:technical_lemmas}
Before we start the analysis for the main results, we investigate the definition of the solution and the trace norm of the inflow boundary data $g$.

In this paper, we consider the transport equation of type
\begin{align}\label{eq:trans_eq}
    \partial_t f + v\cdot \nabla_x f +f(t,x,v)A(t,x,v)= B(t,x,v)
\end{align}
for nonnegative functions $A$ and $B$, which may also depend on $f$ like $\mathcal{M}(f)$ or $Q^\pm(f,f)$. Our spatial domain is $\mathbb{R}^n_+$, and we consider the initial data $f_0(x,v)$ and inflow boundary data $g(t,x_h,v)$ on $\partial \mathbb{R}^n_+ = \{z=0\}$. If $f$, $A$, and $B$ are sufficiently regular, then integrating the transport equation \eqref{eq:trans_eq} along the characteristic trajectories yields
\begin{align}\label{eq:mild_def}
    \begin{split}
        f(t,x,v) &= e^{-\int_0^{t_b(x,v)} A(t-\tau, x-v\tau, v)\,d\tau}g(t-t_b(x,v), x-vt_b, v)\mathbf{1}_{t_b(x,v)\leq t} \\
        &\quad + e^{-\int_0^t A(t-\tau, x-v\tau, v)\,d\tau}f_0(x-vt, v)\mathbf{1}_{t_b(x,v)>t}\\
        &\quad + \int_0^t e^{-\int_0^\tau A(t-s, x-vs, v)\,ds}B(t-\tau, x-v\tau, v)\mathbf{1}_{z-v_z\tau> 0}\,d\tau.
    \end{split}
\end{align}
We call $f$ a mild solution of \eqref{eq:trans_eq} if it satisfies \eqref{eq:mild_def} for a.e. $(t,x,v)\in[0, t_0]\times \mathbb{R}^n_+\times\mathbb{R}^n$ for some $0<t_0\leq +\infty$. Note that we use a slightly different time-integral convention from \eqref{eq:BGK_sol_def} and \eqref{eq:Boltzmann_sol_def} to simplify the proof of the following lemma.

We can also consider another form of the transport equation by
\begin{align*}
    \partial_t f + v\cdot \nabla_x f= B(t,x,v) - f(t,x,v)A(t,x,v),
\end{align*}
and write the solution $f$ as
\begin{align}\label{eq:mild_def_2}
    \begin{split}
        f(t,x,v) &= g(t-t_b(x,v), x-vt_b, v)\mathbf{1}_{t_b(x,v)\leq t}+ f_0(x-vt, v)\mathbf{1}_{t_b(x,v)>t}\\
        &\quad + \int_0^t \left(B(t-\tau, x-v\tau, v) - f A(t-\tau, x-v\tau, v)\right)\mathbf{1}_{z-v_z\tau> 0}\,d\tau.
    \end{split}
\end{align}
Even though \eqref{eq:mild_def} and \eqref{eq:mild_def_2} have different forms, the following lemma shows that the two forms are equivalent when $A(t,x,v)$ and $B(t,x,v)$ are locally integrable.

\begin{lemma}\label{lem:mild_sol_equiv}
    If $A(t,x,v),B(t,x,v)\in L^1_{loc}([0,t_0]\times \overline{\mathbb{R}^n_+}\times \mathbb{R}_v^n)$, then \eqref{eq:mild_def} and \eqref{eq:mild_def_2} are equivalent, where $\overline{\mathbb{R}^n_+}$ is the closure of $\mathbb{R}^n_+$.
\end{lemma}
\begin{proof}
    We prove the direction $\eqref{eq:mild_def}\Rightarrow \eqref{eq:mild_def_2}$, and the opposite direction follows by reversing the proof. We first extend the $f(t,x,v)$ on $\mathbb{R}\times \mathbb{R}^n\times \mathbb{R}^n$ by choosing $f(t,x,v) = 0$ if $(t,x,v)\not\in [0,t_0]\times\mathbb{R}^n_+\times\mathbb{R}^n$. We introduce an auxiliary variable $w\in\mathbb{R}$ and consider $f(t+w, x+vw, v)$. From \eqref{eq:mild_def}, we obtain
    \begin{align}\label{eq:mild_def_1_1}
        \begin{split}
            &f(t+w, x+vw, v) \\
            &= e^{-\int_0^{t_b(x+vw,v)} A(t-(\tau-w), x-v(\tau-w), v)\,d\tau}g(t+w-t_b(x+vw,v), x+vw-vt_b(x+vw,v), v)\mathbf{1}_{t_b(x+vw,v)\leq t+w}\\
            &\quad + e^{-\int_0^{t+w} A(t-(\tau-w), x-v(\tau-w), v)\,d\tau}f(0, x-vt, v)\mathbf{1}_{t_b(x+vw,v)>t+w}\\
            &\quad + \int_0^{t+w} e^{-\int_0^\tau A(t-(s-w), x-v(s-w), v)\,ds}B(t-(\tau-w), x-v(\tau-w), v)\mathbf{1}_{z-v_z(\tau-w)>0}\,d\tau
        \end{split}
    \end{align}
    when $0\leq t+w\leq t_0$ and $z+v_zw>0$, while $f(t+w, x+vw, v) = 0$ otherwise. We note that \eqref{eq:mild_def} and \eqref{eq:mild_def_1_1} are equivalent in a.e. sense as
    \begin{align}\label{eq:a.e._equiv}
        \begin{split}
            &\iint_{\mathbb{R}\times\mathbb{R}\times \mathbb{R}^n\times \mathbb{R}^n}\left|f(t+w, x+vw, v) - (\text{RHS of \eqref{eq:mild_def_1_1}})\right|\mathbf{1}_{0<t+w<t_0}\mathbf{1}_{z+v_zw>0}\,dtdwdxdv\\
            &=\iint_{\mathbb{R}\times\mathbb{R}\times \mathbb{R}^n\times \mathbb{R}^n}\left|f(t, x, v) - (\text{RHS of \eqref{eq:mild_def}})\right|\mathbf{1}_{0<t<t_0}\mathbf{1}_{z>0}\,dtdwdxdv
        \end{split}
    \end{align}
    by the change of variables.

    For technical use, we claim that
    \begin{align*}
        \dv{w} \left(t_b(x+vw, v) - w\right) = 0
    \end{align*}
    for $v$ and $w$ satisfying $z+v_zw>0$ and $v_z>0$, so $t_b(x+vw, v)\neq \infty$. Indeed, for any small $h$ making $z+v_z(w+h)>0$, we have
    \begin{align*}
        t_b(x+v(w+h), v) - (w + h) = t_b(x+vw, v) + h - (w + h) = t_b(x+vw, v) - w,
    \end{align*}
    so it is a constant function.

    Using this identity, we rewrite \eqref{eq:mild_def_1_1} as
    \begin{align}\label{eq:+w}
        \begin{split}
            &f(t+w, x+vw, v) \\
            &= e^{-\int_{-w}^{t_b(x+vw,v) - w} A(t-\tau, x-v\tau, v)\,d\tau}g\left(t-(t_b(x+vw,v)-w), x-v(t_b(x+vw,v)-w), v\right)\mathbf{1}_{t_b(x+vw,v) - w\leq t}\\
            &\quad + e^{-\int_{-w}^t A(t-\tau, x-v\tau, v)\,d\tau}f(0, x-vt, v)\mathbf{1}_{t_b(x+vw,v) - w>t}\\
            &\quad + \int_{-w}^t e^{-\int_{-w}^\tau A(t-s, x-vs, v)\,ds}B(t-\tau, x-v\tau, v)\mathbf{1}_{z-v_z\tau>0}\,d\tau.
        \end{split}
    \end{align}
    Since $A(t,x,v)$ and $B(t,x,v)$ are locally integrable, by Fubini's theorem, the right-hand side of \eqref{eq:+w} is differentiable for a.e. $w$ satisfying $0<t+w<t_0$ and $z+v_zw>0$, and
    \begin{align*}
        \dv{w}\left(\text{R.H.S}\right) = -A(t+w, x+vw, v) f(t+w, x+vw, v) + B(t+w, x+vw, v)\mathbf{1}_{z+v_zw>0}.
    \end{align*}
    Integrating both sides by $w$, we get
    \begin{align*}
        f(t+w, x+vw, v)&=g(t+w-t_b(x+vw,v), x+vw-vt_b(x+vw,v), v)\mathbf{1}_{t_b(x+vw,v)\leq t+w}\\
        &\quad + f(0, x-vt, v)\mathbf{1}_{t_b(x+vw,v)>t+w}\\
        &\quad + \int_0^{t+w} B(t+w-\tau, x-v(\tau-w), v)\mathbf{1}_{z-v_z(\tau-w)>0} - A f(t+w-\tau, x-v(\tau-w), v)\,d\tau.
    \end{align*}
    As $f(t+w,x+vw,v) = 0$ if $x+vw\not\in \mathbb{R}^n_+$, the final term is same as
    \begin{align*}
        \int_0^{t+w} \left(B(t+w-\tau, x-v(\tau-w), v)- Af(t+w-\tau, x-v(\tau-w), v)\right)\mathbf{1}_{z-v_z(\tau-w)>0}\,d\tau
    \end{align*}
    Since the equality holds for a.e. $(t,x,v)$ and $w$ for $0<t+w<t_0$ and $z+v_zw>0$, the same argument as in \eqref{eq:a.e._equiv} yields \eqref{eq:mild_def_2} for a.e. $(t,x,v)$. Conversely, \eqref{eq:mild_def} follows from \eqref{eq:mild_def_2} by reversing the preceding argument.
\end{proof}
For the BGK model, we take $A = 1$ and $B = \mathcal{M}(f)$, whereas for the Boltzmann equation, we set $A = L(f)$ and $B = Q^+(f,f)$. For the solutions constructed in this paper, the following local integrability property holds on every time interval $[0,t_0]$ contained in their interval of existence:
\begin{align*}
    \mathcal{M}(f)&\in L^1_{\mathrm{loc}}\left([0,t_0]\times\overline{\mathbb{R}^n_+}\times\mathbb{R}^n\right)\quad \text{for the BGK model},\\
    Q^+(f,f),L(f)&\in L^1_{\mathrm{loc}}\left([0,t_0]\times\overline{\mathbb{R}^n_+}\times\mathbb{R}^n\right)\quad\text{for the Boltzmann equation}.
\end{align*}
Therefore, Lemma~\ref{lem:mild_sol_equiv} allows us to use either mild formulation.\\

The next lemma concerns the well-definedness of the trace norm of the inflow-boundary data on $\partial \mathbb{R}^n_+$.
\begin{lemma}\label{lem:trace}
    (Well-definedness of the trace) Let $g_1(t,x_h,v)$ and $g_2(t,x_h,v)$ be measurable functions on $[0, \infty)\times \partial\mathbb{R}^n_+\times \mathbb{R}^n_+$. Then, $g_1(t,x_h,v) = g_2(t,x_h,v)$ for a.e. $(t,x,v)\in [0,\infty)\times \partial \mathbb{R}^n_+\times\mathbb{R}^n_+$ if and only if for any $t> 0$,
    \begin{align*}
        g_1(t-t_b, x-vt_b, v)\mathbf{1}_{t\geq t_b} = g_2(t-t_b, x-vt_b, v)\mathbf{1}_{t\geq t_b}
    \end{align*}
    for a.e. $(x,v)$.
    Furthermore, if $1\leq p,q\leq \infty$ with notation $1/\infty = 0$, we have
    \begin{align*}
        \|g(t-t_b, x-vt_b, v)\mathbf{1}_{v_z\geq \frac{z}{t}}\|_{L^p_v(\mathbb{R}^n, L^q_x(\mathbb{R}^n_+))}=\|v_z^{1/q} g(\tau,x_h,v)\|_{L^p_v(\mathbb{R}^n, L^q_{\tau,x_h}([0,t]\times \partial \mathbb{R}^n_+))}
    \end{align*}
    for any $t> 0$.
    Equivalently,
    \begin{align*}
        \|e^{-t_b}g(t-t_b, x-vt_b, v)\mathbf{1}_{v_z\geq \frac{z}{t}}\|_{L^p_v(\mathbb{R}^n, L^q_x(\mathbb{R}^n_+))}=e^{-t}\|v_z^{1/q} e^\tau g(\tau,x_h,v)\|_{L^p_v(\mathbb{R}^n, L^q_{\tau,x_h}([0,t]\times \partial \mathbb{R}^n_+))}
    \end{align*}
    for any $t>0$.
\end{lemma}
\begin{proof}
    The first statement is easily checked by taking the integral of $|g_1-g_2|$. For fixed $t$, writing $t_b = \frac{z}{v_z}\geq 0$, we have
    \begin{align*}
        &\int_{\mathbb{R}^n_+\times\mathbb{R}^n_+} |g_1(t-t_b, x-vt_b, v)-g_2(t-t_b, x-vt_b, v)|\mathbf{1}_{t\geq t_b}\,dxdv \\
        &=\int_{\mathbb{R}^{n-1}\times\mathbb{R}_+\times\mathbb{R}^n_+} |g_1(t-\frac{z}{v_z}, x-v\frac{z}{v_z}, v)-g_2(t-\frac{z}{v_z}, x-v\frac{z}{v_z}, v)|\mathbf{1}_{t\geq \frac{z}{v_z}}\,dx_h dz dv \\
        &= \int_{\mathbb{R}^{n-1}\times\mathbb{R}_+\times\mathbb{R}^n_+} v_z|g_1(t-t_b, x_h, v)-g_2(t-t_b, x_h, v)|\mathbf{1}_{0\leq t_b\leq t}\,dx_hdt_b dv\\
        &= \int_0^t \left(\int_{\mathbb{R}^{n-1}\times\mathbb{R}^n_+} v_z|g_1(t_b, x_h, v)-g_2(t_b, x_h, v)|\,dx_h dv\right)\,dt_b.
    \end{align*}
    Since the integrands are nonnegative and $v_z>0$ almost everywhere, the identity above shows that, for each $t>0$, the transported functions agree almost everywhere on $\mathbb{R}_+^n\times\mathbb{R}_+^n$ if and only if $g_1=g_2$ almost everywhere on $[0,t]\times\partial\mathbb{R}_+^n\times\mathbb{R}_+^n$. Since $t>0$ is arbitrary, the first assertion follows.

    For the second statement, we first suppose $1\leq q<\infty$. Then,
    \begin{align*}
        \int \left|g\left(t-\frac{z}{v_z}, x-v\frac{z}{v_z}, v\right)\right|^q \mathbf{1}_{t\geq t_b}\,dx &= v_z\int |g(t-t_b, x_h, v)|^q \mathbf{1}_{0\leq t_b\leq t}\,dx_h dt_b \\
        &= v_z \|g(t_b,x_h,v)\|^q_{L^q_{t_b, x_h}([0,t]\times \partial \mathbb{R}^n_+)}
    \end{align*}
    for any $t>0$. Therefore, we have the lemma for $1\leq q<\infty$. The $q=\infty$ case follows from the pointwise estimate. For the final one, we choose $g(\tau, x, v)\mapsto e^\tau g(\tau, x, v)$ in the second result.
\end{proof}

We now turn to the technical estimates, which will be used in the proofs of the main results. The following lemma provides a self-contained proof of a variant of Tricomi's relation~\cite{Tricomi1950}.
\begin{lemma}\label{lem:exp_integral}
Let $n\geq 1$, $\alpha>0$, $0< \beta\leq 2$, $\delta_1\geq 0$, $\delta_2\in\mathbb{R}$, and $M\geq 0$. Then, there exist constants $C_1$ and $C_2$ depending on $n$, $\alpha$, $\beta$, $\delta_1$, and $\delta_2$ such that
    \begin{align*}
        C_1\leq \frac{1}{\langle M\rangle^{(n-\beta) + (1-\beta)\delta_1 + \delta_2} e^{-\alpha M^\beta}}\int_{\mathbb{R}^n} \left(|v|-M\right)^{\delta_1}\langle v\rangle^{\delta_2} e^{-\alpha|v|^\beta}\mathbf{1}_{|v|\geq M}\,dv\leq C_2.
    \end{align*}
\end{lemma}
\begin{proof}
    Using the spherical coordinates, we write
    \begin{align*}
        \int_{\mathbb{R}^n} (|v|-M)^{\delta_1}\langle v\rangle^{\delta_2} e^{-\alpha|v|^\beta}\mathbf{1}_{|v|\geq M}\,dv = |\mathbb{S}^{n-1}|\int_M^\infty (r-M)^{\delta_1}\langle r\rangle^{\delta_2} e^{-\alpha r^\beta}r^{n-1}\,dr.
    \end{align*}
    We further use the change of variable $r^\beta\mapsto r$. Then, we get
    \begin{align}\label{eq:ul_rm}
        \int_M^\infty (r-M)^{\delta_1}\langle r\rangle^{\delta_2} e^{-\alpha r^\beta}r^{n-1}\,dr = \frac{1}{\beta}\int_{M^\beta}^\infty \left(r^{1/\beta}-M\right)^{\delta_1}\big\langle r^{1/\beta} \big\rangle^{\delta_2} e^{-\alpha r} r^{\frac{n}{\beta}- 1}\,dr.
    \end{align}
    If $M^\beta\leq 10$, the integral is bounded above and below by positive constants. Thus, we safely assume $M^\beta\geq 10$. For every $x\geq 10$, we have
    \begin{align*}
        x\leq \langle x\rangle \leq 2x,
    \end{align*}
    so we can replace the right-hand side of \eqref{eq:ul_rm} by
    \begin{align*}
        \int_{M^\beta}^\infty \left(r^{1/\beta}-M\right)^{\delta_1} r^{\frac{\delta_2 + n}{\beta}- 1} e^{-\alpha r} \,dr.
    \end{align*}
    
    To estimate this integral, we partition the interval $\left[M^\beta, \infty\right)$ by setting
    $r_i := M^\beta + i$ for $i \in \mathbb{N}\cup\{0\}$, and estimate the integral on each sub-interval $[r_i, r_{i+1})$. To deal with the term $\left(r^{1/\beta} -  M\right)^{\delta_1}$, we put $r=r_i$ and write
    \begin{align*}
        r_i^{1/\beta} -  M &=\left(M^\beta + i\right)^{1/\beta} -  M = M\left((1+iM^{-\beta})^{1/\beta}-1\right).
    \end{align*}
    For $0\leq x\leq 1$ and $0< \beta\leq 2$, we have
    \begin{align*}
        1+\frac{x}{2\beta}\leq (1+x)^{1/\beta}\leq 1+\frac{\max\{2^{1/\beta - 1}, 1\}}{\beta}x
    \end{align*}
    Therefore, if $0\leq i\leq M^\beta$, we obtain
    \begin{align}\label{eq:M_bounds}
        \begin{split}
            M\left((1+iM^{-\beta})^{1/\beta}-1\right)&\geq i\frac{M^{1-\beta}}{2\beta}\\
            M\left((1+iM^{-\beta})^{1/\beta}-1\right)&\leq i\frac{\max\{2^{1/\beta - 1}, 1\}}{\beta}M^{1-\beta}.
        \end{split}
    \end{align}
    We first derive the lower bound of the integral by
    \begin{align}
        &\int_{M^\beta}^\infty \left(r^{1/\beta}-M\right)^{\delta_1} r^{\frac{\delta_2 + n}{\beta}- 1} e^{-\alpha r} \,dr\geq \int_{M^\beta}^{2M^\beta + 1} \left(r^{1/\beta}-M\right)^{\delta_1} r^{\frac{\delta_2 + n}{\beta}- 1} e^{-\alpha r} \,dr\notag\\
        &\geq \sum_{i=0}^{\lfloor M^\beta\rfloor} \left(i\frac{M^{1-\beta}}{2\beta}\right)^{\delta_1}\min\left\{(M^\beta+i)^{\frac{\delta_2+n}{\beta}-1}, (M^\beta+(i+1))^{\frac{\delta_2+n}{\beta}-1}\right\} e^{-\alpha(M^\beta + (i+1))}\notag\\
        &=(2\beta)^{-\delta_1}M^{(1-\beta)\delta_1 + \delta_2 + n - \beta}e^{-\alpha M^\beta}\sum_{i=0}^{\lfloor M^\beta\rfloor} i^{\delta_1}\min\left\{(1+iM^{-\beta})^{\frac{\delta_2+n}{\beta}-1}, (1+(i+1)M^{-\beta})^{\frac{\delta_2+n}{\beta}-1}\right\} e^{-\alpha(i+1)}.\label{eq:lower}
    \end{align}
    We need to show that the summation in \eqref{eq:lower} is bounded below by a positive constant independent of $M$. If $\frac{\delta_2+n}{\beta}-1\geq 0$, then the minimum term is not smaller than $1$. Hence, it is lower bounded by $\sum_{i=0}^{\lfloor M^\beta\rfloor} i^{\delta_1}e^{-\alpha(i+1)}$. On the other hand, if $\frac{\delta_2+n}{\beta}-1<0$, then, since $M^\beta\geq 10$, we have
    \begin{align*}
        i^{\delta_1}\min\left\{(1+iM^{-\beta})^{\frac{\delta_2+n}{\beta}-1}, (1+(i+1)M^{-\beta})^{\frac{\delta_2+n}{\beta}-1}\right\}\geq i^{\delta_1}(1+(i+1))^{\frac{\delta_2+n}{\beta}-1} e^{-\alpha(i+1)}.
    \end{align*}
    Therefore, the summation is lower bounded by a positive constant independent of $M$. Combining the two cases yields
    \begin{align*}
        \int_{\mathbb{R}^n} \left(|v|-M\right)^{\delta_1}\langle v\rangle^{\delta_2} e^{-\alpha|v|^\beta}\mathbf{1}_{|v|\geq M}\,dv\geq C_{\alpha,\beta,\delta,n}\langle M\rangle ^{(1-\beta)\delta_1 + \delta_2 + n - \beta}e^{-\alpha M^\beta}.
    \end{align*}
    
    For the upper bound, we split the integral region into $[M^\beta, 2M^\beta - 1)$ and $[2M^\beta - 1, \infty)$. For the integral on $[M^\beta, 2M^\beta - 1)$, we take the same partitioning by $[r_i, r_{i+1})$ as above. Then, using \eqref{eq:M_bounds}, we have
    \begin{align*}
        &\int_{M^\beta}^{2M^\beta - 1} \left(r^{1/\beta}-M\right)^{\delta_1} r^{\frac{\delta_2 + n}{\beta}- 1} e^{-\alpha r} \,dr\\
        &\leq \sum_{i=0}^{\lfloor M^\beta\rfloor} \left((i+1)\frac{\max\{2^{1/\beta - 1}, 1\}}{\beta}M^{1-\beta}\right)^{\delta_1}\max\left\{(M^\beta+i)^{\frac{\delta_2+n}{\beta}-1}, (M^\beta+(i+1))^{\frac{\delta_2+n}{\beta}-1}\right\} e^{-\alpha(M^\beta + i)}.
    \end{align*}
    For the summation, we apply the same argument as in the lower bound estimate and get
    \begin{align}\label{eq:2_3:upper_1}
        \begin{split}
            &\sum_{i=0}^{\lfloor M^\beta\rfloor} \left((i+1)\frac{\max\{2^{1/\beta - 1}, 1\}}{\beta}M^{1-\beta}\right)^{\delta_1}\max\left\{(M^\beta+i)^{\frac{\delta_2+n}{\beta}-1}, (M^\beta+(i+1))^{\frac{\delta_2+n}{\beta}-1}\right\} e^{-\alpha(M^\beta + i)}\\
            &\leq C_{\alpha,\beta,\delta,n}\langle M\rangle ^{(1-\beta)\delta_1 + \delta_2 + n - \beta}e^{-\alpha M^\beta}.
        \end{split}
    \end{align}
    For the integral on $[2M^\beta - 1, \infty)$, we crudely bound the polynomial part by $(r^{1/\beta}-M)^{\delta_1}\leq r^{\delta_1/\beta}$. Partitioning this region by $r_i = 2M^\beta + i$ for $i\in\mathbb{N}\cup\{-1, 0\}$, we have
    \begin{align*}
        &\int_{2M^\beta - 1}^\infty \left(r^{1/\beta}-M\right)^{\delta_1} r^{\frac{\delta_2 + n}{\beta}- 1} e^{-\alpha r} \,dr\leq \int_{2M^\beta - 1}^\infty r^{\frac{\delta_1 + \delta_2 + n}{\beta}- 1} e^{-\alpha r} \,dr\\
        &\leq \sum_{i=-1}^\infty \max\left\{(2M^\beta+i)^{\frac{\delta_1+\delta_2+n}{\beta}-1}, (2M^\beta+(i+1))^{\frac{\delta_1+\delta_2+n}{\beta}-1}\right\} e^{-\alpha(2M^\beta + i)}\\
        &\leq C_{\alpha,\beta,\delta,n}\langle M\rangle^{\delta_1 + \delta_2 + n - \beta}e^{-2\alpha M^\beta}.
    \end{align*}
    Since the exponential decay is much stronger than \eqref{eq:2_3:upper_1}, we get
    \begin{align*}
        \int_{M^\beta}^\infty \left(r^{1/\beta}-M\right)^{\delta_1} r^{\frac{\delta_2 + n}{\beta}- 1} e^{-\alpha r} \,dr\leq C_{\alpha,\beta,\delta,n}\langle M\rangle ^{(1-\beta)\delta_1 + \delta_2 + n - \beta}e^{-\alpha M^\beta}.
    \end{align*}
\end{proof}

The next lemma collects several inequalities for estimating \((x+y)^\beta\) when \(0\leq\beta\leq3\). These inequalities will be used to control terms arising from the cutoff function and the exponential weight in the next section.
\begin{lemma}\label{lem:beta_pert}
    For $x\geq 1$, the following inequalities hold.
    \begin{enumerate}
        \item If $1\leq \beta\leq 3$ and $0\leq c\leq x^\beta$, then
        \begin{align*}
            (x+c x^{1-\beta})^\beta\leq x^\beta +4\beta c.
        \end{align*}
        \item If $1\leq \beta\leq 3$ and $-\frac{x^\beta}{2}\leq c\leq 0$, then
        \begin{align*}
            (x+c x^{1-\beta})^\beta\leq x^\beta +\frac{\beta}{4}c.
        \end{align*}
        \item For $1\leq \beta\leq 3$ and $-x^{\beta}\leq c$,
        \begin{align*}
            x^\beta +\beta c\leq (x+c x^{1-\beta})^\beta.
        \end{align*}
        \item If $0\leq \beta\leq 1$ and $-x^\beta\leq c\leq 0$, then
        \begin{align*}
            x^\beta + c\leq (x + c x^{1-\beta})^\beta.
        \end{align*}
    \end{enumerate}
    If $0\leq \beta \leq 1$ and $x,y> 0$, then
    \begin{align}\label{eq:beta_pert_5}
        (x+y)^\beta\leq \min\left\{x^\beta + \beta y x^{\beta-1}, y^\beta + \beta x y^{\beta-1}\right\}.
    \end{align}
\end{lemma}
\begin{proof}
    For the first to fourth inequalities, we have
    \begin{align*}
        (x+cx^{1-\beta})^\beta = x^\beta(1+cx^{-\beta})^\beta.
    \end{align*}
    Therefore, we can rewrite the inequalities as
    \begin{align}
        (1+x)^\beta&\leq 1 + 4\beta x\quad \text{for}\quad 1\leq \beta\leq 3,\quad 0\leq x\leq 1,\label{eq:beta_pert_1}\\
        (1+x)^\beta&\leq 1 + \frac{\beta}{4} x\quad \text{for}\quad 1\leq \beta\leq 3,\quad -\frac{1}{2}\leq x\leq 0,\label{eq:beta_pert_2}\\
        (1+x)^\beta&\geq 1 + \beta x\quad \text{for}\quad 1\leq \beta\leq 3,\quad -1\leq x,\label{eq:beta_pert_3}\\
        (1+x)^\beta&\geq 1 + x\quad \text{for}\quad 0\leq \beta\leq 1,\quad -1\leq x\leq 0\label{eq:beta_pert_4}.
    \end{align}
    
    The inequality \eqref{eq:beta_pert_4} is trivial inequality when $0\leq x+1\leq 1$ and $0\leq \beta\leq 1$. So, we prove the first to third inequalities in steps. For the first inequality, we first check that
    \begin{align*}
        \beta (1+x)^{\beta-1}-4\beta\leq 0
    \end{align*}
    for $1\leq \beta\leq 3$ and $0\leq x\leq 1$. Integrating the inequality for $0\leq x\leq 1$, we have \eqref{eq:beta_pert_1}. For the second and third inequalities, we similarly observe that
    \begin{align*}
        \beta(1+x)^{\beta-1} - \frac{\beta}{4}\geq 0,\quad -\frac{1}{2}\leq x\leq 0,
    \end{align*}
    and
    \begin{align*}
        \beta(1+x)^{\beta-1} - \beta\begin{cases}
            \geq 0 & x\geq 0,\\
            \leq 0 & -1< x\leq 0
        \end{cases}
    \end{align*}
    for $1\leq \beta\leq 3$. Therefore, applying the same argument, we obtain \eqref{eq:beta_pert_2} and \eqref{eq:beta_pert_3}.

    For the final one \eqref{eq:beta_pert_5}, it is sufficient to show that
    \begin{align*}
        (x+y)^\beta\leq x^\beta + \beta y x^{\beta-1}
    \end{align*}
    by the exchange symmetry $x\leftrightarrow y$. We first homogenize the inequality by $t =\frac{y}{x}$. Then, it is equivalent to show
    \begin{align*}
        (1+t)^\beta\leq \beta t + 1.
    \end{align*}
    Differentiating with respect to $t$, we get
    \begin{align*}
        \beta (1+t)^{\beta-1}-\beta \leq 0
    \end{align*}
    for $t>0$, and the inequality is true when $t=0$. Therefore, it holds for $t\geq 0$.
\end{proof}
When $x$ is large and $y$ is a scale of $x^{1-\beta}$, the first four inequalities imply that $(x+y)^\beta$ is almost the same as $x^\beta$ with a small perturbation term.

The following lemma directly uses Lemma \ref{lem:beta_pert}; it estimates the cutoff function in the inflow boundary data \eqref{eq:boundary_data} during the computation.
\begin{lemma}\label{lem:gamma_pert}
    Assume that $2\leq \gamma\leq 3$, $x,C\geq 0$, $0\leq t\leq \frac{1}{4}$, and $x+C\geq 1$. If $r\geq (x+C)^{\frac{1}{\gamma}} + 2\frac{t}{\gamma}(x+C)^{\frac{2-\gamma}{\gamma}}$, then
    \begin{align*}
        r\geq ((x+rt)+C)^{\frac{1}{\gamma}}.
    \end{align*}
    If $r\leq (x+C)^{\frac{1}{\gamma}} - \frac{4t}{\gamma}(x+C)^{\frac{2-\gamma}{\gamma}}$, then
    \begin{align*}
        r\leq (|x-rt|+C)^{\frac{1}{\gamma}}.
    \end{align*}
\end{lemma}
\begin{proof}
    Without loss of generality, we assume $t>0$. We first apply Lemma \ref{lem:beta_pert} to obtain
    \begin{align}
        \left((x+C)^{\frac{1}{\gamma}} + 2\frac{t}{\gamma}(x+C)^{\frac{2-\gamma}{\gamma}}\right)^\gamma&\geq (x+C) + 2t(x+C)^{\frac{1}{\gamma}},\label{eq:gamma_pert_1}\\
        \left((x+C)^{\frac{1}{\gamma}} - \frac{4t}{\gamma}(x+C)^{\frac{2-\gamma}{\gamma}}\right)^\gamma&\leq (x+C) - t(x+C)^{\frac{1}{\gamma}}.\label{eq:gamma_pert_2}
    \end{align}
    Starting from these facts, we will show that $r = (x+C)^{\frac{1}{\gamma}} + 2\frac{t}{\gamma}(x+C)^{\frac{2-\gamma}{\gamma}}$ (resp. $r = (x+C)^{\frac{1}{\gamma}} - \frac{4t}{\gamma}(x+C)^{\frac{2-\gamma}{\gamma}}$) fulfills the first (resp. second) inequality. After that, we enlarge the region such that each inequality holds.
    
    For the first inequality in the lemma, we need to verify
    \begin{align*}
        r^\gamma - (x+r t) - C\geq 0.
    \end{align*}
    The left-hand side is an increasing function when $r\geq (t/\gamma)^{\frac{1}{\gamma-1}}$, so when $r\geq (x+C)^{\frac{1}{\gamma}} + 2\frac{t}{\gamma}(x+C)^{\frac{2-\gamma}{\gamma}}$ as $x+C\geq 1$. Using \eqref{eq:gamma_pert_1}, we have
    \begin{align*}
        &\left((x+C)^{\frac{1}{\gamma}} + 2\frac{t}{\gamma}(x+C)^{\frac{2-\gamma}{\gamma}}\right)^\gamma - \left((x+C)^{\frac{1}{\gamma}} + 2\frac{t}{\gamma}(x+C)^{\frac{2-\gamma}{\gamma}}\right) t - (x + C)\\
        &\geq t(x+C)^{\frac{1}{\gamma}} - t^2(x+C)^{\frac{2-\gamma}{\gamma}} = t(x+C)^{\frac{1}{\gamma}}\left(1 - t(x+C)^{\frac{1-\gamma}{\gamma}}\right)\\
        &\geq 0.
    \end{align*}
    It means that $r\geq ((x+rt)+C)^{\frac{1}{\gamma}}$ when $r\geq (x+C)^{\frac{1}{\gamma}} + 2\frac{t}{\gamma}(x+C)^{\frac{2-\gamma}{\gamma}}$.

    For the second inequality, we need to check
    \begin{align*}
        r^\gamma - |x - r t| - C\leq 0.
    \end{align*}
    We divide into two cases: $(x+C)^{\frac{1}{\gamma}} - \frac{4t}{\gamma}(x+C)^{\frac{2-\gamma}{\gamma}}\leq \frac{x}{t}$ and the other case. For the former case, when $r\leq \frac{x}{t}$, then the inequality is same as
    \begin{align}\label{eq:gamma_pert_2_1}
        r^\gamma - x + r t - C\leq 0.
    \end{align}
    If $r= (x+C)^{\frac{1}{\gamma}} - \frac{4t}{\gamma}(x+C)^{\frac{2-\gamma}{\gamma}}$, we obtain from \eqref{eq:gamma_pert_2} that
    \begin{align*}
        &\left((x+C)^{\frac{1}{\gamma}} - \frac{4t}{\gamma}(x+C)^{\frac{2-\gamma}{\gamma}}\right)^\gamma + \left((x+C)^{\frac{1}{\gamma}} - \frac{4t}{\gamma}(x+C)^{\frac{2-\gamma}{\gamma}}\right)t -x - C\\
        &\leq (x+C) - t(x+C)^{\frac{1}{\gamma}} + \left((x+C)^{\frac{1}{\gamma}} - \frac{4t}{\gamma}(x+C)^{\frac{2-\gamma}{\gamma}}\right)t -x - C\\
        &\leq 0.
    \end{align*}
    Since left-hand side of \eqref{eq:gamma_pert_2_1} is an increasing function of $r$, the inequality holds whenever $r\leq (x+C)^{\frac{1}{\gamma}} - \frac{4t}{\gamma}(x+C)^{\frac{2-\gamma}{\gamma}}$.

    Next, we assume $(x+C)^{\frac{1}{\gamma}} - \frac{4t}{\gamma}(x+C)^{\frac{2-\gamma}{\gamma}}> \frac{x}{t}$. Then, we need to check the inequality
    \begin{align*}
        r^\gamma + x - r t - C\leq 0
    \end{align*}
    for $\frac{x}{t}\leq r<(x+C)^{\frac{1}{\gamma}} - \frac{4t}{\gamma}(x+C)^{\frac{2-\gamma}{\gamma}}$; since the original inequality is an increasing function for $r\leq \frac{x}{t}$, if it is negative at $r = \frac{x}{t}$, it still holds for $r\leq \frac{x}{t}$. First, when $r = (x+C)^{\frac{1}{\gamma}} - \frac{4t}{\gamma}(x+C)^{\frac{2-\gamma}{\gamma}}$, we have
    \begin{align*}
        &\left((x+C)^{\frac{1}{\gamma}} - \frac{4t}{\gamma}(x+C)^{\frac{2-\gamma}{\gamma}}\right)^\gamma - \left((x+C)^{\frac{1}{\gamma}} - \frac{4t}{\gamma}(x+C)^{\frac{2-\gamma}{\gamma}}\right)t + x - C\\
        &\leq (x+C) - t(x+C)^{\frac{1}{\gamma}} - \left((x+C)^{\frac{1}{\gamma}} - \frac{4t}{\gamma}(x+C)^{\frac{2-\gamma}{\gamma}}\right)t + x - C\\
        &=2x - 2t\left((x+C)^{\frac{1}{\gamma}} - \frac{2t}{\gamma}(x+C)^{\frac{2-\gamma}{\gamma}}\right)\\
        &\leq 2t\left((x+C)^{\frac{1}{\gamma}} - \frac{4t}{\gamma}(x+C)^{\frac{2-\gamma}{\gamma}}\right) - 2t\left((x+C)^{\frac{1}{\gamma}} - \frac{2t}{\gamma}(x+C)^{\frac{2-\gamma}{\gamma}}\right)\leq 0
    \end{align*}
    as $x<t\left((x+C)^{\frac{1}{\gamma}} - \frac{4t}{\gamma}(x+C)^{\frac{2-\gamma}{\gamma}}\right)$.

    Next, we put $r = \frac{x}{t}$ and verify
    \begin{align*}
        \left(\frac{x}{t}\right)^\gamma - C\leq 0.
    \end{align*}
    We claim that $\left(\frac{x}{t}\right)^\gamma>C$ can not happen. If this were the case, then
    \begin{align*}
        C< \left(\frac{x}{t}\right)^\gamma &<\left((x+C)^{\frac{1}{\gamma}} - \frac{4t}{\gamma}(x+C)^{\frac{2-\gamma}{\gamma}}\right)^\gamma\\
        &\leq (x+C) - t(x+C)^{\frac{1}{\gamma}},
    \end{align*}
    so
    \begin{align*}
        (x+C)^{\frac{1}{\gamma}}<\frac{x}{t}.
    \end{align*}
    It contradicts the assumption. Therefore, we get $\left(\frac{x}{t}\right)^\gamma - C\leq 0$.

    Finally, as $r^\gamma + x - r t - C$ is a convex function in $r$, we get $r^\gamma + x - r t - C\leq 0$ for $\frac{x}{t}\leq r\leq (x+C)^{\frac{1}{\gamma}} - \frac{4t}{\gamma}(x+C)^{\frac{2-\gamma}{\gamma}}$.
\end{proof}

The following technical lemma will be used in the proof of Proposition~\ref{prop:M[f]_integral} to bound a sum of two exponents, as in \eqref{eq:concaten_example1}.
\begin{lemma}\label{lem:concaten}
    Let $\beta>0$ and $C>0$ be constants. For $x,y>0$, we have
    \begin{align*}
        x^{\frac{\beta}{1+\beta}} + \frac{y^2}{Cx^{\frac{2}{1+\beta}}}\geq (x+y)^{\frac{\beta}{1+\beta}} - \left((1+C)^{\frac{\beta}{1+\beta}} + \frac{\beta}{1+\beta}C\right).
    \end{align*}
\end{lemma}
\begin{proof}
    If $y\leq Cx^{\frac{1}{1+\beta}}$, then by Lemma~\ref{lem:beta_pert} (for $x\geq 1$), we have
    \begin{align*}
        (x+y)^{\frac{\beta}{1+\beta}} - (1+C)^{\frac{\beta}{1+\beta}} - \frac{\beta}{1+\beta}C&\leq (x+Cx^{\frac{1}{1+\beta}})^{\frac{\beta}{1+\beta}} - (1+C)^{\frac{\beta}{1+\beta}}- \frac{\beta}{1+\beta}C\\
        &\leq \begin{cases}
            (1+C)^{\frac{\beta}{1+\beta}}- (1+C)^{\frac{\beta}{1+\beta}}& x\leq 1\\
            x^{\frac{\beta}{1+\beta}}+\frac{\beta}{1+\beta}C  - \frac{\beta}{1+\beta}C& x\geq 1
        \end{cases}\\
        &\leq x^{\frac{\beta}{1+\beta}}.
    \end{align*}
    If $y\geq Cx^{\frac{1}{1+\beta}}$, then applying Lemma~\ref{lem:beta_pert} again, we get
    \begin{align*}
        (x+y)^{\frac{\beta}{1+\beta}} - x^{\frac{\beta}{1+\beta}}\leq \frac{\beta}{1+\beta}yx^{-\frac{1}{1+\beta}}\leq \frac{y^2}{Cx^{\frac{2}{1+\beta}}}
    \end{align*}
    as
    \begin{align*}
        \frac{\beta}{1+\beta}C x^{\frac{1}{1+\beta}}\leq y.
    \end{align*}
    This proves the inequality.
\end{proof}

\section{A priori estimates for the macroscopic fields}\label{sec:Macro_quan}
In Section \ref{sec:Macro_quan}, we choose a specific inflow boundary data $g$ on $\partial \mathbb{R}^n_+$, which will generate the desired exponential ill-posedness for Theorem~\ref{thm:main_BGK}. We compute the sharp lower and upper bounds of the macroscopic fields of a solution with the inflow boundary data $g$ by controlling the nonlinear local Maxwellian part $\mathcal{M}(f)(t,x,v)$. These \textit{a priori} estimates will play crucial roles in showing the existence, uniqueness, and the exponential ill-posedness of the solution.

For $\alpha>0$, $1\leq\beta\leq 2$, and $m\geq 0$, we define the inflow boundary data $g(t,x_h,v)$ for $t\geq 0$ and $(x_h,v)\in \partial \mathbb{R}^n_+ \times \mathbb{R}^n$ by
\begin{align}\label{eq:boundary_data}
    g(t,x_h,v) = e^{-t}\langle v_h\rangle^{-m}e^{-\alpha |v_h|^\beta}\mathbf{1}_{|v_h|\geq \left||x_h-v_ht| + v_z\right|^{\frac{1}{1+\beta}}}\mathbf{1}_{v_z\geq 0}.
\end{align}
It will be the inflow data for Theorem~\ref{thm:main_BGK} exhibiting the ill-posedness. If we define
\begin{align*}
    g_0(x_h,v) = \langle v_h\rangle^{-m}e^{-\alpha|v_h|^\beta}\mathbf{1}_{|v_h|\geq \left||x_h| + v_z\right|^{\frac{1}{1+\beta}}}\mathbf{1}_{v_z\geq 0},
\end{align*}
then $g(t,x_h,v) = e^{-t}g_0(x_h-v_ht, v)$.

We recall that a solution of the BGK model $f$ under the initial data $f_0(x,v) = 0$ and the inflow boundary data $g$ should satisfy \eqref{eq:BGK_sol_def}:
\begin{align*}
    f(t,x,v) = e^{-t_b}g(t-t_b, x - vt_b, v)\mathbf{1}_{t\geq t_b} + \int_0^t e^{-(t - \tau)} \mathcal{M}(f)(\tau, x - v(t - \tau), v)\mathbf{1}_{z-v_z(t-\tau)> 0}\,d\tau.
\end{align*}

For any given constant $\epsilon>0$ in Theorem~\ref{thm:main_BGK}, we can scale $g$ by $\epsilon g$ and the solution $f$ by $\epsilon f$ to write
\begin{align*}
    \epsilon f(t,x,v) &= e^{-t_b}\epsilon g(t-t_b, x - vt_b, v)\mathbf{1}_{t\geq t_b} + \int_0^t e^{-(t - \tau)} \mathcal{M}(\epsilon f)(\tau, x - v(t - \tau), v)\mathbf{1}_{z-v_z(t-\tau)> 0}\,d\tau\\
    &=\epsilon\left(e^{-t_b}g(t-t_b, x - vt_b, v)\mathbf{1}_{t\geq t_b} + \int_0^t e^{-(t - \tau)} \mathcal{M}(f)(\tau, x - v(t - \tau), v)\mathbf{1}_{z-v_z(t-\tau)> 0}\,d\tau\right).
\end{align*}
Therefore, it is enough to assume $\epsilon = 1$ and use the original $g$ to prove Theorem~\ref{thm:main_BGK}.

We recall the macroscopic fields generated by the damped free transport of the inflow boundary data $g$, as in the second condition of Definition \ref{def:BGK_sol_space}.
\begin{definition}\label{Def:macro_0}
    For $t> 0$ and $x\in\mathbb{R}^n_+$, we define $(\rho_g, u_g, T_g)$ by
    \begin{align*}
        \rho_g(t,x)\left(1, u_g(t,x), nT_g(t,x) + u_g^2(t,x)\right) &\coloneqq \int_{\mathbb{R}^n} (1,v,|v|^2) e^{-t_b}g(t-t_b, x-vt_b, v)\mathbf{1}_{t\geq t_b}\,dv\\
        &= e^{-t}\int_{\mathbb{R}^n} (1,v,|v|^2) g_0(x_h-v_ht, v)\mathbf{1}_{t\geq t_b}\,dv.
    \end{align*}
\end{definition}
Note that when $t = 0$, $\rho_g(0,x) = 0$, so $u_g(0,x)$ and $T_g(0,x)$ are not determined. To match the vacuum initial data, we choose $u_g(0,x) = 0$ and $T_g(0,x) = 0$. It will not be a problem for our further analysis since $t=0$ is a measure-null set in the time interval.

The next proposition states that the macroscopic fields generated by the free transport of the inflow boundary data $g$ are well-posed with local-in-time explicit behavior estimates. Furthermore, we check that the free transport function is finite in the exponentially weighted norm. This shows that the boundary term itself does not cause weighted-norm divergence.
\begin{proposition}\label{prop:boundary_int}
    For $\alpha>0$, $1\leq\beta\leq 2$, $m\geq 0$, we set the inflow boundary data $g$ by \eqref{eq:boundary_data}. Then,
    \begin{align}\label{eq:macro_0_bound}
        \begin{split}
            &C_1\left(|x_h| + \frac{z}{t}+1\right)^{\frac{(n-m)-\beta}{1+\beta}} e^{-\alpha(|x_h| + \frac{z}{t}+1)^{\frac{\beta}{1+\beta}}}\leq \rho_g(t,x)\leq C_2\left(|x_h| + \frac{z}{t}+1\right)^{\frac{(n-m)-\beta}{1+\beta}} e^{-\alpha(|x_h| + \frac{z}{t}+1)^{\frac{\beta}{1+\beta}}},\\
            &\left|u_g(t,x) - \frac{z}{t}\hat{z}\right|\leq C_3\left(|x_h| + \frac{z}{t}+1\right)^{\frac{1}{1+\beta}},\\
            &C_4\left(|x_h| + \frac{z}{t}+1\right)^{\frac{2}{1+\beta}} \leq T_g(t,x)\leq C_5\left(|x_h| + \frac{z}{t}+1\right)^{\frac{2}{1+\beta}}
        \end{split}
    \end{align}
    for $0< t\leq \frac{1}{4}$, where $C_i>0$, $1\leq i\leq 5$, are constants depending only on $\alpha$, $\beta$, $m$, and dimension $n$. Moreover, for given $1\leq p,q\leq \infty$ and $\delta\geq 0$, suppose $m$ satisfies
    \begin{align}\label{eq:m_cond1}
        \frac{n+\beta}{p}+(1+\beta)\left(\delta + \frac{n}{q}\right)<m
    \end{align}
    with the convention $1/\infty =0$. Then, we have
    \begin{align*}
        \left\|\langle v\rangle^\delta e^{-t_b}e^{\alpha|v_h|^\beta}g(t-t_b,x_h-v_ht_b,v)\mathbf{1}_{v_z\geq \frac{z}{t}}\right\|_{L^q_x(\mathbb{R}^n_+, L^p_v(\mathbb{R}^n))}\leq C t^{1/q}(1+t^2)^{n/q}e^{-t},
    \end{align*}
    and
    \begin{align*}
        &\left\|\langle v\rangle^\delta e^{-t_b}e^{\alpha|v_h|^\beta}g(t-t_b,x_h-v_ht_b,v)\mathbf{1}_{v_z\geq \frac{z}{t}}\right\|_{L^p_v(\mathbb{R}^n, L^q_x(\mathbb{R}^n_+))}\\
        &=e^{-t}\left\|v_z^{1/q}e^\tau\langle v\rangle^\delta e^{\alpha|v_h|^\beta}g(\tau,x_h,v)\right\|_{L^p_v(\mathbb{R}^n, L^q_{\tau,x_h}([0,t]\times \partial \mathbb{R}^{n}_+))}\leq Ct^{1/q}e^{-t}
    \end{align*}
    for all $t\geq 0$, where the constant $C$ depends on $\alpha$, $\beta$, $m$, $n$, $p$, $q$, and $\delta$.
\end{proposition}

\begin{proof}
    We first compute the macroscopic fields $(\rho_g, u_g, T_g)$ step by step and then prove finiteness in the exponentially weighted norm. The main difficulty in estimating the macroscopic fields lies in handling the cutoff functions
    \begin{align*}
        \mathbf{1}_{|v_h|\geq \left||x_h-v_ht| + v_z\right|^{\frac{1}{1+\beta}}}\qquad \text{and}\qquad \mathbf{1}_{t\geq t_b} = \mathbf{1}_{v_z\geq \frac{z}{t}}.
    \end{align*}\\
    
    \noindent \textbullet{} \textit{$\rho_g$ computation.} Throughout the computation, it is convenient to employ the change of variable $v_z\mapsto v_z + \frac{z}{t}$, so we write
    \begin{align*}
        \rho_g(t,x) = e^{-t}\int_{\mathbb{R}^n} g_0(x_h-v_ht, v)\mathbf{1}_{v_z\geq \frac{z}{t}}\,dv = e^{-t}\int_{\mathbb{R}^n} \langle v_h\rangle^{-m}e^{-\alpha |v_h|^\beta}\mathbf{1}_{|v_h|\geq \left(|x_h-v_ht|+\frac{z}{t} + v_z\right)^{\frac{1}{1+\beta}}}\mathbf{1}_{v_z\geq 0}\,dv.
    \end{align*}
    Using the elementary bound
    \begin{align*}
        \left||x_h| - |v_h|t\right|\leq |x_h-v_ht|\leq |x_h|+|v_h|t,
    \end{align*}
    we bound $\rho_g(t,x)$ as follows:
    \begin{align*}
        \rho_g(t,x) &\leq \int_0^\infty \int_{\mathbb{R}^{n-1}} \langle v_h\rangle^{-m}e^{-\alpha|v_h|^\beta} \mathbf{1}_{|v_h|\geq \left(\left||x_h| - |v_h| t\right|+\frac{z}{t}+v_z\right)^{\frac{1}{1+\beta}}} \,dv_hdv_z,
        \intertext{and}
        \rho_g(t,x)&\geq e^{-\frac{1}{4}}\int_0^\infty \int_{\mathbb{R}^{n-1}} \langle v_h\rangle^{-m}e^{-\alpha|v_h|^\beta} \mathbf{1}_{|v_h|\geq \left((|x_h| + |v_h| t)+\frac{z}{t}+v_z\right)^{\frac{1}{1+\beta}}} \,dv_hdv_z
    \end{align*}
    for $0< t\leq \frac{1}{4}$. To employ Lemma~\ref{lem:gamma_pert}, we divide cases depending on the size of $|x_h| + \frac{z}{t}$.\\
    
    \noindent \textbf{(a)} If $|x_h| + \frac{z}{t}<1$, then we further divide $v_z$ into $0\leq v_z\leq 1$ and $v_z\geq 1$. We first assume $0\leq v_z\leq 1$. If $|v_h|= 2$ and $t\leq \frac{1}{4}$, then
    \begin{align*}
        ((|x_h| + |v_h| t)+\frac{z}{t}+v_z)^{\frac{1}{1+\beta}}\leq \left(\frac{1}{2}+2\right)^{\frac{1}{2}}\leq 2
    \end{align*}
    for $1\leq \beta\leq 2$. It implies that $|v_h|\geq ((|x_h| + |v_h| t)+\frac{z}{t}+v_z)^{\frac{1}{1+\beta}}$ for $|v_h|\geq 2$. Therefore, we get
    \begin{align*}
        &\int_0^1 \int_{\mathbb{R}^{n-1}} \langle v_h\rangle^{-m}e^{-\alpha|v_h|^\beta} \mathbf{1}_{|v_h|\geq \left(\left||x_h| - |v_h| t\right|+\frac{z}{t}+v_z\right)^{\frac{1}{1+\beta}}}\,dv_hdv_z\leq \int_0^1 \int_{\mathbb{R}^{n-1}} \langle v_h\rangle^{-m}e^{-\alpha|v_h|^\beta}\,dv_hdv_z\leq C'_1,
        \intertext{and}
        &\int_0^1 \int_{\mathbb{R}^{n-1}} \langle v_h\rangle^{-m}e^{-\alpha|v_h|^\beta} \mathbf{1}_{|v_h|\geq \left((|x_h|+|v_h|t)+\frac{z}{t}+v_z\right)^{\frac{1}{1+\beta}}}\,dv_hdv_z\geq \int_0^1 \int_{\mathbb{R}^{n-1}} \langle v_h\rangle^{-m}e^{-\alpha|v_h|^\beta}\mathbf{1}_{|v_h|\geq 2}\,dv_hdv_z\geq C'_2
    \end{align*}
    for some constants $C'_1$ and $C'_2$.
    
    If $v_z\geq 1$, using Lemma~\ref{lem:gamma_pert} for $C = \frac{z}{t} + v_z$, we can take the upper and lower bounds of the step function by
    \begin{align*}
        &\int_1^\infty \int_{\mathbb{R}^{n-1}} \langle v_h\rangle^{-m}e^{-\alpha|v_h|^\beta} \mathbf{1}_{|v_h|\geq \left(\left||x_h|-|v_h|t\right|+\frac{z}{t}+v_z\right)^{\frac{1}{1+\beta}}} \,dv_hdv_z \\
        &\leq \int_1^\infty \int_{\mathbb{R}^{n-1}} \langle v_h\rangle^{-m}e^{-\alpha|v_h|^\beta} \mathbf{1}_{|v_h|\geq \left(|x_h| + \frac{z}{t} + v_z\right)^{\frac{1}{1+\beta}} - \frac{4t}{1+\beta}\left(|x_h| + \frac{z}{t} + v_z\right)^{\frac{1-\beta}{1+\beta}}} \,dv_hdv_z,
        \intertext{and}
        &\int_1^\infty \int_{\mathbb{R}^{n-1}} \langle v_h\rangle^{-m}e^{-\alpha|v_h|^\beta} \mathbf{1}_{|v_h|\geq \left((|x_h|+|v_h|t)+\frac{z}{t}+v_z\right)^{\frac{1}{1+\beta}}} \,dv_hdv_z \\
        &\geq \int_1^\infty \int_{\mathbb{R}^{n-1}} \langle v_h\rangle^{-m}e^{-\alpha|v_h|^\beta} \mathbf{1}_{|v_h|\geq \left(|x_h| + \frac{z}{t} + v_z\right)^{\frac{1}{1+\beta}} + \frac{2t}{1+\beta}\left(|x_h| + \frac{z}{t} + v_z\right)^{\frac{1-\beta}{1+\beta}}} \,dv_hdv_z
    \end{align*}
    since $|x_h| + \frac{z}{t}+v_z\geq 1$ and $0<t\leq \frac{1}{4}$. As $|x_h|+\frac{z}{t}\leq 1$, each integral is bounded by
    \begin{align*}
        &\int_1^\infty \int_{\mathbb{R}^{n-1}} \langle v_h\rangle^{-m}e^{-\alpha|v_h|^\beta} \mathbf{1}_{|v_h|\geq v_z^{\frac{1}{1+\beta}} - \frac{4t}{1+\beta}(1 + v_z)^{\frac{1-\beta}{1+\beta}}} \,dv_hdv_z\leq C''_2,\\
        \intertext{and}
        &\int_1^\infty \int_{\mathbb{R}^{n-1}} \langle v_h\rangle^{-m}e^{-\alpha|v_h|^\beta} \mathbf{1}_{|v_h|\geq (1 + v_z)^{\frac{1}{1+\beta}} + \frac{2t}{1+\beta}(1+v_z)^{\frac{1-\beta}{1+\beta}}} \,dv_hdv_z\geq C''_1
    \end{align*}
    for some uniform constants $C_1''$ and $C_2''$. Combining the two cases $0\leq v_z\leq 1$ and $v_z\geq 1$, we have
    \begin{align}\label{eq:rho_0_bound1}
        C_1\leq \rho_g(t,x)\leq C_2
    \end{align}
    for some constants $C_1$ and $C_2$ when $|x_h| + \frac{z}{t}\leq 1$.\\
    
    \noindent \textbf{(b)} When $|x_h|+\frac{z}{t}>1$, then we directly use Lemma~\ref{lem:gamma_pert} for $|x_h| +\frac{z}{t} + v_z\geq 1$ to bound
    \begin{align*}
        &\int_0^\infty \int_{\mathbb{R}^{n-1}} \langle v_h\rangle^{-m}e^{-\alpha|v_h|^\beta} \mathbf{1}_{|v_h|\geq \left(\left||x_h|-|v_h|t\right|+\frac{z}{t}+v_z\right)^{\frac{1}{1+\beta}}} \,dv_hdv_z \\
        &\leq \int_0^\infty \int_{\mathbb{R}^{n-1}} \langle v_h\rangle^{-m}e^{-\alpha|v_h|^\beta} \mathbf{1}_{|v_h|\geq \left(|x_h| + \frac{z}{t} + v_z\right)^{\frac{1}{1+\beta}} - \frac{4t}{1+\beta}\left(|x_h| + \frac{z}{t} + v_z\right)^{\frac{1-\beta}{1+\beta}}} \,dv_hdv_z,
        \intertext{and}
        &\int_0^\infty \int_{\mathbb{R}^{n-1}} \langle v_h\rangle^{-m}e^{-\alpha|v_h|^\beta} \mathbf{1}_{|v_h|\geq \left((|x_h|+|v_h|t)+\frac{z}{t}+v_z\right)^{\frac{1}{1+\beta}}} \,dv_hdv_z \\
        &\geq \int_0^\infty \int_{\mathbb{R}^{n-1}} \langle v_h\rangle^{-m}e^{-\alpha|v_h|^\beta} \mathbf{1}_{|v_h|\geq \left(|x_h| + \frac{z}{t} + v_z\right)^{\frac{1}{1+\beta}} + \frac{2t}{1+\beta}\left(|x_h| + \frac{z}{t} + v_z\right)^{\frac{1-\beta}{1+\beta}}} \,dv_hdv_z.
    \end{align*}
    From (1) and (3) in Lemma~\ref{lem:beta_pert}, we can check that
    \begin{align*}
        \left(\left(|x_h| + \frac{z}{t} + v_z\right)^{\frac{1}{1+\beta}} + \frac{2t}{1+\beta}\left(|x_h| + \frac{z}{t} + v_z\right)^{\frac{1-\beta}{1+\beta}}\right)^\beta\leq \left(|x_h| + \frac{z}{t} + v_z\right)^{\frac{\beta}{1+\beta}} + \frac{8\beta t}{1+\beta},
        \intertext{and}
        \left(\left(|x_h| + \frac{z}{t} + v_z\right)^{\frac{1}{1+\beta}} - \frac{4t}{1+\beta}\left(|x_h| + \frac{z}{t} + v_z\right)^{\frac{1-\beta}{1+\beta}}\right)^\beta\geq \left(|x_h| + \frac{z}{t} + v_z\right)^{\frac{\beta}{1+\beta}} -\frac{4\beta t}{1+\beta}
    \end{align*}
    for $1\leq \beta\leq 2$. Using Lemma~\ref{lem:exp_integral} for $\delta_1 = 0$ and $\delta_2 = -m$, we have
    \begin{align}\label{eq:rho_upper_v_h}
        \begin{split}
            &\int_0^\infty \int_{\mathbb{R}^{n-1}} \langle v_h\rangle^{-m}e^{-\alpha|v_h|^\beta} \mathbf{1}_{|v_h|\geq \left(|x_h| + \frac{z}{t} + v_z\right)^{\frac{1}{1+\beta}} - \frac{4t}{1+\beta}\left(|x_h| + \frac{z}{t} + v_z\right)^{\frac{1-\beta}{1+\beta}}} \,dv_hdv_z\\
            &\leq C\int_0^\infty \left(\left(|x_h| + \frac{z}{t} + v_z\right)^{\frac{1}{1+\beta}}  + 1\right)^{n-1-\beta-m}e^{-\alpha(|x_h| + \frac{z}{t} + v_z)^{\frac{\beta}{1+\beta}}}\,dv_z\\
            &=C\int_{|x_h| + \frac{z}{t}}^\infty \left(v_z^{\frac{1}{1+\beta}}  + 1\right)^{n-1-\beta-m}e^{-\alpha v_z^{\frac{\beta}{1+\beta}}}\,dv_z.
        \end{split}
    \end{align}
    We again apply Lemma~\ref{lem:exp_integral} for $\delta_1 = 0$ and $\delta_2 = \frac{n-1-\beta-m}{1+\beta}$ to bound the final integral by
    \begin{align*}
        \int_{|x_h| + \frac{z}{t}}^\infty \left(v_z^{\frac{1}{1+\beta}}  + 1\right)^{n-1-\beta-m}e^{-\alpha v_z^{\frac{\beta}{1+\beta}}}\,dv_z\leq C_2\left(|x_h| + \frac{z}{t}+1\right)^{\frac{n-\beta-m}{1+\beta}}e^{-\alpha (|x_h| + \frac{z}{t})^{\frac{\beta}{1+\beta}}}.
    \end{align*}
    For the same reason, we can take a lower bound of the integral by
    \begin{align*}
        &\int_0^\infty \int_{\mathbb{R}^{n-1}} \langle v_h\rangle^{-m}e^{-\alpha|v_h|^\beta} \mathbf{1}_{|v_h|\geq \left(|x_h| + \frac{z}{t} + v_z\right)^{\frac{1}{1+\beta}} + \frac{2t}{1+\beta}\left(|x_h| + \frac{z}{t} + v_z\right)^{\frac{1-\beta}{1+\beta}}} \,dv_hdv_z\\
        &\geq C_1\left(|x_h| + \frac{z}{t}+1\right)^{\frac{n-\beta-m}{1+\beta}}e^{-\alpha (|x_h| + \frac{z}{t})^{\frac{\beta}{1+\beta}}}
    \end{align*}
    for some constants $C_1$ and $C_2$. It proves
    \begin{align}\label{eq:rho_0_bound2}
        C_1\left(|x_h| + \frac{z}{t}+1\right)^{\frac{(n-m)-\beta}{1+\beta}} e^{-\alpha(|x_h| + \frac{z}{t}+1)^{\frac{\beta}{1+\beta}}}\leq \rho_g(t,x)\leq C_2\left(|x_h| + \frac{z}{t}+1\right)^{\frac{(n-m)-\beta}{1+\beta}} e^{-\alpha(|x_h| + \frac{z}{t}+1)^{\frac{\beta}{1+\beta}}}
    \end{align}
    when $|x_h|+\frac{z}{t}\geq 1$. (As $\frac{\beta}{1+\beta}\leq 1$, it is safe to add $1$ to the exponent $(|x_h| + \frac{z}{t})^{\frac{\beta}{1+\beta}}$ like above.)
    
    Combining \eqref{eq:rho_0_bound1} and \eqref{eq:rho_0_bound2}, we conclude that \eqref{eq:rho_0_bound2} holds for whole $0< t\leq \frac{1}{4}$ and $x\in\mathbb{R}^n_+$ for some appropriate constants $C_1$ and $C_2$. Note that the left and right side in \eqref{eq:rho_0_bound2} converge to $0$ as $|x_h| + \frac{z}{t}\rightarrow \infty$. For the later use, we also note that \eqref{eq:rho_0_bound2} holds for $m\leq 0$ as Lemma~\ref{lem:exp_integral} still true for the $m\leq 0$ case.\\
    
    \noindent \textbullet{} \textit{$u_g$ computation.}
    
    We again take the change of variable $v_z\mapsto v_z+\frac{z}{t}$ to write
    \begin{align*}
        \rho_g\left(u_g-\frac{z}{t}\hat{z}\right)(t,x) &= e^{-t}\int_{\mathbb{R}^n} \left(v-\frac{z}{t}\hat{z}\right) g_0(x_h-v_ht, v)\mathbf{1}_{v_z\geq \frac{z}{t}}\,dv\\
        &=e^{-t}\int_0^\infty \int_{\mathbb{R}^{n-1}} v \langle v_h\rangle^{-m}e^{-\alpha|v_h|^\beta} \mathbf{1}_{|v_h|\geq \left(|x_h-v_ht|+\frac{z}{t}+v_z\right)^{\frac{1}{1+\beta}}}\,dv_h dv_z.
    \end{align*}
    We take the upper bound of the integral by
    \begin{align*}
        &\int_0^\infty \int_{\mathbb{R}^{n-1}} |v|\langle v_h\rangle^{-m}e^{-\alpha|v_h|^\beta} \mathbf{1}_{|v_h|\geq \left(|x_h-v_ht|+\frac{z}{t}+v_z\right)^{\frac{1}{1+\beta}}}\,dv_h dv_z\\
        &\leq \int_0^\infty \int_{\mathbb{R}^{n-1}} (|v_h| + v_z)\langle v_h\rangle^{-m}e^{-\alpha|v_h|^\beta} \mathbf{1}_{|v_h|\geq \left(|x_h-v_ht|+\frac{z}{t}+v_z\right)^{\frac{1}{1+\beta}}}\,dv_h dv_z.
    \end{align*}
    For $|v_h| \langle v_h\rangle^{-m}$, we can use the results for $\rho_g$ in \eqref{eq:rho_0_bound2} for $m\mapsto m-1$ and obtain
    \begin{align*}
        \begin{split}
            &\int_0^\infty \int_{\mathbb{R}^{n-1}} \langle v_h\rangle^{-m+1}e^{-\alpha|v_h|^\beta} \mathbf{1}_{|v_h|\geq \left(|x_h-v_ht|+\frac{z}{t}+v_z\right)^{\frac{1}{1+\beta}}}\,dv_h dv_z\\
            &\leq C\left(|x_h| + \frac{z}{t}+1\right)^{\frac{(n+1-m)-\beta}{1+\beta}} e^{-\alpha(|x_h| + \frac{z}{t}+1)^{\frac{\beta}{1+\beta}}}
        \end{split}
    \end{align*}
    for a constant $C$.
    
    For $v_z \langle v_h\rangle^{-m}$, we first bound the $v_h$-integral using \eqref{eq:rho_upper_v_h}. Then,
    \begin{align*}
        &\int_0^\infty v_z\int_{\mathbb{R}^{n-1}}\langle v_h\rangle^{-m}e^{-\alpha|v_h|^\beta} \mathbf{1}_{|v_h|\geq \left(|x_h-v_ht|+\frac{z}{t}+v_z\right)^{\frac{1}{1+\beta}}} \,dv_hdv_z\\
        &\leq C\int_0^\infty v_z\left(\left(|x_h| + \frac{z}{t} + v_z\right)^{\frac{1}{1+\beta}}  + 1\right)^{n-1-\beta-m}e^{-\alpha(|x_h| + \frac{z}{t} + v_z)^{\frac{\beta}{1+\beta}}}\,dv_z\\
        &= C\int_{|x_h|+\frac{z}{t}}^\infty \left(v_z-|x_h|-\frac{z}{t}\right)\left(v_z^{\frac{1}{1+\beta}}  + 1\right)^{n-1-\beta-m}e^{-\alpha v_z^{\frac{\beta}{1+\beta}}}\,dv_z.
    \end{align*}
    Employing Lemma~\ref{lem:exp_integral} for $(\delta_1,\delta_2) = (1, \frac{n-1-\beta-m}{1+\beta})$ to the final $v_z$-integral, we have
    \begin{align*}
        \int_{|x_h|+\frac{z}{t}}^\infty \left(v_z-|x_h|-\frac{z}{t}\right)\left(v_z^{\frac{1}{1+\beta}}  + 1\right)^{n-1-\beta-m}e^{-\alpha v_z^{\frac{\beta}{1+\beta}}}\,dv_z\leq C\left(|x_h| + \frac{z}{t}+1\right)^{\frac{n-\beta-m+1}{1+\beta}}e^{-\alpha (|x_h| + \frac{z}{t})^{\frac{\beta}{1+\beta}}}.
    \end{align*}
    Combining the two results and dividing it by $\rho_g$ in \eqref{eq:rho_0_bound2}, we get a constant $C_3$ such that
    \begin{align}\label{eq:u_0_bound}
        \left|u_g - \frac{z}{t}\hat{z}\right|(t,x)\leq C_3\left(|x_h| + \frac{z}{t}+1\right)^{\frac{1}{1+\beta}}
    \end{align}
    for $0<t\leq \frac{1}{4}$ and $x\in\mathbb{R}^n_+$.\\
    
    \noindent \textbullet{} \textit{$T_g$ computation.}
    
    Taking the change of variable $v_z\mapsto v_z+\frac{z}{t}$, we write
    \begin{align*}
        n\rho_gT_g(t,x)&= e^{-t}\int_{\mathbb{R}^n}|v-u_g(t,x)|^2 \langle v_h\rangle^{-m}e^{-\alpha|v_h|^\beta}\mathbf{1}_{|v_h|\geq (|x_h-v_ht| + v_z)^{\frac{1}{1+\beta}}}\mathbf{1}_{v_z\geq \frac{z}{t}}\,dv\\
        &=e^{-t}\int_0^\infty \int_{\mathbb{R}^{n-1}} \left|v -\left(u_g(t,x) - \frac{z}{t}\hat{z}\right)\right|^2 \langle v_h\rangle^{-m}e^{-\alpha|v_h|^\beta}\mathbf{1}_{|v_h|\geq (|x_h-v_ht| + \frac{z}{t} + v_z)^{\frac{1}{1+\beta}}}\,dv_h dv_z.
    \end{align*}
    If $|x_h| + \frac{z}{t}\leq 2$, by \eqref{eq:u_0_bound}, we have
    \begin{align*}
        \left|u_g - \frac{z}{t}\hat{z}\right|^2\leq 3^{\frac{2}{1+\beta}} C^2_3.
    \end{align*}
    Therefore, the integral is well-bounded below and above by constants as in the $\rho_g$ and $\rho_g\left(u_g-\frac{z}{t}\hat{z}\right)$ computations, so we assume $|x_h|+\frac{z}{t}\geq 2$. Applying Lemma~\ref{lem:gamma_pert} again, we have
    \begin{align*}
        &\int_0^\infty \int_{\mathbb{R}^{n-1}} \left|v -\left(u_g(t,x) - \frac{z}{t}\hat{z}\right)\right|^2 \langle v_h\rangle^{-m}e^{-\alpha|v_h|^\beta}\mathbf{1}_{|v_h|\geq \left(|x_h-v_ht| + \frac{z}{t} + v_z\right)^{\frac{1}{1+\beta}}}\,dv_h dv_z\\
        &\leq \int_0^\infty \int_{\mathbb{R}^{n-1}} \left|v -\left(u_g(t,x) - \frac{z}{t}\hat{z}\right)\right|^2 \langle v_h\rangle^{-m}e^{-\alpha|v_h|^\beta}\mathbf{1}_{|v_h|\geq \left(|x_h| + \frac{z}{t} + v_z\right)^{\frac{1}{1+\beta}} - \frac{4t}{1+\beta}\left(|x_h| + \frac{z}{t} + v_z\right)^{\frac{1-\beta}{1+\beta}}}\,dv_h dv_z,\\
        \intertext{and}
        &\int_0^\infty \int_{\mathbb{R}^{n-1}} \left|v -\left(u_g(t,x) - \frac{z}{t}\hat{z}\right)\right|^2 \langle v_h\rangle^{-m}e^{-\alpha|v_h|^\beta}\mathbf{1}_{|v_h|\geq \left(|x_h-v_ht| + v_z + \frac{z}{t}\right)^{\frac{1}{1+\beta}}}\,dv_h dv_z\\
        &\geq \int_0^\infty \int_{\mathbb{R}^{n-1}} \left|v -\left(u_g(t,x) - \frac{z}{t}\hat{z}\right)\right|^2 \langle v_h\rangle^{-m}e^{-\alpha|v_h|^\beta}\mathbf{1}_{|v_h|\geq \left(|x_h| + \frac{z}{t} + v_z\right)^{\frac{1}{1+\beta}} + \frac{2t}{1+\beta}\left(|x_h| + \frac{z}{t} + v_z\right)^{\frac{1-\beta}{1+\beta}}}\,dv_h dv_z.
    \end{align*}
    
    We first consider the upper bound of the integral. Decomposing the square term, we write
    \begin{align*}
        &\int_0^\infty \int_{\mathbb{R}^{n-1}} \left|v -\left(u_g(t,x) - \frac{z}{t}\hat{z}\right)\right|^2 \langle v_h\rangle^{-m}e^{-\alpha|v_h|^\beta}\mathbf{1}_{|v_h|\geq \left(|x_h| + \frac{z}{t} + v_z\right)^{\frac{1}{1+\beta}} - \frac{4t}{1+\beta}\left(|x_h| + \frac{z}{t} + v_z\right)^{\frac{1-\beta}{1+\beta}}}\,dv_h dv_z\\
        &\leq 2\int_0^\infty \int_{\mathbb{R}^{n-1}} \left(\left(|v_h|^2 + |v_z|^2\right) +\left(u_g(t,x) - \frac{z}{t}\hat{z}\right)^2\right)\langle v_h\rangle^{-m}e^{-\alpha|v_h|^\beta}\mathbf{1}_{|v_h|\geq \left(|x_h| + \frac{z}{t} + v_z\right)^{\frac{1}{1+\beta}} - \frac{4t}{1+\beta}\left(|x_h| + \frac{z}{t} + v_z\right)^{\frac{1-\beta}{1+\beta}}}\,dv_h dv_z.
    \end{align*}
    Iteratively applying Lemma~\ref{lem:exp_integral} as in the $u_g$ computation, we have
    \begin{align*}
        &\int_0^\infty \int_{\mathbb{R}^{n-1}} \left(|v_h|^2 + |v_z|^2\right)\langle v_h\rangle^{-m}e^{-\alpha|v_h|^\beta}\mathbf{1}_{|v_h|\geq \left(|x_h| + \frac{z}{t} + v_z\right)^{\frac{1}{1+\beta}} - \frac{4t}{1+\beta}\left(|x_h| + \frac{z}{t} + v_z\right)^{\frac{1-\beta}{1+\beta}}}\,dv_h dv_z\\
        &\leq C\left(|x_h| + \frac{z}{t}+1\right)^{\frac{(n+2-m)-\beta}{1+\beta}} e^{-\alpha(|x_h| + \frac{z}{t}+1)^{\frac{\beta}{1+\beta}}}
    \end{align*}
    for some constant $C$. Also, from \eqref{eq:u_0_bound}, we take the upper bound of the integral of the second term by
    \begin{align*}
        &\left(u_g(t,x) - \frac{z}{t}\hat{z}\right)^2\int_0^\infty \int_{\mathbb{R}^{n-1}} \langle v_h\rangle^{-m}e^{-\alpha|v_h|^\beta}\mathbf{1}_{|v_h|\geq \left(|x_h| + \frac{z}{t} + v_z\right)^{\frac{1}{1+\beta}} - \frac{4t}{1+\beta}\left(|x_h| + \frac{z}{t} + v_z\right)^{\frac{1-\beta}{1+\beta}}}\,dv_h dv_z\\
        &\leq C\left(|x_h| + \frac{z}{t}+1\right)^{\frac{(n+2-m)-\beta}{1+\beta}} e^{-\alpha(|x_h| + \frac{z}{t}+1)^{\frac{\beta}{1+\beta}}}.
    \end{align*}
    Combining the two upper bounds, we obtain
    \begin{align}\label{eq:T_g_upper_1}
        \begin{split}
            &\int_0^\infty \int_{\mathbb{R}^{n-1}} \left|v -\left(u_g(t,x) - \frac{z}{t}\hat{z}\right)\right|^2 \langle v_h\rangle^{-m}e^{-\alpha|v_h|^\beta}\mathbf{1}_{|v_h|\geq \left(|x_h-v_ht| + \frac{z}{t} + v_z\right)^{\frac{1}{1+\beta}}}\,dv_h dv_z\\
            &\leq C'_5\left(|x_h| + \frac{z}{t}+1\right)^{\frac{(n+2-m)-\beta}{1+\beta}} e^{-\alpha(|x_h| + \frac{z}{t}+1)^{\frac{\beta}{1+\beta}}}
        \end{split}
    \end{align}
    for a constant $C'_5$.
    
    For the lower bound of the integral, we need additional techniques to obtain the same asymptotic behavior as the upper bound. Our $v_h$ integration for the lower bound is on the region $|v_h|\geq R$, where
    \begin{align*}
        R=\left(|x_h| + \frac{z}{t} + v_z\right)^{\frac{1}{1+\beta}} + \frac{2t}{1+\beta}\left(|x_h| + \frac{z}{t} + v_z\right)^{\frac{1-\beta}{1+\beta}}.
    \end{align*}
    For fixed $t$ and $x$, $u_g(t,x) - \frac{z}{t}\hat{z}$ is a fixed vector in $\mathbb{R}^n$. If we measure the distance of each point on the sphere $|v_h| = R$ from the $\left(u_g(t,x) - \frac{z}{t}\hat{z}\right)_h = u_{g,h}(t,x)$, where $u_{g,h}$ is the horizontal component of $u_g$, the hemisphere opposite to $u_{g,h}(t,x)$ should have a distance not less than $R$. It implies that
    \begin{align}\label{eq:set_inclusion}
        \begin{split}
            &\left\{v_h:|v_h|\geq \left(|x_h| + \frac{z}{t} + v_z\right)^{\frac{1}{1+\beta}} + \frac{2t}{1+\beta}\left(|x_h| + \frac{z}{t} + v_z\right)^{\frac{1-\beta}{1+\beta}}, v_h\cdot u_{g,h}\leq 0\right\}\\
            &\subset \left\{v_h:|v_h|\geq \left(|x_h| + \frac{z}{t} + v_z\right)^{\frac{1}{1+\beta}} + \frac{2t}{1+\beta}\left(|x_h| + \frac{z}{t} + v_z\right)^{\frac{1-\beta}{1+\beta}}\right\}\cap \left\{v_h:\left|v_h -u_{g,h}\right|\geq\left(|x_h| + \frac{z}{t} + v_z\right)^{\frac{1}{1+\beta}}\right\}.
        \end{split}
    \end{align}
    
    We use this idea to compute a lower bound for the integral. We first take the horizontal component of the square term and write
    \begin{align*}
        &\int_0^\infty \int_{\mathbb{R}^{n-1}} \left|v -\left(u_g(t,x) - \frac{z}{t}\hat{z}\right)\right|^2 \langle v_h\rangle^{-m}e^{-\alpha|v_h|^\beta}\mathbf{1}_{|v_h|\geq \left(|x_h| + \frac{z}{t} + v_z\right)^{\frac{1}{1+\beta}} + \frac{2t}{1+\beta}\left(|x_h| + \frac{z}{t} + v_z\right)^{\frac{1-\beta}{1+\beta}}}\\
        &\geq \int_0^\infty \int_{\mathbb{R}^{n-1}} \left|v_h -\left(u_g(t,x) - \frac{z}{t}\hat{z}\right)_h\right|^2 \langle v_h\rangle^{-m}e^{-\alpha|v_h|^\beta}\mathbf{1}_{|v_h|\geq \left(|x_h| + \frac{z}{t} + v_z\right)^{\frac{1}{1+\beta}} + \frac{2t}{1+\beta}\left(|x_h| + \frac{z}{t} + v_z\right)^{\frac{1-\beta}{1+\beta}}}\\
        &\qquad \times \mathbf{1}_{\left|v_h -\left(u_g(t,x) - \frac{z}{t}\hat{z}\right)_h\right|\geq \left(|x_h| + \frac{z}{t} + v_z\right)^{\frac{1}{1+\beta}}}\,dv_h dv_z\\
        &\geq C\left(|x_h| + \frac{z}{t}+1\right)^{\frac{2}{1+\beta}}\int_0^\infty \int_{\mathbb{R}^{n-1}} \langle v_h\rangle^{-m}e^{-\alpha|v_h|^\beta}\mathbf{1}_{|v_h|\geq \left(|x_h| + \frac{z}{t} + v_z\right)^{\frac{1}{1+\beta}} + \frac{2t}{1+\beta}\left(|x_h| + \frac{z}{t} + v_z\right)^{\frac{1-\beta}{1+\beta}}}\\
        &\qquad \times \mathbf{1}_{\left|v_h -u_{g,h}\right|\geq \left(|x_h| + \frac{z}{t} + v_z\right)^{\frac{1}{1+\beta}}}\,dv_h dv_z.
    \end{align*}
    By the set inclusion \eqref{eq:set_inclusion} and the $\rho_g$ computation \eqref{eq:rho_0_bound2}, we have
    \begin{align*}
        &\int_0^\infty \int_{\mathbb{R}^{n-1}} \langle v_h\rangle^{-m}e^{-\alpha|v_h|^\beta}\mathbf{1}_{|v_h|\geq \left(|x_h| + \frac{z}{t} + v_z\right)^{\frac{1}{1+\beta}} + \frac{2t}{1+\beta}\left(|x_h| + \frac{z}{t} + v_z\right)^{\frac{1-\beta}{1+\beta}}}\mathbf{1}_{\left|v_h -u_{g,h}\right|\geq \left(|x_h| + \frac{z}{t} + v_z\right)^{\frac{1}{1+\beta}}}\,dv_h dv_z\\
        &\geq \frac{1}{2}\int_0^\infty \int_{\mathbb{R}^{n-1}} \langle v_h\rangle^{-m}e^{-\alpha|v_h|^\beta}\mathbf{1}_{|v_h|\geq \left(|x_h| + \frac{z}{t} + v_z\right)^{\frac{1}{1+\beta}} + \frac{2t}{1+\beta}\left(|x_h| + \frac{z}{t} + v_z\right)^{\frac{1-\beta}{1+\beta}}}\,dv_h dv_z\\
        &\geq \frac{C_1}{2}\left(|x_h| + \frac{z}{t}+1\right)^{\frac{(n-m)-\beta}{1+\beta}} e^{-\alpha(|x_h| + \frac{z}{t}+1)^{\frac{\beta}{1+\beta}}}.
    \end{align*}
    
    Finally, we assemble the lower bound results and combine with the upper bound result \eqref{eq:T_g_upper_1}. Then, we obtain
    \begin{align*}
        C'_4\left(|x_h| + \frac{z}{t}+1\right)^{\frac{(n+2-m)-\beta}{1+\beta}} e^{-\alpha(|x_h| + \frac{z}{t}+1)^{\frac{\beta}{1+\beta}}}\leq n\rho_gT_g(t,x)\leq C'_5 \left(|x_h| + \frac{z}{t}+1\right)^{\frac{(n+2-m)-\beta}{1+\beta}} e^{-\alpha(|x_h| + \frac{z}{t}+1)^{\frac{\beta}{1+\beta}}}
    \end{align*}
    for some constants $C'_4$ and $C'_5$. Dividing it by $\rho_g$ in \eqref{eq:rho_0_bound2}, we obtain
    \begin{align*}
        C_4\left(|x_h| + \frac{z}{t}+1\right)^{\frac{2}{1+\beta}}\leq T_g(t,x)\leq C_5 \left(|x_h| + \frac{z}{t}+1\right)^{\frac{2}{1+\beta}}
    \end{align*}
    for some constants $C_4$ and $C_5$ when $0<t\leq \frac{1}{4}$ and $x\in\mathbb{R}^n_+$.\\
    
    \noindent \textbullet{} {\bf Finiteness of the boundary data in the exponential norm:}
    
    Now, we will show that the boundary data is finite in the exponentially weighted norm. Multiplying the exponential weight $w^h_{\alpha,\beta,\delta}$ to $e^{-t_b}g(t-t_b,x_h-v_h t_b, v)\mathbf{1}_{v_z\geq \frac{z}{t}}$, we have
    \begin{align*}
        \langle v\rangle^\delta e^{-t_b}e^{\alpha|v_h|^\beta}g(t-t_b,x_h-v_ht_b,v)\mathbf{1}_{v_z\geq \frac{z}{t}} = e^{-t}\langle v\rangle^\delta \langle v_h\rangle^{-m}\mathbf{1}_{|v_h|\geq \left(|x_h-v_h t| + v_z\right)^{\frac{1}{1+\beta}}}\mathbf{1}_{v_z\geq \frac{z}{t}}.
    \end{align*}
    We first estimate the $L^q_x(\mathbb{R}^n_+,L^p_v(\mathbb{R}^n))$ norm. We initially consider $1\leq p<\infty$ since the case $p=\infty$ easily follows by a similar argument. The cutoff condition implies that $\frac{z}{t}\leq v_z\leq |v_h|^{1+\beta}$. Consequently,
    \begin{align*}
        \langle v\rangle^\delta \langle v_h\rangle^{-m}\mathbf{1}_{|v_h|\geq \left(|x_h-v_h t|+v_z\right)^{\frac{1}{1+\beta}}} \mathbf{1}_{v_z\geq \frac{z}{t}}\leq \langle v\rangle^\delta \langle v_h\rangle^{-m}\mathbf{1}_{|v_h|\geq \left(|x_h-v_h t|+\frac{z}{t}\right)^{\frac{1}{1+\beta}}}\mathbf{1}_{\frac{z}{t}\leq v_z\leq |v_h|^{1+\beta}}.
    \end{align*}
    We take $v_z$-integral to obtain
    \begin{align}\label{eq:LqxLpv_0}
        \begin{split}
            &\int_{\mathbb{R}}\left\langle v\right\rangle^{p\delta} \langle v_h\rangle^{-pm}\mathbf{1}_{|v_h|\geq \left(|x_h-v_h t|+\frac{z}{t}\right)^{\frac{1}{1+\beta}}}\mathbf{1}_{\frac{z}{t}\leq v_z\leq |v_h|^{1+\beta}}\,dv_z\\
            &=\langle v_h\rangle^{-pm}\mathbf{1}_{|v_h|\geq \left(|x_h-v_h t|+\frac{z}{t}\right)^{\frac{1}{1+\beta}}}\int_{\frac{z}{t}}^{|v_h|^{1+\beta}}\left\langle v\right\rangle^{p\delta} \,dv_z\\
            &\leq C\langle v_h\rangle^{-pm + (1+\beta)(p\delta + 1)}\mathbf{1}_{|v_h|\geq \left(|x_h-v_h t|+\frac{z}{t}\right)^{\frac{1}{1+\beta}}}.
        \end{split}
    \end{align}
    
    It remains to estimate the $v_h$-integral on the the right-hand side of \eqref{eq:LqxLpv_0}. We first consider the case $|x_h| + \frac{z}{t}\geq 1+t^2$. If
    \begin{align*}
        |v_h|\leq \frac{1}{2}(|x_h| + \frac{z}{t})^{\frac{1}{1+\beta}},
    \end{align*}
    then we have
    \begin{align*}
        |x_h-v_ht| + \frac{z}{t}&\geq \left(|x_h| + \frac{z}{t}\right) - |v_h|t\\
        &\geq \left(|x_h| + \frac{z}{t}\right) - \frac{1}{2}\left(|x_h| + \frac{z}{t}\right)^{\frac{1}{1+\beta}}\left(|x_h| + \frac{z}{t}\right)^{\frac{1}{2}}\\
        &\geq \frac{1}{2}(|x_h| + \frac{z}{t})\\
        &\geq |v_h|^{1+\beta}.
    \end{align*}
    It means that $|v_h|\geq \left(|x_h-v_h t|+\frac{z}{t}\right)^{\frac{1}{1+\beta}}$ implies $|v_h|\geq \frac{1}{2}(|x_h| + \frac{z}{t})^{\frac{1}{1+\beta}}$. Using this fact, we have
    \begin{align*}
        \mathbf{1}_{|v_h|\geq \left(|x_h-v_h t|+\frac{z}{t}\right)^{\frac{1}{1+\beta}}}\leq \mathbf{1}_{|v_h|\geq \frac{1}{2}(|x_h| + \frac{z}{t})^{\frac{1}{1+\beta}}}.
    \end{align*}
    So, we bound the remaining $v_h$-integral by
    \begin{align}\label{eq:LqxLpv_1}
        \begin{split}
            &\int_{\mathbb{R}^{n-1}}\langle v_h\rangle^{-pm + (1+\beta)(p\delta + 1)}\mathbf{1}_{|v_h|\geq \left(|x_h-v_h t|+\frac{z}{t}\right)^{\frac{1}{1+\beta}}} \mathbf{1}_{|x_h| + \frac{z}{t}\geq 1 + t^2}\,dv_h\\
            &\leq \mathbf{1}_{|x_h|+\frac{z}{t}\geq 1 + t^2}\int_{\mathbb{R}^{n-1}} \langle v_h\rangle^{-pm + (1+\beta)(p\delta + 1)} \mathbf{1}_{|v_h|\geq \frac{1}{2}(|x_h| + \frac{z}{t})^{\frac{1}{1+\beta}}}\,dv_h\\
            &\leq C\left\langle |x_h| + \frac{z}{t} \right\rangle^{p\delta +\frac{n+\beta-pm}{1+\beta}}\mathbf{1}_{|x_h|+\frac{z}{t}\geq 1 + t^2}
        \end{split}
    \end{align}
    when $p\delta +\frac{n+\beta-pm}{1+\beta}<0$.
    
    If $0\leq |x_h|+\frac{z}{t}\leq 1 + t^2$, then by the direct computation, we have
    \begin{align}\label{eq:LqxLpv_2}
        \begin{split}
            &\int_{\mathbb{R}^{n-1}}\langle v_h\rangle^{-pm + (1+\beta)(p\delta + 1)}\mathbf{1}_{|v_h|\geq \left(|x_h-v_h t|+\frac{z}{t}\right)^{\frac{1}{1+\beta}}}\mathbf{1}_{|x_h|+\frac{z}{t}\leq 1 + t^2}\, dv_h\\
            &\leq \mathbf{1}_{|x_h|+\frac{z}{t}\leq 1 + t^2}\int_{\mathbb{R}^{n-1}}\langle v_h\rangle^{-pm + (1+\beta)(p\delta + 1)}\, dv_h\\
            &\leq C\mathbf{1}_{|x_h|+\frac{z}{t}\leq 1 + t^2}
        \end{split}
    \end{align}
    for some constant $C$ when $p\delta +\frac{n+\beta-pm}{1+\beta}<0$. Combining \eqref{eq:LqxLpv_0}, \eqref{eq:LqxLpv_1}, and \eqref{eq:LqxLpv_2}, we get
    \begin{align*}
        &\left(\int_{\mathbb{R}^n}\left(\langle v\rangle^\delta e^{\alpha|v_h|^\beta}e^{-t_b}g(t-t_b,x_h-v_ht_b,v)\mathbf{1}_{v_z\geq \frac{z}{t}}\right)^p\,dv\right)^{1/p}\\
        &\leq C\left(\left\langle |x_h| + \frac{z}{t} \right\rangle^{\delta+\frac{(n+\beta)/p-m}{1+\beta}}\mathbf{1}_{|x_h|+\frac{z}{t}\geq 1 + t^2} + \mathbf{1}_{|x_h|+\frac{z}{t}\leq 1 + t^2}\right)
    \end{align*}
    for $1\leq p< \infty$.
    
    Next, suppose $1\leq q<\infty$. If we choose $m$ satisfying $\delta+\frac{(n+\beta)/p-m}{1+\beta}<-\frac{n}{q}$,
    \begin{align*}
        \int_{\mathbb{R}^{n-1}}\int_0^\infty \left\langle |x_h| + \frac{z}{t}\right\rangle^{q\left(\delta+\frac{(n+\beta)/p-m}{1+\beta}\right)}\,dz dx_h&=t\int_{\mathbb{R}^{n-1}}\int_0^\infty \left\langle |x_h| + z\right\rangle^{q\left(\delta+\frac{(n+\beta)/p-m}{1+\beta}\right)}\,dz dx_h\\
        &<Ct
    \end{align*}
    for some constant $C$. Also,
    \begin{align*}
        \int_{\mathbb{R}^{n-1}}\int_0^\infty \mathbf{1}_{|x_h|+\frac{z}{t}\leq 1 + t^2}\,dzdx_h\leq C(1+t^2)^n t.
    \end{align*}
    It shows that
    \begin{align}\label{eq:LqxLpv_3}
        \left\|e^{-t_b}\langle v\rangle^\delta e^{\alpha|v_h|^\beta}g(t-t_b,x_h-v_ht_b,v)\mathbf{1}_{v_z\geq \frac{z}{t}}\right\|_{L^q_x(\mathbb{R}^n_+, L^p_v(\mathbb{R}^n))}\leq C(1+t^2)^{n/q} t^{1/q}e^{-t}
    \end{align}
    when
    \begin{align*}
        \frac{n+\beta}{p}+\frac{(n+q\delta)(1+\beta)}{q}<m
    \end{align*}
    for $1\leq q<\infty$. Note that the $e^{-t}$ term is from the definition of the $g$ in \eqref{eq:boundary_data}. When $p=\infty$ or $q=\infty$, we can repeat a similar calculation and conclude that the same results still hold with the convention $1/\infty = 0$. Therefore, for any fixed $1\leq p,q\leq \infty$, the boundary data $g$ satisfies \eqref{eq:LqxLpv_3} when $m$ meets the condition \eqref{eq:m_cond1}.\\
    
    Next, we do the $x$ integral first to compute 
    \begin{align}
        &\left\|e^{-t_b}\langle v\rangle^\delta e^{\alpha|v_h|^\beta}g(t-t_b,x_h-v_ht_b,v)\mathbf{1}_{v_z\geq \frac{z}{t}}\right\|_{L^p_v(\mathbb{R}^n, L^q_x(\mathbb{R}^n_+))}\notag\\
        &=e^{-t}\left\|v_z^{1/q}e^\tau\langle v\rangle^\delta e^{\alpha|v_h|^\beta}g(\tau,x_h,v)\right\|_{L^p_v(\mathbb{R}^n, L^q_{\tau,x_h}([0,t]\times \partial \mathbb{R}^{n}_+))}\notag\\
        &=e^{-t}\left\|v_z^{1/q}\langle v\rangle^\delta\langle v_h\rangle^{-m}\mathbf{1}_{|v_h|\geq \left||x_h-v_h\tau| + v_z\right|^{\frac{1}{1+\beta}}}\mathbf{1}_{v_z\geq 0}\right\|_{L^p_v(\mathbb{R}^n, L^q_{\tau,x_h}([0,t]\times \partial \mathbb{R}^{n}_+))}\notag\\
        &=t^{1/q}e^{-t}\left\|v_z^{1/q}\langle v\rangle^\delta\langle v_h\rangle^{-m}\mathbf{1}_{|v_h|\geq \left||x_h| + v_z\right|^{\frac{1}{1+\beta}}}\mathbf{1}_{v_z\geq 0}\right\|_{L^p_v(\mathbb{R}^n, L^q_{x_h}(\partial \mathbb{R}^{n}_+))}.\label{eq:g_LpvLqx_form}
    \end{align}
    For the first equality, we used Lemma~\ref{lem:trace} for the norm equivalence. For $1\leq q<\infty$, we have
    \begin{align*}
        &\left(\int_{\mathbb{R}^{n-1}} \left(v_z^{1/q}\langle v\rangle^\delta\langle v_h\rangle^{-m}\mathbf{1}_{|v_h|\geq \left||x_h| + v_z\right|^{\frac{1}{1+\beta}}}\mathbf{1}_{v_z\geq 0}\right)^q\,dx_h \right)^{1/q}\\
        &=v_z^{1/q}\langle v\rangle^\delta\langle v_h\rangle^{-m}\left(\int_{\mathbb{R}^{n-1}} \mathbf{1}_{|v_h|\geq \left||x_h| + v_z\right|^{\frac{1}{1+\beta}}}\mathbf{1}_{v_z\geq 0}\,dx_h\right)^{1/q}\\
        &=\left(\frac{|\mathbb{S}^{n-2}|}{n-1}\right)^{1/q}v_z^{1/q}\langle v\rangle^\delta\langle v_h\rangle^{-m}\left(\max\{|v_h|^{1+\beta}-v_z, 0\}\right)^{\frac{n-1}{q}}\mathbf{1}_{v_z\geq 0}.
    \end{align*}
    When $1\leq p<\infty$, we also have
    \begin{align*}
        &\int_{\mathbb{R}^{n-1}}\int_0^{|v_h|^{1+\beta}} v_z^{p/q}\langle v\rangle^{p\delta}\langle v_h\rangle^{-pm}\left(|v_h|^{1+\beta}-v_z\right)^{\frac{p(n-1)}{q}}\,dv_z dv_h\\
        &\leq \int_{\mathbb{R}^{n-1}}\int_0^{|v_h|^{1+\beta}} v_z^{p/q}\langle v\rangle^{p\delta}\langle v_h\rangle^{-pm}|v_h|^{\frac{p(n-1)(1+\beta)}{q}}\,dv_z dv_h\\
        &\leq C\int_{\mathbb{R}^{n-1}}\langle v_h\rangle^{(p/q + p\delta + 1)(1+\beta)-pm+\frac{p(n-1)(1+\beta)}{q}}\,dv_h.
    \end{align*}
    It is integrable if
    \begin{align*}
        (\delta+1/p)(1+\beta)-m+\frac{n(1+\beta)}{q}<-\frac{n-1}{p}
    \end{align*}
    Therefore,
    \begin{align*}
        &\left\|e^{-t_b}\langle v\rangle^\delta e^{\alpha|v_h|^\beta}g(t-t_b,x_h-v_ht_b,v)\mathbf{1}_{v_z\geq \frac{z}{t}}\right\|_{L^p_v(\mathbb{R}^n, L^q_x(\mathbb{R}^n_+))}\leq Ct^{1/q}e^{-t}
    \end{align*}
    if the constants satisfy \eqref{eq:m_cond1} for $1\leq p,q<\infty$. One can do the same calculation for $p=\infty$ or $q=\infty$ and obtain the identical result under the convention $1/\infty = 0$.
\end{proof}

The next proposition is a technical lemma to provide a quantitative control of the $\mathcal{M}(f)$ term when $f$ has the macroscopic fields $\rho_g$, $u_g$, and $T_g$. It shows that the macroscopic fields generated by $\mathcal{M}(f)$ also follow the same behavior as the macroscopic fields of $f$ locally in time. It will be frequently used throughout the proof in constructing a unique solution $f(t,x,v)$ of the BGK model with the given inflow boundary data $g$.
\begin{proposition}\label{prop:M[f]_integral}
    Let $n\geq2$, $m_1\in\mathbb{R}$, $m_2\geq 0$, $\alpha\geq 0$, and $1\leq\beta\leq 2$. Assume that there exist positive constants $C_i$ such that
    \begin{align}\label{eq:macro_bound}
        \begin{split}
            &C_1\left(|x_h| + \frac{z}{t}+1\right)^{\frac{m_1}{1+\beta}} e^{-\alpha(|x_h| + \frac{z}{t}+1)^{\frac{\beta}{1+\beta}}}\leq \rho(t,x)\leq C_2\left(|x_h| + \frac{z}{t}+1\right)^{\frac{m_1}{1+\beta}} e^{-\alpha(|x_h| + \frac{z}{t}+1)^{\frac{\beta}{1+\beta}}},\\
            &\left|u(t,x)-\frac{z}{t}\hat{z}\right|\leq C_3 \left(|x_h| + \frac{z}{t}+1\right)^{\frac{1}{1+\beta}}\\
            &C_4\left(|x_h| + \frac{z}{t}+1\right)^{\frac{2}{1+\beta}}\leq T(t,x)\leq C_5\left(|x_h| + \frac{z}{t}+1\right)^{\frac{2}{1+\beta}}
        \end{split}
    \end{align}
    for $0< t\leq \frac{1}{4}$. Then, there exists a constant $C>0$, depending on $C_i$, $n$, $m_1$, $m_2$, $\alpha$, and $\beta$ such that
    \begin{align}\label{eq:M[f]_integral_res}
        \begin{split}
            &\sup_{0<\tau\leq t}\frac{t}{\tau}\int_{\mathbb{R}^n}\left|v-\frac{z}{t}\hat{z}\right|^{m_2} \frac{\rho(\tau, x-v(t-\tau))}{(2\pi T(\tau, x-v(t-\tau)))^{n/2}}\exp\left(-\frac{|v-u(\tau, x-v(t-\tau))|^2}{2T(\tau, x-v(t-\tau))}\right)\mathbf{1}_{z-v_z(t-\tau)>0}\,dv\\
            &\leq C\left(|x_h| + \frac{z}{t} + 1\right)^{\frac{m_1+m_2}{1+\beta}} e^{-\alpha(|x_h| + \frac{z}{t}+1)^{\frac{\beta}{1+\beta}}}
        \end{split}
    \end{align}
    for $0<t\leq \frac{1}{4}$.
\end{proposition}
\begin{proof}
    Since the proof is lengthy and technically involved, we first provide a brief overview of the main ideas. The first step is to simplify the complicated local Maxwellian integrand using the macroscopic field conditions \eqref{eq:macro_bound}. For example, because of the boundary effect, we separate the velocity integration into the tangential $v_h$-integration and normal $v_z$-integration and treat the two integrals differently. Through a chain of change of variables, we will observe that the term $\frac{z}{t}$ naturally appears in the integral in the $v_z$-integral as in the proof of Proposition ~\ref{prop:boundary_int}.
    
    We then derive explicit and sharp bounds for the simplified integrals. Since the transport term shifts the spatial argument from $x$ to $x-v(t-\tau)$, it is natural to partition the velocity-integration domains according to the relevant spatial scales. More precisely, we decompose the $v_z$-integration domain into $A_1$ and $A_2$, and the $v_h$-integration domain into $B_1$ and $B_2$, as defined below, and estimate the integrals over the subdomains. At first, we use Lemma~\ref{lem:concaten} to extract the same exponential decay factor as that appearing in the bound for $\rho$ in \eqref{eq:macro_bound}. We then bound the remaining integral by the desired polynomial term, thereby obtaining \eqref{eq:M[f]_integral_res}.
    
    We start the proof by simplifying the integral. Let $f(t,x,v)$ be a function whose macroscopic fields satisfy \eqref{eq:macro_bound}. Then the associated local Maxwellian can be written as follows:
    \begin{align*}
        \mathcal{M}(f)(t, x, v) = \frac{\rho(t, x)}{(2\pi T(t, x))^{n/2}}\exp\left(-\frac{|v-u(t, x)|^2}{2T(t, x)}\right).
    \end{align*}
    By change of variable $v_z\mapsto v_z + \frac{z}{t}$, we first have
    \begin{align*}
        &\frac{t}{\tau}\int_{\mathbb{R}^n} \left|v-\frac{z}{t}\hat{z}\right|^{m_2} \mathcal{M}(f)(\tau, x-v(t-\tau), v)\mathbf{1}_{z-v_z(t-\tau)>0}\,dv\\
        &=\frac{t}{\tau}\int_{\mathbb{R}^n}|v|^{m_2}\mathcal{M}(f)\left(\tau, x-\left(v+\frac{z}{t}\hat{z}\right)(t-\tau), v+\frac{z}{t}\hat{z}\right)\mathbf{1}_{\frac{\tau}{t}z - v_z(t-\tau) >0}\,dv.
    \end{align*}
    We expand the local Maxwellian by
    \begin{align*}
        &\mathcal{M}(f)\left(\tau, x-\left(v+\frac{z}{t}\hat{z}\right)(t-\tau), v+\frac{z}{t}\hat{z}\right)\\
        &=\frac{\rho(\tau, x-(v+\frac{z}{t}\hat{z})(t-\tau))}{(2\pi T(\tau, x-(v+\frac{z}{t}\hat{z})(t-\tau)))^{n/2}}\exp\left(-\frac{|(v+\frac{z}{t}\hat{z})-u(\tau, x-(v+\frac{z}{t}\hat{z})(t-\tau))|^2}{2T(\tau, x-(v+\frac{z}{t}\hat{z})(t-\tau))}\right).
    \end{align*}
    The $z$-component of the velocity part in the exponent is rearranged into
    \begin{align*}
        &v_z + \frac{z}{t}-u_z(\tau, x-(v+\frac{z}{t}\hat{z})(t-\tau))\\
        &= v_z + \frac{z}{t} - \left(u_z(\tau, x-(v+\frac{z}{t}\hat{z})(t-\tau)) - \frac{z-(v_z+\frac{z}{t})(t-\tau)}{\tau}\right) - \frac{z-(v_z+\frac{z}{t})(t-\tau)}{\tau}\\
        &=\frac{t}{\tau}v_z - \left(u_z(\tau, x-(v+\frac{z}{t}\hat{z})(t-\tau)) - \frac{z-(v_z+\frac{z}{t})(t-\tau)}{\tau}\right).
    \end{align*}
    Therefore, we can rewrite the exponent to
    \begin{align*}
        \left(v+\frac{z}{t}\hat{z}\right)-u(\tau, x-(v+\frac{z}{t}\hat{z})(t-\tau)) = \left(v_h,\frac{t}{\tau}v_z\right) - \left(u(\tau, x-(v+\frac{z}{t}\hat{z})(t-\tau)) - \frac{z-(v_z+\frac{z}{t})(t-\tau)}{\tau}\hat{z}\right).
    \end{align*}
    Furthermore, from \eqref{eq:macro_bound}, the second term is bounded by
    \begin{align*}
        \left|u(\tau, x-(v+\frac{z}{t}\hat{z})(t-\tau)) - \frac{z-(v_z+\frac{z}{t})(t-\tau)}{\tau}\hat{z}\right|\leq \frac{C_3}{\sqrt{C_4}}\sqrt{T(\tau, x-(v+\frac{z}{t}\hat{z})(t-\tau))}.
    \end{align*}
    Also, using the simple inequality $(x-y)^2\geq \frac{1}{2}x^2 - y^2$ for $x,y\geq 0$, we can bound the exponent in exponential term by
    \begin{align*}
        &-\frac{\left|(v_h, \frac{t}{\tau}v_z) - \left(u(\tau, x-(v+\frac{z}{t}\hat{z})(t-\tau)) - \frac{z-(v_z + \frac{z}{t})(t-\tau)}{\tau}\hat{z}\right)\right|^2}{2T(\tau, x-(v+\frac{z}{t}\hat{z})(t-\tau))}\\
        &\leq -\frac{\frac{1}{2}|(v_h, \frac{t}{\tau}v_z)|^2 - |u(\tau, x-(v+\frac{z}{t}\hat{z})(t-\tau)) - \frac{z-(v_z + \frac{z}{t})(t-\tau)}{\tau}\hat{z}|^2}{2T(\tau, x-(v+\frac{z}{t}\hat{z})(t-\tau))}\\
        &\leq \frac{C^2_3}{2C_4} - \frac{\frac{1}{2}|(v_h, \frac{t}{\tau}v_z)|^2}{2T(\tau, x-(v+\frac{z}{t}\hat{z})(t-\tau))}.
    \end{align*}
    These bounds yield
    \begin{align*}
        &\int_{\mathbb{R}^n} |v|^{m_2}\frac{\rho(\tau, x-(v+\frac{z}{t}\hat{z})(t-\tau))}{T(\tau, x-(v+\frac{z}{t}\hat{z})(t-\tau))^{n/2}}\exp\left(-\frac{|(v+\frac{z}{t}\hat{z})-u(\tau, x-(v+\frac{z}{t}\hat{z})(t-\tau))|^2}{2T(\tau, x-(v+\frac{z}{t}\hat{z})(t-\tau))}\right)\mathbf{1}_{\frac{\tau}{t}z - v_z(t-\tau) >0}\,dv\\
        &\leq e^{\frac{C_3^2}{2C_4}}\int_{\mathbb{R}^n} |v|^{m_2}\frac{\rho(\tau, x-(v+\frac{z}{t}\hat{z})(t-\tau))}{T(\tau, x-(v+\frac{z}{t}\hat{z})(t-\tau))^{n/2}}\exp\left(-\frac{|(v_h, \frac{t}{\tau}v_z)|^2}{4T(\tau, x-(v+\frac{z}{t}\hat{z})(t-\tau))}\right)\mathbf{1}_{\frac{\tau}{t}z - v_z(t-\tau) >0}\,dv.
    \end{align*}
    If we assume $\tau = t$, then the integral is the same as the plain Gaussian integral
    \begin{align*}
        \int_{\mathbb{R}^n} |v|^{m_2}\frac{\rho(t, x)}{T(t, x)^{n/2}}\exp\left(-\frac{|v|^2}{4T(t, x)}\right)\mathbf{1}_{z>0}\,dv,
    \end{align*}
    and we directly obtain the target bound \eqref{eq:M[f]_integral_res}. Therefore, we assume $0<\tau<t$ from now on and write the cutoff function $\mathbf{1}_{\frac{\tau}{t}z - v_z(t-\tau) >0} = \mathbf{1}_{v_z < \frac{\tau z}{t(t-\tau)}}$.
    
    We split $|v|^{m_2} \leq C(|v_h|^{m_2} + |v_z|^{m_2})$ for some constant $C$ for given $m_2\geq 0$. Replacing $\rho$ and $T$ by \eqref{eq:macro_bound}, we have
    \begin{align*}
        &\frac{t}{\tau}\int_{\mathbb{R}^n} |v|^{m_2}\frac{\rho(\tau, x-(v+\frac{z}{t}\hat{z})(t-\tau))}{T(\tau, x-(v+\frac{z}{t}\hat{z})(t-\tau))^{n/2}}\exp\left(-\frac{|(v_h, \frac{t}{\tau}v_z)|^2}{4T(\tau, x-(v+\frac{z}{t}\hat{z})(t-\tau))}\right)\mathbf{1}_{v_z< \frac{\tau z}{t(t-\tau)}}\,dv\\
        &\leq C\frac{t}{\tau}\int_{\mathbb{R}^n} (|v_h|^{m_2} + |v_z|^{m_2})\frac{(|x_h-v_h(t-\tau)| + \frac{z}{t} - \frac{(t-\tau)v_z}{\tau} + 1)^{\frac{m_1}{1+\beta}} e^{-\alpha (|x_h-v_h(t-\tau)| + \frac{z}{t} - \frac{(t-\tau)v_z}{\tau} + 1)^{\frac{\beta}{1+\beta}}}}{(|x_h-v_h(t-\tau)| + \frac{z}{t} - \frac{(t-\tau)v_z}{\tau} + 1)^{\frac{n}{1+\beta}}}\\
        &\qquad \times \exp\left(-\frac{|v_h|^2 + |\frac{t}{\tau}v_z|^2}{4C_5(|x_h-v_h(t-\tau)| + \frac{z}{t} - \frac{(t-\tau)v_z}{\tau} + 1)^{\frac{2}{1+\beta}}}\right)\mathbf{1}_{v_z< \frac{\tau z}{t(t-\tau)}}\,dv\\
        &= C\frac{t}{t-\tau}\int_{\mathbb{R}^n} \left(|v_h|^{m_2} + \left(\frac{\tau}{t-\tau}|v_z|\right)^{m_2}\right) \frac{e^{-\alpha(|x_h-v_h(t-\tau)| + \frac{z}{t} - v_z + 1)^{\frac{\beta}{1+\beta}}}}{(|x_h-v_h(t-\tau)| + \frac{z}{t} - v_z + 1)^{\frac{n-m_1}{1+\beta}}}\\
        &\qquad \times \exp\left(-\frac{|v_h|^2 + |\frac{t}{t-\tau}v_z|^2}{4C_5(|x_h-v_h(t-\tau)| + \frac{z}{t} - v_z + 1)^{\frac{2}{1+\beta}}}\right)\mathbf{1}_{v_z< \frac{z}{t}}\,dv.
    \end{align*}
    In the last equality, we used a change of variable $v_z\mapsto \frac{\tau}{t-\tau}v_z$. For later use, we generalize the notation and write the desired bound by
    \begin{align}\label{eq:M[f]_integral_form1}
        \begin{split}
            &\frac{t}{t-\tau}\int_{\mathbb{R}^n} |v_h|^{m_{2,1}}\left(\frac{\tau}{t-\tau}|v_z|\right)^{m_{2,2}} \frac{e^{-\alpha(|x_h-v_h(t-\tau)| + \frac{z}{t} - v_z + 1)^{\frac{\beta}{1+\beta}}}}{(|x_h-v_h(t-\tau)| + \frac{z}{t} - v_z + 1)^{\frac{n-m_1}{1+\beta}}}\\
            &\qquad \times \exp\left(-\frac{|v_h|^2 + |\frac{t}{t-\tau}v_z|^2}{4C_5(|x_h-v_h(t-\tau)| + \frac{z}{t} - v_z + 1)^{\frac{2}{1+\beta}}}\right)\mathbf{1}_{v_z< \frac{z}{t}}\,dv\\
            &\leq C(|x_h| + \frac{z}{t} + 1)^{\frac{m_1+m_{2,1} + m_{2,2}}{1+\beta}} e^{-\alpha(|x_h| + \frac{z}{t}+1)^{\frac{\beta}{1+\beta}}}
        \end{split}
    \end{align}
    for a constant depending on $C_i$, $n$, $m_1$, $m_2$, $\alpha$, and $\beta$. The original integral can be recovered by choosing $(m_{2,1},m_{2,2}) = (m_2, 0)$ and $(m_{2,1},m_{2,2}) = (0, m_2)$.

    To bound this integral, we separate the $v_z$ and $v_h$-integrals and decompose each integration domain by
    \begin{align*}
        A_1&:=\left\{v_z\in\mathbb R:-2\left(|x_h-v_h(t-\tau)|+\frac{z}{t}+1\right)\leq v_z<\frac{z}{t}\right\},\\
        A_2&:=\left\{v_z\in\mathbb R:v_z\leq-2\left(|x_h-v_h(t-\tau)|+\frac{z}{t}+1\right) \right\}
    \end{align*}
    for the $v_z$ domain $(-\infty, \frac{z}{t}]$ and
    \begin{equation}\label{eq:v_h_division}
        \begin{aligned}
            B_1&:=\left\{v_h\in\mathbb R^{n-1}:|v_h|(t-\tau)\leq 2\left(|x_h|+\frac{z}{t}+1\right)\right\},\\
            B_2&:=\left\{v_h\in\mathbb R^{n-1}:|v_h|(t-\tau)\geq 2\left(|x_h|+\frac{z}{t}+1\right)\right\}
        \end{aligned}
    \end{equation}
    for the $v_h$ domain $\mathbb{R}^{n-1}$.
    
    We first consider the domain $A_1\times \mathbb{R}^{n-1}$. We will perform $v_z$-integral on $A_1$ and then $v_h$-integral on $B_1$ and $B_2$ and derive the desired bound \eqref{eq:M[f]_integral_res}. Next, we prove that the integral on $A_2\times \mathbb{R}^{n-1}$ is a negligible term compared to the integral on $A_1\times\mathbb{R}^{n-1}$ and conclude the result.\\
    
    \noindent\textbf{Step 1: Integral over $A_1\times\mathbb{R}^{n-1}$.} 
    
    Now, we consider the domain $A_1\times \mathbb{R}^{n-1}$, so
    \begin{align*}
        -2\left(|x_h-v_h(t-\tau)|+\frac{z}{t}+1\right)< v_z\leq\frac{z}{t}.
    \end{align*}
    For pretreatment, we extract $e^{-\alpha(|x_h| + \frac{z}{t} + 1)^{\frac{\beta}{1+\beta}}}$ term from the exponential parts in \eqref{eq:M[f]_integral_form1}. We first split the exponent function in the local Maxwellian for the $v_z$ direction into half and multiply with the exponential function in the density $\rho$. Then, by Lemma \ref{lem:concaten},
    \begin{align}\label{eq:concaten_example1}
        \begin{split}
            &e^{-\alpha(|x_h-v_h(t-\tau)| + \frac{z}{t} - v_z + 1)^{\frac{\beta}{1+\beta}}}\exp\left(-\frac{|\frac{t}{t-\tau}v_z|^2}{4C_5(|x_h-v_h(t-\tau)| + \frac{z}{t} - v_z + 1)^{\frac{2}{1+\beta}}}\right)\\
            &=e^{-\alpha(|x_h-v_h(t-\tau)| + \frac{z}{t} - v_z + 1)^{\frac{\beta}{1+\beta}}}\left(\exp\left(-\frac{|\frac{t}{t-\tau}v_z|^2}{8C_5(|x_h-v_h(t-\tau)| + \frac{z}{t} - v_z + 1)^{\frac{2}{1+\beta}}}\right)\right)^2\\
            &\leq Ce^{-\alpha(|x_h-v_h(t-\tau)| + \frac{z}{t} +\left(\frac{t}{t-\tau}-1\right)|v_z| + 1)^{\frac{\beta}{1+\beta}}}\exp\left(-\frac{|\frac{t}{t-\tau}v_z|^2}{8C_5(|x_h-v_h(t-\tau)| + \frac{z}{t} - v_z + 1)^{\frac{2}{1+\beta}}}\right)\\
            &\leq Ce^{-\alpha(|x_h-v_h(t-\tau)| + \frac{z}{t} + 1)^{\frac{\beta}{1+\beta}}}\exp\left(-\frac{|\frac{t}{t-\tau}v_z|^2}{8C_5(|x_h-v_h(t-\tau)| + \frac{z}{t} - v_z + 1)^{\frac{2}{1+\beta}}}\right)
        \end{split}
    \end{align}
    for a constant $C$ depending on $C_5$ and $\beta$. We can repeat this for the $v_h$ variable. For the $v_h$ part of the local Maxwellian in \eqref{eq:M[f]_integral_form1}, as $v_z\geq -2\left(|x_h-v_h(t-\tau)| + \frac{z}{t} + 1\right)$, we write
    \begin{align*}
        \exp\left(-\frac{|v_h|^2}{4C_5(|x_h-v_h(t-\tau)| + \frac{z}{t} - v_z + 1)^{\frac{2}{1+\beta}}}\right)\leq \exp\left(-\frac{|v_h|^2}{4C_5\left(3(|x_h-v_h(t-\tau)| + \frac{z}{t} + 1)\right)^{\frac{2}{1+\beta}}}\right).
    \end{align*}
    We now split this local Maxwellian function and apply Lemma \ref{lem:concaten}, as in \eqref{eq:concaten_example1}, to obtain
    \begin{align*}
        &e^{-\alpha(|x_h-v_h(t-\tau)| + \frac{z}{t} + 1)^{\frac{\beta}{1+\beta}}}\exp\left(-\frac{|v_h|^2}{4C_5\left(3(|x_h-v_h(t-\tau)| + \frac{z}{t} + 1)\right)^{\frac{2}{1+\beta}}}\right)\\
        &\leq C e^{-\alpha(|x_h-v_h(t-\tau)| + |v_h| + \frac{z}{t} + 1)^{\frac{\beta}{1+\beta}}}\exp\left(-\frac{|v_h|^2}{8C_5\left(3(|x_h-v_h(t-\tau)| + \frac{z}{t} + 1)\right)^{\frac{2}{1+\beta}}}\right)\\
        &\leq C e^{-\alpha(|x_h| + \frac{z}{t} + 1)^{\frac{\beta}{1+\beta}}}\exp\left(-\frac{|v_h|^2}{8C_5\left(3(|x_h-v_h(t-\tau)| + \frac{z}{t} + 1)\right)^{\frac{2}{1+\beta}}}\right)
    \end{align*}
    as $0\leq t-\tau\leq \frac{1}{4}$. Finally, the integral \eqref{eq:M[f]_integral_form1} on $A_1\times\mathbb{R}^{n-1}$ is bounded by the integral
    \begin{align*}
        &e^{-\alpha(|x_h| + \frac{z}{t} + 1)^{\frac{\beta}{1+\beta}}}\frac{t}{t-\tau}\int_{A_1\times \mathbb{R}^{n-1}} |v_h|^{m_{2,1}}\left(\frac{\tau}{t-\tau}|v_z|\right)^{m_{2,2}}\frac{1}{(|x_h-v_h(t-\tau)| + \frac{z}{t} - v_z + 1)^{\frac{n-m_1}{1+\beta}}}\\
        &\qquad \times \exp\left(-\frac{|v_h|^2}{8C_5\left(3(|x_h-v_h(t-\tau)| + \frac{z}{t} + 1)\right)^{\frac{2}{1+\beta}}}\right)\exp\left(-\frac{|\frac{t}{t-\tau}v_z|^2}{8C_5(|x_h-v_h(t-\tau)| + \frac{z}{t} - v_z + 1)^{\frac{2}{1+\beta}}}\right)\,dv
    \end{align*}
    with some constant multiplication. To make the notation clean, we introduce a constant $C_T>0$ and write the goal by
    \begin{align}\label{eq:M[f]_integral_form2}
        \begin{split}
            &\frac{t}{t-\tau}\int_{A_1\times \mathbb{R}^{n-1}} |v_h|^{m_{2,1}}\left(\frac{\tau}{t-\tau}|v_z|\right)^{m_{2,2}}\frac{1}{(|x_h-v_h(t-\tau)| + \frac{z}{t} - v_z + 1)^{\frac{n-m_1}{1+\beta}}}\\
            &\qquad \times \exp\left(-\frac{|v_h|^2}{C_T(|x_h-v_h(t-\tau)| + \frac{z}{t} + 1)^{\frac{2}{1+\beta}}}\right)\exp\left(-\frac{|\frac{t}{t-\tau}v_z|^2}{C_T(|x_h-v_h(t-\tau)| + \frac{z}{t} - v_z + 1)^{\frac{2}{1+\beta}}}\right)\,dv\\
            &\leq C(|x_h| + \frac{z}{t} + 1)^{\frac{m_1+m_{2,1} + m_{2,2}}{1+\beta}}.
        \end{split}
    \end{align}

    We start from the $v_z$-integral on $A_1$. We will show that
    \begin{align}\label{eq:v_z_int_res}
        \begin{split}
            &\frac{t}{t-\tau}\int_{A_1} \frac{\left(\frac{\tau}{t-\tau}|v_z|\right)^{m_{2,2}}}{(|x_h-v_h(t-\tau)| + \frac{z}{t} - v_z + 1)^{\frac{n-m_1}{1+\beta}}}\exp\left(-\frac{|\frac{t}{t-\tau}v_z|^2}{C_T(|x_h-v_h(t-\tau)| + \frac{z}{t} - v_z + 1)^{\frac{2}{1+\beta}}}\right)\,dv_z\\
            &\leq C\left(\frac{\tau}{t}\right)^{m_{2,2}}(|x_h-v_h(t-\tau)| + \frac{z}{t} + 1)^{\frac{m_1+m_{2,2}-n+1}{1+\beta}}
        \end{split}
    \end{align}
    as the result of the integration. For this, we further divide the integral region into two parts:
    \begin{align*}
        \frac{t}{t-\tau}|v_z| > \frac{1}{2}(|x_h-v_h(t-\tau)| + \frac{z}{t} + 1)^{\frac{1}{1+\beta}}
    \end{align*}
    and another part depending on the scaling of the exponential function in \eqref{eq:v_z_int_res}. If $\frac{t}{t-\tau}|v_z| > \frac{1}{2}(|x_h-v_h(t-\tau)| + \frac{z}{t} + 1)^{\frac{1}{1+\beta}}$, choose $N$ to be the smallest integer satisfying
    \begin{align*}
        N> \max\{m_{2,2}+2, -(m_1-n)\}.
    \end{align*}
    Using $x^N e^{-x^2}\leq C_N$ for a constant $C_N$, we have
    \begin{align}\label{eq:A1_bound_3}
        \begin{split}
            &\frac{t}{t-\tau}\left(\frac{\tau}{t-\tau}\right)^{m_{2,2}}\int_{-2\left(|x_h-v_h(t-\tau)|+\frac{z}{t}+1\right)}^{\frac{z}{t}} \frac{|v_z|^{m_{2,2}}}{(|x_h-v_h(t-\tau)| + \frac{z}{t} - v_z + 1)^{\frac{n-m_1}{1+\beta}}}\\
            &\qquad \times \exp\left(-\frac{|\frac{t}{t-\tau}v_z|^2}{C_T(|x_h-v_h(t-\tau)| + \frac{z}{t} - v_z + 1)^{\frac{2}{1+\beta}}}\right) \mathbf{1}_{\frac{t}{t-\tau}|v_z| > \frac{1}{2}(|x_h-v_h(t-\tau)| + \frac{z}{t} + 1)^{\frac{1}{1+\beta}}}\,dv_z\\
            &\leq C_N\frac{t}{t-\tau}\left(\frac{\tau}{t-\tau}\right)^{m_{2,2}}\int_{-2\left(|x_h-v_h(t-\tau)|+\frac{z}{t}+1\right)}^{\frac{z}{t}} \frac{|v_z|^{m_{2,2}}}{(|x_h-v_h(t-\tau)| + \frac{z}{t} - v_z + 1)^{\frac{n-m_1}{1+\beta}}}\\
            &\qquad \times \left(\frac{\sqrt{C_T}(|x_h-v_h(t-\tau)| + \frac{z}{t} - v_z + 1)^{\frac{1}{1+\beta}}}{|\frac{t}{t-\tau}v_z|}\right)^N \mathbf{1}_{\frac{t}{t-\tau}|v_z| > \frac{1}{2}(|x_h-v_h(t-\tau)| + \frac{z}{t} + 1)^{\frac{1}{1+\beta}}}\,dv_z\\
            &\leq C\frac{t}{t-\tau}\left(\frac{\tau}{t-\tau}\right)^{m_{2,2}}\left(\frac{t-\tau}{t}\right)^N\\
            &\qquad \times \int_{-2\left(|x_h-v_h(t-\tau)|+\frac{z}{t}+1\right)}^{\frac{z}{t}} (|x_h-v_h(t-\tau)| + \frac{z}{t} - v_z + 1)^{\frac{N+m_1-n}{1+\beta}}\frac{1}{|v_z|^{N-m_{2,2}}}\mathbf{1}_{\frac{t}{t-\tau}|v_z| > \frac{1}{2}(|x_h-v_h(t-\tau)| + \frac{z}{t} + 1)^{\frac{1}{1+\beta}}}\,dv_z.
        \end{split}
    \end{align}
    As $N+m_1-n\geq 0$ and $v_z\geq -2(|x_h-v_h(t-\tau)| + \frac{z}{t} + 1)$, we have
    \begin{align*}
        (|x_h-v_h(t-\tau)| + \frac{z}{t} - v_z + 1)^{\frac{N+m_1-n}{1+\beta}}\leq \left(3(|x_h-v_h(t-\tau)| + \frac{z}{t} + 1)\right)^{\frac{N+m_1-n}{1+\beta}}.
    \end{align*}
    Also, since $\frac{t}{t-\tau}|v_z| > (|x_h-v_h(t-\tau)| + \frac{z}{t} + 1)^{\frac{1}{1+\beta}}$ and $N-m_{2,2}\geq 2$, we have
    \begin{align}\label{eq:A1_bound_4}
        \begin{split}
            &\left(\frac{t-\tau}{t}\right)^N\int_{-2\left(|x_h-v_h(t-\tau)|+\frac{z}{t}+1\right)}^{\frac{z}{t}} \frac{1}{|v_z|^{N-m_{2,2}}}\mathbf{1}_{\frac{t}{t-\tau}|v_z| > \frac{1}{2}(|x_h-v_h(t-\tau)| + \frac{z}{t} + 1)^{\frac{1}{1+\beta}}}\,dv_z\\
            &\leq 2\left(\frac{t-\tau}{t}\right)^N\int_{\frac{1}{2}\frac{t-\tau}{t}(|x_h-v_h(t-\tau)| + \frac{z}{t} + 1)^{\frac{1}{1+\beta}}}^\infty\frac{1}{|v_z|^{N-m_{2,2}}}\,dv_z\\
            &\leq C\left(\frac{t-\tau}{t}\right)^{m_{2,2}+1}(|x_h-v_h(t-\tau)| + \frac{z}{t} + 1)^{\frac{m_{2,2} + 1 - N}{1+\beta}}.
        \end{split}
    \end{align}
    Applying \eqref{eq:A1_bound_3} and \eqref{eq:A1_bound_4}, we obtain
    \begin{align}\label{eq:A1_bound_5}
        \begin{split}
            &\frac{t}{t-\tau}\int_{A_1} \frac{\left(\frac{\tau}{t-\tau}|v_z|\right)^{m_{2,2}}}{(|x_h-v_h(t-\tau)| + \frac{z}{t} - v_z + 1)^{\frac{n-m_1}{1+\beta}}}\\
            &\qquad \times \exp\left(-\frac{|\frac{t}{t-\tau}v_z|^2}{C_T(|x_h-v_h(t-\tau)| + \frac{z}{t} - v_z + 1)^{\frac{2}{1+\beta}}}\right)\mathbf{1}_{\frac{t}{t-\tau}|v_z| > \frac{1}{2}(|x_h-v_h(t-\tau)| + \frac{z}{t} + 1)^{\frac{1}{1+\beta}}}\,dv_z\\
            &\leq C\left(\frac{\tau}{t}\right)^{m_{2,2}}(|x_h-v_h(t-\tau)| + \frac{z}{t} + 1)^{\frac{m_1+m_{2,2}-n + 1}{1+\beta}}.
        \end{split}
    \end{align}
        
    Next, suppose $\frac{t}{t-\tau}|v_z| \leq \frac{1}{2}(|x_h-v_h(t-\tau)| + \frac{z}{t} + 1)^{\frac{1}{1+\beta}}$. In this case, we easily obtain
    \begin{align*}
        \frac{1}{2}\left(|x_h-v_h(t-\tau)| + \frac{z}{t} + 1\right)\leq |x_h-v_h(t-\tau)| + \frac{z}{t} - v_z + 1\leq \frac{3}{2}\left(|x_h-v_h(t-\tau)| + \frac{z}{t} + 1\right)
    \end{align*}
    Therefore, crudely bounding the exponential function by $1$, we obtain
    \begin{align}\label{eq:A1_bound_6}
        \begin{split}
            &\frac{t}{t-\tau}\int_{-2\left(|x_h-v_h(t-\tau)|+\frac{z}{t}+1\right)}^{\frac{z}{t}} \frac{\left(\frac{\tau}{t-\tau}|v_z|\right)^{m_{2,2}}}{(|x_h-v_h(t-\tau)| + \frac{z}{t} - v_z + 1)^{\frac{n-m_1}{1+\beta}}}\\
            &\qquad \times \exp\left(-\frac{|\frac{t}{t-\tau}v_z|^2}{C_T(|x_h-v_h(t-\tau)| + \frac{z}{t} - v_z + 1)^{\frac{2}{1+\beta}}}\right)\mathbf{1}_{\frac{t}{t-\tau}|v_z| \leq \frac{1}{2}(|x_h-v_h(t-\tau)| + \frac{z}{t} + 1)^{\frac{1}{1+\beta}}}\,dv_z\\
            &\leq C\frac{t}{t-\tau}(|x_h-v_h(t-\tau)| + \frac{z}{t} + 1)^{\frac{m_1-n}{1+\beta}}\int_{-\infty}^\infty \left(\frac{\tau}{t-\tau}|v_z|\right)^{m_{2,2}}\mathbf{1}_{\frac{t}{t-\tau}|v_z| \leq \frac{1}{2}(|x_h-v_h(t-\tau)| + \frac{z}{t} + 1)^{\frac{1}{1+\beta}}} \,dv_z\\
            &\leq C\left(\frac{\tau}{t}\right)^{m_{2,2}}(|x_h-v_h(t-\tau)| + \frac{z}{t} + 1)^{\frac{m_1+m_{2,2}-n+1}{1+\beta}}.
        \end{split}
    \end{align}
    
    Combining \eqref{eq:A1_bound_5} and \eqref{eq:A1_bound_6}, we obtain \eqref{eq:v_z_int_res}.\\

    We move on to the $v_h$-integral. Inserting the $v_z$-integral result \eqref{eq:v_z_int_res} to \eqref{eq:M[f]_integral_form2}, we need to obtain
    \begin{align}\label{eq:M[f]_integral_form3}
        \begin{split}
            &\int_{\mathbb{R}^{n-1}} |v_h|^{m_{2,1}}(|x_h-v_h(t-\tau)| + \frac{z}{t} + 1)^{\frac{m_1+m_{2,2}-n+1}{1+\beta}} \exp\left(-\frac{|v_h|^2}{C_T(|x_h-v_h(t-\tau)| + \frac{z}{t} + 1)^{\frac{2}{1+\beta}}}\right)\,dv_h\\
            &\leq C(|x_h| + \frac{z}{t} + 1)^{\frac{m_1+m_{2,1} + m_{2,2}}{1+\beta}}.
        \end{split}
    \end{align}
    We estimate this integral over the regions $B_1$ and $B_2$ defined in \eqref{eq:v_h_division}.\\

    \noindent\textbullet{} \textit{Region $B_1$.} For
    \begin{align*}
        |v_h|(t-\tau)\leq 2\left(|x_h|+\frac{z}{t}+1\right),
    \end{align*}
    we further split the integral region by
    \begin{align*}
        |v_h|\geq \frac{1}{2}\left(|x_h|+\frac{z}{t}+1\right)^{\frac{1}{1+\beta}}
    \end{align*}
    and the other as in the $v_z$-integral case. If $|v_h|\geq \frac{1}{2}\left(|x_h|+\frac{z}{t}+1\right)^{\frac{1}{1+\beta}}$, we choose $N\in\mathbb{N}$ satisfying
    \begin{align*}
        N> \max\{-(m_1+m_{2,2} - n + 1), m_{2,1} + n + 1\}.
    \end{align*}
    Using the inequality $x^N e^{-x^2}\leq C_N$ for some constant $C_N$, we bound the integral by
    \begin{align*}
        &\int_{|v_h|(t-\tau)\leq 2(|x_h| + \frac{z}{t} + 1)} |v_h|^{m_{2,1}}\left(|x_h-v_h(t-\tau)| + \frac{z}{t} + 1\right)^{\frac{m_1+m_{2,2} - n + 1}{1+\beta}}\\
        &\qquad \times \exp\left(-\frac{|v_h|^2}{C_T(|x_h-v_h(t-\tau)| + \frac{z}{t} + 1)^{\frac{2}{1+\beta}}}\right)\mathbf{1}_{|v_h|\geq \frac{1}{2}\left(|x_h|+\frac{z}{t}+1\right)^{\frac{1}{1+\beta}}}\,dv_h\\
        &\leq C_N\int_{|v_h|(t-\tau)\leq 2(|x_h| + \frac{z}{t} + 1)} |v_h|^{m_{2,1}}\left(|x_h-v_h(t-\tau)| + \frac{z}{t} + 1\right)^{\frac{m_1+m_{2,2} - n + 1}{1+\beta}}\\
        &\qquad \times \left(\frac{\sqrt{C_T}(|x_h-v_h(t-\tau)| + \frac{z}{t} + 1)^{\frac{1}{1+\beta}}}{|v_h|}\right)^N\mathbf{1}_{|v_h|\geq \frac{1}{2}\left(|x_h|+\frac{z}{t}+1\right)^{\frac{1}{1+\beta}}}\,dv_h\\
        &\leq C\int_{|v_h|(t-\tau)\leq 2(|x_h| + \frac{z}{t} + 1)} \left(|x_h-v_h(t-\tau)| + \frac{z}{t} + 1\right)^{\frac{N+m_1+m_{2,2} - n + 1}{1+\beta}}\frac{1}{|v_h|^{N-m_{2,1}}}\mathbf{1}_{|v_h|\geq \frac{1}{2}\left(|x_h|+\frac{z}{t}+1\right)^{\frac{1}{1+\beta}}}\,dv_h.
    \end{align*}
    Since $N+m_1+m_{2,2} - n + 1\geq 0$, we have
    \begin{align*}
        \left(|x_h-v_h(t-\tau)| + \frac{z}{t} + 1\right)^{\frac{N+m_1+m_{2,2} - n + 1}{1+\beta}}\leq \left(3\left(|x_h| + \frac{z}{t} + 1\right)\right)^{\frac{N+m_1+m_{2,2} - n + 1}{1+\beta}}
    \end{align*}
    for $|v_h|(t-\tau)\leq 2\left(|x_h|+\frac{z}{t}+1\right)$. Also, as $N - m_{2,1}> n + 1$ and $0\leq (t-\tau)\leq \frac{1}{4}$, we have
    \begin{align}\label{eq:B1_bound_res1}
        \begin{split}
            &\int_{|v_h|(t-\tau)\leq 2(|x_h| + \frac{z}{t} + 1)} \left(|x_h-v_h(t-\tau)| + \frac{z}{t} + 1\right)^{\frac{N+m_1+m_{2,2} - n + 1}{1+\beta}}\frac{1}{|v_h|^{N-m_{2,1}}}\mathbf{1}_{|v_h|\geq \frac{1}{2}\left(|x_h|+\frac{z}{t}+1\right)^{\frac{1}{1+\beta}}}\,dv_h\\
            &\leq \left(3\left(|x_h| + \frac{z}{t} + 1\right)\right)^{\frac{N+m_1+m_{2,2} - n + 1}{1+\beta}}\int_{\mathbb{R}^{n-1}} \frac{1}{|v_h|^{N-m_{2,1}}}\mathbf{1}_{|v_h|\geq \frac{1}{2}\left(|x_h|+\frac{z}{t}+1\right)^{\frac{1}{1+\beta}}}\,dv_h\\
            &\leq C\left(|x_h| + \frac{z}{t} + 1\right)^{\frac{N+m_1+m_{2,2} - n + 1}{1+\beta}}\left(|x_h| + \frac{z}{t} + 1\right)^{\frac{m_{2,1} - N + n-1}{1+\beta}}\\
            &=C\left(|x_h| + \frac{z}{t} + 1\right)^{\frac{m_1+m_{2,2} + m_{2,1}}{1+\beta}}.
        \end{split}
    \end{align}
    
    If $|v_h|\leq \frac{1}{2}\left(|x_h|+\frac{z}{t}+1\right)^{\frac{1}{1+\beta}}$, then
    \begin{align*}
        \frac{1}{2}\left(|x_h| + \frac{z}{t} + 1\right)\leq |x_h-v_h(t-\tau)| + \frac{z}{t} + 1\leq \frac{3}{2}\left(|x_h| + \frac{z}{t} + 1\right).
    \end{align*}
    Therefore, bounding the exponential function by $1$, we get
    \begin{align}\label{eq:B1_bound_res2}
        \begin{split}
            &\int_{|v_h|(t-\tau)\leq 2(|x_h| + \frac{z}{t} + 1)} |v_h|^{m_{2,1}}\left(|x_h-v_h(t-\tau)| + \frac{z}{t} + 1\right)^{\frac{m_1+m_{2,2} - n + 1}{1+\beta}}\\
            &\qquad \times \exp\left(-\frac{|v_h|^2}{C_T(|x_h-v_h(t-\tau)| + \frac{z}{t} + 1)^{\frac{2}{1+\beta}}}\right)\mathbf{1}_{|v_h|\leq \frac{1}{2}\left(|x_h|+\frac{z}{t}+1\right)^{\frac{1}{1+\beta}}}\,dv_h\\
            &\leq C\left(|x_h| + \frac{z}{t} + 1\right)^{\frac{m_1+m_{2,2} - n + 1}{1+\beta}}\int_{\mathbb{R}^{n-1}} |v_h|^{m_{2,1}}\mathbf{1}_{|v_h|\leq \frac{1}{2}\left(|x_h|+\frac{z}{t}+1\right)^{\frac{1}{1+\beta}}}\,dv_h\\
            &\leq C\left(|x_h| + \frac{z}{t} + 1\right)^{\frac{m_1+m_{2,2} + m_{2,1}}{1+\beta}}.
        \end{split}
    \end{align}
    We merge the two results \eqref{eq:B1_bound_res1} and \eqref{eq:B1_bound_res2} and obtain
    \begin{align}\label{eq:B1_bound_res}
        \begin{split}
            &\int_{B_1} |v_h|^{m_{2,1}}\left(|x_h-v_h(t-\tau)| + \frac{z}{t} + 1\right)^{\frac{m_1+m_{2,2} - n + 1}{1+\beta}}\exp\left(-\frac{|v_h|^2}{C_T(|x_h-v_h(t-\tau)| + \frac{z}{t} + 1)^{\frac{2}{1+\beta}}}\right)\,dv_h\\
            &\leq C\left(|x_h| + \frac{z}{t} + 1\right)^{\frac{m_1+m_{2,1} + m_{2,2}}{1+\beta}}.
        \end{split}
    \end{align}
    It matches with \eqref{eq:M[f]_integral_form3}.\\
    
    \noindent\textbullet{} \textit{Region $B_2$.} When
    \begin{align*}
        |v_h|(t-\tau) \geq 2\left(|x_h|+\frac{z}{t}+1\right),
    \end{align*}
    by the triangle inequality, we have
    \begin{align*}
        \frac{|v_h(t-\tau)|}{2}\leq |x_h-v_h(t-\tau)| + \frac{z}{t} + 1\leq 2|v_h(t-\tau)|.
    \end{align*}
    Consequently, we have
    \begin{align*}
        &\int_{|v_h|(t-\tau) \geq 2\left(|x_h|+\frac{z}{t}+1\right)} |v_h|^{m_{2,1}}\left(|x_h-v_h(t-\tau)| + \frac{z}{t} + 1\right)^{\frac{m_1+m_{2,2}-n+1}{1+\beta}}\exp\left(-\frac{|v_h|^2}{C_T(|x_h-v_h(t-\tau)| + \frac{z}{t} + 1)^{\frac{2}{1+\beta}}}\right)\,dv_h\\
        &\leq C\int_{|v_h|(t-\tau) \geq 2\left(|x_h|+\frac{z}{t}+1\right)} |v_h|^{m_{2,1}}(|v_h|(t-\tau))^{\frac{m_1+m_{2,2}-n+1}{1+\beta}}\exp\left(-\frac{|v_h|^2}{C_T(2|v_h|(t-\tau))^{\frac{2}{1+\beta}}}\right)\,dv_h\\
        &= C(t-\tau)^{\frac{m_1+m_{2,2}-n+1}{1+\beta}}\int_{|v_h|(t-\tau) \geq 2\left(|x_h|+\frac{z}{t}+1\right)} |v_h|^{m_{2,1} + \frac{m_1+m_{2,2}-n+1}{1+\beta}}\exp\left(-\frac{|v_h|^{\frac{2\beta}{1+\beta}}}{C_T(2(t-\tau))^{\frac{2}{1+\beta}}}\right)\,dv_h.
    \end{align*}
    By taking $v_h\mapsto (t-\tau)^{1/\beta}v_h$ and using Lemma~\ref{lem:exp_integral},
    \begin{align}
        &(t-\tau)^{\frac{m_1+m_{2,2}-n+1}{1+\beta}}\int_{|v_h|(t-\tau) \geq 2\left(|x_h|+\frac{z}{t}+1\right)} |v_h|^{m_{2,1} + \frac{m_1+m_{2,2}-n+1}{1+\beta}}\exp\left(-\frac{|v_h|^{\frac{2\beta}{1+\beta}}}{C_T(2(t-\tau))^{\frac{2}{1+\beta}}}\right)\,dv_h\notag\\
        &=(t-\tau)^{\frac{m_1 + m_{2,1}+m_{2,2}}{\beta}}\int_{\mathbb{R}^{n-1}} |v_h|^{m_{2,1} + \frac{m_1+m_{2,2}-n+1}{1+\beta}}\exp\left(-\frac{|v_h|^{\frac{2\beta}{1+\beta}}}{2^{\frac{2}{1+\beta}}C_T}\right)\mathbf{1}_{|v_h|(t-\tau)^{\frac{1+\beta}{\beta}}\geq 2(|x_h| + \frac{z}{t} + 1)}\,dv_h\notag\\
        &\leq C'(t-\tau)^{-m_{2,1}-n+3}(|x_h| + \frac{z}{t} + 1)^{m_{2,1} + \frac{m_1+m_{2,2} + (n-3)\beta}{1+\beta}}\exp\left(-\frac{(2(|x_h| + \frac{z}{t} + 1))^{\frac{2\beta}{1+\beta}}}{2^{\frac{2}{1+\beta}}C_T(t-\tau)^2}\right)\label{eq:B2_bound_2}
    \end{align}
    for some constant $C'$. If $-m_{2,1}-n+3\geq 0$, the desired estimate follows because
    \begin{align*}
        &(t-\tau)^{-m_{2,1}-n+3}(|x_h| + \frac{z}{t} + 1)^{m_{2,1} + \frac{m_1+m_{2,2} + (n-3)\beta}{1+\beta}}\\
        &= (|x_h| + \frac{z}{t} + 1)^{\frac{m_1+m_{2,1} + m_{2,2}}{1+\beta}}\left(\frac{t-\tau}{(|x_h| + \frac{z}{t} + 1)^{\frac{\beta}{1+\beta}}}\right)^{-m_{2,1}-n+3}\\
        &\leq (|x_h| + \frac{z}{t} + 1)^{\frac{m_1+m_{2,1} + m_{2,2}}{1+\beta}}
    \end{align*}
    for $0<t-\tau\leq \frac{1}{4}$. If not, let $\eta = m_{2,1} + n - 3$.
    Since $x^\eta e^{-x^2}\leq C_\eta$ for some constant $C_\eta$, applying
    \begin{align*}
        \exp\left(-\frac{(2(|x_h| + \frac{z}{t} + 1))^{\frac{2\beta}{1+\beta}}}{2^{\frac{2}{1+\beta}}C_T(t-\tau)^2}\right)\leq C_\eta\left(\frac{(2(|x_h| + \frac{z}{t} + 1))^{\frac{\beta}{1+\beta}}}{2^{\frac{1}{1+\beta}}\sqrt{C_T}(t-\tau)}\right)^{-\eta}
    \end{align*}
    to the final formula, we get
    \begin{align*}
        \eqref{eq:B2_bound_2}&\leq C(|x_h| + \frac{z}{t} + 1)^{\frac{m_1+m_{2,1} + m_{2,2}}{1+\beta}}\left(\frac{t-\tau}{(|x_h| + \frac{z}{t} + 1)^{\frac{\beta}{1+\beta}}}\right)^{-m_{2,1}-n+3}\left(\frac{2^{\frac{1}{1+\beta}}\sqrt{C_T}(t-\tau)}{(2(|x_h| + \frac{z}{t} + 1))^{\frac{\beta}{1+\beta}}}\right)^{m_{2,1}+n-3}\\
        &\leq C(|x_h| + \frac{z}{t} + 1)^{\frac{m_1+m_{2,1} + m_{2,2}}{1+\beta}}.
    \end{align*}
    Therefore, we obtain
    \begin{align}\label{eq:B2_bound_res1}
        \begin{split}
            &\int_{B_2} |v_h|^{m_{2,1}}\left(|x_h-v_h(t-\tau)| + \frac{z}{t} + 1\right)^{\frac{m_1+m_{2,2}-n+1}{1+\beta}}\exp\left(-\frac{|v_h|^2}{C_T(|x_h-v_h(t-\tau)| + \frac{z}{t} + 1)^{\frac{2}{1+\beta}}}\right)\,dv_h\\
            &\leq C(|x_h| + \frac{z}{t} + 1)^{\frac{m_1+m_{2,1} + m_{2,2}}{1+\beta}}.
        \end{split}
    \end{align}
    It also matches with \eqref{eq:M[f]_integral_form3}. Combining \eqref{eq:B1_bound_res} and \eqref{eq:B2_bound_res1}, we conclude \eqref{eq:M[f]_integral_form2} on the region $A_1\times\mathbb{R}^{n-1}$.\\
    
    \noindent\textbf{Step 2: Integral over $A_2\times\mathbb{R}^{n-1}$.} 
    
    We move on to the remaining region $A_2\times\mathbb{R}^{n-1}$, where $v_z\leq -2\left(|x_h-v_h(t-\tau)|+\frac{z}{t}+1\right)$. Here, for any fixed $m_{2,1},m_{2,2}\geq 0$ and $N\geq 0$, we bound the $v_z$ and $v_h$-integrals in steps by
    \begin{align}
        &\frac{t}{t-\tau}\int_{-\infty}^{-2\left(|x_h-v_h(t-\tau)|+\frac{z}{t}+1\right)} \left(\frac{\tau}{t-\tau}|v_z|\right)^{m_{2,2}} \frac{e^{-\alpha(|x_h-v_h(t-\tau)| + \frac{z}{t} - v_z + 1)^{\frac{\beta}{1+\beta}}}}{(|x_h-v_h(t-\tau)| + \frac{z}{t} - v_z + 1)^{\frac{n-m_1}{1+\beta}}}\notag\\
        &\qquad \times \exp\left(-\frac{|v_h|^2 + |\frac{t}{t-\tau}v_z|^2}{C_T(|x_h-v_h(t-\tau)| + \frac{z}{t} - v_z + 1)^{\frac{2}{1+\beta}}}\right)\,dv_z\notag\\
        &\leq C\left(|x_h-v_h(t-\tau)| + \frac{z}{t} + 1\right)^{-N} e^{-3^{\frac{\beta}{1+\beta}}\alpha(|x_h-v_h(t-\tau)| + \frac{z}{t} + 1)^{\frac{\beta}{1+\beta}}}\exp\left(-\frac{|\frac{t}{t-\tau}|^{\frac{2}{1+\beta}}|v_h|^{\frac{2\beta}{1+\beta}}}{C_T}\right)\label{eq:A2_v_z_int},
        \intertext{and}
        &\int_{\mathbb{R}^{n-1}}|v_h|^{m_{2,1}}\left(|x_h-v_h(t-\tau)| + \frac{z}{t} + 1\right)^{-N} e^{-3^{\frac{\beta}{1+\beta}}\alpha(|x_h-v_h(t-\tau)| + \frac{z}{t} + 1)^{\frac{\beta}{1+\beta}}}\exp\left(-\frac{|\frac{t}{t-\tau}|^{\frac{2}{1+\beta}}|v_h|^{\frac{2\beta}{1+\beta}}}{C_T}\right)\,dv_h\notag\\
        &\leq C \left(|x_h| + \frac{z}{t} + 1\right)^{-N}e^{-3^{\frac{\beta}{1+\beta}}\alpha(|x_h| + \frac{z}{t} + 1)^{\frac{\beta}{1+\beta}}}\label{eq:A2_v_h_int}
    \end{align}
    for some constant $C$ depending on $\alpha$, $\beta$, $m_1$, $n$, $m_{2,1}$, $m_{2,2}$, $C_T$, and $N$. It shows that 
    \begin{align*}
        &\frac{t}{t-\tau}\int_{A_2\times\mathbb{R}^{n-1}} |v_h|^{m_{2,1}}\left(\frac{\tau}{t-\tau}|v_z|\right)^{m_{2,2}} \frac{e^{-\alpha(|x_h-v_h(t-\tau)| + \frac{z}{t} - v_z + 1)^{\frac{\beta}{1+\beta}}}}{(|x_h-v_h(t-\tau)| + \frac{z}{t} - v_z + 1)^{\frac{n-m_1}{1+\beta}}}\\
        &\qquad \times \exp\left(-\frac{|v_h|^2 + |\frac{t}{t-\tau}v_z|^2}{C_T(|x_h-v_h(t-\tau)| + \frac{z}{t} - v_z + 1)^{\frac{2}{1+\beta}}}\right)\,dv\\
        &\leq C \left(|x_h| + \frac{z}{t} + 1\right)^{-N}e^{-3^{\frac{\beta}{1+\beta}}\alpha(|x_h| + \frac{z}{t} + 1)^{\frac{\beta}{1+\beta}}}.
    \end{align*}
    By choosing $N\geq \max\{0, -\frac{m_1+m_2}{1+\beta}\}$, we verify that the integral on $A_2\times\mathbb{R}^{n-1}$ for the weight $|v|^{m_2}$ is very small compared to the integral on $A_1\times \mathbb{R}^{n-1}$.
    
    We start from the $v_z$-integral. Since $-v_z\geq 2\left(|x_h-v_h(t-\tau)|+\frac{z}{t}+1\right)$, we have
    \begin{align*}
        |v_z|\leq (|x_h-v_h(t-\tau)| + \frac{z}{t} - v_z + 1)\leq 2|v_z|.
    \end{align*}
    From this simple bound, we deduce
    \begin{align*}
        \frac{1}{(|x_h-v_h(t-\tau)| + \frac{z}{t} - v_z + 1)^{\frac{n-m_1}{1+\beta}}}\leq C\frac{1}{|v_z|^{\frac{n-m_1}{1+\beta}}},
    \end{align*}
    and
    \begin{align*}
        \exp\left(-\frac{|v_h|^2 + |\frac{t}{t-\tau}v_z|^2}{C_T(|x_h-v_h(t-\tau)| + \frac{z}{t} - v_z + 1)^{\frac{2}{1+\beta}}}\right)&\leq \exp\left(-\frac{|v_h|^2 + |\frac{t}{t-\tau}v_z|^2}{2^{\frac{2}{1+\beta}}C_T|v_z|^{\frac{2}{1+\beta}}}\right) \\
        &= \exp\left(-\frac{|v_h|^2 + \frac{1}{2}|\frac{t}{t-\tau}v_z|^2}{2^{\frac{2}{1+\beta}}C_T|v_z|^{\frac{2}{1+\beta}}}\right)\exp\left(-\frac{1}{2}\left(\frac{t}{t-\tau}\right)^2\frac{|v_z|^{\frac{2\beta}{1+\beta}}}{2^{\frac{2}{1+\beta}}C_T}\right).
    \end{align*}
    We use Young's inequality
    \begin{align*}
        |v_h|^{\frac{2\beta}{1+\beta}}\left|\frac{\sqrt{\beta}}{\sqrt{2}} \frac{t}{t-\tau}v_z\right|^{\frac{2}{1+\beta}}\leq \frac{\beta}{1+\beta}\left(|v_h|^2 + \frac{1}{2}\left|\frac{t}{t-\tau}v_z\right|^2\right)
    \end{align*}
    to further bound the first exponential function as
    \begin{align*}
        \exp\left(-\frac{|v_h|^2 + \frac{1}{2}|\frac{t}{t-\tau}v_z|^2}{2^{\frac{2}{1+\beta}}C_T|v_z|^{\frac{2}{1+\beta}}}\right)&\leq \exp\left(-\frac{1+\beta}{\beta}\left(\frac{\beta}{2}\right)^{\frac{1}{1+\beta}}\frac{|v_h|^{\frac{2\beta}{1+\beta}}|\frac{t}{t-\tau}v_z|^{\frac{2}{1+\beta}}}{2^{\frac{2}{1+\beta}}C_T|v_z|^{\frac{2}{1+\beta}}}\right)\\
        &\leq \exp\left(-\frac{1+\beta}{\beta^{\frac{\beta}{1+\beta}}}\frac{|\frac{t}{t-\tau}|^{\frac{2}{1+\beta}}|v_h|^{\frac{2\beta}{1+\beta}}}{2^{\frac{3}{1+\beta}}C_T}\right).
    \end{align*}
    Applying these inequalities and $(|x_h-v_h(t-\tau)| + \frac{z}{t} - v_z + 1)\geq 3(|x_h-v_h(t-\tau)| + \frac{z}{t} + 1)$ under this $v_z$ region, we bound the $v_z$-integral by
    \begin{align}
        &\frac{t}{t-\tau}\int_{-\infty}^{-2\left(|x_h-v_h(t-\tau)| + \frac{z}{t} + 1\right)} \left(\frac{\tau}{t-\tau}|v_z|\right)^{m_{2,2}} \frac{e^{-\alpha(|x_h-v_h(t-\tau)| + \frac{z}{t} - v_z + 1)^{\frac{\beta}{1+\beta}}}}{(|x_h-v_h(t-\tau)| + \frac{z}{t} - v_z + 1)^{\frac{n-m_1}{1+\beta}}}\notag\\
        &\qquad \times \exp\left(-\frac{|v_h|^2 + |\frac{t}{t-\tau}v_z|^2}{C_T(|x_h-v_h(t-\tau)| + \frac{z}{t} - v_z + 1)^{\frac{2}{1+\beta}}}\right)\,dv_z\notag\\
        &\leq C\frac{t}{t-\tau}\left(\frac{\tau}{t-\tau}\right)^{m_{2,2}}e^{-3^{\frac{\beta}{1+\beta}}\alpha(|x_h-v_h(t-\tau)| + \frac{z}{t} + 1)^{\frac{\beta}{1+\beta}}}\exp\left(-\frac{1+\beta}{\beta^{\frac{\beta}{1+\beta}}}\frac{|\frac{t}{t-\tau}|^{\frac{2}{1+\beta}}|v_h|^{\frac{2\beta}{1+\beta}}}{2^{\frac{3}{1+\beta}}C_T}\right)\notag\\
        &\qquad \times \int_{-\infty}^{-2\left(|x_h-v_h(t-\tau)| + \frac{z}{t} + 1\right)} |v_z|^{m_{2,2}+\frac{m_1-n}{1+\beta}}\exp\left(-\frac{\left(\frac{t}{t-\tau}\right)^2|v_z|^{\frac{2\beta}{1+\beta}}}{2^{\frac{3+\beta}{1+\beta}}C_T}\right)\,dv_z\label{eq:A2_v_z_int_1}.
    \end{align}
    For the final integral, using Lemma~\ref{lem:exp_integral}, we have
    \begin{align*}
        &\int_{-\infty}^{-2\left(|x_h-v_h(t-\tau)| + \frac{z}{t} + 1\right)} |v_z|^{m_{2,2}+\frac{m_1-n}{1+\beta}}\exp\left(-\frac{\left(\frac{t}{t-\tau}\right)^2|v_z|^{\frac{2\beta}{1+\beta}}}{2^{\frac{3+\beta}{1+\beta}}C_T}\right)\,dv_z\\
        &\leq C\left(\frac{t-\tau}{t}\right)^2 \left(|x_h-v_h(t-\tau)| + \frac{z}{t} + 1\right)^{\frac{1-\beta}{1+\beta} + m_{2,2} + \frac{m_1-n}{1+\beta}}\exp\left(-\left(\frac{t}{t-\tau}\right)^2\frac{\left(|x_h-v_h(t-\tau)| + \frac{z}{t} + 1\right)^{\frac{2\beta}{1+\beta}}}{2^{\frac{3-\beta}{1+\beta}}C_T}\right).
    \end{align*}
    Using the exponential factor, we can eliminate the $\frac{1}{t-\tau}$ singularities in \eqref{eq:A2_v_z_int_1} and produce an arbitrary order of polynomial decay. Therefore, for any $N\geq 0$, we prove \eqref{eq:A2_v_z_int}.
    
    Next, we bound the $v_h$-integral of \eqref{eq:A2_v_z_int}: we define $C'_T$ by
    \begin{align*}
        C'_T = \frac{\beta^{\frac{\beta}{1+\beta}}}{1+\beta}2^{\frac{3}{1+\beta}}C_T,
    \end{align*}
    and compute the integral
    \begin{align*}
        \int_{\mathbb{R}^{n-1}}|v_h|^{m_{2,1}} \left(|x_h-v_h(t-\tau)| + \frac{z}{t} + 1\right)^{-N}e^{-3^{\frac{\beta}{1+\beta}}\alpha(|x_h-v_h(t-\tau)| + \frac{z}{t} + 1)^{\frac{\beta}{1+\beta}}}\exp\left(-\frac{|\frac{t}{t-\tau}|^{\frac{2}{1+\beta}}|v_h|^{\frac{2\beta}{1+\beta}}}{C'_T}\right)\,dv_h.
    \end{align*}
    
    Since $0\leq\frac{\beta}{1+\beta}\leq 1$, we crudely split the exponential function by
    \begin{align*}
        e^{-3^{\frac{\beta}{1+\beta}}\alpha(|x_h-v_h(t-\tau)| + \frac{z}{t} + 1)^{\frac{\beta}{1+\beta}}}&\leq e^{-3^{\frac{\beta}{1+\beta}}\alpha(|x_h| + \frac{z}{t} + 1)^{\frac{\beta}{1+\beta}}}e^{3^{\frac{\beta}{1+\beta}}\alpha|v_h(t-\tau)|^{\frac{\beta}{1+\beta}}}\\
        &\leq e^{-3^{\frac{\beta}{1+\beta}}\alpha(|x_h| + \frac{z}{t} + 1)^{\frac{\beta}{1+\beta}}}e^{3^{\frac{\beta}{1+\beta}}\alpha|v_h|^{\frac{\beta}{1+\beta}}}
    \end{align*}
    for $0\leq t-\tau\leq \frac{1}{4}$. Using this inequality, we get
    \begin{align*}
        &\int_{\mathbb{R}^{n-1}} |v_h|^{m_{2,1}}\left(|x_h-v_h(t-\tau)| + \frac{z}{t} + 1\right)^{-N} e^{-3^{\frac{\beta}{1+\beta}}\alpha(|x_h-v_h(t-\tau)| + \frac{z}{t} + 1)^{\frac{\beta}{1+\beta}}}\exp\left(-\frac{|\frac{t}{t-\tau}|^{\frac{2}{1+\beta}}|v_h|^{\frac{2\beta}{1+\beta}}}{C'_T}\right)\,dv_h\\
        &\leq e^{-3^{\frac{\beta}{1+\beta}}\alpha(|x_h| + \frac{z}{t} + 1)^{\frac{\beta}{1+\beta}}}\int_{\mathbb{R}^{n-1}} |v_h|^{m_{2,1}}\left(|x_h-v_h(t-\tau)| + \frac{z}{t} + 1\right)^{-N} e^{3^{\frac{\beta}{1+\beta}}\alpha|v_h|^{\frac{\beta}{1+\beta}}}\exp\left(-\frac{|v_h|^{\frac{2\beta}{1+\beta}}}{C'_T}\right)\,dv_h\\
        &\leq \exp\left(\frac{9C'_T\alpha^2}{4}\right)e^{-3^{\frac{\beta}{1+\beta}}\alpha(|x_h| + \frac{z}{t} + 1)^{\frac{\beta}{1+\beta}}}\\
        &\qquad \times \int_{\mathbb{R}^{n-1}} |v_h|^{m_{2,1}}\left(|x_h-v_h(t-\tau)| + \frac{z}{t} + 1\right)^{-N}\exp\left(-\frac{1}{C'_T}\left(|v_h|^{\frac{\beta}{1+\beta}} - \frac{3^{\frac{\beta}{1+\beta}} C'_T\alpha}{2}\right)^2\right)\,dv_h.
    \end{align*}
    To bound the final integral, we split the $v_h$ domain by
    \begin{align*}
        |v_h|(t-\tau)\leq \frac{1}{2}\left(|x_h| + \frac{z}{t} + 1\right)
    \end{align*}
    and the other part. When $|v_h|(t-\tau)\leq \frac{1}{2}\left(|x_h| + \frac{z}{t} + 1\right)$, we have
    \begin{align*}
        &\int_{\mathbb{R}^{n-1}} |v_h|^{m_{2,1}}\left(|x_h-v_h(t-\tau)| + \frac{z}{t} + 1\right)^{-N}\exp\left(-\frac{1}{C'_T}\left(|v_h|^{\frac{\beta}{1+\beta}} - \frac{3^{\frac{\beta}{1+\beta}} C'_T\alpha}{2}\right)^2\right)\mathbf{1}_{|v_h|(t-\tau)\leq \frac{1}{2}(|x_h| + \frac{z}{t} + 1)}\,dv_h\\
        &\leq C\left(|x_h| + \frac{z}{t} + 1\right)^{-N}\int_{\mathbb{R}^{n-1}} |v_h|^{m_{2,1}}\exp\left(-\frac{1}{C'_T}\left(|v_h|^{\frac{\beta}{1+\beta}} - \frac{3^{\frac{\beta}{1+\beta}} C'_T\alpha}{2}\right)^2\right)\,dv_h\\
        &\leq C\left(|x_h| + \frac{z}{t} + 1\right)^{-N}
    \end{align*}
    as the final integral is uniformly bounded. If $|v_h|(t-\tau)\geq \frac{1}{2}\left(|x_h| + \frac{z}{t} + 1\right)$, we further assume
    \begin{align}\label{eq:A2_v_h_cond}
        \frac{1}{4}\left(|x_h| + \frac{z}{t} + 1\right)\geq \left(\frac{3^{\frac{\beta}{1+\beta}} C'_T\alpha}{2}\right)^{\frac{1+\beta}{\beta}}.
    \end{align}
    Otherwise, since $|x_h| + \frac{z}{t} + 1$ is bounded above,
    \begin{align*}
        &\int_{\mathbb{R}^{n-1}} |v_h|^{m_{2,1}}\left(|x_h-v_h(t-\tau)| + \frac{z}{t} + 1\right)^{-N}\exp\left(-\frac{1}{C'_T}\left(|v_h|^{\frac{\beta}{1+\beta}} - \frac{3^{\frac{\beta}{1+\beta}} C'_T\alpha}{2}\right)^2\right)\mathbf{1}_{|v_h|(t-\tau)\geq \frac{1}{2}(|x_h| + \frac{z}{t} + 1)}\,dv_h\\
        &\leq \int_{\mathbb{R}^{n-1}} |v_h|^{m_{2,1}}\exp\left(-\frac{1}{C'_T}\left(|v_h|^{\frac{\beta}{1+\beta}} - \frac{3^{\frac{\beta}{1+\beta}} C'_T\alpha}{2}\right)^2\right)\,dv_h,
    \end{align*}
    and the final integral is well-bounded by a constant. Under this region \eqref{eq:A2_v_h_cond}, we have
    \begin{align*}
        |v_h|^{\frac{\beta}{1+\beta}} - \frac{3^{\frac{\beta}{1+\beta}} C'_T\alpha}{2}\geq |v_h|^{\frac{\beta}{1+\beta}} - \left(\frac{|v_h|(t-\tau)}{2}\right)^{\frac{\beta}{1+\beta}}\geq \frac{1}{2}|v_h|^{\frac{\beta}{1+\beta}}
    \end{align*}
    as $0\leq t-\tau\leq \frac{1}{4}$. Thus, using Lemma~\ref{lem:exp_integral}, we get
    \begin{align*}
        &\int_{\mathbb{R}^{n-1}} |v_h|^{m_{2,1}}\left(|x_h-v_h(t-\tau)| + \frac{z}{t} + 1\right)^{-N}\exp\left(-\frac{1}{C'_T}\left(|v_h|^{\frac{\beta}{1+\beta}} - \frac{3^{\frac{\beta}{1+\beta}} C'_T\alpha}{2}\right)^2\right)\mathbf{1}_{|v_h|(t-\tau)\geq \frac{1}{2}(|x_h| + \frac{z}{t} + 1)}\,dv_h\\
        &\leq \int_{\mathbb{R}^{n-1}} |v_h|^{m_{2,1}}\exp\left(-\frac{|v_h|^{\frac{2\beta}{1+\beta}}}{4C'_T}\right)\mathbf{1}_{|v_h|(t-\tau)\geq \frac{1}{2}(|x_h| + \frac{z}{t} + 1)}\,dv_h\\
        &\leq C\left(\frac{1}{2(t-\tau)}\left(|x_h| + \frac{z}{t} + 1\right)\right)^{n-1-\frac{2\beta}{1+\beta} + m_{2,1}}\exp\left(-\frac{1}{4C'_T}\left(\frac{1}{2(t-\tau)}(|x_h| + \frac{z}{t} + 1)\right)^{\frac{2\beta}{1+\beta}}\right).
    \end{align*}
    Again, we can construct an arbitrary polynomial decay from the exponentially decaying term and bound the final result by $C_N\left(|x_h| + \frac{z}{t} + 1\right)^{-N}$ for any $N\geq 0$. It proves \eqref{eq:A2_v_h_int}. Applying the two results \eqref{eq:A2_v_z_int} and \eqref{eq:A2_v_h_int} for $C_T = 4C_5$ to \eqref{eq:M[f]_integral_form1}, for any $N\geq 0$, we obtain
    \begin{align}\label{eq:M[f]_integral_res2}
        \begin{split}
            &\frac{t}{t-\tau}\int_{\mathbb{R}^n} \left(|v_h|^{m_2} + \left(\frac{\tau}{t-\tau}|v_z|\right)^{m_2}\right) \frac{e^{-\alpha(|x_h-v_h(t-\tau)| + \frac{z}{t} - v_z + 1)^{\frac{\beta}{1+\beta}}}}{(|x_h-v_h(t-\tau)| + \frac{z}{t} - v_z + 1)^{\frac{n-m_1}{1+\beta}}}\\
            &\qquad \times \exp\left(-\frac{|v_h|^2 + |\frac{t}{t-\tau}v_z|^2}{4C_5(|x_h-v_h(t-\tau)| + \frac{z}{t} - v_z + 1)^{\frac{2}{1+\beta}}}\right)\mathbf{1}_{v_z\leq -2\left(|x_h-v_h(t-\tau)|+\frac{z}{t}+1\right)}\,dv\\
            &\leq C\left(|x_h| + \frac{z}{t} + 1\right)^{-N}e^{-3^{\frac{\beta}{1+\beta}}\alpha(|x_h| + \frac{z}{t} + 1)^{\frac{\beta}{1+\beta}}}.
        \end{split}
    \end{align}
    Combining \eqref{eq:M[f]_integral_form2} and \eqref{eq:M[f]_integral_res2}, we obtain the result \eqref{eq:M[f]_integral_res}.
\end{proof}

The next corollary is an application of Proposition~\ref{prop:M[f]_integral} for $\alpha = 0$. It will be used in the contraction mapping argument for the construction of the BGK solution in Proposition~\ref{prop:maxwell}.
\begin{corollary}\label{prop:bound}
    For $1\leq \beta\leq 2$, assume
    \begin{align*}
        &\left|u-\frac{z}{t}\hat{z}\right|(t,x)\leq C_3 \left(|x_h| + \frac{z}{t}+1\right)^{\frac{1}{1+\beta}}\\
        &C_4\left(|x_h| + \frac{z}{t}+1\right)^{\frac{2}{1+\beta}}\leq T(t,x)\leq C_5\left(|x_h| + \frac{z}{t}+1\right)^{\frac{2}{1+\beta}}
    \end{align*}
    for some constants $C_i$ and $0<t\leq \frac{1}{4}$. Then, the following holds.
    \begin{enumerate}
        \item There exists a constant $C>0$ for each $i=0, 1$ depending on $\beta$, $n$, and $C_i$ such that
        \begin{align*}
            &\sup_{0<\tau\leq t}\int_{\mathbb{R}^n} \bigl|v - \tfrac{z}{t}\hat z\bigr|^{2i} \frac{1}{T(\tau, x-v(t-\tau))^{\frac{n}{2}+i}} \exp\!\left(-\frac{\bigl|v - u(\tau, x-v(t-\tau))\bigr|^2}{4T(\tau, x-v(t-\tau))}\right)\mathbf{1}_{z-v_z(t-\tau)>0}\, dv \leq C.
        \end{align*}
        \item There exists a constant $C>0$ depending on $\alpha$, $\beta$, $n$, and $C_i$ such that
        \begin{align*}
            &\sup_{0<\tau\leq t}\int_{\mathbb{R}^n} \left(|v_h(t-\tau)| + \frac{t-\tau}{\tau}\left|v_z - \frac{z}{t}\right|\right)^{\frac{2}{1+\beta}} \frac{1}{T(\tau, x-v(t-\tau))^{n/2 + 1}}\exp\left(-\frac{|v-u(\tau, x-v(t-\tau))|^2}{4T(\tau, x-v(t-\tau))}\right)\mathbf{1}_{z-v_z(t-\tau)>0}\,dv\\
            &\leq C\left(|x_h|+\frac{z}{t} + 1\right)^{-\frac{2\beta}{(1+\beta)^2}}.
        \end{align*}
    \end{enumerate}
\end{corollary}
\begin{proof}
    In \eqref{eq:macro_bound}, we choose $\rho = 1$, so $\alpha = 0$, $m_1 = 0$, and $C_1 = C_2 = 1$. Choosing $m_2 = 0$, we get the result for $i=0$. If $i=1$, we choose $\rho = \left(|x_h| + \frac{z}{t}+1\right)^{-\frac{2}{1+\beta}}$ and $m_2 = 2$ to get a similar result.

    For the second one, we choose $\rho = \left(|x_h| + \frac{z}{t}+1\right)^{-\frac{2}{1+\beta}}$. In \eqref{eq:M[f]_integral_form1}, we select
    \begin{align*}
        (m_{2,1}, m_{2,2}) = \left(\frac{2}{1+\beta}, 0\right)\quad\text{or}\quad(m_{2,1}, m_{2,2}) = \left(0, \frac{2}{1+\beta}\right).
    \end{align*}
    Then,
    \begin{align*}
        &\sup_{0<\tau\leq t}\int_{\mathbb{R}^n} \left(|v_h(t-\tau)| + \frac{t-\tau}{\tau}\left|v_z - \frac{z}{t}\right|\right)^{\frac{2}{1+\beta}} \frac{1}{T^{n/2 + 1}}\exp\left(-\frac{|v-u|^2}{4T}\right)(\tau, x-v(t-\tau))\mathbf{1}_{z-v_z(t-\tau)>0}\,dv\\
        &\leq C\sup_{0<\tau\leq t}\int_{\mathbb{R}^n} \left(\left(|v_h(t-\tau)|\right)^{\frac{2}{1+\beta}} + \left(\frac{t-\tau}{\tau}\left|v_z - \frac{z}{t}\right|\right)^{\frac{2}{1+\beta}}\right) \frac{1}{T^{n/2 + 1}}\exp\left(-\frac{|v-u|^2}{4T}\right)(\tau, x-v(t-\tau))\mathbf{1}_{z-v_z(t-\tau)>0}\,dv\\
        &\leq C\left(|x_h|+\frac{z}{t} + 1\right)^{-\frac{2\beta}{(1+\beta)^2}}
    \end{align*}
    for a constant $C$.
\end{proof}

We recall the macroscopic fields $(\rho_g, u_g, T_g)$ given in Proposition~\ref{prop:boundary_int}. If we choose $m_1 = (n-m)-\beta$, then the macroscopic fields satisfy the conditions of Proposition~\ref{prop:M[f]_integral}, providing the sharp control of the macroscopic fields generated by the local Maxwellian $\mathcal{M}(f)$. Using this proposition, we can check the \textit{a priori} local-in-time well-posedness of the macroscopic fields of the BGK solution.

\begin{proposition}\label{prop:apriori}
    Let $\rho_g$, $u_g$, and $T_g$ be given by \eqref{eq:macro_0_bound}. Then there exists $0<t_0\leq \frac{1}{4}$ depending on the $C_i$, $m$, $\alpha$, $\beta$, and $n$ such that the following holds. Assume $f_1$ satisfies
    \begin{align}\label{eq:macro_bound_f_1}
        \begin{split}
            &|\rho_1-\rho_g|(t,x)\leq \frac{1}{2}\rho_g(t,x),\\
            &\left|u_1-\frac{z}{t}\hat{z}\right|(t,x)\leq \frac{C_3 + 1}{\sqrt{C_4}}\sqrt{T_g(t,x)},\\
            &|T_1-T_g|(t,x)\leq \frac{1}{2}T_g(t,x)
        \end{split}
    \end{align}
    for $0<t\leq t_0$, where $(\rho_1,u_1,T_1)$ are macroscopic quantities generated by $f_1$. If $f_2$ satisfies
    \begin{align*}
        f_2(t,x,v) = e^{-t_b}g(t-t_b, x - vt_b, v)\mathbf{1}_{t\geq t_b} + \int_0^t e^{-(t - \tau)} \mathcal{M}(f_1)(\tau, x - v(t - \tau), v)\mathbf{1}_{z-v_z(t-\tau)> 0}\,d\tau,
    \end{align*}
    then the $(\rho_2, u_2, T_2)$ generated by $f_2$ also satisfy \eqref{eq:macro_bound_f_1} for $0<t\leq t_0$.
\end{proposition}

\begin{proof}
    We first write $(\rho_2, u_2, T_2)$ as
    \begin{align}\label{eq:macro_f_2}
        \begin{split}
            &\rho_2(t,x)\left(1, u_2(t,x)-\frac{z}{t}\hat{z}, n T_2(t,x)\right)=\int_{\mathbb{R}^n} \left(1,v-\frac{z}{t}\hat{z},|v-u_2(t,x)|^2\right)f_2(t,x,v)\,dv\\
            &=\int_{\mathbb{R}^n} \left(1,v-\frac{z}{t}\hat{z},|v-u_2(t,x)|^2\right)e^{-\frac{z}{v_z}}g(t-t_b, x-vt_b, v)\mathbf{1}_{v_z\geq \frac{z}{t}}\,dv\\
            &\quad + \int_0^t \int_{\mathbb{R}^n} \left(1,v-\frac{z}{t}\hat{z},|v-u_2(t,x)|^2\right)e^{-(t-\tau)}\mathcal{M}(f_1)(\tau, x-v(t-\tau), v)\mathbf{1}_{z-v_z(t-\tau)>0}\,dv d\tau.
        \end{split}
    \end{align}
    The first term represents the macroscopic fields generated by the damped free transport of the boundary data $g(t,x_h,v)$. For the second term, we can bound the integral using Proposition~\ref{prop:boundary_int} and~\ref{prop:M[f]_integral} as
    \begin{align}\label{eq:macro_M(f_1)}
        \begin{split}
            &\int_0^t e^{-(t-\tau)}\int_{\mathbb{R}^n} \left(1,\left|v-\frac{z}{t}\hat{z}\right|,\left|v-\frac{z}{t}\hat{z}\right|^2\right)\mathcal{M}(f_1)(\tau, x-v(t-\tau), v)\mathbf{1}_{z-v_z(t-\tau)>0}\,dv d\tau\\
            &\leq C_{\mathcal{M}}t\left(\left(|x_h| + \frac{z}{t} + 1\right)^{\frac{n-\beta-m}{1+\beta}}, \left(|x_h| + \frac{z}{t} + 1\right)^{\frac{n+1-\beta-m}{1+\beta}}, \left(|x_h| + \frac{z}{t} + 1\right)^{\frac{n+2-\beta-m}{1+\beta}}\right) e^{-\alpha(|x_h| + \frac{z}{t}+1)^{\frac{\beta}{1+\beta}}}\\
            &\leq C_{\mathcal{M}}C_1^{-1}t\rho_g\left(1, \sqrt{C_4}^{-1}\sqrt{T_g}, C_4^{-1}T_g\right)(t,x)
        \end{split}
    \end{align}
    for some constant $C_{\mathcal{M}}$ depending on $C_i$ in \eqref{eq:macro_0_bound} and other fixed constants in Proposition \ref{prop:M[f]_integral}. 
    From \eqref{eq:macro_0_bound}, we further have
    \begin{align*}
        \left|u_g(t,x)-\frac{z}{t}\hat{z}\right|\leq \frac{C_3}{\sqrt{C_4}}\sqrt{T_g}(t,x),
    \end{align*}
    
    Now, we check each line of \eqref{eq:macro_bound_f_1}. From the definition of $\rho_g$ in Definition \ref{Def:macro_0}, we obtain
    \begin{align}\label{eq:rho_diff}
        \begin{split}
            0\leq (\rho_2-\rho_g)(t,x) &= \int_0^t  e^{-(t-\tau)}\int_{\mathbb{R}^n} \mathcal{M}(f_1)(\tau, x-v(t-\tau), v)\mathbf{1}_{z-v_z(t-\tau)>0}\,dv d\tau\\
            &\leq \frac{C_{\mathcal{M}}}{C_1}t\rho_g(t,x).
        \end{split}
    \end{align}
    Therefore, we choose $t\leq \frac{C_1}{2C_{\mathcal{M}}}$ to get $|\rho_2-\rho_g|(t,x)\leq \frac{1}{2}\rho_g(t,x)$. From this result, we also have $\frac{\rho_g}{\rho_2}\leq 1$.
    
    To compute $u_2(t,x)$, we directly use \eqref{eq:macro_f_2} and \eqref{eq:macro_M(f_1)} by
    \begin{align*}
        \rho_2\left|u_2-\frac{z}{t}\hat{z}\right|(t,x)&\leq \left|\int_{\mathbb{R}^n} \left(v-\frac{z}{t}\hat{z}\right)e^{-\frac{z}{v_z}}g(t-t_b, x-vt_b, v)\mathbf{1}_{v_z\geq \frac{z}{t}}\,dv\right|\\
        &\quad + \int_0^t  e^{-(t-\tau)}\int_{\mathbb{R}^n} \left|v-\frac{z}{t}\hat{z}\right|\mathcal{M}(f_1)(\tau, x-v(t-\tau), v)\mathbf{1}_{z-v_z(t-\tau)>0}\,dv d\tau \\
        &\leq \rho_g\left|u_g-\frac{z}{t}\hat{z}\right|(t,x) + \frac{C_{\mathcal{M}}}{C_1\sqrt{C_4}}t\rho_g\sqrt{T_g}(t,x)\\
        &\leq \frac{1}{\sqrt{C_4}}\left(C_3+\frac{C_{\mathcal{M}}}{C_1}t\right)\rho_g\sqrt{T_g}(t,x).
    \end{align*}
    As $\frac{\rho_g}{\rho_2}\leq 1$, it shows
    \begin{align}\label{eq:u2_bound}
        \left|u_2-\frac{z}{t}\hat{z}\right|(t,x)\leq \frac{C_3 + 1}{\sqrt{C_4}}\sqrt{T_g}(t,x)
    \end{align}
    for $t\leq \frac{C_1}{C_{\mathcal{M}}}$.
    
    We can further compute $|u_2-u_g|$ as follows. From the definition of $\rho_g u_g$, we have
    \begin{align*}
        \left|\rho_2\left(u_2-\frac{z}{t}\hat{z}\right) - \rho_g\left(u_g-\frac{z}{t}\hat{z}\right)\right|(t,x)&\leq \int_0^t e^{-(t-\tau)}\int_{\mathbb{R}^n} \left|v-\frac{z}{t}\hat{z}\right|\mathcal{M}(f_1)(\tau, x-v(t-\tau), v)\mathbf{1}_{z-v_z(t-\tau)>0}\,dv d\tau \\
        &\leq \frac{C_{\mathcal{M}}}{C_1\sqrt{C_4}}t\rho_g\sqrt{T_g}(t,x).
    \end{align*}
    From \eqref{eq:rho_diff} and \eqref{eq:u2_bound}, we get
    \begin{align*}
        \rho_g\left|u_2-u_g\right|(t,x)&\leq\left|\rho_2\left(u_2-\frac{z}{t}\hat{z}\right) - \rho_g\left(u_g-\frac{z}{t}\hat{z}\right)\right|(t,x) + \left|(\rho_g-\rho_2)\left(u_2-\frac{z}{t}\hat{z}\right)\right|(t,x)\\
        &\leq \frac{C_{\mathcal{M}}}{C_1\sqrt{C_4}}t\rho_g\sqrt{T_g}(t,x) + \frac{C_{\mathcal{M}}}{C_1}\frac{C_3 + 1}{\sqrt{C_4}}t\rho_g \sqrt{T_g}(t,x)\\
        &\leq \frac{C_{\mathcal{M}}}{C_1}\frac{C_3 + 2}{\sqrt{C_4}}t\rho_g\sqrt{T_g}(t,x).
    \end{align*}
    It shows
    \begin{align}\label{eq:u_diff}
        \left|u_2-u_g\right|(t,x)\leq \frac{C_{\mathcal{M}}}{C_1}\frac{C_3 + 2}{\sqrt{C_4}}t\sqrt{T_g}(t,x).
    \end{align}
    
    Next, we compute $n|\rho_g T_g - \rho_2 T_2|$. We explicitly write down the formula and bound it by
    \begin{align}\label{eq:T_diff_1}
        \begin{split}
            n|\rho_2 T_2 - \rho_g T_g|(t,x) &\leq \left|\int_{\mathbb{R}^n} \left(|v-u_2(t,x)|^2 - |v-u_g(t,x)|^2\right)e^{-\frac{z}{v_z}}g(t-t_b, x-vt_b, v)\mathbf{1}_{v_z\geq \frac{z}{t}}\,dv\right|\\
            &\quad + \int_0^t e^{-(t-\tau)}\int_{\mathbb{R}^n} \left||v-u_2(t,x)|^2 - \left|v-\frac{z}{t}\hat{z}\right|^2\right|\mathcal{M}(f_1)(\tau, x-v(t-\tau), v)\mathbf{1}_{z-v_z(t-\tau)>0}\,dv d\tau\\
            &\quad + \int_0^t e^{-(t-\tau)}\int_{\mathbb{R}^n} \left|v-\frac{z}{t}\hat{z}\right|^2 \mathcal{M}(f_1)(\tau, x-v(t-\tau), v)\mathbf{1}_{z-v_z(t-\tau)>0}\,dv d\tau.
        \end{split}
    \end{align}
    The difference term in the first integral is given by
    \begin{align*}
        (v-u_2(t,x))^2 - (v-u_g(t,x))^2 &= \left(u_2(t,x) + u_g(t,x) - 2v\right)\cdot \left(u_2(t,x) - u_g(t,x)\right)
    \end{align*}
    By the same calculation, we also have
    \begin{align*}
        (v-u_2(t,x))^2 - \left(v-\frac{z}{t}\hat{z}\right)^2=\left(u_2(t,x) + \frac{z}{t}\hat{z} - 2v\right)\cdot \left(u_2(t,x) - \frac{z}{t}\hat{z}\right)
    \end{align*}
    in the second integral.
    
    Using this equality and \eqref{eq:u_diff}, we bound the first integral by
    \begin{align*}
        &\int_{\mathbb{R}^n} \left(|v-u_2(t,x)|^2 - |v-u_g(t,x)|^2\right)e^{-\frac{z}{v_z}}g(t-t_b, x-vt_b, v)\mathbf{1}_{v_z\geq \frac{z}{t}}\,dv\\
        &= (u_2-u_g)\cdot \int_{\mathbb{R}^n}\left(u_2(t,x) + u_g(t,x) - 2v\right)e^{-\frac{z}{v_z}}g(t-t_b, x-vt_b,v)\mathbf{1}_{v_z\geq \frac{z}{t}}\,dv\\
        &= \rho_g (u_2-u_g)^2\leq \left(\frac{C_{\mathcal{M}}}{C_1}\frac{C_3 + 2}{\sqrt{C_4}} t\right)^2\rho_g T_g(t,x).
    \end{align*}
    Using \eqref{eq:macro_M(f_1)} and \eqref{eq:u2_bound}, the second integral can be bounded similarly:
    \begin{align*}
        &\int_0^t e^{-(t-\tau)}\int_{\mathbb{R}^n} \left||v-u_2(t,x)|^2 - \left|v-\frac{z}{t}\hat{z}\right|^2\right|\mathcal{M}(f_1)(\tau, x-v(t-\tau), v)\mathbf{1}_{z-v_z(t-\tau)>0}\,dv d\tau\\
        &\leq \left|u_2(t,x) - \frac{z}{t}\hat{z}\right|\int_0^t e^{-(t-\tau)}\int_{\mathbb{R}^n}\left(\left|u_2(t,x) - \frac{z}{t}\hat{z}\right| + 2\left|v-\frac{z}{t}\hat{z}\right|\right)\mathcal{M}(f_1)(\tau, x-v(t-\tau), v)\mathbf{1}_{z-v_z(t-\tau)>0}\,dv d\tau\\
        &\leq C_{\mathcal{M}}C_1^{-1} t\rho_g \left(\left|u_2 - \frac{z}{t}\hat{z}\right|^2 + 2\left|u_2 - \frac{z}{t}\hat{z}\right|\sqrt{C_4}^{-1} \sqrt{T_g}\right)(t,x)\\
        &\leq \frac{C_{\mathcal{M}}}{C_1 C_4}(C_3 + 1)(C_3+3) t \rho_g T_g(t,x).
    \end{align*}
    Finally, the last integral in \eqref{eq:T_diff_1} is bounded by $\frac{C_{\mathcal{M}}}{C_1C_4}t\rho_gT_g(t,x)$ by \eqref{eq:macro_M(f_1)}. 
    
    Combining all the results, we conclude
    \begin{align*}
        n|\rho_2 T_2-\rho_g T_g|(t,x)\leq \frac{C_{\mathcal{M}}}{C_1C_4}\left(\frac{C_{\mathcal{M}}}{C_1} (C_3 + 2)^2 t +(C_3+1)(C_3+3) + 1\right) t\rho_g T_g(t,x).
    \end{align*}
    From this and \eqref{eq:rho_diff}, we have
    \begin{align*}
        \rho_2|T_2-T_g|&\leq |\rho_2 T_2-\rho_g T_g| + |(\rho_g - \rho_2)T_g|\\
        &\leq \frac{1}{n}\frac{C_{\mathcal{M}}}{C_1C_4}\left(\frac{C_{\mathcal{M}}}{C_1} (C_3 + 2)^2 t +(C_3+1)(C_3+3) + 1\right) t\rho_g T_g(t,x) + \frac{C_{\mathcal{M}}}{C_1}t\rho_gT_g(t,x),
    \end{align*}
    and as $\frac{\rho_g}{\rho_2}\leq 1$, we get
    \begin{align*}
        |T_2-T_g|\leq \frac{C_{\mathcal{M}}}{C_1}\left(\frac{1}{nC_4}\left(\frac{C_{\mathcal{M}}}{C_1} (C_3 + 2)^2 t +(C_3+1)(C_3+3) + 1\right) + 1\right) tT_g(t,x).
    \end{align*}
    Now, we can choose small enough $t$ depending on $C_{\mathcal{M}}$ and $C_i$ to make
    \begin{align*}
        |T_2-T_g|(t,x)\leq \frac{1}{2}T_g(t,x).
    \end{align*}
\end{proof}
If we choose $C_u = \frac{C_3 + 1}{\sqrt{C_4}}$, Proposition~\ref{prop:apriori} verifies that the BGK solution fulfills \eqref{eq:X_T_A3} \textit{a priori}.

\section{Fixed-point construction of the BGK solution}\label{sec:sol_construction}
In this section, we establish a unique solution to the BGK model with the boundary data $g(t,x_h,v)$ given in \eqref{eq:boundary_data} and vacuum initial data. Possessing the \textit{a priori} estimate \eqref{eq:macro_bound_f_1}, we define a solution map on the weighted space $\Omega_{\frac{1}{1+\beta}, t_0}$ defined in Definition~\ref{def:BGK_sol_space} and prove that the map is a contraction map in a short time interval $[0, t_0]$; the local well-posedness of the solution follows immediately by the fixed-point argument.

We recall the time-dependent weight function and the weighted norm
\begin{align*}
    \mathfrak{w}(t,x,v) &= 1+\left(|x_h| + \frac{z}{t}\right)^{\frac{2}{1+\beta}} + \left(v-\frac{z}{t}\hat{z}\right)^2,\\
    \text{norm}&:\sup_{0< t\leq t_0}\sup_{x\in\mathbb{R}^n_+}\int_{v\in \mathbb{R}^n}\mathfrak{w}(t,x,v)f(t,x,v)\,dv
\end{align*}
in Definition \ref{def:BGK_sol_space} for $\Omega_{\frac{1}{1+\beta}, t_0}$. We note that $\Omega_{\frac{1}{1+\beta}, t_0}$ is non-empty when $0<t_0\leq \frac{1}{4}$ since $f(t,x,v) = e^{-t_b}g(t-t_b,x-vt_b, v)\mathbf{1}_{t\geq t_b}$ is in $\Omega_{\frac{1}{1+\beta}, t_0}$ by the results in Proposition~\ref{prop:boundary_int}.

\subsection{Lipschitz continuity of the Maxwellian operator}
The following proposition establishes the Lipschitz continuity of the local Maxwellian operator $\mathcal{M}$ on $\Omega_{\frac{1}{1+\beta}, t_0}$. It will be used in the proof of the contraction property of the solution map.
\begin{proposition}\label{prop:maxwell}
    For given inflow boundary data $g$ in \eqref{eq:boundary_data}, assume that $f_1,f_2\in \Omega_{\frac{1}{1+\beta}, t_0}$. By \eqref{eq:X_T_A3}, there exist constants $C_{u,1}$ and $C_{u,2}$ such that
    \begin{align*}
        \left|u_1-\frac{z}{t}\hat{z}\right|(t,x)&\leq C_{u,1}\sqrt{T_g(t,x)}\\
        \left|u_2-\frac{z}{t}\hat{z}\right|(t,x)&\leq C_{u,2}\sqrt{T_g(t,x)},
    \end{align*}
    where $(\rho_i, u_i, T_i)$ are the macroscopic fields generated by $f_i$ for $i=1,2$. Then 
    \begin{align}\label{Inequ:maxwell}
        \begin{split}
            &\esssup_{x\in\mathbb{R}^n_+}\int_{\mathbb{R}^n}\mathfrak{w}(t,x,v)\left|\mathcal{M}(f_1)-\mathcal{M}(f_2)\right|(\tau, x-v(t-\tau), v)\mathbf{1}_{z-v_z(t-\tau)>0}\,dv\\
            &\leq C\esssup_{x\in\mathbb{R}^n_+}\int_{\mathbb{R}^n}\mathfrak{w}(\tau,x,v)|f_1-f_2|(\tau,x,v)\,dv
        \end{split}
    \end{align}
    for almost every $0<\tau\leq t$ and all $0<t\leq t_0$. Here, the constant $C<\infty$ depends on the boundary data $g$ in \eqref{eq:boundary_data}, $\max\{C_{u,1}, C_{u,2}\}$, and dimension $n$.
\end{proposition}
\begin{proof}
    First, we define a linear combination of $f_1$ and $f_2$ by $f_\lambda = (1-\lambda)f_1 + \lambda f_2$ for $0\leq \lambda\leq 1$. The corresponding macroscopic fields are defined as follows:
    \begin{align*}
        \rho_\lambda &= \int_{\mathbb{R}^n}(1-\lambda)f_1 + \lambda f_2\,dv,\\
        \rho_\lambda u_\lambda &= \int_{\mathbb{R}^n}v\left((1-\lambda)f_1 + \lambda f_2\right)\,dv,\\
        \rho_\lambda\left(nT_\lambda + u_\lambda^2\right) &= \int_{\mathbb{R}^n}v^2\left((1-\lambda)f_1 + \lambda f_2\right)\,dv.
    \end{align*}
    We first claim that $\rho_\lambda$, $u_\lambda$, and $T_\lambda$ satisfy
    \begin{align}\label{eq:macro_lambda_bound}
        \begin{split}
            &|\rho_\lambda-\rho_g|(t,x)\leq \frac{1}{2}\rho_g(t,x),\\
            &\left|u_\lambda-\frac{z}{t}\hat{z}\right|(t,x)\leq C_{u,max}\sqrt{T_g(t,x)},\\
            &\frac{1}{2}T_g(t,x)\leq T_\lambda(t,x)\leq \left(C_{u,max}^2 + \frac{3}{2}\right)T_g(t,x)
        \end{split}
    \end{align}
    for all $0\leq \lambda\leq 1$, where $C_{u,max} =\max\{C_{u,1},C_{u,2}\}$, and $C_3$ and $C_4$ are from \eqref{eq:macro_0_bound}. Indeed, $\rho_\lambda$ is a linear combination of $\rho_1$ and $\rho_2$, so it satisfies
    \begin{align*}
        |\rho_\lambda-\rho_g|(t,x)\leq \frac{1}{2}\rho_g(t,x)
    \end{align*}
    by \eqref{eq:X_T_A3}.
    
    For $u_\lambda$, we write
    \begin{align*}
        \left(u_\lambda - \frac{z}{t}\hat{z}\right)(t,x) = \frac{1}{\rho_\lambda}\left((1-\lambda)\rho_1\left(u_1-\frac{z}{t}\hat{z}\right) + \lambda\rho_2\left(u_2-\frac{z}{t}\hat{z}\right)\right),
    \end{align*}
    so we have
    \begin{align*}
        \left|u_\lambda-\frac{z}{t}\hat{z}\right|(t,x)\leq \max\left\{\left|u_1-\frac{z}{t}\hat{z}\right|,\left|u_2-\frac{z}{t}\hat{z}\right|\right\}\leq \max\{C_{u,1},C_{u,2}\}\sqrt{T_g(t,x)}.
    \end{align*}
    
    For $T_\lambda$, it is given by
    \begin{align*}
        n T_\lambda = \frac{1}{\rho_\lambda}\left((1-\lambda)\rho_1(n T_1 + u_1^2) + \lambda\rho_2(n T_2 + u_2^2) - \rho_\lambda u_\lambda^2\right).
    \end{align*}
    By an explicit calculation, we have
    \begin{align*}
        (1-\lambda) \rho_1 u_1^2 +\lambda\rho_2u_2^2 - \rho_\lambda u_\lambda^2 &=(1-\lambda) \rho_1 u_1^2 +\lambda\rho_2u_2^2 - \frac{((1-\lambda)\rho_1u_1 + \lambda\rho_2u_2)^2}{(1-\lambda) \rho_1 + \lambda\rho_2}\\
        &=\lambda(1-\lambda)\rho_1\rho_2\frac{u_1^2 + u_2^2 - 2u_1\cdot u_2}{(1-\lambda) \rho_1 + \lambda\rho_2}\\
        &=\lambda(1-\lambda)\rho_1\rho_2\frac{(u_1-u_2)^2}{\rho_\lambda}.
    \end{align*}
    Therefore,
    \begin{align*}
        nT_\lambda = \frac{\lambda(1-\lambda)\rho_1\rho_2}{\rho_\lambda^2}|u_1-u_2|^2 + \frac{(1-\lambda) \rho_1}{\rho_\lambda} nT_1 + \frac{\lambda \rho_2}{\rho_\lambda} nT_2.
    \end{align*}
    The upper and lower bounds of $T_\lambda - T_g$ are computed as follows.
    \begin{align*}
        n(T_\lambda - T_g) &= \frac{\lambda(1-\lambda)\rho_1\rho_2}{\rho_\lambda^2}|u_1-u_2|^2 + \frac{(1-\lambda) \rho_1}{\rho_\lambda} nT_1 + \frac{\lambda \rho_2}{\rho_\lambda} nT_2 - nT_g\\
        &\leq \frac{1}{4}|u_1-u_2|^2 + \frac{n}{2}T_g,\\
        n(T_\lambda - T_g) &\geq \frac{(1-\lambda) \rho_1}{\rho_\lambda} nT_1 + \frac{\lambda \rho_2}{\rho_\lambda} nT_2 - nT_g\geq -n\max\{|T_1-T_g|,|T_2-T_g|\}\\
        &\geq -\frac{n}{2}T_g.
    \end{align*}
    Since
    \begin{align*}
        \frac{\lambda(1-\lambda)\rho_1\rho_2}{((1-\lambda)\rho_1 + \lambda \rho_2)^2}\leq \frac{1}{4}\quad \text{and}\quad |u_1-u_2|\leq \left|u_1-\frac{z}{t}\hat{z}\right| + \left|u_2-\frac{z}{t}\hat{z}\right|\leq 2\max\{C_{u,1},C_{u,2}\}\sqrt{T_g(t,x)},
    \end{align*}
    we get
    \begin{align*}
        \frac{1}{2}T_g(t,x)\leq T_\lambda(t,x)\leq \left((\max\{C_{u,1},C_{u,2}\})^2 + \frac{3}{2}\right)T_g(t,x).
    \end{align*}
    This proves the claim.

    By the Fundamental Theorem of Calculus, the left-hand side of \eqref{Inequ:maxwell} can be written as
    \begin{align}\label{eq:contract_prep1}
        \begin{split}
            &\int_{\mathbb{R}^n}\mathfrak{w}(t,x,v)\left|\mathcal{M}(f_1)-\mathcal{M}(f_2)\right|(\tau, x-v(t-\tau), v)\mathbf{1}_{z-v_z(t-\tau)> 0}\,dv\\
            &\leq \int_{\mathbb{R}^n}\int_0^1 \mathfrak{w}(t,x,v)\left|\dv{\lambda}\mathcal{M}(f_\lambda)\right|(\tau, x-v(t-\tau), v)\mathbf{1}_{z-v_z(t-\tau)> 0}\,d\lambda dv.
        \end{split}
    \end{align}
    Since the local Maxwellian depends smoothly on the macroscopic fields
    $(\rho_\lambda,u_\lambda,T_\lambda)$, differentiation with respect to the
    interpolation parameter $\lambda$ yields
    \begin{align*}
        \frac{d}{d\lambda}\mathcal{M}(f_\lambda):=\frac{d\rho_\lambda}{d\lambda}\frac{\partial}{\partial\rho_\lambda}\mathcal{M}(f_\lambda)+\frac{d u_\lambda}{d\lambda}\frac{\partial}{\partial u_\lambda}\mathcal{M}(f_\lambda)+\frac{d T_\lambda}{d\lambda}\frac{\partial}{\partial T_\lambda}\mathcal{M}(f_\lambda).
    \end{align*}
    Now, the following equalities can be obtained by a straightforward differentiation starting from the definitions of $\rho_\lambda$, $u_\lambda$, and $T_\lambda$.
    \begin{align*}
        \dv{\lambda}\rho_\lambda(t,x) &= \int_{\mathbb{R}^n}(f_2-f_1)(t,x,w)\,dw,\\
        \frac{\rho_\lambda^2}{\rho_1}\dv{\lambda} u_\lambda(t,x) &=
        \int_{\mathbb{R}^n}\left(w-\frac{z}{t}\hat{z}\right)(f_2-f_1)(t,x,w)\,dw - \left(u_1(t,x)-\frac{z}{t}\hat{z}\right)\int_{\mathbb{R}^n}(f_2-f_1)(t,x,w)\,dw,\\
        n\frac{\rho_\lambda^2}{\rho_1}\dv{\lambda} T_\lambda(t,x) &=\int_{\mathbb{R}^n}\left(w-\frac{z}{t}\hat{z}\right)^2(f_2-f_1)(t,x,w)\,dw - \left(nT_1 + \left(u_1 - \frac{z}{t}\hat{z}\right)^2\right)\int_{\mathbb{R}^n}(f_2-f_1)(t,x,w)\,dw\\
        &\quad - 2\left(u_\lambda-\frac{z}{t}\hat{z}\right) \cdot\left(\int_{\mathbb{R}^n}\left(w-\frac{z}{t}\hat{z}\right)(f_2-f_1)(t,x,w)\,dw - \left(u_1-\frac{z}{t}\hat{z}\right)\int_{\mathbb{R}^n}(f_2-f_1)(t,x,w)\,dw\right).
    \end{align*}
    Generally, the same conclusion remains valid when $\frac{z}{t}\hat{z}$ is replaced by any fixed vector $c(t,x)$. From \eqref{eq:X_T_A3} and \eqref{eq:macro_lambda_bound}, we further simplify these equations by
    \begin{align}\label{eq:macro_deri}
        \begin{split}
            \left|\dv{\lambda}\rho_\lambda\right|(t,x) &\leq  \int_{\mathbb{R}^n}|f_2-f_1|(t,x,w)\,dw,\\
            \frac{\rho_\lambda^2}{\rho_1}\left|\dv{\lambda} u_\lambda\right|(t,x) &\leq
            \int_{\mathbb{R}^n}\left|w-\frac{z}{t}\hat{z}\right||f_2-f_1|(t,x,w)\,dw + C_{u,max}\sqrt{T_g(t,x)}\int_{\mathbb{R}^n}|f_2-f_1|(t,x,w)\,dw,\\
            n\frac{\rho_\lambda^2}{\rho_1}\left|\dv{\lambda} T_\lambda\right|(t,x) &\leq\int_{\mathbb{R}^n}\left|w-\frac{z}{t}\hat{z}\right|^2|f_2-f_1|(t,x,w)\,dw + C T_g(t,x)\int_{\mathbb{R}^n}|f_2-f_1|(t,x,w)\,dw\\
            &\quad + 2C_{u,max}\sqrt{T_g(t,x)} \int_{\mathbb{R}^n}\left|w-\frac{z}{t}\hat{z}\right||f_2-f_1|(t,x,w)\,dw,
        \end{split}
    \end{align}
    where $C = \frac{3n}{2} +3C_{u,max}^2$.
    
    We next estimate the derivative of the local Maxwellian $\mathcal{M}(f_\lambda)$. Combining \eqref{eq:macro_deri} with the explicit form of the local Maxwellian, we obtain
    \begin{align*}
        \left|\dv{\rho_\lambda}{\lambda}\dv{\rho_\lambda} \mathcal{M}(f_\lambda)\right|(t,x,v)&\leq \frac{1}{(2\pi T_\lambda)^{n/2}} \exp\left(-\frac{|v-u_\lambda|^2}{2T_\lambda}\right)\int_{\mathbb{R}^n}|f_2-f_1|(t,x,w)\,dw,\\
        \left|\dv{u_\lambda}{\lambda}\dv{u_\lambda} \mathcal{M}(f_\lambda)\right|(t,x,v)&\leq \frac{\rho_1\rho_\lambda^{-1}}{(2\pi T_\lambda)^{n/2}} \frac{|v - u_\lambda|}{T_\lambda} \exp\left(-\frac{|v-u_\lambda|^2}{2T_\lambda}\right)\\
        &\qquad \times \left( \int_{\mathbb{R}^n}\left|w-\frac{z}{t}\hat{z}\right||f_2-f_1|(t,x,w)\,dw + C_{u,max}\sqrt{T_g}\int_{\mathbb{R}^n}|f_2-f_1|(t,x,w)\,dw\right),\\
        \left|\dv{T_\lambda}{\lambda}\dv{T_\lambda} \mathcal{M}(f_\lambda)\right|(t,x,v)&\leq \frac{\rho_1\rho_\lambda^{-1}}{n(2\pi T_\lambda)^{n/2}}\left(\frac{n}{2}\frac{1}{T_\lambda} + \frac{|v-u_\lambda|^2}{2T_\lambda^2}\right) \exp\left(-\frac{|v-u_\lambda|^2}{2T_\lambda}\right)\\
        &\qquad \times \left(\int_{\mathbb{R}^n}\left|w-\frac{z}{t}\hat{z}\right|^2|f_2-f_1|(t,x,w)\,dw + CT_g\int_{\mathbb{R}^n}|f_2-f_1|(t,x,w)\,dw\right.\\
        &\qquad \quad \left.+ 2C_{u,max}\sqrt{T_g} \int_{\mathbb{R}^n}\left|w-\frac{z}{t}\hat{z}\right||f_2-f_1|(t,x,w)\,dw\right).
    \end{align*}
    Now, we use the elementary inequalities 
    \[
    x e^{-\frac{x^2}{4}} \le 1, \qquad
    x^2 e^{-\frac{x^2}{4}} \le 2 \quad \text{for } x \ge 0.
    \]
    for \( x = \frac{|v-u_\lambda|}{\sqrt{T_\lambda}} \) to get
    \begin{align*}
        |u_\lambda - v| \exp\left(-\frac{|v-u_\lambda|^2}{4T_\lambda}\right)\leq \sqrt{T_\lambda},\quad |u_\lambda - v|^2\exp\left(-\frac{|v-u_\lambda|^2}{4T_\lambda}\right)\leq 2T_\lambda.
    \end{align*}
    Substituting these bounds into the previous estimates with some rearrangements, we obtain
    \begin{align*}
        \left|\dv{\rho_\lambda}{\lambda}\dv{\rho_\lambda} \mathcal{M}(f_\lambda)(t,x,v)\right|&\leq \frac{2\pi}{(2\pi T_\lambda)^{n/2 + 1}} \exp\left(-\frac{|v-u_\lambda|^2}{4T_\lambda}\right)\int_{\mathbb{R}^n}T_\lambda |f_2-f_1|(t,x,w)\,dw,\\
        \left|\dv{u_\lambda}{\lambda}\dv{u_\lambda} \mathcal{M}(f_\lambda)(t,x,v)\right|&\leq \frac{2\pi\rho_1\rho_\lambda^{-1}}{(2\pi T_\lambda)^{n/2 + 1}} \exp\left(-\frac{|v-u_\lambda|^2}{4T_\lambda}\right)\\
        &\qquad \times \left(\int_{\mathbb{R}^n}\sqrt{T_\lambda}\left|w-\frac{z}{t}\hat{z}\right||f_2-f_1|(t,x,w)\,dw + C_{u,max}\sqrt{T_\lambda T_g}\int_{\mathbb{R}^n}|f_2-f_1|(t,x,w)\,dw\right),\\
        \left|\dv{T_\lambda}{\lambda}\dv{T_\lambda} \mathcal{M}(f_\lambda)(t,x,v)\right|&\leq \frac{2\pi\rho_1\rho_\lambda^{-1}}{(2\pi T_\lambda)^{n/2 + 1}}\left(\frac{1}{2} + \frac{1}{n}\right) \exp\left(-\frac{|v-u_\lambda|^2}{4T_\lambda}\right)\\
        &\qquad \times \int_{\mathbb{R}^n}\left(\left|w-\frac{z}{t}\hat{z}\right|^2 + 2C_{u,max}\sqrt{T_g}\left|w-\frac{z}{t}\hat{z}\right| +  CT_g\right)|f_2-f_1|(t,x,w)\,dw.
    \end{align*}
    From \eqref{eq:macro_lambda_bound}, $\rho_1 \rho_\lambda^{-1}\leq 3$ and $T_\lambda \leq \left(C_{u,max}^2 + \frac{3}{2}\right) T_g$, so combining all the bounds, we get
    \begin{align*}
        &\int_{\mathbb{R}^n} \mathfrak{w}(t,x,v)\left|\dv{\lambda}\mathcal{M}(f_\lambda)\right|(\tau, x-v(t-\tau), v)\mathbf{1}_{z-v_z(t-\tau)> 0}\,dv\\
        &\leq \int_{\mathbb{R}^n}\mathfrak{w}(t,x,v)\left(\left|\frac{d\rho_\lambda}{d\lambda}\frac{\partial}{\partial\rho_\lambda}\mathcal{M}(f_\lambda)\right|+\left|\frac{d u_\lambda}{d\lambda}\frac{\partial}{\partial u_\lambda}\mathcal{M}(f_\lambda)\right|+\left|\frac{d T_\lambda}{d\lambda}\frac{\partial}{\partial T_\lambda}\mathcal{M}(f_\lambda)\right|\right)(\tau, x-v(t-\tau), v)\mathbf{1}_{z-v_z(t-\tau)> 0}\,dv\\
        &\leq C\int_{\mathbb{R}^n}\frac{\mathfrak{w}(t,x,v)}{(2\pi T_\lambda(\tau, x-v(t-\tau)))^{n/2 + 1}} \exp\left(-\frac{|v-u_\lambda(\tau, x-v(t-\tau))|^2}{4T_\lambda(\tau, x-v(t-\tau))}\right)\mathbf{1}_{z-v_z(t-\tau)> 0}\\
        &\qquad \times \int_{\mathbb{R}^n}\left(\left|w-\frac{z-v_z(t-\tau)}{\tau}\hat{z}\right| + \sqrt{T_g(\tau, x-v(t-\tau))}\right)^2 |f_2-f_1|(\tau, x-v(t-\tau),w)\,dw \,dv\\
        &\leq C\int_{\mathbb{R}^n}\frac{\mathfrak{w}(t,x,v)}{(2\pi T_\lambda(\tau, x-v(t-\tau)))^{n/2 + 1}} \exp\left(-\frac{|v-u_\lambda(\tau, x-v(t-\tau))|^2}{4T_\lambda(\tau, x-v(t-\tau))}\right)\mathbf{1}_{z-v_z(t-\tau)> 0} \,dv\\
        &\qquad \times \esssup_{x\in\mathbb{R}^n_+}\int_{\mathbb{R}^n}\left(\left|w-\frac{z}{\tau}\hat{z}\right| + \sqrt{T_g(\tau, x)}\right)^2 |f_2-f_1|(\tau, x,w)\,dw.
    \end{align*}
    for a constant $C$ depending on the dimension $n$ and $C_u$ for almost every $x$ and $\tau$.
    
    As $T_g(\tau,x)\leq C_5\left(|x_h|+\frac{z}{\tau} + 1\right)^{\frac{2}{1+\beta}}$ from \eqref{eq:macro_0_bound}, we have
    \begin{align*}
        \int_{\mathbb{R}^n}\left(\left|w-\frac{z}{\tau}\hat{z}\right| + \sqrt{T_g(\tau, x)}\right)^2 |f_2-f_1|(\tau, x,w)\,dw\leq C\int_{\mathbb{R}^n}\mathfrak{w}(\tau,x,w)|f_2-f_1|(\tau,x,w)\,dw.
    \end{align*}
    Summarizing the current result, we have
    \begin{align*}
        &\int_{\mathbb{R}^n} \mathfrak{w}(t,x,v)\left|\dv{\lambda}\mathcal{M}(f_\lambda)\right|(\tau, x-v(t-\tau), v)\mathbf{1}_{z-v_z(t-\tau)> 0}\,dv\\
        &\leq C\int_{\mathbb{R}^n} \mathfrak{w}(t,x,v)\frac{1}{(2\pi T_\lambda(\tau, x-v(t-\tau)))^{n/2 + 1}} \exp\left(-\frac{|v-u_\lambda(\tau, x-v(t-\tau))|^2}{4T_\lambda(\tau, x-v(t-\tau))}\right)\mathbf{1}_{z-v_z(t-\tau)> 0}\,dv\\
        &\qquad \times\esssup_{x\in\mathbb{R}^n_+}\int_{\mathbb{R}^n}\mathfrak{w}(\tau,x,w)|f_2-f_1|(\tau,x,w)\,dw.
    \end{align*}
    Our remaining task is to bound the $v$-integral. We first decompose the $x_h$ and $\frac{z}{t}$ in $\mathfrak{w}(t,x,v)$ by
    \begin{align*}
        x_h = x_h-v_h(t-\tau) + v_h(t-\tau)\quad \text{and} \quad \frac{z}{t} = \frac{z-v_z(t-\tau)}{\tau} + \frac{(t-\tau)}{\tau}\left(v_z-\frac{z}{t}\right)
    \end{align*}
    to extract $x_h-v_h(t-\tau) + \frac{z-v_z(t-\tau)}{\tau}$ of $T(\tau, x-v(t-\tau))$. Separating the terms, we rearrange the integral by
    \begin{align*}
        &\int_{\mathbb{R}^n} \left(1+(|x_h|+\frac{z}{t})^{\frac{2}{1+\beta}} + \left(v-\frac{z}{t}\hat{z}\right)^2\right) \frac{1}{T_\lambda^{n/2 + 1}}\exp\left(-\frac{|v-u_\lambda|^2}{4T_\lambda}\right)(\tau, x-v(t-\tau))\mathbf{1}_{z - v_z(t-\tau)>0}\,dv\\
        &\leq\int_{\mathbb{R}^n} \left(1+\left(|x_h-v_h(t-\tau)|+\frac{z-v_z(t-\tau)}{\tau}\right)^{\frac{2}{1+\beta}} + \left(|v_h(t-\tau)|+\left|\frac{(t-\tau)}{\tau}\left(v_z-\frac{z}{t}\right)\right|\right)^{\frac{2}{1+\beta}} + \left(v-\frac{z}{t}\hat{z}\right)^2\right)\\
        &\qquad \times \frac{1}{T_\lambda(\tau, x-v(t-\tau))^{n/2 + 1}}\exp\left(-\frac{|v-u_\lambda(\tau, x-v(t-\tau))|^2}{4T_\lambda(\tau, x-v(t-\tau))}\right)\mathbf{1}_{z - v_z(t-\tau)>0}\,dv\\
        &\leq C\int_{\mathbb{R}^n} \frac{1}{T_\lambda(\tau, x-v(t-\tau))^{n/2}}\exp\left(-\frac{|v-u_\lambda(\tau, x-v(t-\tau))|^2}{4T_\lambda(\tau, x-v(t-\tau))}\right)\mathbf{1}_{z - v_z(t-\tau)>0}\,dv\\
        &\quad + \int_{\mathbb{R}^n} \left(\left(|v_h(t-\tau)|+\frac{t-\tau}{\tau}\left|(v_z - \frac{z}{t})\right|\right)^{\frac{2}{1+\beta}} + \left(v-\frac{z}{t}\hat{z}\right)^2\right)\\
        &\qquad \times \frac{1}{T_\lambda(\tau, x-v(t-\tau))^{n/2 + 1}}\exp\left(-\frac{|v-u_\lambda(\tau, x-v(t-\tau))|^2}{4T_\lambda(\tau, x-v(t-\tau))}\right)\mathbf{1}_{z - v_z(t-\tau)>0}\,dv.
    \end{align*}
    Applying Proposition~\ref{prop:boundary_int} to \eqref{eq:macro_lambda_bound}, we obtain an upper bound of order
    $(|x_h|+\frac{z}{t}+1)^{\frac{1}{1+\beta}}$ for $|u_\lambda-\frac{z}{t}\hat{z}|$ and two-sided bounds of order $(|x_h|+\frac{z}{t}+1)^{\frac{2}{1+\beta}}$ for $T_\lambda$. Thus, we can directly employ Corollary~\ref{prop:bound} to bound the two integrals by a uniform constant. In conclusion, we obtain the following result
    \begin{align*}
        &\sup_{0\leq \lambda\leq 1}\int_{\mathbb{R}^n} \mathfrak{w}(t,x,v)\left|\dv{\lambda}\mathcal{M}(f_\lambda)\right|(\tau, x-v(t-\tau), v)\mathbf{1}_{z - v_z(t-\tau)>0}\,dv\\
        &\leq C\esssup_{x\in\mathbb{R}^n_+}\int_{\mathbb{R}^n}\mathfrak{w}(\tau,x,w)|f_2-f_1|(\tau,x,w)\,dw
    \end{align*}
    for almost every $0<\tau\leq t$ and $x\in\mathbb{R}^n_+$ for all $0<t\leq t_0$. Since the constant $C$ only depends on the inflow boundary data $g$, $C_{u,max}$, and the dimension $n$, by \eqref{eq:contract_prep1}, we prove the proposition.
    \end{proof}
    
    \subsection{Existence and uniqueness via a contraction mapping}
    With Proposition~\ref{prop:maxwell}, we are now ready to prove the existence and uniqueness of solutions via the Banach fixed point theorem.
    
    \begin{theorem}\label{thm:BGK_sol}
        For the inflow boundary data $g$ defined in \eqref{eq:boundary_data} and the initial data $f_0(x,v)=0$ for $x\in \mathbb{R}^n_+$, there exists $t_0>0$ and a unique solution $f(t,x,v)$ of the BGK model in the space $\Omega_{\frac{1}{1+\beta}, t_0}$. Furthermore, we can choose a version of $f(t,x,v)$ such that it satisfies \eqref{eq:X_T_A3} pointwise.
    \end{theorem}
    
    \begin{proof} 
    For given inflow boundary data $g$ in \eqref{eq:boundary_data}, we define the solution map $\mathcal{T}$ by
    \begin{align*}
        \mathcal{T}f(t,x,v) := e^{-t_b}g(t-t_b, x - vt_b, v)\mathbf{1}_{t\geq t_b} + \int_0^t e^{-(t - \tau)} \mathcal{M}(f)(\tau, x - v(t - \tau), v)\mathbf{1}_{z-v_z(t-\tau)> 0}\,d\tau.
    \end{align*}
    From Proposition~\ref{prop:boundary_int}, we have the macroscopic fields bounds \eqref{eq:macro_0_bound} of $(\rho_g,u_g,T_g)$ generated by $e^{-t_b}g(t-t_b, x - vt_b, v)\mathbf{1}_{t\geq t_b}$. We choose $C_u = \frac{C_3+1}{\sqrt{C_4}}$ and $0<t_0^*\leq \frac{1}{4}$ depending on the inflow boundary data $g$ from Proposition~\ref{prop:apriori}. We construct a function sequence by
    \begin{align*}
        f^1(t,x,v) = e^{-t_b}g(t-t_b, x - vt_b, v)\mathbf{1}_{t\geq t_b},
    \end{align*}
    and $f^i = \mathcal{T}f^{i-1}$ for $i\geq 2$. Then, we claim that $\{f^i\}_{i\in\mathbb{N}}\subset \Omega_{\frac{1}{1+\beta}, t_0^*}$. Indeed,
    \begin{enumerate}
        \item by the construction, $f^i(0,x,v) = 0$ for all $i\geq 1$.
        \item By Proposition~\ref{prop:apriori}, since $f^1$ trivially satisfies \eqref{eq:macro_bound_f_1} by definition, $f^i = \mathcal{T}f^{i-1}$ for all $i\geq 2$ satisfies \eqref{eq:macro_bound_f_1}, so \eqref{eq:X_T_A3} for all $(t,x)$.
        \item Since all $f^i$ satisfies \eqref{eq:X_T_A3} in the pointwise sense, using Propositions~\ref{prop:boundary_int} and \ref{prop:M[f]_integral}, we have
        \begin{align*}
            &\int_{\mathbb{R}^n} \mathfrak{w}(t,x,v)e^{-t_b}g(t-t_b, x - vt_b, v)\mathbf{1}_{t\geq t_b}\,dv\\
            &\leq C\left(|x_h| + \frac{z}{t} + 1\right)^{\frac{n+2-\beta-m}{1+\beta}}e^{-\alpha(|x_h| + \frac{z}{t}+1)^{\frac{\beta}{1+\beta}}},
            \intertext{and}
            &\sup_{0<\tau\leq t}\int_{\mathbb{R}^n} \mathfrak{w}(t,x,v)\mathcal{M}(f^i)(\tau, x-v(t-\tau), v)\mathbf{1}_{z-v_z(t-\tau)> 0}\,dv\\
            &\leq C\left(|x_h| + \frac{z}{t} + 1\right)^{\frac{n+2-\beta-m}{1+\beta}}e^{-\alpha(|x_h| + \frac{z}{t}+1)^{\frac{\beta}{1+\beta}}}
        \end{align*}
        for all $0<t\leq t_0^*$ and $x\in\mathbb{R}^n_+$. Therefore,
        \begin{align*}
            \sup_{0<t\leq t_0^*}\sup_{x\in\mathbb{R}^n_+}\int_{\mathbb{R}^n} \mathfrak{w}(t,x,v)f^i(t,x,v)\,dv\leq C<\infty
        \end{align*}
        for a constant $C$ independent of $i$.
    \end{enumerate}
    Now, we can employ Proposition~\ref{prop:maxwell} for $f^i$: for $f^i$ and $f^{i-1}$ for $i\geq 2$, we have
    \begin{align*}
        &\sup_{x\in\mathbb{R}^n_+}\int_{\mathbb{R}^n} \mathfrak{w}(t,x,v)|\mathcal{T}f^i - \mathcal{T}f^{i-1}|(t,x,v)\,dv\\
        &\leq\sup_{x\in\mathbb{R}^n_+}\int_0^t\int_{\mathbb{R}^n} \mathfrak{w}(t,x,v)|\mathcal{M}(f^i)-\mathcal{M}(f^{i-1})|(\tau,x-v(t-\tau),v)\mathbf{1}_{z-v_z(t-\tau)> 0}\,dv d\tau\\
        &\leq Ct\sup_{0< \tau\leq t}\sup_{x\in\mathbb{R}^n_+} \int_{\mathbb{R}^n} \mathfrak{w}(\tau,x,v)|f^i-f^{i-1}|(\tau,x,v) \,dv
    \end{align*}
    for all $0<t\leq t_0^*$; since the macroscopic fields of $f^i$ satisfy \eqref{eq:X_T_A3} pointwise, we can use $\sup$ instead of $\esssup$. We choose $t_0 = \min\{\frac{1}{2C}, t^*_0\}$. Since $\Omega_{\frac{1}{1+\beta},t_0^*}\subset \Omega_{\frac{1}{1+\beta},t_0}$, we have $\{f^i\}_{i\in\mathbb{N}}\subset \Omega_{\frac{1}{1+\beta},t_0}$. Also, we obtain
    \begin{align*}
        \sup_{0<t\leq t_0} \sup_{x\in\mathbb{R}^n_+}\int_{\mathbb{R}^n} \mathfrak{w}(t,x,v)|\mathcal{T}f^i-\mathcal{T}f^{i-1}|(t,x,v)\,dv\leq \frac{1}{2}\sup_{0<t\leq t_0} \sup_{x\in\mathbb{R}^n_+}\int_{\mathbb{R}^n} \mathfrak{w}(t,x,v)|f^i - f^{i-1}|(t, x,v) \,dv,
    \end{align*}
    so it is a contraction mapping $\mathcal{T}$ for $\{f^i\}$ on $\Omega_{\frac{1}{1+\beta}, t_0}\subset L^\infty_{t,x}((0,t_0]\times\mathbb{R}^n_+, L^1_v(\mathfrak{w}))$. By the standard measure theoretic argument, we have $f^i(t,x,v)\rightarrow f(t,x,v) = f^1 + \sum_{i=2}^\infty (f^i-f^{i-1})$ (i) in $L^\infty_{t,x}((0,t_0]\times\mathbb{R}^n_+, L^1_v(\mathfrak{w}))$ and (ii) almost everywhere in $(t,x,v)$. Furthermore, the macroscopic fields of $f$ satisfy \eqref{eq:X_T_A3} for all $(t,x)$ since
    \begin{align*}
        \sum_{i=2}^\infty \sup_{0<t\leq t_0} \sup_{x\in\mathbb{R}^n_+} \int \mathfrak{w}(t,x,v)|f^i-f^{i-1}|(t,x,v)\,dv<\infty
    \end{align*}
    and the density of $f^i$ are all bounded below by $\frac{1}{2}\rho_g$. Therefore, it satisfies $f\in \Omega_{\frac{1}{1+\beta}, t_0}$ and $\mathcal{T}f = f$ with $f(0,x,v) = 0$. This proves the existence part.
    
    For the uniqueness, assume there exists another solution $h\in \Omega_{\frac{1}{1+\beta},t_0}$ with some $C_u'$ in \eqref{eq:X_T_A3}. For $t>0$, by Proposition~\ref{prop:maxwell}, we have
    \begin{align*}
        &\esssup_{x\in\mathbb{R}^n_+}\int_{\mathbb{R}^n} \mathfrak{w}(t,x,v)|f - h|(t,x,v)\,dv\\
        &=\esssup_{x\in\mathbb{R}^n_+}\int_{\mathbb{R}^n} \mathfrak{w}(t,x,v)|\mathcal{T}f - \mathcal{T}h|(t,x,v)\,dv\\
        &\leq \esssup_{x\in\mathbb{R}^n_+}\int_0^t\int_{\mathbb{R}^n} \mathfrak{w}(t,x,v)|\mathcal{M}(f)-\mathcal{M}(h)|(\tau,x-v(t-\tau),v)\mathbf{1}_{z-v_z(t-\tau)> 0}\,dv d\tau\\
        &\leq C\int_0^t \esssup_{x\in\mathbb{R}^n_+} \int_{\mathbb{R}^n} \mathfrak{w}(\tau,x,v)|f-h|(\tau,x,v) \,dv\,d\tau
    \end{align*}
    for almost every $0<t\leq t_0$ and for some constant $C$ depending on the boundary data $g$ in \eqref{eq:boundary_data}, $\max\{C_{u,1}, C_{u,2}\}$, and dimension $n$. By Gronwall's inequality, we have
    \begin{align*}
        \esssup_{0<t\leq t_0}\esssup_{x\in\mathbb{R}^n_+}\int_{\mathbb{R}^n} \mathfrak{w}(t,x,v)|f - h|(t,x,v)\,dv = 0.
    \end{align*}
    By the definition of $\Omega_{\frac{1}{1+\beta}, t_0}$, $h(0,x,v) = 0$ at $t=0$. Therefore, we conclude $f=h$ a.e. in $\Omega_{\frac{1}{1+\beta}, t_0}$.
\end{proof}

\section{Ill-posedness of the constructed BGK solution}\label{sec:Ill_posed}
In this section, we complete the proof of Theorem~\ref{thm:main_BGK} by proving that the constructed local BGK solution in Theorem~\ref{thm:BGK_sol} exhibits the desired ill-posedness in the weight $e^{\alpha'|v_h|^\beta}$. One of the interesting points of the next proposition is that any solutions of the BGK model sharing the same solution near $t=0$ have the same ill-posedness for any extended time.
\begin{proposition}\label{prop:Ill_posed}
    For the given boundary data $g(t,x_h,v)$ in \eqref{eq:boundary_data}, let $f(t,x,v)\in \Omega_{\frac{1}{1+\beta}, t_0}$ be the solution of the BGK model constructed in Theorem~\ref{thm:BGK_sol}. Then, for any $0<\alpha'\leq \alpha$, we have
    \begin{align*}
        \inf_{0<t\leq t_0}\left\|e^{\alpha'|v_h|^\beta}f(t,x,v)\right\|_{L^p_v(\mathbb{R}^n, L^q_x(\mathbb{R}^n_+))} = \inf_{0<t\leq t_0}\left\|e^{\alpha'|v_h|^\beta}f(t,x,v)\right\|_{L^q_x(\mathbb{R}^n_+, L^p_v(\mathbb{R}^n))} = \infty
    \end{align*}
    for any $1\leq p,q\leq \infty$. Furthermore, let $\tilde{f}(t,x,v)$ be any solution on $0\leq t\leq t^*_0$, where $t^*_0\geq t_0$, such that $\tilde{f}(t,x,v) = f(t,x,v)$ for $0\leq t\leq t_0$. Then, it (more precisely, a version of $\tilde{f}$) satisfies
    \begin{align*}
        \inf_{0<t\leq t^*_0}\left\|e^{\alpha'|v_h|^\beta}\tilde{f}(t,x,v)\right\|_{L^p_v(\mathbb{R}^n, L^q_x(\mathbb{R}^n_+))} = \inf_{0<t\leq t^*_0}\left\|e^{\alpha'|v_h|^\beta}\tilde{f}(t,x,v)\right\|_{L^q_x(\mathbb{R}^n_+, L^p_v(\mathbb{R}^n))} = \infty
    \end{align*}
    for any $1\leq p,q\leq \infty$. Here, we need not assume that $\tilde{f}(t,x,v)$ belongs to $\Omega_{\frac{1}{1+\beta}, t^*_0}$.
\end{proposition}
\begin{proof}
    We directly prove the ill-posedness on any extended time interval since it includes the first statement. For notational convenience, we present the proof for $t_0^*=\infty$ and show that
    \begin{align*}
        \inf_{t>0}\left\|e^{\alpha'|v_h|^\beta}f(t,x,v)\right\|_{L^p_v(\mathbb{R}^n, L^q_x(\mathbb{R}^n_+))} = \inf_{t>0}\left\|e^{\alpha'|v_h|^\beta}f(t,x,v)\right\|_{L^q_x(\mathbb{R}^n_+, L^p_v(\mathbb{R}^n))} = \infty
    \end{align*}
    using only the assumption that $f(t,x,v) \in \Omega_{\frac{1}{1+\beta},t_0}$. The same argument applies to any finite $t_0^*$ by restricting $t$ to $0<t\leq t_0^*$.

    Let $0<\alpha'\leq \alpha$. We first check
    \begin{align}\label{eq:Lp_vLq_x_res}
        \inf_{t>0}\left\|e^{\alpha'|v_h|^\beta}f(t,x,v)\right\|_{L^p_v(\mathbb{R}^n, L^q_x(\mathbb{R}^n_+))} = \infty.
    \end{align}
    By the definition of the solution, we have
    \begin{align*}
        e^{\alpha' |v_h|^\beta}f(t,x,v)\geq \int_0^t e^{-(t-\tau)}e^{\alpha' |v_h|^\beta}\mathcal{M}(f)(\tau, x-v(t-\tau), v)\mathbf{1}_{z-v_z(t-\tau)> 0}\,d\tau.
    \end{align*}
    Taking $L^q_x(\mathbb{R}^n_+)$ on both sides, we write
    \begin{align*}
        \int_{\mathbb{R}^n_+} \left(e^{\alpha'|v_h|^\beta} f(t,x,v)\right)^q\,dx\geq \int_{\mathbb{R}^n} e^{q\alpha' |v_h|^\beta}\left(\int_0^t e^{-(t-\tau)}\mathcal{M}(f)(\tau, x-v(t-\tau), v)\mathbf{1}_{z\geq v_z(t-\tau)}\mathbf{1}_{z\geq 0}\,d\tau\right)^q\,dx.
    \end{align*}
    We confine $v_z> 0$, so $\mathbf{1}_{z> v_z(t-\tau)}\mathbf{1}_{z> 0} = \mathbf{1}_{z> v_z(t-\tau)}$.

    Our goal is to show that for any fixed
    \begin{align}\label{eq:Lp_vLq_x_choice2}
        t>0,\quad \frac{\min\{t, t_0\}}{2}\leq z, \tau\leq \min\{t, t_0\},\quad\text{and}\quad 0< v_z\leq \frac{\min\{t, t_0\}}{2t},
    \end{align}
    there exists sufficiently large $M>0$ depending on $t$, $t_0$, $\alpha$, $\alpha'$, $\beta$, and $C_T$ such that for any $|v_h|\geq M$ and $x_h$ satisfying
    \begin{align}\label{eq:Lp_vLq_x_choice3}
        \left(\left(\frac{2}{\alpha'C_T}\right)^{\frac{1+\beta}{2}} + t\right)|v_h|\leq |x_h|\leq 2\left(\left(\frac{2}{\alpha'C_T}\right)^{\frac{1+\beta}{2}} + t\right)|v_h|,
    \end{align}
    the local Maxwellian satisfies
    \begin{align}\label{eq:Lp_vLq_x_1}
        e^{\alpha' |v_h|^\beta}\mathcal{M}(f)(\tau, x-v(t-\tau), v)\mathbf{1}_{z - v_z(t-\tau)>0}\geq e^{\frac{\alpha'}{8}|v_h|^\beta}.
    \end{align}
    If we prove this, then for $1\leq q<\infty$, we obtain
    \begin{align*}
        &\int_{\mathbb{R}^n_+} \left(e^{\alpha'|v_h|^\beta}f(t,x,v)\right)^q\,dx\\
        &\geq e^{-qt}\int_{\mathbb{R}^{n-1}} \int_{\frac{\min\{t, t_0\}}{2}}^{\min\{t, t_0\}} \Bigg(\int_{\frac{\min\{t, t_0\}}{2}}^{\min\{t, t_0\}} e^{\alpha' |v_h|^\beta}\mathcal{M}(f)(\tau, x-v(t-\tau), v)\\
        &\qquad \times \mathbf{1}_{\left(\left(\frac{2}{\alpha'C_T}\right)^{\frac{1+\beta}{2}} + t\right)|v_h|\leq |x_h|\leq 2\left(\left(\frac{2}{\alpha'C_T}\right)^{\frac{1+\beta}{2}} + t\right)|v_h|}\mathbf{1}_{z-v_z(t-\tau)>0}\,d\tau\Bigg)^q\,dz dx_h\\
        &\geq e^{-qt}e^{\frac{q\alpha'}{8}|v_h|^\beta}\int_{\mathbb{R}^{n-1}} \int_{\frac{\min\{t, t_0\}}{2}}^{\min\{t, t_0\}}\Bigg(\int_{\frac{\min\{t, t_0\}}{2}}^{\min\{t, t_0\}}\mathbf{1}_{\left(\left(\frac{2}{\alpha'C_T}\right)^{\frac{1+\beta}{2}} + t\right)|v_h|\leq |x_h|\leq 2\left(\left(\frac{2}{\alpha'C_T}\right)^{\frac{1+\beta}{2}} + t\right)|v_h|}\,d\tau\Bigg)^q\,dz dx_h\\
        &\geq C\left(\min\{t, t_0\}\right)^{q+1}\left(\left(\frac{2}{\alpha'C_T}\right)^{\frac{1+\beta}{2}} + t\right)^{n-1} e^{-qt}e^{\frac{q\alpha'}{8}|v_h|^\beta} |v_h|^{n-1}
    \end{align*}
    for any $t>0$ and $0< v_z\leq \frac{\min\{t, t_0\}}{2t}$. This yields the desired ill-posedness \eqref{eq:Lp_vLq_x_res} for any $1\leq p\leq \infty$ due to the exponential term $e^{\frac{\alpha'}{8}|v_h|^\beta}$. If $q=\infty$, it is simpler as
    \begin{align*}
        \left\|e^{\alpha'|v_h|^\beta}f(t,x,v)\right\|_{L^\infty_x(\mathbb{R}^n_+)} \geq e^{-t}\frac{\min\{t, t_0\}}{2}e^{\frac{\alpha'}{8}|v_h|^\beta}\mathbf{1}_{0< v_z\leq \frac{\min\{t, t_0\}}{2t}},
    \end{align*}
    and we again get the ill-posedness \eqref{eq:Lp_vLq_x_res} for $1\leq p\leq \infty$.
    
   We first derive an explicit lower bound for the $\mathcal M(f)$-term. By the choices of $z, \tau$, and $v_z$ in \eqref{eq:Lp_vLq_x_choice2}, we have
    \begin{align}\label{eq:z_v_z_cond}
        z - v_z(t-\tau) > 0\quad \text{and}\quad \frac{z-v_z(t-\tau)}{\tau}\leq 2.
    \end{align}
    We estimate $\frac{|v-u(\tau, x-v(t-\tau))|^2}{2T(\tau, x-v(t-\tau))}$. From \eqref{eq:X_T_A3}, we have
    \begin{align*}
        \left|u(\tau, x-v(t-\tau)) - \frac{z-v_z(t-\tau)}{\tau}\hat{z}\right|\leq C_u\sqrt{T_g(\tau, x-v(t-\tau))}.
    \end{align*}
    Therefore, from $(x-y)^2\leq 2(x^2 + y^2)$, we have
    \begin{align*}
        \frac{|v-u(\tau, x-v(t-\tau))|^2}{2T(\tau, x-v(t-\tau))}&\leq \frac{|v|^2 + |u(\tau, x-v(t-\tau))|^2}{T(\tau, x-v(t-\tau))}\\
        &\leq \frac{|v|^2 + \left(\left|\frac{z-v_z(t-\tau)}{\tau}\right| + C_u\sqrt{T_g(\tau, x-v(t-\tau))}\right)^2}{T(\tau, x-v(t-\tau))}\\
        &\leq C + \frac{|v|^2}{T(\tau, x-v(t-\tau))}
    \end{align*}
    for some constant $C$. Thus, we can write
    \begin{align}\label{eq:Lp_vLq_x_2}
        \begin{split}   
            &\mathcal{M}(f)(\tau, x-v(t-\tau), v)\mathbf{1}_{z - v_z(t-\tau)>0}\\
            &=\frac{\rho(\tau, x-v(t-\tau))}{(2\pi T(\tau, x-v(t-\tau)))^{n/2}}\exp\left(-\frac{|v-u(\tau, x-v(t-\tau))|^2}{2T(\tau, x-v(t-\tau))}\right)\mathbf{1}_{z - v_z(t-\tau)>0}\\
            &\geq C \frac{e^{-\alpha(|x_h-v_h(t-\tau)| + \frac{z-v_z(t-\tau)}{\tau} + 1)^{\frac{\beta}{1+\beta}}}}{(|x_h-v_h(t-\tau)| + \frac{z-v_z(t-\tau)}{\tau} + 1)^{\frac{m+\beta}{1+\beta}}}\exp\left(-\frac{|v|^2}{C_T(|x_h-v_h(t-\tau)| + \frac{z-v_z(t-\tau)}{\tau} + 1)^{\frac{2}{1+\beta}}}\right)\mathbf{1}_{z - v_z(t-\tau)>0}
        \end{split}
    \end{align}
    for some constant $C_T$, replacing $\rho$ and $T$ by \eqref{eq:X_T_A3} and \eqref{eq:macro_0_bound}.
    
    By the cutoff condition \eqref{eq:Lp_vLq_x_choice3} and \eqref{eq:z_v_z_cond}, for $|v_h|\geq 1$, it satisfies
    \begin{align*}
        \left(\frac{2}{\alpha'C_T}\right)^{\frac{1+\beta}{2}}|v_h|\leq |x_h-v_h(t-\tau)| + \frac{z-v_z(t-\tau)}{\tau} + 1\leq \left(\left(\frac{4}{\alpha'C_T}\right)^{\frac{1+\beta}{2}} + 4t + 4\right)|v_h|.
    \end{align*}
    Indeed,
    \begin{align*}
        &|x_h-v_h(t-\tau)| + \frac{z-v_z(t-\tau)}{\tau} + 1\geq |x_h|-|v_h|(t-\tau)\geq \left(\frac{2}{\alpha'C_T}\right)^{\frac{1+\beta}{2}}|v_h|,\\
        &|x_h-v_h(t-\tau)| + \frac{z-v_z(t-\tau)}{\tau} + 1\leq |x_h|+|v_h|(t-\tau) + 3 \leq \left(\left(\frac{4}{\alpha'C_T}\right)^{\frac{1+\beta}{2}} + 4t + 4\right)|v_h|
    \end{align*}
    for $|v_h|\geq 1$. Applying these inequalities to \eqref{eq:Lp_vLq_x_2}, we obtain
    \begin{align*}
        e^{\alpha'|v_h|^\beta}\mathcal{M}(f)(\tau, x-v(t-\tau), v)\mathbf{1}_{z> v_z(t-\tau)}\geq C e^{\alpha'|v_h|^\beta}\frac{e^{-\alpha (C')^{\frac{\beta}{1+\beta}} |v_h|^{\frac{\beta}{1+\beta}}}}{(C'|v_h|)^{\frac{m+\beta}{1+\beta}}}\exp\left(-\frac{1}{2}\alpha'\frac{|v|^2}{|v_h|^{\frac{2}{1+\beta}}}\right)
    \end{align*}
    for some constant $C'$.
    
    If we take $|v_h|$ large enough so that $|v_z|^2\leq \frac{1}{2}|v_h|^2$, then
    \begin{align*}
        e^{\alpha'|v_h|^\beta}e^{-\alpha(C')^{\frac{\beta}{1+\beta}}|v_h|^{\frac{\beta}{1+\beta}}}\exp\left(-\frac{1}{2}\alpha'\frac{|v|^2}{|v_h|^{\frac{2}{1+\beta}}}\right)&\geq e^{\alpha'|v_h|^\beta}e^{-\alpha(C')^{\frac{\beta}{1+\beta}}|v_h|^{\frac{\beta}{1+\beta}}}\exp\left(-\frac{3}{4}\alpha'|v_h|^{\frac{2\beta}{1+\beta}}\right).
    \end{align*}
    When $1\leq \beta\leq 2$, we have $\beta\geq \frac{\beta}{1+\beta}$ and $\beta \geq \frac{2\beta}{1+\beta}$. Therefore, the exponential growth term dominates the other decaying terms as $|v_h|\rightarrow\infty$, and there exists a sufficiently large $M>0$ such that the desired result \eqref{eq:Lp_vLq_x_1} holds for $|v_h|\geq M$.\\

    Next, we perform the $v$-integral first and then prove the ill-posedness
    \begin{align}\label{eq:Lq_xLp_v_res}
        \inf_{t>0}\left\|e^{\alpha'|v_h|^\beta}f(t,x,v)\right\|_{L^q_x(\mathbb{R}^n_+, L^p_v(\mathbb{R}^n))} = \infty.
    \end{align}
    Similar to \eqref{eq:Lp_vLq_x_choice2}, we fix
    \begin{align}\label{eq:Lq_xLp_v_choice}
        t>0,\quad \frac{\min\{t, t_0\}}{2}\leq z, \tau\leq \min\{t, t_0\},\quad 0< v_z\leq \frac{z}{4t}.
    \end{align}
    For $|v_h|$, we take scaling
    \begin{align}\label{eq:Lq_xLp_v_v_h_cond}
        |v_h| = c(|x_h| + 1)^{\frac{2}{1+\beta}}
    \end{align}
    for $0<c\leq \frac{1}{4\max\{t, 1\}}$. Under these choices, we have
    \begin{align*}
        v_z\leq 1,\quad \frac{3}{8}\leq \frac{z-v_z(t-\tau)}{\tau}\leq 2,
    \end{align*}
    and
    \begin{align*}
        \left|x_h-v_h(t-\tau)\right| + \frac{z-v_z(t-\tau)}{\tau} + 1&\geq |x_h| + 1 - \frac{1}{4}\left(|x_h| + 1\right)^{\frac{2}{1+\beta}}\geq \frac{3}{4}\left(|x_h| + 1\right),\\
        \left|x_h-v_h(t-\tau)\right| + \frac{z-v_z(t-\tau)}{\tau} + 1&\leq|x_h|+ 3 + \frac{1}{4}\left(|x_h| + 1\right)^{\frac{2}{1+\beta}}\\
        &\leq 4\left(|x_h|+ 1\right)
    \end{align*}
    for $1\leq \beta\leq 2$.
    
    Applying these inequalities to \eqref{eq:Lp_vLq_x_2}, we have
    \begin{align*}
        &e^{\alpha' |v_h|^\beta}\mathcal{M}(f)(\tau, x-v(t-\tau), v)\mathbf{1}_{z> v_z(t-\tau)}\\
        &\geq C e^{\alpha'|v_h|^\beta}\frac{e^{-\alpha(|x_h-v_h(t-\tau)| + \frac{z-v_z(t-\tau)}{\tau} + 1)^{\frac{\beta}{1+\beta}}}}{(|x_h-v_h(t-\tau)| + \frac{z-v_z(t-\tau)}{\tau} + 1)^{\frac{m+\beta}{1+\beta}}}\exp\left(-\frac{|v|^2}{C_T(|x_h-v_h(t-\tau)| + \frac{z-v_z(t-\tau)}{\tau} + 1)^{\frac{2}{1+\beta}}}\right)\mathbf{1}_{z> v_z(t-\tau)}\\
        &\geq C e^{\alpha'|v_h|^\beta}\frac{e^{-4^{\frac{\beta}{1+\beta}}\alpha (|x_h| + 1)^{\frac{\beta}{1+\beta}}}}{4^{\frac{m+\beta}{1+\beta}}(|x_h|+ 1)^{\frac{m+\beta}{1+\beta}}}\exp\left(-\frac{|v_h|^2 + 1}{\left(\frac{3}{4}\right)^{\frac{2}{1+\beta}}C_T(|x_h| + 1)^{\frac{2}{1+\beta}}}\right).
    \end{align*}
    Finally, directly replacing $|v_h|$ by \eqref{eq:Lq_xLp_v_v_h_cond}, we get
    \begin{align*}
        &e^{\alpha'|v_h|^\beta}\frac{e^{-4^{\frac{\beta}{1+\beta}}\alpha (|x_h| + 1)^{\frac{\beta}{1+\beta}}}}{4^{\frac{m+\beta}{1+\beta}}(|x_h|+ 1)^{\frac{m+\beta}{1+\beta}}}\exp\left(-\frac{|v_h|^2 + 1}{\left(\frac{3}{4}\right)^{\frac{2}{1+\beta}}C_T(|x_h| + 1)^{\frac{2}{1+\beta}}}\right)\\
        &\geq e^{\alpha'c^\beta (|x_h| + 1)^{\frac{2\beta}{1+\beta}}}\frac{e^{-4^{\frac{\beta}{1+\beta}}\alpha (|x_h| + 1)^{\frac{\beta}{1+\beta}}}}{4^{\frac{m+\beta}{1+\beta}}(|x_h|+ 1)^{\frac{m+\beta}{1+\beta}}}\exp\left(-\frac{c^2(|x_h| + 1)^{\frac{4}{1+\beta}} + 1}{\left(\frac{3}{4}\right)^{\frac{2}{1+\beta}}C_T(|x_h| + 1)^{\frac{2}{1+\beta}}}\right).
    \end{align*}
    
    If $1<\beta\leq 2$, as $\frac{2\beta}{1+\beta}> \frac{\beta}{1+\beta}, \frac{2}{1+\beta}$, the first growth term overwhelms the other decay terms as $|x_h|\rightarrow \infty$ for $c>0$. Therefore, there exists sufficiently large $M>0$ and small $\epsilon>0$ such that
    \begin{align*}
        e^{\alpha' |v_h|^\beta}\mathcal{M}(f)(\tau, x-v(t-\tau), v)\mathbf{1}_{z\geq v_z(t-\tau)} \geq e^{\epsilon (|x_h| + 1)^{\frac{2\beta}{1+\beta}}}
    \end{align*}
    for $|x_h|\geq M$ under the choices of \eqref{eq:Lq_xLp_v_choice} and
    \begin{align*}
        \frac{1}{8\max\{t, 1\}}(|x_h|+1)^{\frac{2}{1+\beta}}\leq |v_h|\leq \frac{1}{4\max\{t, 1\}}(|x_h|+1)^{\frac{2}{1+\beta}}
    \end{align*}
    from \eqref{eq:Lq_xLp_v_v_h_cond}.
    
    If $\beta = 1$, then the leading term in the exponents is given by
    \begin{align*}
        \left(\alpha'c -\frac{4}{3}\frac{c^2}{C_T}\right)(|x_h| + 1).
    \end{align*}
    If we choose $c$ from
    \begin{align*}
        c\in\left[\frac{1}{2}\min\left\{\frac{3\alpha' C_T}{8}, \frac{1}{4\max\{t, 1\}}\right\}, \min\left\{\frac{3\alpha' C_T}{8}, \frac{1}{4\max\{t, 1\}}\right\}\right],
    \end{align*}
    we have
    \begin{align*}
        \left(\alpha'c -\frac{4}{3}\frac{c^2}{C_T}\right)(|x_h| + 1)\geq C(|x_h| + 1)
    \end{align*}
    for some constant $C>0$ depending on $\min\left\{\frac{3\alpha' C_T}{8}, \frac{1}{4\max\{t, 1\}}\right\}$. It shows that the growth term is much bigger than the decay term in the given range of $c$. Therefore, there again exist sufficiently large $M>0$ and small $\epsilon>0$ such that
    \begin{align*}
        e^{\alpha' |v_h|}\mathcal{M}(f)(\tau, x-v(t-\tau), v)\mathbf{1}_{z\geq v_z(t-\tau)}\geq e^{\epsilon (|x_h| + 1)}
    \end{align*}
    for $|x_h|\geq M$ under the choices of \eqref{eq:Lq_xLp_v_choice} and
    \begin{align*}
        \frac{1}{2}\min\left\{\frac{3\alpha' C_T}{8}, \frac{1}{4\max\{t, 1\}}\right\} (|x_h| + 1)\leq |v_h|\leq \min\left\{\frac{3\alpha' C_T}{8}, \frac{1}{4\max\{t, 1\}}\right\} (|x_h| + 1).
    \end{align*}
    
    Finally, for $1\leq p<\infty$, we obtain the lower bound of $L^p_v$ norm of $f(t,x,v)$ as
    \begin{align*}
        &\int_{\mathbb{R}^n} \left(e^{\alpha'|v_h|^\beta}f(t,x,v)\right)^p\,dv\\
        &\geq e^{-pt}\int_{\mathbb{R}^n} \left(\int_0^t e^{\alpha'|v_h|^\beta}\mathcal{M}(f)(\tau,x-v(t-\tau),v)\mathbf{1}_{z - v_z(t-\tau)>0}\,d\tau\right)^p\,dv\\
        &\geq e^{-pt}\int_{\mathbb{R}^{n-1}}\int_0^{\frac{z}{4t}} \left(\int_{\frac{1}{2}\min\{t, t_0\}}^{\min\{t, t_0\}} e^{\alpha'|v_h|^\beta}\mathcal{M}(f)(\tau,x-v(t-\tau),v)\mathbf{1}_{v_h\in A}\,d\tau\right)^p\,dv_z dv_h\\
        &\geq e^{-pt}e^{p\epsilon (|x_h| + 1)^{\frac{2\beta}{1+\beta}}}\int_{\mathbb{R}^{n-1}}\int_0^{\frac{z}{4t}} \left(\int_{\frac{1}{2}\min\{t, t_0\}}^{\min\{t, t_0\}} \mathbf{1}_{v_h\in A}\,d\tau\right)^p\,dv_z dv_h\\
        &\geq C(\min\{t, t_0\})^p \frac{z}{4t}e^{-pt}e^{p\epsilon (|x_h| + 1)^{\frac{2\beta}{1+\beta}}}\int_{\mathbb{R}^{n-1}}\mathbf{1}_{v_h\in A}\,dv_h
    \end{align*}
    for $\frac{\min\{t, t_0\}}{2}\leq z\leq \min\{t, t_0\}$, where $A$ is the set given by
    \begin{align*}
        A:=\begin{dcases}
            \left\{v_h:\frac{1}{8\max\{t, 1\}}(|x_h|+1)^{\frac{2}{1+\beta}}\leq |v_h|\leq \frac{1}{4\max\{t, 1\}}(|x_h|+1)^{\frac{2}{1+\beta}}\right\} & 1<\beta\leq 2,\\
            \left\{v_h:\frac{1}{2}\min\left\{\frac{3\alpha' C_T}{8}, \frac{1}{4\max\{t, 1\}}\right\} (|x_h| + 1)\leq |v_h|\leq \min\left\{\frac{3\alpha' C_T}{8}, \frac{1}{4\max\{t, 1\}}\right\} (|x_h| + 1)\right\}& \beta = 1.
        \end{dcases}
    \end{align*}
    The integral of the set is bounded by
    \begin{align*}
        \int_{\mathbb{R}^{n-1}}\mathbf{1}_{v_h\in A}\,dv_h\geq C\left(\min\left\{\frac{3\alpha' C_T}{8}, \frac{1}{4\max\{t, 1\}}\right\} (|x_h| + 1)^{\frac{2}{1+\beta}}\right)^{n-1}.
    \end{align*}
    Due to the exponential factor $e^{\epsilon (|x_h| + 1)^{\frac{2\beta}{1+\beta}}}$, we obtain the exponential ill-posedness \eqref{eq:Lq_xLp_v_res} for any $1\leq q\leq \infty$. If $p=\infty$, since
    \begin{align*}
        \left\|e^{\alpha'|v_h|^\beta}f(t,x,v)\right\|_{L^\infty_v(\mathbb{R}^n)}\geq \frac{\min\{t, t_0\}}{2} e^{-t}e^{\epsilon (|x_h| + 1)^{\frac{2\beta}{1+\beta}}}\mathbf{1}_{\frac{\min\{t, t_0\}}{2}\leq z\leq \min\{t, t_0\}},
    \end{align*}
    we again get \eqref{eq:Lq_xLp_v_res} for $1\leq q\leq \infty$.
\end{proof}

We conclude this section with the proof of Theorem~\ref{thm:main_BGK}.
\begin{proof}[Proof of Theorem~\ref{thm:main_BGK}]
    For given constants, we choose the boundary data $g(t,x_h,v)$ of the form \eqref{eq:boundary_data} satisfying the condition \eqref{eq:m_cond1}; as noted in the introduction of Section~\ref{sec:Macro_quan}, we can rescale $g$ by $\eta g$ for any sufficiently small $\eta>0$ to meet the condition \eqref{3_thm:main_BGK_1_cond}. From Proposition \ref{prop:boundary_int}, we have the boundedness of the boundary data under the exponential weight. To make the weighted norm of the boundary data arbitrarily small, we choose the boundary data as $\epsilon g$ for any $\epsilon>0$ by the homogeneity of the BGK model equation; we refer to the argument right after \eqref{eq:boundary_data}. Using Theorem \ref{thm:BGK_sol}, we construct a unique solution $f$ of the BGK model in $\Omega_{\frac{1}{1+\beta},t_0}$ such that the macroscopic fields of $f$ satisfy \eqref{eq:X_T_A3} pointwise. For any BGK solution $\tilde{f}$ such that $f(t,x,v) = \tilde{f}(t,x,v)$ for $0\leq t\leq t_0$, we have the desired exponential-norm blow-up by Proposition \ref{prop:Ill_posed}.
\end{proof}

\section{Well-posedness theory of the Boltzmann equation}\label{sec:Well_Bol}
In this section, we establish the local well-posedness result for the Boltzmann equation in exponentially weighted spaces. As noted in the introduction, we use the exponential norm with $e^{\alpha|v|^\beta}$, which differs from the one used in the BGK model. This is due to the incompatibility of the exponential weight between the two models. To put it simply, the BGK model suffers from the instability of the macroscopic fields when we use the weight $e^{\alpha |v|^\beta}$, and the Boltzmann equation loses the symmetric properties of the velocity variables when we use $e^{\alpha |v_h|^\beta}$. We will discuss these differences in detail in the Appendix. As a result, it gives a complementary local well-posedness result to the instantaneous blow-up in the BGK model.

We mainly follow and adapt the proof steps of \cite{LPY2024} to our boundary problem. First, we check \textit{a priori} estimate that a solution $f$ of the Boltzmann equation satisfies a $w_{\alpha,\beta,\delta}$-weighted $L^\infty_{x,v}(\mathbb{R}^n_+\times\mathbb{R}^n)$ boundedness locally in time. Using this \textit{a priori} estimate, we construct a unique solution for the boundary data. Finally, having weighted $L^\infty_{x,v}$ estimate in hand, we construct $w_{\alpha,\beta,\delta}$-weighted $L^p_v(\mathbb{R}^n, L^q_x(\mathbb{R}^n_+))$ boundedness.

Same as the BGK model, we consider the vacuum initial data $f_0 = 0$ and the inflow boundary data $g$. We recall that the mild solution of the Boltzmann equation is defined by
\begin{align*}
    f(t,x,v) &= e^{-\int_{t-t_b}^t L(f)(\tau, x-v(t-\tau), v)\,d\tau}g(t-t_b, x-vt_b, v)\mathbf{1}_{v_z\geq \frac{z}{t}}\\
    &\quad + \int_0^t e^{-\int_\tau^t L(f)(s, x-v(t-s), v)\,ds} Q^+(f,f)(\tau, x-v(t-\tau), v)\mathbf{1}_{z-v_z(t-\tau)> 0}\,d\tau.
\end{align*}
Our toy model for the boundary data $g$ is
\begin{align*}
    g(t,x_h,v) = e^{-t}\langle v\rangle^{-m}e^{-\alpha |v|^\beta}\mathbf{1}_{v_z\geq 0}
\end{align*}
for some $m\geq 0$, which is similar to \eqref{eq:boundary_data}. We also recall that our collision kernel satisfies \eqref{eq:B_condition}, so
\begin{align*}
    B(v-v_*,\omega)\leq C|v-v_*|^\kappa |\cos\theta|,\quad -n<\kappa\leq 1.
\end{align*}

The following lemma establishes a key estimate for the gain operator $Q^+(f,f)$ with respect to the weight $w_{\alpha,\beta,\delta}$. This estimate will be used repeatedly in the proof of the following proposition.
\begin{lemma}[Lemmas 5.5 and 5.10 in \cite{LPY2024}]\label{lem:LPY_lem}
    Let $n\geq 2$, $\alpha>0$, $\max\left\{0, \frac{2}{n-1}\kappa, \frac{2}{3}(1+\kappa)\right\}<\beta\leq 2$, and $\delta> n$. Then, for functions $f$ and $h$, we have
    \begin{align*}
        \left\|w_{\alpha,\beta,\delta}Q^+(f,h)(x,v)\right\|_{L^\infty_{x,v}(\mathbb{R}^n_+\times\mathbb{R}^n)}\leq C_{gain,1}\|w_{\alpha,\beta,\delta}f\|_{L^\infty_{x,v}(\mathbb{R}^n_+\times\mathbb{R}^n)}\|w_{\alpha,\beta,\delta}h\|_{L^\infty_{x,v}(\mathbb{R}^n_+\times\mathbb{R}^n)}
    \end{align*}
    for a constant $C_{gain,1}$ depending on $\alpha$, $\beta$, $\delta$, $\kappa$, and the dimension $n$.
    For general $1\leq p,q\leq \infty$, we further assume $\delta>n+1$. If $n=2$, we also impose $1+\kappa\leq \beta$. Then, we have
    \begin{align*}
        \left\|w_{\alpha,\beta,\delta}Q^+(f,f)(x,v)\right\|_{L^p_v(\mathbb{R}^n, L^q_x(\mathbb{R}^n_+))}\leq C_{gain,2}\|w_{\alpha,\beta,\delta}f\|_{L^\infty_{x,v}}\|w_{\alpha,\beta,\delta}f\|_{L^p_v(\mathbb{R}^n, L^q_x(\mathbb{R}^n_+))}
    \end{align*}
    for a constant $C_{gain,2}$ depending on $\alpha$, $\beta$, $\delta$, $\kappa$, and $n$.
\end{lemma}
\begin{proof}
    The original results in \cite{LPY2024} were established in the whole space setting $\mathbb{R}^n$. To extend these results to the half space, we just extend the function $f$ on $\mathbb{R}^n_+$ by
    \begin{align*}
        \tilde{f}(x,v) = \begin{cases}
            f(x,v) & z\geq 0,\\
            0 & z<0
        \end{cases}
    \end{align*}
    and apply Lemma 5.5 and 5.10 in \cite{LPY2024} for $\tilde{f}$.
\end{proof}

The following proposition establishes the existence and uniqueness of the solution in the $L^\infty_{x,v}(w_{\alpha,\beta,\delta})$ space. We temporarily assume $t_0^* = \infty$ in Theorem \ref{thm:main_Boltzmann} to simplify the statement.
\begin{proposition}\label{prop:boltzmann_existence}
    Let $n\geq 2$, $\alpha>0$, $\max\left\{0, \frac{2}{n-1}\kappa, \frac{2}{3}(1+\kappa)\right\}<\beta\leq 2$, and $\delta> n+1$. If the boundary data on $\partial \mathbb{R}^n_+$ satisfies
    \begin{align*}
        \|w_{\alpha, \beta,\delta} g(t,x_h,v)\|_{L_{t,x_h,v}^\infty([0, \infty)\times \partial\mathbb{R}^n_+\times\mathbb{R}^n)}<\infty,
    \end{align*}
    then there exists a unique solution of the Boltzmann equation $f(t,x,v)$ in $L_{loc, t}^\infty([0,t_0), L^\infty_{x,v}(w_{\alpha, \beta,\delta}))$ which satisfies
    \begin{align*}
        \|w_{\alpha,\beta,\delta}f(t,x,v)\|_{L^\infty_{x,v}(\mathbb{R}^n_+\times\mathbb{R}^n)}\leq \frac{1}{\|w_{\alpha, \beta,\delta}g(t,x_h,v)\|^{-1}_{L^\infty_{t,x_h,v}([0, \infty)\times \partial\mathbb{R}^n_+\times\mathbb{R}^n)} - C_{gain,1}t},
    \end{align*}
    where $C_{gain,1}$ is given in Lemma~\ref{lem:LPY_lem}. Here, $t_0$ is given by
    \begin{align}\label{eq:t_0_def}
        t_0 = \left(C_{gain,1}\|w_{\alpha, \beta,\delta}g(t,x_h,v)\|_{L^\infty_{t,x_h,v}([0, \infty)\times \partial\mathbb{R}^n_+\times\mathbb{R}^n)}\right)^{-1}.
    \end{align}
\end{proposition}

\begin{proof}
    Define a solution map $\mathcal{T}$ by
    \begin{align*}
        \mathcal{T}f(t,x,v) &= e^{-\int_{t-t_b}^t L(f)(\tau, x-v(t-\tau), v)\,d\tau}g(t-t_b, x-vt_b, v)\mathbf{1}_{v_z\geq \frac{z}{t}}\\
        &\quad + \int_0^t e^{-\int_\tau^t L(f)(s, x-v(t-s), v)\,ds} Q^+(f,f)(\tau, x-v(t-\tau), v)\mathbf{1}_{z-v_z(t-\tau)> 0}\,d\tau.
    \end{align*}
    We claim that if
    \begin{align}\label{eq:apriori_esti}
        \|w_{\alpha,\beta,\delta}f(t)\|_{L^\infty_{x,v}(\mathbb{R}^n_+\times\mathbb{R}^n)}\leq \frac{1}{\|w_{\alpha, \beta,\delta} g(t,x_h,v)\|^{-1}_{L^\infty_{t,x_h,v}([0,\infty)\times \partial\mathbb{R}^n_+\times\mathbb{R}^n)} - C_{gain,1}t},
    \end{align}
    then $\|w_{\alpha,\beta,\delta}\mathcal{T}f(t)\|_{L^\infty_{x,v}(\mathbb{R}^n_+\times\mathbb{R}^n)}$ also follows the same bound along time, and so the solution map
    \begin{align*}
        \mathcal{T}:\mathcal{A}\rightarrow \mathcal{A}
    \end{align*}
    is well-defined, where $\mathcal{A}$ is the collection of functions in $L^\infty_{loc, t}([0, t_0), L^\infty_{x,v}(w_{\alpha,\beta,\delta}))$ satisfying \eqref{eq:apriori_esti} and $f\geq 0$. Indeed, if $f\in \mathcal{A}$ for some $T'>0$, by Lemma~\ref{lem:LPY_lem} and \ref{lem:trace}, we have
    
    \begin{align*}
        \|w_{\alpha, \beta,\delta}\mathcal{T}f(t,x,v)\|_{L^\infty_{x,v}(\mathbb{R}^n_+\times\mathbb{R}^n)}&\leq \|w_{\alpha, \beta,\delta}g(t-t_b, x-vt_b, v)\mathbf{1}_{v_z\geq \frac{z}{t}}\|_{L^\infty_{x,v}(\mathbb{R}^n_+\times\mathbb{R}^n)}\\
        &\quad + \int_0^t \|w_{\alpha, \beta,\delta} Q^+(f,f)(\tau, x-v(t-\tau), v)\mathbf{1}_{z-v_z(t-\tau)> 0}\|_{L^\infty_{x,v}(\mathbb{R}^n_+\times\mathbb{R}^n)}\,d\tau\\
        &\leq \|w_{\alpha, \beta,\delta} g(t,x_h,v)\|_{L^\infty_{t,x_h,v}([0,\infty)\times \partial\mathbb{R}^n_+\times\mathbb{R}^n)} + C_{gain,1}\int_0^t \|w_{\alpha, \beta,\delta} f(\tau,x,v)\|^2_{L^\infty_{x,v}(\mathbb{R}^n_+\times\mathbb{R}^n)}\,d\tau
    \end{align*}
    for each $0\leq t\leq T'$. Therefore, if $f$ satisfies \eqref{eq:apriori_esti}, then
    \begin{align*}
        \|w_{\alpha, \beta,\delta}\mathcal{T}f(t,x,v)\|_{L^\infty_{x,v}(\mathbb{R}^n_+\times\mathbb{R}^n)}&\leq \|w_{\alpha, \beta,\delta} g(t,x_h,v)\|_{L^\infty_{t,x_h,v}([0,\infty)\times \partial\mathbb{R}^n_+\times\mathbb{R}^n)}\\
        &\quad + C_{gain,1}\int_0^t \frac{1}{(\|w_{\alpha, \beta,\delta} g(t,x_h,v)\|^{-1}_{L^\infty_{t,x_h,v}([0,\infty)\times \partial\mathbb{R}^n_+\times\mathbb{R}^n)} - C_{gain,1}\tau)^2}\,d\tau\\
        &=\frac{1}{\|w_{\alpha, \beta,\delta} g(t,x_h,v)\|^{-1}_{L^\infty_{t,x_h,v}([0,\infty)\times \partial\mathbb{R}^n_+\times\mathbb{R}^n)} - C_{gain,1}t}.
    \end{align*}
    
    With this \textit{a priori} estimate, we can prove the existence of a fixed point of $\mathcal{T}$ in $\mathcal{A}$ satisfying \eqref{eq:apriori_esti} for small enough $T_1$ by checking the contraction property of $\mathcal{T}$. We estimate difference $|\mathcal{T}f_1-\mathcal{T}f_2|$ as follows:
    \begin{align}
        &|\mathcal{T}f_1-\mathcal{T}f_2|(t,x,v)\notag\\
        &\leq \left|e^{-\int_{t-t_b}^t L[f_1](\tau, x-v(t-\tau), v)\,d\tau} - e^{-\int_{t-t_b}^t L[f_2](\tau, x-v(t-\tau), v)\,d\tau}\right|g(t-t_b, x-vt_b, v)\mathbf{1}_{v_z\geq \frac{z}{t}}\label{eq:Boltzmann_contraction_1}\\
        &\quad + \int_0^t \left|e^{-\int_\tau^t L[f_1](s, x-v(t-s), v)\,ds} - e^{-\int_\tau^t L[f_2](s, x-v(t-s), v)\,ds}\right| Q^+(f_1,f_1)(\tau, x-v(t-\tau), v)\mathbf{1}_{z-v_z(t-\tau)> 0}\,d\tau\label{eq:Boltzmann_contraction_2}\\
        &\quad + \int_0^t e^{-\int_\tau^t L[f_2](s, x-v(t-s), v)\,ds} \left|Q^+(f_1,f_1) - Q^+(f_2,f_2)\right|(\tau, x-v(t-\tau), v)\mathbf{1}_{z-v_z(t-\tau)> 0}\,d\tau\label{eq:Boltzmann_contraction_3}.
    \end{align}
    We split the proof into the cases $\kappa>0$ and $\kappa\leq 0$.\\
    
    \noindent \textit{(i) The case $\kappa>0$.}
    When $\kappa>0$, we multiply $w_{\alpha,\beta,\delta-\kappa}(v)$ on the both sides. Since $|e^{-x}-e^{-y}|\leq |x-y|$ for $x,y\geq 0$, we have
    \begin{align*}
        &w_{\alpha,\beta,\delta-\kappa}(v)\mathbf{1}_{z-v_z(t-\tau)> 0}\left|\left(e^{-\int_\tau^t L[f_1](s, x-v(t-s), v)\,ds} - e^{-\int_\tau^t L[f_2](s, x-v(t-s), v)\,ds}\right)\right|\\
        &\leq w_{\alpha,\beta,\delta-\kappa}(v)\mathbf{1}_{z-v_z(t-\tau)> 0}\int_\tau^t \left|L[f_1]-L[f_2]\right|(s, x-v(t-s), v)\,ds\\
        &\leq w_{\alpha,\beta,\delta-\kappa}(v)(1+|v|^2)^{\kappa/2}\int_\tau^t \iint_{\mathbb{S}^{n-1}\times\mathbb{R}^n}(1+|v_*|^2)^{\kappa/2}b(\cos\theta)(w_{\alpha,\beta,\delta-\kappa}(v_*))^{-1}\|w_{\alpha,\beta,\delta-\kappa}(f_1-f_2)(s)\|_{L^\infty_{x,v}}\,dv_*d\omega ds\\
        &\leq C w_{\alpha,\beta,\delta}(v)\int_0^t \left\|w_{\alpha,\beta,\delta-\kappa}(f_1-f_2)(s)\right\|_{L^\infty_{x,v}}\,ds
    \end{align*}
    for $0\leq \tau\leq t$. In the middle, we used
    \begin{align}\label{eq:L_bound}
        \begin{split}
            \left|L[f_1]-L[f_2]\right|(v) &\leq C\int_{\mathbb{R}^n\times\mathbb{S}^{n-1}}|v-v_*|^\kappa |\cos\theta| |f_1-f_2|(v_*)\,dv_* d\omega\\
            &\leq C\int_{\mathbb{R}^n\times\mathbb{S}^{n-1}}(1+|v|^2)^{\kappa/2}(1+|v_*|^2)^{\kappa/2} |\cos\theta| |f_1-f_2|(v_*)\,dv_* d\omega\\
            &\leq C(1+|v|^2)^{\kappa/2} \|w_{\alpha,\beta,\delta-\kappa}(f_1-f_2)\|_{L^\infty_v}\int_{\mathbb{R}^n\times\mathbb{S}^{n-1}} |\cos\theta|(1+|v_*|^2)^{\kappa-\delta/2}e^{-\alpha |v_*|^\beta}\,dv_*d\omega\\
            &\leq C(1+|v|^2)^{\kappa/2} \|w_{\alpha,\beta,\delta-\kappa}(f_1-f_2)\|_{L^\infty_v}.
        \end{split}
    \end{align}
    Thus, combining this bound with Lemma~\ref{lem:trace}, we obtain
    \begin{align}\label{eq:diff_esti2}
        \begin{split}
            &w_{\alpha,\beta,\delta-\kappa}(v)\left|e^{-\int_{t-t_b}^t L[f_1](\tau, x-v(t-\tau), v)\,d\tau} - e^{-\int_{t-t_b}^t L[f_2](\tau, x-v(t-\tau), v)\,d\tau}\right| g(t-t_b, x-vt_b, v)\mathbf{1}_{v_z\geq \frac{z}{t}}\\
            &\leq C \|w_{\alpha, \beta,\delta} g(t,x_h,v)\|_{L^\infty_{t,x_h,v}([0,\infty)\times \partial\mathbb{R}^n_+\times\mathbb{R}^n)}\int_0^t \left\|w_{\alpha,\beta,\delta-\kappa}(f_1-f_2)(\tau)\right\|_{L^\infty_{x,v}}\,d\tau
        \end{split}
    \end{align}
    for \eqref{eq:Boltzmann_contraction_1}. For \eqref{eq:Boltzmann_contraction_2}, we use Lemma~\ref{lem:LPY_lem} to obtain
    \begin{align}\label{eq:diff_esti3}
        \begin{split}
            &w_{\alpha,\beta,\delta-\kappa}(v)\int_0^t \left|e^{-\int_\tau^t L[f_1](s, x-v(t-s), v)\,ds} - e^{-\int_\tau^t L[f_2](s, x-v(t-s), v)\,ds}\right| Q^+(f_1,f_1)(\tau, x-v(t-\tau), v)\mathbf{1}_{z-v_z(t-\tau)> 0}\,d\tau\\
            &\leq C\int_0^t w_{\alpha,\beta,\delta}(v) Q^+(f_1,f_1)(\tau, x-v(t-\tau), v)\mathbf{1}_{z-v_z(t-\tau)> 0}\int_\tau^t \left\|w_{\alpha,\beta,\delta-\kappa}(f_1-f_2)(s)\right\|_{L^\infty_{x,v}}\,ds \,d\tau\\
            &\leq C\int_0^t \|w_{\alpha,\beta,\delta} f_1(\tau)\|^2_{L^\infty_{x,v}}\int_\tau^t \left\|w_{\alpha,\beta,\delta-\kappa}(f_1-f_2)(s)\right\|_{L^\infty_{x,v}}\,ds \,d\tau\\
            &\leq C\sup_{0\leq \tau\leq t}\|w_{\alpha,\beta,\delta} f_1(\tau)\|^2_{L^\infty_{x,v}}\int_0^t \left\|w_{\alpha,\beta,\delta-\kappa}(f_1-f_2)(s)\right\|_{L^\infty_{x,v}}\,ds
        \end{split}
    \end{align}
    for $t\leq 1$. Finally, we decompose the gain term $Q^+$ and apply Lemma~\ref{lem:LPY_lem} to bound \eqref{eq:Boltzmann_contraction_3} by
    \begin{align}\label{eq:diff_esti4}
        \begin{split}
            &w_{\alpha,\beta,\delta-\kappa}(v)\int_0^t e^{-\int_\tau^t L[f_2](s, x-v(t-s), v)\,ds} \left|Q^+(f_1,f_1) - Q^+(f_2,f_2)\right|(\tau, x-v(t-\tau), v)\mathbf{1}_{z-v_z(t-\tau)> 0}\,d\tau\\
            &\leq w_{\alpha,\beta,\delta-\kappa}(v)\int_0^t \left(\left|Q^+(f_1-f_2,f_1)\right| + \left|Q^+(f_2,f_2-f_1)\right|\right)(\tau, x-v(t-\tau), v)\mathbf{1}_{z-v_z(t-\tau)> 0}\,d\tau\\
            &\leq C\int_0^t (\|w_{\alpha,\beta,\delta} f_1(\tau)\|_{L^\infty_{x,v}} + \|w_{\alpha,\beta,\delta} f_2(\tau)\|_{L^\infty_{x,v}})\|w_{\alpha,\beta,\delta-\kappa} (f_1-f_2)(\tau)\|_{L^\infty_{x,v}}\,d\tau.
        \end{split}
    \end{align}

    Combining \eqref{eq:diff_esti2}, \eqref{eq:diff_esti3}, and \eqref{eq:diff_esti4}, we get
    \begin{align}\label{eq:contract_kappa>0}
        \begin{split}
            &\|w_{\alpha,\beta, \delta-\kappa}(\mathcal{T}f_1-\mathcal{T}f_2)(t)\|_{L^\infty_{x,v}(\mathbb{R}^n_+\times\mathbb{R}^n)}\\
            &\leq C\left[\|w_{\alpha, \beta,\delta} g(t,x_h,v)\|_{L^\infty_{t,x_h,v}} + \esssup_{0\leq \tau\leq t}\left(\|w_{\alpha,\beta,\delta} f_1(\tau)\|_{L^\infty_{x,v}}^2 + \|w_{\alpha,\beta,\delta}f_1(\tau)\|_{L^\infty_{x,v}} + \|w_{\alpha,\beta,\delta}f_2(\tau)\|_{L^\infty_{x,v}}\right)\right]\\
            &\qquad \times \int_0^t \|w_{\alpha,\beta,\delta-\kappa} (f_1-f_2)(\tau)\|_{L^\infty_{x,v}}\,d\tau
        \end{split}
    \end{align}
    for $t\leq 1$.
    
    Now, we construct a function sequence $f^i = \mathcal{T}f^{i-1}$ with $f^0(t,x,v) = 0$. Then, $f^i$ satisfies \eqref{eq:apriori_esti} and $f^i\geq 0$ for each $i\in\mathbb{N}\cup\{0\}$. Therefore, setting $f_1 = f^i$ and $f_2 = f^{i-1}$ in \eqref{eq:contract_kappa>0}, we get
    \begin{align*}
        &\|w_{\alpha,\beta, \delta-\kappa}(f^{i+1}-f^i)(t)\|_{L^\infty_{x,v}}\\
        &\leq C\left[\|w_{\alpha, \beta,\delta} g(t,x_h,v)\|_{L^\infty_{t,x_h,v}} + \esssup_{0\leq \tau\leq t}\left(\|w_{\alpha,\beta,\delta} f^i(\tau)\|_{L^\infty_{x,v}}^2 + \|w_{\alpha,\beta,\delta}f^i(\tau)\|_{L^\infty_{x,v}} + \|w_{\alpha,\beta,\delta}f^{i-1}(\tau)\|_{L^\infty_{x,v}}\right)\right]\\
        &\qquad \times \int_0^t \|w_{\alpha,\beta,\delta-\kappa} (f^i-f^{i-1})(\tau)\|_{L^\infty_{x,v}}\,d\tau.
    \end{align*}
    Choosing small enough $T_1\leq 1$, we have
    \begin{align*}
        \|w_{\alpha,\beta, \delta-\kappa}(f^{i+1}-f^i)(t)\|_{L^\infty_{x,v}(\mathbb{R}^n_+\times\mathbb{R}^n)}\leq \frac{1}{2}\sup_{0\leq \tau\leq T_1}\|w_{\alpha,\beta, \delta-\kappa} (f^i-f^{i-1})(\tau)\|_{L^\infty_{x,v}(\mathbb{R}^n_+\times\mathbb{R}^n)}
    \end{align*}
    for $0\leq t\leq T_1$. Therefore, there exists a fixed point $f(t,x,v)$ of $\mathcal{T}$ in $L^\infty_t([0, T_1], L^\infty_{x,v}(w_{\alpha,\beta,\delta-\kappa}))$, which is a solution of the Boltzmann equation. Since $w_{\alpha,\beta,\delta}f^i(t,x,v)\rightarrow w_{\alpha,\beta,\delta}f(t,x,v)$ in a.e. $(t,x,v)$, and $f^i\in \mathcal{A}$ for all $i$, we conclude that $f\in L^\infty_t([0, T_1], L^\infty_{x,v}(w_{\alpha,\beta,\delta}))$ satisfying \eqref{eq:apriori_esti} and $f\geq 0$.
    
    To prove uniqueness of the solution in $L^\infty_t([0, T_1], L^\infty_{x,v}(w_{\alpha,\beta,\delta}))$, suppose there is another solution $h(t,x,v)$ in $L^\infty_t([0, T_1], L^\infty_{x,v}(w_{\alpha,\beta,\delta}))$. Plugging $f_1 = f$ and $f_2 = h$ into \eqref{eq:contract_kappa>0}, we obtain
    \begin{align*}
        &\|w_{\alpha,\beta, \delta-\kappa}(f-h)(t)\|_{L^\infty_{x,v}}\\
        &\leq C\left[\|w_{\alpha, \beta,\delta} g(t,x_h,v)\|_{L^\infty_{t,x_h,v}} + \esssup_{0\leq \tau\leq t}\left(\|w_{\alpha,\beta,\delta} f(\tau)\|_{L^\infty_{x,v}}^2 + \|w_{\alpha,\beta,\delta}f(\tau)\|_{L^\infty_{x,v}} + \|w_{\alpha,\beta,\delta}h(\tau)\|_{L^\infty_{x,v}}\right)\right]\\
        &\qquad \times \int_0^t \|w_{\alpha,\beta,\delta-\kappa} (f-h)(\tau)\|_{L^\infty_{x,v}}\,d\tau.
    \end{align*}
    for a.e. $t$. By Gr{\"o}nwall's inequality, we have $f=h$ for a.e. $(t,x,v)$, so in $L^\infty_t([0, T_1], L^\infty_{x,v}(w_{\alpha,\beta,\delta}))$.

    To extend the solution in time, assume we have constructed the solution $f(t,x,v)$ for $t\in [0, T']$ and $\|w_{\alpha,\beta, \delta}f(T', x, v)\|_{L^\infty_{x,v}}<\infty$. Treating $f(T', x, v)$ as initial data, we can extend the solution map starting from $T'$ by
    \begin{align*}
        \mathcal{T}f(t+T',x,v) &= e^{-\int_{t-t_b}^t L(f)(\tau+T', x-v(t-\tau), v)\,d\tau}g(t+T'-t_b, x-vt_b, v)\mathbf{1}_{v_z\geq \frac{z}{t}} \\
        &\quad + e^{-\int_0^t L(f)(\tau+T', x-v(t-\tau), v)\,d\tau}f(T', x-vt, v)\mathbf{1}_{v_z<\frac{z}{t}}\\
        &\quad + \int_0^t e^{-\int_\tau^t L(f)(s+T', x-v(t-s), v)\,ds} Q^+(f,f)(\tau+T', x-v(t-\tau), v)\mathbf{1}_{z-v_z(t-\tau)> 0}\,d\tau.
    \end{align*}
    We can repeat the same computations as above and obtain
    \begin{align*}
        \begin{split}
            &\|w_{\alpha,\beta, \delta-\kappa}(\mathcal{T}f_1-\mathcal{T}f_2)(t+T')\|_{L^\infty_{x,v}}\\
            &\leq C\bigg[\|w_{\alpha, \beta,\delta} g(t,x_h,v)\|_{L^\infty_{t,x_h,v}} + \|w_{\alpha, \beta,\delta} f(T',x,v)\|_{L^\infty_{x,v}} \\
            &\quad + \esssup_{0\leq \tau\leq t}\left(\|w_{\alpha,\beta,\delta} f_1(\tau+T')\|_{L^\infty_{x,v}}^2 + \|w_{\alpha,\beta,\delta}f_1(\tau+T')\|_{L^\infty_{x,v}} + \|w_{\alpha,\beta,\delta}f_2(\tau+T')\|_{L^\infty_{x,v}}\right)\bigg]\\
            &\qquad \times \int_0^t \|w_{\alpha,\beta,\delta-\kappa} (f_1-f_2)(\tau+T')\|_{L^\infty_{x,v}}\,d\tau.
        \end{split}
    \end{align*}
    Therefore, as long as the right-hand side of \eqref{eq:apriori_esti} is finite, we can extend the time domain for the solution with uniqueness. It concludes the existence and uniqueness of the solution $f(t,x,v)$ in $L^\infty_{loc, t}([0,t_0), L^\infty_{x,v}(w_{\alpha, \beta,\delta}))$ satisfying \eqref{eq:apriori_esti}.\\

    \noindent \textit{(ii) The case $\kappa\leq 0$.} In this case, we use the original weight $w_{\alpha,\beta,\delta}(v)$. The only problematic part in the previous case was to bound the difference term $|L[f_1]-L[f_2]|$, which generated $(1+|v|)^\kappa$ term from the collision kernel $B$ as in \eqref{eq:L_bound}. If $\kappa\leq 0$, however, it does not produce the $(1+|v|)^\kappa$ factor. Indeed, we write
    \begin{align*}
        \left|L[f_1]-L[f_2]\right|(v)&\leq \int_{\mathbb{S}^{n-1}\times \mathbb{R}^n} |v-v_*|^\kappa b(\cos\theta) |f_1-f_2|(v_*)\,d\omega dv_*\\
        &\leq C\|w_{\alpha,\beta,\delta}(f_1-f_2)\|_{L^\infty_v}\int_{\mathbb{R}^n} |v-v_*|^\kappa (w_{\alpha,\beta,\delta})^{-1}(v_*)\,dv_*.
    \end{align*}
    Since $-n<\kappa\leq 0$, by the Hardy-Littlewood inequality, the last integral is bounded by
    \begin{align*}
        \int_{\mathbb{R}^n} |v-v_*|^\kappa (w_{\alpha,\beta,\delta})^{-1}(v_*)\,dv_*\leq \int_{\mathbb{R}^n} |v_*|^\kappa (w_{\alpha,\beta,\delta})^{-1}(v_*)\,dv_*<\infty.
    \end{align*}
    Therefore, we can modify \eqref{eq:diff_esti2} by
    \begin{align*}
        &w_{\alpha,\beta,\delta}(v)\left|e^{-\int_{t-t_b}^t L[f_1](\tau, x-v(t-\tau), v)\,d\tau} - e^{-\int_{t-t_b}^t L[f_2](\tau, x-v(t-\tau), v)\,d\tau}\right| g(t-t_b, x-vt_b, v)\mathbf{1}_{v_z\geq \frac{z}{t}}\\
        &\leq C \|w_{\alpha, \beta,\delta} g(t,x_h,v)\|_{L^\infty_{t,x_h,v}([0,\infty)\times \partial\mathbb{R}^n_+\times\mathbb{R}^n)}\int_0^t \left\|w_{\alpha,\beta,\delta}(f_1-f_2)(\tau)\right\|_{L^\infty_{x,v}}\,d\tau.
    \end{align*}
    For \eqref{eq:diff_esti3} and \eqref{eq:diff_esti4}, we use a similar computation and get
    \begin{align*}
        &w_{\alpha,\beta,\delta}(v)\int_0^t \left|e^{-\int_\tau^t L[f_1](s, x-v(t-s), v)\,ds} - e^{-\int_\tau^t L[f_2](s, x-v(t-s), v)\,ds}\right| Q^+(f_1,f_1)(\tau, x-v(t-\tau), v)\mathbf{1}_{z-v_z(t-\tau)> 0}\,d\tau\\
        &\leq C\int_0^t \|w_{\alpha,\beta,\delta} f_1(\tau)\|^2_{L^\infty_{x,v}}\int_\tau^t \left\|w_{\alpha,\beta,\delta}(f_1-f_2)(s)\right\|_{L^\infty_{x,v}}\,ds \,d\tau\\
        &\leq C\sup_{0\leq\tau\leq t}\|w_{\alpha,\beta,\delta} f_1(\tau)\|^2_{L^\infty_{x,v}}\int_0^t \left\|w_{\alpha,\beta,\delta}(f_1-f_2)(s)\right\|_{L^\infty_{x,v}}\,ds,
        \intertext{and}
        &w_{\alpha,\beta,\delta}(v)\int_0^t e^{-\int_\tau^t L[f_2](s, x-v(t-s), v)\,ds} \left|Q^+(f_1,f_1) - Q^+(f_2,f_2)\right|(\tau, x-v(t-\tau), v)\mathbf{1}_{z-v_z(t-\tau)> 0}\,d\tau\\
        &\leq C\int_0^t (\|w_{\alpha,\beta,\delta} f_1(\tau)\|_{L^\infty_{x,v}} + \|w_{\alpha,\beta,\delta} f_2(\tau)\|_{L^\infty_{x,v}})\|w_{\alpha,\beta,\delta} (f_1-f_2)(\tau)\|_{L^\infty_{x,v}}\,d\tau
    \end{align*}
    for $-n<\kappa\leq 0$ and $t\leq 1$. Collecting all the bounds, we have
    \begin{align}\label{eq:contract_kappa<=0}
        \begin{split}
            &\|w_{\alpha,\beta, \delta}(\mathcal{T}f_1-\mathcal{T}f_2)(t)\|_{L^\infty_{x,v}}\\
            &\leq C\left[\|w_{\alpha, \beta,\delta} g(t,x_h,v)\|_{L^\infty_{t,x_h,v}} + \esssup_{0\leq \tau\leq t}\left(\|w_{\alpha,\beta,\delta} f_1(\tau)\|_{L^\infty_{x,v}}^2 + \|w_{\alpha,\beta,\delta}f_1(\tau)\|_{L^\infty_{x,v}} + \|w_{\alpha,\beta,\delta}f_2(\tau)\|_{L^\infty_{x,v}}\right)\right]\\
            &\qquad \times \int_0^t \|w_{\alpha,\beta,\delta} (f_1-f_2)(\tau)\|_{L^\infty_{x,v}}\,d\tau.
        \end{split}
    \end{align}
    Starting from \eqref{eq:contract_kappa<=0}, we follow the remaining steps for $\kappa>0$ and establish the existence and uniqueness of the solution satisfying \eqref{eq:apriori_esti}.
\end{proof}

\begin{remark}
    In the proof, we can check that $L(f)(t,x,v)$ and $Q^+(f,f)(t,x,v)$ are in $L^1_{loc}(t,x,v)$ if $f(t,x,v)\in L^\infty_{loc,t}([0, t_0), L^\infty_{x,v}(w_{\alpha,\beta,\delta}))$ as a byproduct. Therefore, by Lemma~\ref{lem:mild_sol_equiv}, the mild solution $f$ also satisfies
    \begin{align*}
        f(t,x,v) = g(t-t_b, x-vt_b, v)\mathbf{1}_{v_z\geq \frac{z}{t}} + \int_0^t Q(f,f)(\tau, x-v(t-\tau), v)\mathbf{1}_{z-v_z(t-\tau)> 0}\,d\tau,
    \end{align*}
    which is another presentation of a mild solution.
\end{remark}
\begin{remark}
    We can slightly modify the proof and get the result that there exists a unique solution $f(t,x,v)$ as long as $\|w_{\alpha,\beta,\delta}f(t,x,v)\|_{L^\infty_{x,v}}<\infty$ by repeatedly extending the time domain of the solution.
\end{remark}

Having the existence and uniqueness in $L^\infty_{loc,t}([0, t_0), L^\infty_{x,v}(w_{\alpha,\beta,\delta}))$, we establish the local-in-time boundedness in $L^\infty_{loc,t}([0, t_0), L^p_vL^q_x(w_{\alpha,\beta,\delta}))$ space.
\begin{proposition}\label{prop:boltzmann_wellposed}
    Assume $g$ satisfies
    \begin{align*}
        \|v_z^{1/q}w_{\alpha,\beta,\delta}g(t,x_h,v)\|_{L^p_v(\mathbb{R}^n, L^q_{t,x_h}([0, t_0)\times \partial \mathbb{R}^n_+))}<\infty
    \end{align*}
    for $\alpha>0$, $\max\left\{0, \frac{2}{n-1}\kappa, \frac{2}{3}(1+\kappa)\right\}<\beta\leq 2$, $\delta>n+1$, and $t_0$ defined by \eqref{eq:t_0_def}. If $n=2$, we further impose a condition $1+\kappa\leq \beta$. Then, the solution of the Boltzmann equation $f(t,x,v)$ for $t\in [0, t_0)$ constructed above satisfies
    \begin{align*}
        \|w_{\alpha,\beta,\delta} f(t)\|_{L^p_v(\mathbb{R}^n, L^q_x(\mathbb{R}^n_+))}\leq \|v_z^{1/q} w_{\alpha,\beta,\delta} g\|_{L^p_v(\mathbb{R}^n, L^q_{t,x_h}([0, t_0)\times \partial \mathbb{R}^n_+))}\exp\left(C_{gain,2}\int_0^t \|w_{\alpha,\beta,\delta} f(\tau)\|_{L^\infty_{x,v}(\mathbb{R}^n\times\mathbb{R}^n_+)}\,d\tau\right)
    \end{align*}
    for $0\leq t<t_0$. Here, $C_{gain, 2}$ is from Lemma~\ref{lem:LPY_lem}.
\end{proposition}
\begin{proof}
    By Lemma~\ref{lem:trace}, we have
    \begin{align*}
        \|w_{\alpha,\beta,\delta} g(t-t_b,x-vt_b,v)\mathbf{1}_{v_z\geq \frac{z}{t}}\|_{L^p_v(\mathbb{R}^n, L^q_x(\mathbb{R}^n_+))}&=\|v_z^{1/q}w_{\alpha,\beta,\delta} g(t,x_h,v)\|_{L^p_v(\mathbb{R}^n, L^q_{t,x_h}([0, t]\times \partial \mathbb{R}^n_+))}\\
        &\leq \|v_z^{1/q}w_{\alpha,\beta,\delta} g(t,x_h,v)\|_{L^p_v(\mathbb{R}^n, L^q_{t,x_h}([0, t_0)\times \partial \mathbb{R}^n_+))}.
    \end{align*}
    Employing Lemma~\ref{lem:LPY_lem}, we obtain
    \begin{align*}
        \|w_{\alpha,\beta,\delta} f(t)\|_{L^p_v(\mathbb{R}^n, L^q_x(\mathbb{R}^n_+))}&\leq \|w_{\alpha,\beta,\delta} g(t-t_b,x-vt_b,v)\mathbf{1}_{v_z\geq \frac{z}{t}}\|_{L^p_v(\mathbb{R}^n, L^q_x(\mathbb{R}^n_+))} \\
        &\quad + \int_0^t \|w_{\alpha,\beta,\delta}Q^+(f,f)(\tau, x-v(t-\tau), v)\mathbf{1}_{z-v_z(t-\tau)> 0}\|_{L^p_v(\mathbb{R}^n, L^q_x(\mathbb{R}^n_+))}\,d\tau\\
        &\leq \|v_z^{1/q} w_{\alpha,\beta,\delta} g(t,x_h, v)\|_{L^p_v(\mathbb{R}^n, L^q_{t,x_h}([0, t_0)\times \partial \mathbb{R}^n_+))}\\
        &\quad + C_{gain,2}\int_0^t \|w_{\alpha,\beta,\delta} f(\tau)\|_{L^\infty_{x,v}(\mathbb{R}^n\times\mathbb{R}^n_+)}\|w_{\alpha,\beta,\delta} f(\tau)\|_{L^p_v(\mathbb{R}^n, L^q_x(\mathbb{R}^n_+))}\,d\tau.
    \end{align*}
    By Gr{\"o}nwall's inequality, we get \textit{a priori} estimate
    \begin{align*}
        \|w_{\alpha,\beta,\delta} f(t)\|_{L^p_v(\mathbb{R}^n, L^q_x(\mathbb{R}^n_+))}\leq \|v_z^{1/q} w_{\alpha,\beta,\delta} g\|_{L^p_v(\mathbb{R}^n, L^q_{t,x_h}([0, t_0)\times \partial \mathbb{R}^n_+))}\exp\left(C_{gain,2}\int_0^t \|w_{\alpha,\beta,\delta} f(\tau)\|_{L^\infty_{x,v}(\mathbb{R}^n\times\mathbb{R}^n_+)}\,d\tau\right)
    \end{align*}
    for $0\leq t< t_0$. Finally, we apply this \textit{a priori} estimate to the iteration scheme of Proposition~\ref{prop:boltzmann_existence} to conclude the proposition.
\end{proof}

\begin{remark}
    In some range of $\kappa$ and $\beta$, we can construct a unique solution of the Boltzmann equation with the inflow boundary data $g(t,x_h,v)$ of the form \eqref{eq:boundary_data}. Let $g$ be given by \eqref{eq:boundary_data} for some $\alpha>0$, $0<\beta\leq 2$, and $m\geq 0$ to be determined later. We choose the weight function $w_{\frac{\alpha}{2},\frac{\beta}{1+\beta}, \delta}$ with $\delta>n + 1$. We write the multiplication of $g$ and the weight function as
    \begin{align*}
        w_{\frac{\alpha}{2},\frac{\beta}{1+\beta}, \delta}(v)g(t,x_h,v) &= e^{-t} w_{\frac{\alpha}{2},\frac{\beta}{1+\beta}, \delta}(v) \langle v_h\rangle^{-m}e^{-\alpha|v_h|^\beta}\mathbf{1}_{|v_h|\geq (|x_h-v_h t| + v_z)^{\frac{1}{1+\beta}}}\mathbf{1}_{v_z\geq 0}\\
        &=e^{-t}\langle v\rangle^\delta\langle v_h\rangle^{-m} e^{-\alpha|v_h|^\beta} \exp\left(\frac{\alpha}{2} |v|^{\frac{\beta}{1+\beta}}\right)\mathbf{1}_{|v_h|\geq (|x_h-v_h t| + v_z)^{\frac{1}{1+\beta}}}\mathbf{1}_{v_z\geq 0}.
    \end{align*}
    If $|v_h|\geq v_z$ and $|v_h|\geq 1$, by Lemma~\ref{lem:beta_pert}, we have
    \begin{align*}
        \frac{1}{2}\left(|v_h|^2 + v_z^2\right)^{\frac{\beta}{2(1+\beta)}} - |v_h|^\beta\leq \frac{1}{2}\left(|v_h|^{\frac{\beta}{1+\beta}} + \frac{\beta}{2(1+\beta)}|v_h|^{\frac{\beta}{1+\beta}}\right) - |v_h|^\beta\leq 0.
    \end{align*}
    If $|v_h|\leq v_z$, as $v_z\leq |v_h|^{1+\beta}$ by the cutoff condition, we have
    \begin{align*}
        \frac{1}{2}\left(|v_h|^2 + v_z^2\right)^{\frac{\beta}{2(1+\beta)}} - |v_h|^\beta\leq \frac{1}{2}\left(v_z^{\frac{\beta}{1+\beta}} + \frac{\beta}{2(1+\beta)}v_z^{\frac{\beta}{1+\beta}}\right) - |v_h|^\beta\leq 0.
    \end{align*}
    Therefore,
    \begin{align*}
        &e^{-t}v_z^{1/q}\langle v\rangle^\delta \langle v_h\rangle^{-m}e^{-\alpha|v_h|^\beta} \exp\left(\frac{\alpha}{2} |v|^{\frac{\beta}{1+\beta}}\right)\mathbf{1}_{|v_h|\geq (|x_h-v_ht| + v_z)^{\frac{1}{1+\beta}}}\mathbf{1}_{v_z\geq 0}\\
        &\leq Ce^{-t} v_z^{1/q}\langle v\rangle^\delta \langle v_h\rangle^{-m}\mathbf{1}_{|v_h|\geq \left(|x_h-v_ht| + v_z\right)^{\frac{1}{1+\beta}}}\mathbf{1}_{v_z\geq 0}
    \end{align*}
    for some constant $C$ depending on $\alpha$, which is similar to the equation right before \eqref{eq:g_LpvLqx_form}. In fact, for $t_0>0$, we have
    \begin{align*}
        &\left\|e^{-t} v_z^{1/q}\langle v\rangle^\delta \langle v_h\rangle^{-m}\mathbf{1}_{|v_h|\geq \left(|x_h-v_ht| + v_z\right)^{\frac{1}{1+\beta}}}\mathbf{1}_{v_z\geq 0}\right\|_{L^p_v(\mathbb{R}^n, L^q_{t,x_h}([0,t_0)\times\partial \mathbb{R}^n_+))} \\
        &= \left(\frac{1-e^{-qt_0}}{q}\right)^{1/q}\left\| v_z^{1/q}\langle v\rangle^\delta \langle v_h\rangle^{-m}\mathbf{1}_{|v_h|\geq \left(|x_h| + v_z\right)^{\frac{1}{1+\beta}}}\mathbf{1}_{v_z\geq 0}\right\|_{L^p_v(\mathbb{R}^n, L^q_{x_h}(\partial \mathbb{R}^n_+))}
    \end{align*}
    for $1\leq q<\infty$ and
    \begin{align*}
        &\left\|e^{-t} v_z^{1/q}\langle v\rangle^\delta \langle v_h\rangle^{-m}\mathbf{1}_{|v_h|\geq \left(|x_h-v_ht| + v_z\right)^{\frac{1}{1+\beta}}}\mathbf{1}_{v_z\geq 0}\right\|_{L^p_v(\mathbb{R}^n, L^\infty_{t,x_h}([0,t_0)\times\partial \mathbb{R}^n_+))} \\
        &= \left\| v_z^{1/q}\langle v\rangle^\delta \langle v_h\rangle^{-m}\mathbf{1}_{|v_h|\geq \left(|x_h| + v_z\right)^{\frac{1}{1+\beta}}}\mathbf{1}_{v_z\geq 0}\right\|_{L^p_v(\mathbb{R}^n, L^\infty_{x_h}(\partial \mathbb{R}^n_+))}
    \end{align*}
    for $q=\infty$. Therefore, if $m$ satisfies \eqref{eq:m_cond1} for $p=q=\infty$, we have
    \begin{align*}
        \left\|w_{\frac{\alpha}{2},\frac{\beta}{1+\beta},\delta} g(t,x_h,v)\right\|_{L^\infty_{t,x_h,v}([0,\infty)\times \partial \mathbb{R}^n_+\times \mathbb{R}^n)}<\infty.
    \end{align*}
    Now, we can apply Proposition~\ref{prop:boltzmann_existence} and check that there exists a unique local-in-time solution of the Boltzmann equation in $L^\infty_{loc, t}([0, t_0), L^\infty_{x,v}(w_{\frac{\alpha}{2},\frac{\beta}{1+\beta}, \delta}))$ space when $\kappa$ and $\beta$ satisfy
    \begin{align}\label{eq:kappa_range}
        \max\left\{0, \frac{2}{n-1}\kappa, \frac{2}{3}(1+\kappa)\right\}<\frac{\beta}{1+\beta}.
    \end{align}
    Furthermore, if $m$ is large enough such that it meets \eqref{eq:m_cond1} for given $1\leq p,q\leq \infty$, then it also satisfies
    \begin{align*}
        \sup_{t\geq 0}\left\|v_z^{1/q} w_{\frac{\alpha}{2},\frac{\beta}{1+\beta},\delta} g(t,x_h,v)\right\|_{L^p_v(\mathbb{R}^n, L^q_{t,x_h}([0,t]\times \partial \mathbb{R}^n_+))}<\infty.
    \end{align*}
    If we further assume $1+\kappa\leq \frac{\beta}{1+\beta}$ when $n=2$, Proposition~\ref{prop:boltzmann_wellposed} shows that the Boltzmann solution is also in $L^\infty_{loc, t}([0, t_0), L^p_vL^q_x(w_{\frac{\alpha}{2},\frac{\beta}{1+\beta}, \delta}))$.
\end{remark}
When $1\leq \beta\leq 2$, the possible range of $\frac{\beta}{1+\beta}$ is between $\frac{1}{2}$ and $\frac{2}{3}$. From \eqref{eq:kappa_range}, we need to impose $\kappa<0$ to ensure the existence.\\

Now, we prove the main theorem for the Boltzmann equation.
\begin{proof}[Proof of Theorem~\ref{thm:main_Boltzmann}]
    For $w_{\alpha,\beta,\delta}g(t, x_h, v)\in L^\infty_{t,x_h,v}\left([0, t^*_0]\times \partial \mathbb{R}^n_+\times \mathbb{R}^n\right)$, we extend it to $t\in [0,\infty)$ by making
    \begin{align*}
        \tilde{g}(t,x,v) = g(t, x_h, v)\mathbf{1}_{t\leq t^*_0}.
    \end{align*}
    Applying Proposition~\ref{prop:boltzmann_existence} for $\tilde{g}$, we find a unique solution of the Boltzmann equation $f$ in $L_{loc, t}^\infty([0,t_0), L^\infty_{x,v}(w_{\alpha, \beta,\delta}))$ for some $0<t_0\leq t_0^*$. Finally, we employ Proposition~\ref{prop:boltzmann_wellposed} to get the $L_{loc, t}^\infty\left([0,t_0), L^p_vL^q_x(w_{\alpha, \beta,\delta}))\right)$ boundedness.
\end{proof}

\appendix

\section{Stability of the macroscopic fields with respect to inflow boundary data}
In the Appendix, we provide a partial explanation of why we need to exclude the $v_z$ component from the weight function for the BGK model in Theorem \ref{thm:main_BGK}. This is due to an instability of the macroscopic fields when there exists only the inflow boundary data. In Proposition \ref{prop:BGK_contradiction}, we prove that any solution with $f_0(x,v) = 0$ and $g(t,x_h,v) = e^{-|v|^2}$, where the inflow boundary data is given by the natural Maxwellian function with fixed temperature, can not satisfy the second condition of Definition \ref{def:BGK_sol_space}. This result indicates that estimating the macroscopic fields of the BGK solution is beyond the reach of standard techniques, even in this case, since these fields cannot be controlled in terms of $(\rho_g,u_g, T_g)$. As the fixed point argument for solution construction requires the sharp estimate of $(\rho,u, T)$, a possible alternative is to use a weak compactness argument, which can not guarantee the uniqueness of the solution.

Proposition \ref{prop:suppressed_v_z} demonstrates that we can avoid this problem by weakening the decay in the normal velocity component in the inflow boundary function. In detail, by choosing $g(t,x_h,v) = e^{-|v_h|^2 - |v_z|^\beta}$ for $0<\beta\leq 1$, we prove that the macroscopic fields of the solution are stable in the sense of Definition \ref{def:BGK_sol_space}, and we can follow the same procedure as in Section \ref{sec:sol_construction} to construct a unique BGK solution. Motivated by this observation, we choose the weight function $e^{|v_h|^\beta}$ for $1\leq \beta\leq 2$ in Theorem \ref{thm:main_BGK} and impose a proper interior cutoff function to construct a unique solution and establish the desired ill-posedness for this weight.

This restriction also reveals an interesting contrast between the BGK model and the Boltzmann equation. For the Boltzmann equation, it is more natural to impose inflow boundary data with the exponential decay $e^{-|v|^\beta}$ than with the tangential exponential decay $e^{-|v_h|^\beta}$. Indeed, the analysis of the collision operator $Q$, such as the one stated in Lemma~\ref{lem:LPY_lem}, relies crucially on the following identity from the energy conservation:
\begin{align*}
    e^{-|v|^2-|v_*|^2} = e^{-|v'|^2 - |v'_*|^2}.
\end{align*}
When the tangential exponential function $e^{-|v_h|^\beta}$ is used, however, no corresponding identity holds, making it impossible to close the weighted estimates for $Q$ using this weight alone. This difference suggests that, with vacuum initial data, the two models respond differently to the boundary data.

We first prove that the macroscopic fields of the solution with the Maxwellian inflow boundary data have an instability.
\begin{proposition}\label{prop:BGK_contradiction}
    For the given initial data $f_0(x,v) = 0$ and inflow boundary data $g(t,x_h,v) = e^{-|v|^2}$, if $f(t,x,v)$ solves the BGK model, then it does not satisfy the second condition of Definition~\ref{def:BGK_sol_space}.
\end{proposition}
\begin{proof}
    We confine $0< t\leq 1$. We start by computing the $(\rho_g,u_g, T_g)(t,x)$. Using Lemma~\ref{lem:exp_integral}, we have
    \begin{align}\label{eq:Appendix_rho_g_bound2}
        \begin{split}
            \rho_g(t,x) &= \int_{\mathbb{R}^n} e^{-t_b}e^{-|v|^2}1_{v_z\geq \frac{z}{t}}\,dv = \int_{\mathbb{R}^{n-1}}e^{-|v_h|^2}\,dv_h\int_{\frac{z}{t}}^\infty e^{-\frac{z}{v_z}}e^{-v_z^2}\,dv_z\geq C_1 e^{-t}\frac{1}{\frac{z}{t}+1} e^{-\frac{z^2}{t^2}},
            \intertext{and}
            \rho_g(t,x)&\leq C_2 \frac{1}{\frac{z}{t}+1} e^{-\frac{z^2}{t^2}}
        \end{split}
    \end{align}
    for some constants $C_1$ and $C_2$. For $u_g$, we proceed with a similar calculation and obtain
    \begin{align*}
        \rho_g(t,x)u_{g,h}(t,x) &=\int_{\mathbb{R}^{n-1}}v_h e^{-|v_h|^2}\,dv_h\int_{\frac{z}{t}}^\infty e^{-\frac{z}{v_z}}e^{-v_z^2}\,dv_z = 0,\\
        \rho_g(t,x)\left|u_{g,z}(t,x)-\frac{z}{t}\right| &\leq \int_{\mathbb{R}^{n-1}}e^{-|v_h|^2}\,dv_h\int_{\frac{z}{t}}^\infty e^{-\frac{z}{v_z}}\left|v_z-\frac{z}{t}\right| e^{-v_z^2}\,dv_z\leq C_3'\frac{1}{\left(\frac{z}{t}+1\right)^2} e^{-\frac{z^2}{t^2}}
    \end{align*}
    for some constants $C_3'$, where $u_{g,h}$ and $u_{g,z}$ are horizontal and vertical components of $u_g$. It shows that
    \begin{align}\label{eq:Appendix_u_g_bound2}
        u_{g,h}(t,x) = 0,\quad \left|u_{g,z}(t,x)-\frac{z}{t}\right|\leq \frac{C_3}{\frac{z}{t} + 1}
    \end{align}
    for a constant $C_3$.
    
    For $T_g$, we write
    \begin{align*}
        n\rho_g(t,x)T_g(t,x)&=\int_{\mathbb{R}^n} e^{-\frac{z}{v_z}}\left(|v_h|^2 + |v_z - u_{g,z}|^2\right)e^{-|v|^2}1_{v_z\geq \frac{z}{t}}\,dv\\
        &= \int_{\mathbb{R}^{n-1}} |v_h|^2 e^{-|v_h|^2} \int_{\frac{z}{t}}^\infty e^{-\frac{z}{v_z}} e^{-v_z^2}\,dv_z dv_h\\
        &\quad + \int_{\mathbb{R}^{n-1}} e^{-|v_h|^2} \int_{\frac{z}{t}}^\infty |v_z - u_{g,z}|^2 e^{-\frac{z}{v_z}} e^{-v_z^2}\,dv_z dv_h.
    \end{align*}
    We first compute an upper bound of $T_g$. The first integral is easily bounded above by
    \begin{align}\label{eq:Appendix_T_g_bound}
        \int_{\mathbb{R}^{n-1}} |v_h|^2 e^{-|v_h|^2} \int_{\frac{z}{t}}^\infty e^{-\frac{z}{v_z}} e^{-v_z^2}\,dv_z dv_h\leq C\frac{1}{\frac{z}{t}+1} e^{-\frac{z^2}{t^2}}.
    \end{align}
    For the second integral, we have
    \begin{align*}
        \int_{\mathbb{R}^{n-1}} e^{-|v_h|^2} \int_{\frac{z}{t}}^\infty |v_z - u_{g,z}|^2 e^{-\frac{z}{v_z}} e^{-v_z^2}\,dv_z dv_h&\leq \int_{\mathbb{R}^{n-1}} e^{-|v_h|^2} \int_{\frac{z}{t}}^\infty 2\left(\left|v_z - \frac{z}{t}\right|^2 + \left|u_{g,z}-\frac{z}{t}\right|^2\right) e^{-\frac{z}{v_z}} e^{-v_z^2}\,dv_z dv_h\\
        &\leq C\frac{1}{\left(\frac{z}{t}+1\right)^3} e^{-\frac{z^2}{t^2}}
    \end{align*}
    by \eqref{eq:Appendix_u_g_bound2} and Lemma~\ref{lem:exp_integral}. Therefore, we conclude
    \begin{align}\label{eq:Appendix_T_g_bound2}
        n\rho_g(t,x)T_g(t,x)\leq C_5'\frac{1}{\frac{z}{t}+1} e^{-\frac{z^2}{t^2}}
    \end{align}
    for some $C_5'$. By a similar computation for \eqref{eq:Appendix_T_g_bound}, we compute the lower bound of $n\rho_g T_g$ by
    \begin{align}\label{eq:Appendix_T_g_bound3}
        n\rho_g(t,x)T_g(t,x)&\geq  \int_{\mathbb{R}^{n-1}} |v_h|^2 e^{-|v_h|^2} \int_{\frac{z}{t}}^\infty e^{-\frac{z}{v_z}} e^{-v_z^2}\,dv_z dv_h\geq C_4' \frac{1}{\frac{z}{t}+1} e^{-\frac{z^2}{t^2}}
    \end{align}
    for a constant $C_4'$. Combining all the macroscopic bound results \eqref{eq:Appendix_rho_g_bound2}, \eqref{eq:Appendix_u_g_bound2}, \eqref{eq:Appendix_T_g_bound2}, and \eqref{eq:Appendix_T_g_bound3}, we get
    \begin{align*}
        &C_1 e^{-t}\frac{1}{\frac{z}{t}+1} e^{-\frac{z^2}{t^2}}\leq \rho_g(t,x)\leq C_2\frac{1}{\frac{z}{t}+1} e^{-\frac{z^2}{t^2}},\\
        &\left|u_g(t,x)-\frac{z}{t}\hat{z}\right|\leq \frac{C_3}{\frac{z}{t} + 1},\\
        &C_4\leq T_g(t,x)\leq C_5
    \end{align*}
    for some constants $C_i$, $1\leq i\leq 5$.

    Now, we assume that there exists $0<t^*\leq 1$ such that
    \begin{align}\label{eq:macro_contradiction}
        \begin{split}
            &|\rho-\rho_g|(t,x)\leq \frac{1}{2}\rho_g(t,x),\\
            &\left|u-\frac{z}{t}\hat{z}\right|(t,x)\leq C_u\sqrt{T_g(t,x)},\\
            &|T-T_g|(t,x)\leq \frac{1}{2}T_g(t,x)
        \end{split}
    \end{align}
    for some constant $C_u>0$ for $0< t\leq t^*$ and derive that $\rho(t,x)$ can not satisfy the condition. From the definition of the BGK solution, we have
    \begin{align*}
        f(t,x,v) &= e^{-t_b}g(t-t_b, x-vt_b, v)\mathbf{1}_{t\geq t_b} + \int_0^t e^{-(t-\tau)}\mathcal{M}(f)(\tau, x-v(t-\tau), v)\mathbf{1}_{z-v_z(t-\tau)> 0}\,d\tau\\
        &=e^{-t_b}e^{-|v|^2}\mathbf{1}_{v_z\geq \frac{z}{t}} + \int_0^t e^{-(t-\tau)}\mathcal{M}(f)(\tau, x-v(t-\tau), v)\mathbf{1}_{z-v_z(t-\tau)> 0}\,d\tau,
    \end{align*}
    we have
    \begin{align}\label{eq:contra_2}
        \rho(t,x) \geq \rho_g(t,x) + \int_0^t e^{-(t-\tau)}\int_{\mathbb{R}^n} \mathcal{M}(f)(\tau, x-v(t-\tau), v)\mathbf{1}_{z-v_z(t-\tau)> 0}\,dv d\tau.
    \end{align}
    We compute the integral of the local Maxwellian term and prove that it is much bigger than $\rho_g$.
    
    Since $(\rho,u,T)$ satisfy the relations, using \eqref{eq:macro_contradiction}, we can write
    \begin{align*}
        \mathcal{M}(f)(t,x,v)&= \frac{\rho(t,x)}{(2\pi T(t,x))^{n/2}}\exp\left(-\frac{|v-u(t,x)|^2}{2T(t,x)}\right)\\
        &\geq C\frac{\rho_g(t,x)}{(2\pi T_g(t,x))^{n/2}}\exp\left(-\frac{2\left(\left|v-\frac{z}{t}\hat{z}\right|^2 + \left|u(t,x) - \frac{z}{t}\hat{z}\right|^2\right)}{T_g(t,x)}\right)\\
        &\geq C\frac{1}{\frac{z}{t}+1} e^{-\frac{z^2}{t^2}}\exp\left(-\frac{|v-\frac{z}{t}\hat{z}|^2}{C_4/4}\right).
    \end{align*}
    for a constant $C$ depending on $C_i$ for $1\leq i\leq 5$, $C_u$, and dimension $n$. Using this lower bound, we have
    \begin{align*}
        &\int_{\mathbb{R}^n} \left|v-\frac{z}{t}\hat{z}\right|^m \mathcal{M}(f)(\tau, x-v(t-\tau), v)\mathbf{1}_{z-v_z(t-\tau)> 0}\,dv\\
        &\geq C\int_{\mathbb{R}^n} \left|v-\frac{z}{t}\hat{z}\right|^m \frac{1}{\frac{z-v_z(t-\tau)}{\tau} +1} e^{-\left(\frac{z-v_z(t-\tau)}{\tau}\right)^2}\exp\left(-\frac{|v_h|^2 + \left|v_z - \frac{z-v_z(t-\tau)}{\tau}\right|^2}{C_4/4}\right)\mathbf{1}_{z-v_z(t-\tau)>0}\,dv\\
        &=C\int_{\mathbb{R}^n} |v|^m \frac{1}{\frac{z}{t} - \frac{t-\tau}{\tau}v_z +1} e^{-\left(\frac{z}{t} - \frac{t-\tau}{\tau}v_z\right)^2}\exp\left(-\frac{|v_h|^2 + \left|\frac{t}{\tau}v_z\right|^2}{C_4/4}\right)\mathbf{1}_{\frac{t-\tau}{\tau}v_z<\frac{z}{t}}\,dv\\
        &=C\frac{\tau}{t-\tau}\int_{\mathbb{R}^n} \left(|v_h|^2 + \left|\frac{\tau}{t-\tau}v_z\right|^2\right)^{m/2} \frac{1}{\frac{z}{t} - v_z +1} e^{-\left(\frac{z}{t} - v_z\right)^2}\exp\left(-\frac{|v_h|^2 + \left|\frac{t}{t-\tau}v_z\right|^2}{C_4/4}\right)\mathbf{1}_{v_z<\frac{z}{t}}\,dv.
    \end{align*}
    In the last two equalities, we used the change of variable $v_z\mapsto v_z + \frac{z}{t}$ and $v_z\mapsto \frac{\tau}{t-\tau}v_z$ in steps. We do the $v_h$-integral first and get
    \begin{align*}
        &\frac{\tau}{t-\tau}\int_{\mathbb{R}^n} \left(\frac{\tau}{t-\tau}|v_z|\right)^m \frac{1}{\frac{z}{t} - v_z +1} e^{-\left(\frac{z}{t} - v_z\right)^2}\exp\left(-\frac{|v_h|^2 + \left|\frac{t}{t-\tau}v_z\right|^2}{C_T}\right)\mathbf{1}_{v_z<\frac{z}{t}}\,dv\\
        &\geq C\frac{\tau}{t-\tau}\int_{\mathbb{R}} \left(\frac{\tau}{t-\tau}|v_z|\right)^m \frac{1}{\frac{z}{t} - v_z +1} e^{-\left(\frac{z}{t} - v_z\right)^2}\exp\left(-\frac{\left|\frac{t}{t-\tau}v_z\right|^2}{C_T}\right)\mathbf{1}_{v_z<\frac{z}{t}}\,dv_z,
    \end{align*}
    where $C_T = C_4/4$.
    
    The sum of the exponents is given by
    \begin{align*}
        \left(\frac{z}{t} - v_z\right)^2 + \frac{\left|\frac{t}{t-\tau}v_z\right|^2}{C_T} = \left(1 + \frac{t^2}{C_T(t-\tau)^2}\right)\left(v_z -\left(\frac{t^2}{C_T(t-\tau)^2}+1\right)^{-1}\frac{z}{t}\right)^2 + \frac{1}{\frac{C_T(t-\tau)^2}{t^2}+1}\frac{z^2}{t^2}.
    \end{align*}
    Therefore, merging the two exponents and taking the change of variable
    \begin{align*}
        v_z\mapsto \left(1 + \frac{t^2}{C_T(t-\tau)^2}\right)^{-1/2}v_z + \frac{1}{1+\frac{t^2}{C_T(t-\tau)^2}}\frac{z}{t},
    \end{align*}
    we get
    \begin{align}
        &\frac{\tau}{t-\tau}\int_{\mathbb{R}} \left(\frac{\tau}{t-\tau}|v_z|\right)^m\frac{1}{\frac{z}{t} - v_z +1} e^{-\left(\frac{z}{t} - v_z\right)^2}\exp\left(-\frac{\left|\frac{t}{t-\tau}v_z\right|^2}{C_T}\right)\mathbf{1}_{v_z< \frac{z}{t}}\,dv_z\notag\\
        &=\frac{\tau}{t-\tau}\exp\left(-\frac{1}{\frac{C_T(t-\tau)^2}{t^2}+1}\frac{z^2}{t^2}\right)\int_{\mathbb{R}} \left(\frac{\tau}{t-\tau}|v_z|\right)^m\frac{1}{\frac{z}{t} - v_z +1}\notag\\
        &\qquad \times \exp\left(- \left(1 + \frac{t^2}{C_T(t-\tau)^2}\right)\left(v_z -\left(\frac{t^2}{C_T(t-\tau)^2}+1\right)^{-1}\frac{z}{t}\right)^2\right)\mathbf{1}_{v_z< \frac{z}{t}}\,dv_z\notag\\
        &\geq \frac{\tau}{t}\frac{\sqrt{C_T}}{\sqrt{\frac{C_T(t-\tau)^2}{t^2} + 1}}\exp\left(-\frac{1}{\frac{C_T(t-\tau)^2}{t^2}+1}\frac{z^2}{t^2}\right)\int_{\mathbb{R}} \left(\frac{\tau}{t}\frac{\sqrt{C_T}}{\sqrt{\frac{C_T(t-\tau)^2}{t^2} + 1}}\left|v_z+\frac{1}{\sqrt{1+\frac{t^2}{C_T(t-\tau)^2}}}\frac{z}{t}\right|\right)^m \notag\\
        &\qquad \times \frac{1}{\frac{1}{\frac{C_T(t-\tau)^2}{t^2}+1}\frac{z}{t} - \frac{1}{\sqrt{1 + \frac{t^2}{C_T(t-\tau)^2}}}v_z +1}\exp\left(- |v_z|^2\right)\mathbf{1}_{v_z< \frac{1}{\sqrt{C_T}}\frac{t}{t-\tau}\frac{1}{\sqrt{\frac{C_T(t-\tau)^2}{t^2} + 1}}\frac{z}{t}}\,dv_z\notag.
    \end{align}
    We confine $0\leq \tau\leq t/2$ and $v_z\geq 0$. Then, we obtain
    \begin{align*}
        &\frac{\tau}{t}\frac{\sqrt{C_T}}{\sqrt{\frac{C_T(t-\tau)^2}{t^2} + 1}}\exp\left(-\frac{1}{\frac{C_T(t-\tau)^2}{t^2}+1}\frac{z^2}{t^2}\right)\int_{\mathbb{R}} \left(\frac{\tau}{t}\frac{\sqrt{C_T}}{\sqrt{\frac{C_T(t-\tau)^2}{t^2} + 1}}\left|v_z+\frac{1}{\sqrt{1+\frac{t^2}{C_T(t-\tau)^2}}}\frac{z}{t}\right|\right)^m\\
        &\qquad \times \frac{1}{\frac{1}{\frac{C_T(t-\tau)^2}{t^2}+1}\frac{z}{t} - \frac{1}{\sqrt{1 + \frac{t^2}{C_T(t-\tau)^2}}}v_z +1}\exp\left(- |v_z|^2\right)\mathbf{1}_{v_z< \frac{1}{\sqrt{C_T}}\frac{t}{t-\tau}\frac{1}{\sqrt{\frac{C_T(t-\tau)^2}{t^2} + 1}}\frac{z}{t}}\,dv_z\\
        &\geq \left(\frac{\tau}{t}\right)^{m+1}\left(\frac{\sqrt{C_T}}{\sqrt{C_T + 1}}\right)^{m+1}\exp\left(-\frac{1}{\frac{C_T}{4}+1}\frac{z^2}{t^2}\right) \frac{1}{\frac{1}{\frac{C_T}{4}+1}\frac{z}{t}+1}\\
        &\qquad \times \int_{\mathbb{R}}\left(v_z + \frac{1}{\sqrt{1+\frac{4}{C_T}}}\frac{z}{t}\right)^m\exp\left(- v_z^2\right)\mathbf{1}_{0\leq v_z< \frac{1}{\sqrt{C_T}}\frac{1}{\sqrt{C_T + 1}}\frac{z}{t}}\,dv_z.
    \end{align*}
    When $\frac{1}{\sqrt{C_T}}\frac{1}{\sqrt{C_T + 1}}\frac{z}{t}\geq 1$, then the final integral is bounded below by
    \begin{align*}
        \int_{\mathbb{R}}\left(v_z + \frac{1}{\sqrt{1+\frac{4}{C_T}}}\frac{z}{t}\right)^m\exp\left(- v_z^2\right)\mathbf{1}_{0\leq v_z< \frac{1}{\sqrt{C_T}}\frac{1}{\sqrt{C_T + 1}}\frac{z}{t}}\,dv_z\geq C\left(1 + \frac{1}{\sqrt{1+\frac{4}{C_T}}}\frac{z}{t}\right)^m.
    \end{align*}
    It shows that
    \begin{align*}
        &\int_0^t e^{-(t-\tau)}\int_{\mathbb{R}^n} \left|v-\frac{z}{t}\hat{z}\right|^m \mathcal{M}(f)(\tau, x-v(t-\tau), v)\mathbf{1}_{z-v_z(t-\tau)> 0}\,dv d\tau\\
        &\geq C\exp\left(-\frac{1}{\frac{C_T}{4}+1}\frac{z^2}{t^2}\right)\frac{1}{\frac{1}{C_T+1}\frac{z}{t}+1}\left(1 + \frac{1}{\sqrt{1+\frac{4}{C_T}}}\frac{z}{t}\right)^m \int_0^{t/2} e^{-(t-\tau)}\left(\frac{\tau}{t}\right)^{m+1}\,d\tau\\
        &\geq Ct\exp\left(-\frac{1}{\frac{C_T}{4}+1}\frac{z^2}{t^2}\right) \left(\frac{z}{t} + 1\right)^{m-1}.
    \end{align*}
    for a constant $C$ depending on $C_T$, $m$, and dimension $n$ when $0<t\leq 1$ and $\frac{1}{\sqrt{C_T}}\frac{1}{\sqrt{C_T + 1}}\frac{z}{t}\geq 1$. When we take $m=0$ and apply it to \eqref{eq:contra_2}, we obtain
    \begin{align*}
        \rho(t,x)\geq \rho_g(t,x) + C\frac{1}{\frac{z}{t} +1}\exp\left(-\frac{1}{\frac{C_T}{4}+1}\frac{z^2}{t^2}\right)t.
    \end{align*}
    Since $C_T>0$, the density generated by $\mathcal{M}(f)$ is much bigger than $\rho_g$ when $\frac{z}{t}\gg 1$, so it can not satisfy \eqref{eq:macro_contradiction}. 
\end{proof}
When we further compute the bulk velocity $u$ and temperature $T$, we need to divide the weighted integral for $\rho u$ and $n\rho T$ by $\rho$. In this way, the instability of $\rho(t,x)$ can propagate to the other macroscopic fields making the problem much harder to analyze. To prevent such instability, one possible way is to suppress the $v_z$ component of the inflow boundary data. If we consider inflow boundary data $g(t,x_h,v)$ like $e^{-|v_h|^2 - |v_z|^\beta}$, or more generally $\phi(t,x_h,v_h) e^{-|v_z|^\beta}$ for some function $\phi$, the next proposition claims that we can take $0<\beta\leq 1$ to make the BGK solution have stable macroscopic fields near $(\rho_g, u_g, T_g)$.
\begin{proposition}\label{prop:suppressed_v_z}
    For $f_0(x,v) = 0$ and $g(t,x_h,v) = e^{-|v_h|^2 - |v_z|^\beta}$, if $0<\beta\leq 1$, then there exists a unique local solution $f$ satisfying the second condition of Definition \ref{def:BGK_sol_space}. In fact, we can take $g(t,x_h, v) = \phi(t,x_h,v_h)e^{-|v_z|^\beta}$ for $0<\beta\leq 1$ and any function $\phi(t,x_h, v_h)\geq 0$ satisfying
    \begin{align}\label{eq:suppressed_v_z_cond}
        \sup_{t\geq 0}\|(1+|v_h|^q)\phi(t,x_h, v_h)\|_{L^\infty_{x_h,v_h}}<\infty\quad \text{and}\quad \inf_{t\geq 0,\,0\leq \tau\leq t,\,x_h\in\mathbb{R}^{n-1}}\int_{\mathbb{R}^{n-1}}\phi(t-\tau, x_h-v_h\tau, v_h)\,dv_h>0
    \end{align}
    for some $q>n+1$.
\end{proposition}
\begin{proof}
    We follow the proof strategy used in Sections~\ref{sec:Macro_quan} and \ref{sec:sol_construction}. We first compute the macroscopic fields $(\rho_g, u_g, T_g)$ generated by the damped free transport of the boundary data $g$, and then verify the \textit{a priori} stability estimate for the solution's macroscopic fields near $(\rho_g, u_g, T_g)$.
    
    By the change of variables $t\mapsto t+\frac{z}{v_z}$ and $x_h\mapsto x_h + v_h\frac{z}{v_z}$, we have
    \begin{align*}
        \sup_{t\geq 0}\left\|(1+|v_h|^q)\phi(t-\frac{z}{v_z},x_h - v_h \frac{z}{v_z},v_h)\mathbf{1}_{v_z t\geq z}\right\|_{L^\infty_{x_h,v_h}} = \sup_{t\geq 0}\|(1+|v_h|^q)\phi(t,x_h, v_h)\|_{L^\infty_{x_h,v_h}}
    \end{align*}
    for all $z\geq 0$ and $v_z>0$. Therefore, applying the inequality of Perthame and Pulvirenti from \cite{PP1993} to $\phi(t-\frac{z}{v_z},x_h - v_h \frac{z}{v_z},v_h)\mathbf{1}_{t\geq \frac{z}{v_z}}$ and using the lower bound condition in \eqref{eq:suppressed_v_z_cond}, we have
    \begin{align}\label{eq:horizontal_int_bounds}
        \begin{split}
            &A_1\leq \int_{\mathbb{R}^{n-1}}\phi(t-t_b,x_h-v_ht_b,v_h)\,dv_h\leq A_2,\\
            &\int_{\mathbb{R}^{n-1}}|v_h|\phi(t-t_b,x_h-v_ht_b,v_h)\,dv_h\leq A_3,\\
            &A_4\leq \int_{\mathbb{R}^{n-1}}|v_h|^2 \phi(t-t_b,x_h-v_ht_b,v_h)\,dv_h\leq A_5
        \end{split}
    \end{align}
    for some uniform constants $A_i$, $1\leq i\leq 5$ for all $x\in\mathbb{R}^n_+$, $v_z>0$, and $t\geq t_b$.
    
    Next, we compute $(\rho_g, u_g, T_g)$. For this, we confine $0<t\leq 1$. By the definition of $g(t,x_h,v)$, we write the $\rho_g$ integral by
    \begin{align*}
        \int_{\mathbb{R}^n}e^{-t_b}g(t-t_b,x_h-v_ht_b,v)\mathbf{1}_{v_z\geq \frac{z}{t}}\,dv&=\int_{\mathbb{R}}e^{-\frac{z}{v_z}}e^{-v_z^\beta}\mathbf{1}_{v_z\geq \frac{z}{t}}\int_{\mathbb{R}^{n-1}}\phi(t-t_b,x_h-v_ht_b,v_h)\,dv_h dv_z.
    \end{align*}
    By \eqref{eq:horizontal_int_bounds} and Lemma \ref{lem:exp_integral}, we estimate $\rho_g$ by
    \begin{align*}
        C_1e^{-t}\left(\frac{z}{t}+1\right)^{1-\beta} e^{-\frac{z^\beta}{t^\beta}}\leq \int_{\mathbb{R}^n}e^{-t_b}g(t-t_b,x_h-v_ht_b,v)\mathbf{1}_{v_z\geq \frac{z}{t}}\,dv&\leq C_2\left(\frac{z}{t}+1\right)^{1-\beta} e^{-\frac{z^\beta}{t^\beta}}
    \end{align*}
    for some constants $C_1$ and $C_2$. By similar calculations, we have
    \begin{align*}
        &\int_{\mathbb{R}^n}e^{-t_b}|v_h|g(t-t_b,x_h-v_ht_b,v)\mathbf{1}_{v_z\geq \frac{z}{t}}\,dv\leq C_3'\left(\frac{z}{t}+1\right)^{1-\beta} e^{-\frac{z^\beta}{t^\beta}},\\
        &\int_{\mathbb{R}^n}e^{-t_b}\left|v_z-\frac{z}{t}\right|g(t-t_b,x_h-v_ht_b,v)\mathbf{1}_{v_z\geq \frac{z}{t}}\,dv\leq C_3' \left(\frac{z}{t}+1\right)^{2(1-\beta)} e^{-\frac{z^\beta}{t^\beta}}.
    \end{align*}
    for a constant $C_3'$. At this stage, we can state that there exist constants $C_i$ for $1\leq i\leq 3$ such that
    \begin{align}\label{eq:Appendix_macro_bound1}
        \begin{split}
            &C_1\left(\frac{z}{t} + 1\right)^{1-\beta}e^{-\left(\frac{z}{t} + 1\right)^\beta}\leq \rho_g(t,x) \leq C_2\left(\frac{z}{t} + 1\right)^{1-\beta}e^{-\frac{z^\beta}{t^\beta}},\\
            &\left|u_g(t,x)-\frac{z}{t}\hat{z}\right|\leq C_3\left(\frac{z}{t}+1\right)^{1-\beta}
        \end{split}
    \end{align}
    for $0<t\leq 1$. To compute $T_g$, we first bound the integral for $\left|v-\frac{z}{t}\hat{z}\right|^2$ weight as
    \begin{align*}
        &\int_{\mathbb{R}^n}e^{-t_b}\left|v-\frac{z}{t}\hat{z}\right|^2 g(t-t_b,x_h-v_ht_b,v)\mathbf{1}_{v_z\geq \frac{z}{t}}\,dv\\
        &=\int_{\mathbb{R}^n}e^{-t_b}\left(|v_h|^2 + \left|v_z-\frac{z}{t}\right|^2\right) g(t-t_b,x_h-v_ht_b,v)\mathbf{1}_{v_z\geq \frac{z}{t}}\,dv\\
        &\leq C_5'\left(\frac{z}{t}+1\right)^{3(1-\beta)}e^{-\frac{z^\beta}{t^\beta}}
    \end{align*}
    for a constant $C_5'$ employing \eqref{eq:horizontal_int_bounds} and Lemma~\ref{lem:exp_integral}. We need to use the weight $|v-u_g(t,x)|^2$ to compute $T_g$. For the upper bound, we easily get
    \begin{align}
        &\int_{\mathbb{R}^n}e^{-t_b}|v-u_g(t,x)|^2 g(t-t_b,x_h-v_ht_b,v)\mathbf{1}_{v_z\geq \frac{z}{t}}\,dv\notag\\
        &=\int_{\mathbb{R}^n}e^{-t_b}\left|\left(v-\frac{z}{t}\hat{z}\right)-\left(u_g(t,x)-\frac{z}{t}\hat{z}\right)\right|^2 g(t-t_b,x_h-v_ht_b,v)\mathbf{1}_{v_z\geq \frac{z}{t}}\,dv\notag\\
        &\leq 2\int_{\mathbb{R}^n}e^{-t_b}\left(\left|v-\frac{z}{t}\hat{z}\right|^2+\left|u_g(t,x)-\frac{z}{t}\hat{z}\right|^2\right) g(t-t_b,x_h-v_ht_b,v)\mathbf{1}_{v_z\geq \frac{z}{t}}\,dv\notag\\
        &\leq 2 (C_5' + C_3^2 C_2)\left(\frac{z}{t}+1\right)^{3(1-\beta)}e^{-\frac{z^\beta}{t^\beta}}\label{eq:Appendix_T_g_upper}.
    \end{align}
    For the lower bound, we take the $z$-component of the weight $|v-u_g(t,x)|^2$. As $u_g$ satisfies \eqref{eq:Appendix_macro_bound1}, if we take $v_z\geq \frac{z}{t} + (C_3+1)\left(\frac{z}{t}+1\right)^{1-\beta}$, we have
    \begin{align}
        &\int_{\mathbb{R}^n}e^{-t_b}\left|v-u_g(t,x)\right|^2 g(t-t_b,x_h-v_ht_b,v)\mathbf{1}_{v_z\geq \frac{z}{t}}\,dv\notag\\
        &\geq \int_{\mathbb{R}^n}e^{-t_b}\left|v_z-u_{g,z}(t,x)\right|^2 g(t-t_b,x_h-v_ht_b,v)\mathbf{1}_{v_z\geq \frac{z}{t} + (C_3+1)\left(\frac{z}{t}+1\right)^{1-\beta}}\,dv\notag\\
        &\geq e^{-t}\left(\frac{z}{t}+1\right)^{2(1-\beta)}\int_{\frac{z}{t} + (C_3+1)\left(\frac{z}{t}+1\right)^{1-\beta}}^\infty e^{-v_z^\beta}\int_{\mathbb{R}^{n-1}} \phi(t-t_b,x_h-v_ht_b,v_h)\,dv_h dv_z\notag\\
        &\geq C_4'\left(\frac{z}{t}+1\right)^{3(1-\beta)} e^{-\frac{z^\beta}{t^\beta}}\label{eq:Appendix_T_g_lower}
    \end{align}
    for a constant $C_4'$. In the final inequality, we assume $\frac{z}{t}\geq 1$ and use Lemma~\ref{lem:exp_integral} with Lemma~\ref{lem:beta_pert} to get
    \begin{align*}
        \left(\frac{z}{t} + (C_3+1)\left(\frac{z}{t}+1\right)^{1-\beta}\right)^\beta&\leq \frac{z^\beta}{t^\beta} + \beta (C_3+1)\left(\frac{z}{t}+1\right)^{1-\beta}\left(\frac{z}{t}\right)^{\beta-1}\\
        &\leq \frac{z^\beta}{t^\beta} + 2^{1-\beta}\beta (C_3+1).
    \end{align*}
    We divide \eqref{eq:Appendix_T_g_upper} and \eqref{eq:Appendix_T_g_lower} by $\rho_g$ in \eqref{eq:Appendix_macro_bound1} to get the upper and lower bounds of $T_g$. If $\frac{z}{t}\leq 1$, the $v_z$ integral is bounded below by a constant, so we again get the lower bound \eqref{eq:Appendix_T_g_lower}.
    
    Adding the bounds of $T_g$ to \eqref{eq:Appendix_macro_bound1}, we prove that there exist constants $C_i$ for $1\leq i\leq 5$ such that
    \begin{align*}
        &C_1\left(\frac{z}{t} + 1\right)^{1-\beta}e^{-\left(\frac{z}{t} + 1\right)^\beta}\leq \rho_g(t,x) \leq C_2\left(\frac{z}{t} + 1\right)^{1-\beta}e^{-\left(\frac{z}{t} + 1\right)^\beta},\\
        &\left|u_g(t,x)-\frac{z}{t}\hat{z}\right|\leq C_3\left(\frac{z}{t}+1\right)^{1-\beta},\\
        &C_4\left(\frac{z}{t} + 1\right)^{2(1-\beta)}\leq T_g(t,x) \leq C_5\left(\frac{z}{t} + 1\right)^{2(1-\beta)}
    \end{align*}
    for $0< t\leq 1$. Note that it is safe to add $1$ in the exponent of $\rho_g$ since $0<\beta\leq 1$. 
    
    Now, as in Proposition \ref{prop:M[f]_integral}, we want to bound the integral of local Maxwellian $\mathcal{M}(f)(\tau, x-v(t-\tau), v)$ generated by $(\rho_g, u_g, T_g)$. In general, we claim that if the macroscopic fields $(\rho, u, T)$ of $f$ satisfy
    \begin{align*}
        &C_1\left(\frac{z}{t} + 1\right)^{m_1}e^{-\alpha\left(\frac{z}{t} + 1\right)^\beta}\leq \rho(t,x) \leq C_2\left(\frac{z}{t} + 1\right)^{m_1}e^{-\alpha\left(\frac{z}{t} + 1\right)^\beta},\\
        &\left|u(t,x)-\frac{z}{t}\hat{z}\right|\leq C_3\left(\frac{z}{t}+1\right)^{1-\beta},\\
        &C_4\left(\frac{z}{t} + 1\right)^{2(1-\beta)}\leq T(t,x) \leq C_5\left(\frac{z}{t} + 1\right)^{2(1-\beta)},
    \end{align*}
    for some $0\leq \alpha\leq 1$ and $m_1\in\mathbb{R}$, we have
    \begin{align}\label{eq:appendix_int_res}
        \sup_{0<\tau\leq t}\frac{t}{\tau}\int_{\mathbb{R}^n} \left|v-\frac{z}{t}\hat{z}\right|^{m_2}\mathcal{M}(f)(\tau, x-v(t-\tau), v)\mathbf{1}_{z-v_z(t-\tau)> 0}\,dv\leq C_{m_1,m_2}(\frac{z}{t} + 1)^{m_1 + m_2(1-\beta)}e^{-\alpha(\frac{z}{t} + 1)^\beta}
    \end{align}
    for each $m_2\geq 0$. If $\tau = t$, then it is an integral of $\mathcal{M}(f)(t,x,v)$ and we directly get \eqref{eq:appendix_int_res}. Thus, we assume $0<\tau<t$. Applying the chain of change of variables used in the proof of Proposition~\ref{prop:M[f]_integral}, we can write the integral of the local Maxwellian as
    \begin{align*}
        &\frac{t}{\tau}\int_{\mathbb{R}^n} \left|v-\frac{z}{t}\hat{z}\right|^{m_2} \mathcal{M}(f)(\tau, x-v(t-\tau), v)\mathbf{1}_{z-v_z(t-\tau)> 0}\,dv\\
        &\leq C\frac{t}{t-\tau}\int_{\mathbb{R}^n} \left(|v_h|^{m_2} + \left(\frac{\tau}{t-\tau}|v_z|\right)^{m_2} \right)\frac{e^{-\alpha(\frac{z}{t}-v_z + 1)^\beta}}{(\frac{z}{t} - v_z + 1)^{n(1-\beta) - m_1}}\exp\left(-\frac{|v_h|^2 + |\frac{t}{t-\tau}v_z|^2}{4C_5(\frac{z}{t} - v_z + 1)^{2(1-\beta)}}\right)\mathbf{1}_{v_z< \frac{z}{t}}\,dv.
    \end{align*}
    We perform the $v_h$-integral first and get
    \begin{align}\label{eq:appendix_int}
        \begin{split}
            &\int_{\mathbb{R}^n} \left(|v_h|^{m_2} + \left(\frac{\tau}{t-\tau}|v_z|\right)^{m_2} \right)\frac{e^{-\alpha(\frac{z}{t}-v_z + 1)^\beta}}{(\frac{z}{t} - v_z + 1)^{n(1-\beta) - m_1}}\exp\left(-\frac{|v_h|^2 + |\frac{t}{t-\tau}v_z|^2}{4C_5(\frac{z}{t} - v_z + 1)^{2(1-\beta)}}\right)\mathbf{1}_{v_z< \frac{z}{t}}\,dv\\
            &\leq C\int_{\mathbb{R}} \left(\frac{z}{t} - v_z + 1\right)^{m_1 + (m_2-1)(1-\beta)}e^{-\alpha(\frac{z}{t}-v_z + 1)^\beta}\exp\left(-\frac{|\frac{t}{t-\tau}v_z|^2}{4C_5(\frac{z}{t} - v_z + 1)^{2(1-\beta)}}\right)\mathbf{1}_{v_z< \frac{z}{t}}\,dv_z\\
            &\quad + C\int_{\mathbb{R}} \left(\frac{\tau}{t-\tau}|v_z|\right)^{m_2}\left(\frac{z}{t} - v_z + 1\right)^{m_1 -(1-\beta)} e^{-\alpha(\frac{z}{t}-v_z + 1)^\beta}\exp\left(-\frac{|\frac{t}{t-\tau}v_z|^2}{4C_5(\frac{z}{t} - v_z + 1)^{2(1-\beta)}}\right)\mathbf{1}_{v_z< \frac{z}{t}}\,dv_z.
        \end{split}
    \end{align}
    The final integrals closely resemble the $v_z$-integrals over the regions $A_i$ in the proof of Proposition \ref{prop:M[f]_integral}, with $x_h=v_h(t-\tau)$. Therefore, we can apply the same proof technique for Proposition \ref{prop:M[f]_integral} to these integrals. For completeness, we provide the details.
    
    We first claim that
    \begin{align*}
        e^{-\alpha(\frac{z}{t}-v_z + 1)^\beta}\exp\left(-\frac{|\frac{t}{t-\tau}v_z|^2}{8C_5(\frac{z}{t} - v_z + 1)^{2(1-\beta)}}\right)&\leq Ce^{-\alpha(\frac{z}{t}-v_z + 1 + \frac{t}{t-\tau}|v_z|)^\beta}\\
        &\leq Ce^{-\alpha(\frac{z}{t} + 1)^\beta}
    \end{align*}
    for some constant $C$. Indeed, when $0<\beta<1$, we can choose $\beta'>0$ satisfying $\beta = \frac{\beta'}{1+\beta'}$ and directly apply Lemma~\ref{lem:concaten} to get the bound. If $\beta = 1$, by direct computation, we have
    \begin{align*}
        e^{-\alpha(\frac{z}{t}-v_z + 1)}\exp\left(-\frac{|\frac{t}{t-\tau}v_z|^2}{8C_5}\right) &= e^{-\alpha(\frac{z}{t} + 1)}\exp\left(-\frac{\left(\frac{t}{t-\tau}v_z - 4\alpha\frac{t-\tau}{t}C_5\right)^2}{8C_5} + 2\alpha^2\frac{(t-\tau)^2}{t^2}C_5\right)\\
        &\leq \exp\left(2\alpha^2 C_5\right) e^{-\alpha(\frac{z}{t} + 1)}.
    \end{align*}
    Therefore, for the final integrals in \eqref{eq:appendix_int}, it is enough to bound the integrals by
    \begin{align}
        &\int_{-\infty}^{\frac{z}{t}} (\frac{z}{t} - v_z + 1)^{m_1 + (m_2-1)(1-\beta)}\exp\left(-\frac{|\frac{t}{t-\tau}v_z|^2}{8C_5(\frac{z}{t} - v_z + 1)^{2(1-\beta)}}\right)\,dv_z\leq C\frac{t-\tau}{t}(\frac{z}{t} + 1)^{m_1 + m_2(1-\beta)},\label{eq:appendix_int_1}\\
        &\int_{-\infty}^{\frac{z}{t}} \left(\frac{\tau}{t-\tau}|v_z|\right)^{m_2}\left(\frac{z}{t} - v_z + 1\right)^{m_1 -(1-\beta)} \exp\left(-\frac{|\frac{t}{t-\tau}v_z|^2}{8C_5(\frac{z}{t} - v_z + 1)^{2(1-\beta)}}\right)\,dv_z\leq C\frac{t-\tau}{t}(\frac{z}{t} + 1)^{m_1 + m_2(1-\beta)}\label{eq:appendix_int_2}
    \end{align}
    to get the result \eqref{eq:appendix_int_res}. For each integral, we decompose the integral domain into three parts: $\frac{1}{2}\frac{z}{t}\leq v_z<\frac{z}{t}$, $-2\left(\frac{z}{t} + 1\right)\leq v_z< \frac{1}{2}\frac{z}{t}$, and $v_z\leq -2\left(\frac{z}{t} + 1\right)$ and estimate the integral. For notational convenience, we set $C_T = 8C_5$ in the exponent of the integrals.
    
    When $\frac{1}{2}\frac{z}{t}\leq v_z<\frac{z}{t}$, we have
    \begin{align*}
        &\int_{\frac{1}{2}\frac{z}{t}}^{\frac{z}{t}} (\frac{z}{t} - v_z + 1)^{m_1 + (m_2-1)(1-\beta)}\exp\left(-\frac{|\frac{t}{t-\tau}v_z|^2}{C_T(\frac{z}{t} - v_z + 1)^{2(1-\beta)}}\right)\,dv_z\\
        &\leq \int_{\frac{1}{2}\frac{z}{t}}^{\frac{z}{t}} (\frac{z}{t} - v_z + 1)^{m_1 + (m_2-1)(1-\beta)}\exp\left(-\frac{|\frac{t}{t-\tau}v_z|^2}{C_T(\frac{1}{2}\frac{z}{t} + 1)^{2(1-\beta)}}\right)\,dv_z\\
        &=\frac{t-\tau}{t}\int_{\frac{t}{t-\tau}\frac{1}{2}\frac{z}{t}}^{\frac{t}{t-\tau}\frac{z}{t}} (\frac{z}{t} - \frac{t-\tau}{t}v_z + 1)^{m_1 + (m_2-1)(1-\beta)}\exp\left(-\frac{|v_z|^2}{C_T(\frac{1}{2}\frac{z}{t} + 1)^{2(1-\beta)}}\right)\,dv_z,
        \intertext{and}
        &\int_{\frac{1}{2}\frac{z}{t}}^{\frac{z}{t}} \left(\frac{\tau}{t-\tau}|v_z|\right)^{m_2}\left(\frac{z}{t} - v_z + 1\right)^{m_1 -(1-\beta)}\exp\left(-\frac{|\frac{t}{t-\tau}v_z|^2}{C_T(\frac{z}{t} - v_z + 1)^{2(1-\beta)}}\right)\,dv_z\\
        &\leq \int_{\frac{1}{2}\frac{z}{t}}^{\frac{z}{t}} \left(\frac{\tau}{t-\tau}|v_z|\right)^{m_2}\left(\frac{z}{t} - v_z + 1\right)^{m_1 -(1-\beta)}\exp\left(-\frac{|\frac{t}{t-\tau}v_z|^2}{C_T(\frac{1}{2}\frac{z}{t} + 1)^{2(1-\beta)}}\right)\,dv_z\\
        &=\frac{t-\tau}{t}\int_{\frac{t}{t-\tau}\frac{1}{2}\frac{z}{t}}^{\frac{t}{t-\tau}\frac{z}{t}} \left(\frac{\tau}{t}|v_z|\right)^{m_2}\left(\frac{z}{t} - \frac{t-\tau}{t}v_z + 1\right)^{m_1 -(1-\beta)}\exp\left(-\frac{|v_z|^2}{C_T(\frac{1}{2}\frac{z}{t} + 1)^{2(1-\beta)}}\right)\,dv_z\\
        &\leq \frac{t-\tau}{t}\int_{\frac{t}{t-\tau}\frac{1}{2}\frac{z}{t}}^{\frac{t}{t-\tau}\frac{z}{t}} \left(\frac{z}{t} - \frac{t-\tau}{t}v_z + 1\right)^{m_1 -(1-\beta)}v_z^{m_2}\exp\left(-\frac{|v_z|^2}{C_T(\frac{1}{2}\frac{z}{t} + 1)^{2(1-\beta)}}\right)\,dv_z.
    \end{align*}
    Regardless of the polynomial terms, due to Lemma~\ref{lem:exp_integral} and $0<\beta\leq 1$, we have exponential decay about $\frac{z}{t}$ and $\frac{t}{t-\tau}$ as the result of the integral. Therefore, we have
    \begin{align}\label{eq:appendix_int_3}
        \begin{split}
            &\int_{\frac{1}{2}\frac{z}{t}}^{\frac{z}{t}} (\frac{z}{t} - v_z + 1)^{m_1 + (m_2-1)(1-\beta)}\exp\left(-\frac{|\frac{t}{t-\tau}v_z|^2}{C_T(\frac{z}{t} - v_z + 1)^{2(1-\beta)}}\right)\,dv_z\leq C\frac{t-\tau}{t}\left(\frac{z}{t} + 1\right)^{m_1 + m_2(1-\beta)}\\
            &\int_{\frac{1}{2}\frac{z}{t}}^{\frac{z}{t}} \left(\frac{\tau}{t-\tau}|v_z|\right)^{m_2}\left(\frac{z}{t} - v_z + 1\right)^{m_1 -(1-\beta)}\exp\left(-\frac{|\frac{t}{t-\tau}v_z|^2}{C_T(\frac{z}{t} - v_z + 1)^{2(1-\beta)}}\right)\,dv_z\leq C\frac{t-\tau}{t}\left(\frac{z}{t} + 1\right)^{m_1 + m_2(1-\beta)}.
        \end{split}
    \end{align}
    
    If $-2\left(\frac{z}{t} + 1\right)\leq v_z<\frac{1}{2}\frac{z}{t}$, then we have
    \begin{align*}
        \frac{1}{2}\frac{z}{t} + \frac{1}{2}\leq \frac{z}{t} - v_z + 1\leq 3\frac{z}{t} + 3.
    \end{align*}
    Therefore, we obtain
    \begin{align}\label{eq:appendix_int_4}
        \begin{split}
            &\int_{-2\left(\frac{z}{t}+1\right)}^{\frac{1}{2}\frac{z}{t}} \left(\frac{z}{t} - v_z + 1\right)^{m_1 + (m_2-1)(1-\beta)}\exp\left(-\frac{|\frac{t}{t-\tau}v_z|^2}{C_T(\frac{z}{t} - v_z + 1)^{2(1-\beta)}}\right)\,dv_z\\
            &\leq C\left(\frac{z}{t} + 1\right)^{m_1 + (m_2-1)(1-\beta)}\int_{-\infty}^\infty \exp\left(-\frac{|\frac{t}{t-\tau}v_z|^2}{C_T(3\frac{z}{t} + 3)^{2(1-\beta)}}\right)\,dv_z\\
            &\leq C\frac{t-\tau}{t}\left(\frac{z}{t} + 1\right)^{m_1 + m_2(1-\beta)}.
        \end{split}
    \end{align}
    For the second integral, by a similar reason,
    \begin{align}\label{eq:appendix_int_5}
        \begin{split}
            &\int_{-2\left(\frac{z}{t}+1\right)}^{\frac{1}{2}\frac{z}{t}} \left(\frac{\tau}{t-\tau}|v_z|\right)^{m_2}\left(\frac{z}{t} - v_z + 1\right)^{m_1 -(1-\beta)} \exp\left(-\frac{|\frac{t}{t-\tau}v_z|^2}{C_T(\frac{z}{t} - v_z + 1)^{2(1-\beta)}}\right)\,dv_z\\
            &\leq C\left(\frac{z}{t} + 1\right)^{m_1 -(1-\beta)}\int_{-\infty}^\infty \left(\frac{\tau}{t-\tau}|v_z|\right)^{m_2}\exp\left(-\frac{|\frac{t}{t-\tau}v_z|^2}{C_T(3\frac{z}{t} + 3)^{2(1-\beta)}}\right)\,dv_z\\
            &\leq C\frac{t-\tau}{t}\left(\frac{z}{t} + 1\right)^{m_1 + m_2(1-\beta)}.
        \end{split}
    \end{align}
    Hence, we get the desired result for each \eqref{eq:appendix_int_1} and \eqref{eq:appendix_int_2} in this case.
    
    Finally, suppose $v_z\leq -2\left(\frac{z}{t} + 1\right)$. In this case, we have $|v_z|\leq \frac{z}{t}-v_z + 1\leq 2|v_z|$. Using these bounds, we split the exponential function into half and get
    \begin{align*}
        &\int_{-\infty}^{-2\left(\frac{z}{t} + 1\right)} \left(\frac{z}{t} - v_z + 1\right)^{m_1 + (m_2-1)(1-\beta)}\exp\left(-\frac{|\frac{t}{t-\tau}v_z|^2}{C_T(\frac{z}{t} - v_z + 1)^{2(1-\beta)}}\right)\,dv_z\\
        &\leq C\int_{-\infty}^{-2\left(\frac{z}{t} + 1\right)} |v_z|^{m_1 + (m_2-1)(1-\beta)}\exp\left(-\frac{\left(\frac{t}{t-\tau}|v_z|^{\beta}\right)^2}{2^{2(1-\beta) }C_T}\right)\,dv_z\\
        &\leq C\exp\left(-\frac{1}{2}\frac{\left(\frac{t}{t-\tau}2^\beta \left(\frac{z}{t} + 1\right)^{\beta}\right)^2}{2^{2(1-\beta)}C_T}\right)\int_{-\infty}^{-2} |v_z|^{m_1 + (m_2-1)(1-\beta)}\exp\left(-\frac{1}{2}\frac{\left(\frac{t}{t-\tau}|v_z|^{\beta}\right)^2}{2^{2(1-\beta)}C_T}\right)\,dv_z\\
        &\leq C\exp\left(-\frac{1}{2}\frac{\left(\frac{t}{t-\tau}2^\beta \left(\frac{z}{t} + 1\right)^{\beta}\right)^2}{2^{2(1-\beta)}C_T}\right).
    \end{align*}
    By a similar calculation, we have
    \begin{align*}
        &\int_{-\infty}^{-2\left(\frac{z}{t}+1\right)} \left(\frac{\tau}{t-\tau}|v_z|\right)^{m_2}\left(\frac{z}{t} - v_z + 1\right)^{m_1 -(1-\beta)} \exp\left(-\frac{|\frac{t}{t-\tau}v_z|^2}{C_T(\frac{z}{t} - v_z + 1)^{2(1-\beta)}}\right)\,dv_z\\
        &\leq C\left(\frac{\tau}{t-\tau}\right)^{m_2}\exp\left(-\frac{1}{2}\frac{\left(\frac{t}{t-\tau}2^\beta \left(\frac{z}{t} + 1\right)^{\beta}\right)^2}{2^{2(1-\beta)}C_T}\right).
    \end{align*}
    Since the exponential decay can suppress $\frac{z}{t}$ and the time singular term $\frac{1}{t-\tau}$ for arbitrary order, we get
    \begin{align}\label{eq:appendix_int_6}
        \begin{split}
            &\int_{-\infty}^{-2\left(\frac{z}{t} + 1\right)} \left(\frac{z}{t} - v_z + 1\right)^{m_1 + (m_2-1)(1-\beta)}\exp\left(-\frac{|\frac{t}{t-\tau}v_z|^2}{C_T(\frac{z}{t} - v_z + 1)^{2(1-\beta)}}\right)\,dv_z\leq C\frac{t-\tau}{t}(\frac{z}{t} + 1)^{m_1 + m_2(1-\beta)},\\
            &\int_{-\infty}^{-2\left(\frac{z}{t}+1\right)} \left(\frac{\tau}{t-\tau}|v_z|\right)^{m_2}\left(\frac{z}{t} - v_z + 1\right)^{m_1 -(1-\beta)} \exp\left(-\frac{|\frac{t}{t-\tau}v_z|^2}{C_T(\frac{z}{t} - v_z + 1)^{2(1-\beta)}}\right)\,dv_z\leq C\frac{t-\tau}{t}(\frac{z}{t} + 1)^{m_1 + m_2(1-\beta)}.
        \end{split}
    \end{align}
    Combining the integral estimates \eqref{eq:appendix_int_3}, \eqref{eq:appendix_int_4}, \eqref{eq:appendix_int_5}, and \eqref{eq:appendix_int_6}, we conclude \eqref{eq:appendix_int_1} and \eqref{eq:appendix_int_2}, and so \eqref{eq:appendix_int_res}.
    
    Now, we have versions of Proposition~\ref{prop:M[f]_integral} and Corollary~\ref{prop:bound} (for $\alpha = 0$ and general $m_1\in\mathbb{R}$) for our boundary data. At this stage, we define a solution space $\Omega'_{\gamma, t_0}$ of the BGK model for our inflow boundary data as follows: For $\gamma\geq 0$, $t_0>0$, and given inflow boundary data $g(t,x_h,v)$, $\Omega'_{\gamma, t_0}$ is the collection of functions $f(t,x,v)\geq 0$ for $0\leq t\leq t_0$ such that
    \begin{enumerate}
        \item it satisfies
        \begin{align*}
            \esssup_{0< t\leq t_0}\esssup_{x\in\mathbb{R}^n_+}\int_{v\in \mathbb{R}^n}\left(1+\left(\frac{z}{t}\right)^{2\gamma} + \left(v-\frac{z}{t}\hat{z}\right)^2\right)f(t,x,v)\,dv<\infty
        \end{align*}
        and $f(0,x,v) = 0$.
        \item For the macroscopic fields $(\rho, u, T)$ of $f$, there exists a constant $C_u>0$ such that
        \begin{align*}
            &|\rho-\rho_g|(t,x)\leq \frac{1}{2}\rho_g(t,x),\\
            &\left|u-\frac{z}{t}\hat{z}\right|(t,x)\leq C_u\sqrt{T_g(t,x)},\\
            &|T-T_g|(t,x)\leq \frac{1}{2}T_g(t,x)
        \end{align*}
        for almost every $0<t\leq t_0$ and $x\in\mathbb{R}^n_+$.
    \end{enumerate}
    Starting from \eqref{eq:appendix_int_res}, we can check that the solution of the BGK model with the inflow boundary data satisfies the conditions for $\Omega'_{1-\beta, t_0^*}$ for some $t_0^*>0$ and $C_u = \frac{C_3+1}{\sqrt{C_4}}$, modifying the proof of Proposition~\ref{prop:apriori}. Furthermore, for $f_1,f_2\in \Omega'_{1-\beta, t_0^*}$, we can reproduce Proposition~\ref{prop:maxwell} by
    \begin{align*}
        &\esssup_{x\in\mathbb{R}^n_+}\int_{\mathbb{R}^n}\left(1+\left(\frac{z}{t}\right)^{2(1-\beta)} + \left(v-\frac{z}{t}\hat{z}\right)^2\right)\left|\mathcal{M}(f_1)-\mathcal{M}(f_2)\right|(\tau, x-v(t-\tau), v)\mathbf{1}_{z-v_z(t-\tau)>0}\,dv\\
        &\leq C\esssup_{x\in\mathbb{R}^n_+}\int_{\mathbb{R}^n}\left(1+\left(\frac{z}{\tau}\right)^{2(1-\beta)} + \left(v-\frac{z}{\tau}\hat{z}\right)^2\right)|f_1-f_2|(\tau,x,v)\,dv
    \end{align*}
    for almost every $0<\tau\leq t$ and all $0<t\leq t_0^*$ from the pointwise estimate of $(\rho, u, T)$ in the second condition. Finally, following the contraction argument in Theorem~\ref{thm:BGK_sol}, we get a unique solution $f$ in $\Omega'_{1-\beta, t_0}$ for some $t_0$, which satisfies the second condition of Definition \ref{def:BGK_sol_space}.
\end{proof}
In the proof of Theorem \ref{thm:main_BGK}, we choose the tangential exponential decay $e^{-|v_h|^\beta}$ and the mitigated cutoff function $\mathbf{1}_{|v_h|\geq (|x_h| + v_z)^{\frac{1}{1+\beta}}}$. This choice removes the $v_z$-dependence from the exponential weight while retaining its effect through the interior cutoff. Because of this cutoff, the resulting expression remains integrable with respect to $v_z$.\\

\noindent{\bf Data availability:} No data was used for the research described in the article.
\newline

\noindent{\bf Conflict of interest:} The authors declare that they have no conflict of interest.\newline

\noindent{\bf Disclaimer on the use of AI tools} During the preparation of this manuscript, the authors used ChatGPT for writing assistance and to help check the correctness of the mathematical arguments. The authors reviewed the AI-assisted suggestions and take full responsibility for the content of the manuscript.\newline

\noindent{\bf Acknowledgement} D. Lee and S. Park are supported by the National Research Foundation of Korea (NRF) grant funded by the Korea government(MSIT)(No.RS-2023-00212304 and No.RS-2023-00219980). S. Park is also supported by the Basic Science Research Institute Fund(BRSI)(No.RS-2021-NR060139). S.-J. Son is supported by National Research Foundation of Korea(NRF) grants funded by the Korea government(MSIT and MOE)(RS-2023-00210484 and RS-2025-25419038). Yun was supported by the National Research Foundation of Korea(NRF) grant funded by the Korea government(MSIT)(No.RS-2023-NR076676).

\nocite{}
\bibliographystyle{abbrv}
\bibliography{ref.bib}

\end{document}

%% file: boundary_geometry.tex
\begin{figure}[t]
    \centering
    \resizebox{0.7\textwidth}{!}{
\begin{tikzpicture}[
    x=1.1cm,
    y=1.1cm,
    >=Stealth,
    axis/.style={
        ->,
        line width=.9pt
    },
    boundary/.style={
        dash pattern=on 4pt off 3pt,
        line width=.6pt
    },
    main line/.style={
        line width=.7pt
    },
    point/.style={
        circle,
        fill=black,
        inner sep=1.8pt
    },
    every node/.style={
        font=\small
    }
]

\coordinate (O)     at (0,0);
\coordinate (X)     at (1.7,1.8);
\coordinate (XV)    at (0.5,-1);

\begin{scope}
    \clip (-3.2,-1.5) rectangle (5,0);

    \foreach \x in {
        -4.5,-3.9,...,10.5
    }{
        \draw[
            gray!45,
            line width=.45pt
        ]
        (\x,-3.2) -- ++(-5,3.2);
    }
\end{scope}

\draw[axis] (-3.2,0) -- (5,0);
\draw[axis] (0,-1.5) -- (0,2.5);
\node[xshift=2.5mm] at (0,2.25) {$z$};
\node[yshift=-2.5mm] at (5, 0) {$x_h$};
\node[circle, fill=black, inner sep=0pt, minimum size=4pt] at (0.92857143, 0) {};
\node[above right=1pt] at (O) {$0$};

\node[xshift=16mm, yshift=-1.5mm] at (X) {$x = (x_h, z)$ at time $t$};
\node[xshift=9mm, yshift=-2mm] at (1.7+0.3,1.8+0.7) {$v = (v_h, v_z)$};
\node[circle, fill=black, inner sep=0pt, minimum size=4pt] at (X) {};

\node[yshift = -1.5mm, xshift = 6mm] at (XV){$x-vt$};

\draw[->, thick] (1.7,1.8) -- (1.7+0.3,1.8+0.7);

\node[yshift=2.5mm, xshift = 17mm] at (0.92857143, 0) {$x-v t_b$ at time $t-t_b$};

\node at (XV) {$\times$};

\draw[main line] (X) -- (0.92857143, 0);
\draw[boundary] (0.92857143, 0) -- (XV);

\node[anchor=west] at (-2.7,1.5) {Initially vacuum};

\node[anchor=west] at (-3.5,-0.75) {Filled \& Free transport};

\node[anchor=west] at (-2.7,.25) {inflow at $z=0$};

\end{tikzpicture}
}

\caption{Geometry of the half-space boundary configuration. 
        Particles enter the domain through the boundary plane and propagate along characteristic trajectories. For the point $x$ in the interior, the backward characteristic intersects the boundary at $x_b = x - v t_b$.}
    \label{fig:boundary_geometry}
\end{figure}
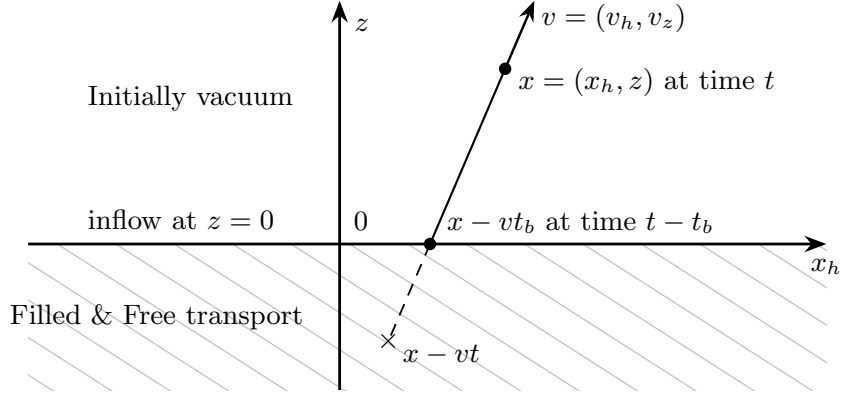